\documentclass[12pt]{amsart}
\usepackage{amssymb,latexsym,mathtools}
\usepackage[margin=1.25in]{geometry}
\usepackage{mathrsfs}
\usepackage{mathtools}
\usepackage{graphicx}
\usepackage{esint}
\usepackage{braket}
\usepackage{url, hyperref}
\usepackage{thmtools}
\usepackage{tikz}
\newtheorem{theorem}{Theorem}
\newtheorem*{theorem*}{Theorem}
\newtheorem{lemma}{Lemma}
\newtheorem{definition}{Definition}

\newtheorem{prop}{Proposition}

\newtheorem{remark}{Remark}
\newtheorem{corollary}{Corollary}
\newtheorem{assumption}{Assumption}

\begin{document}

\title[Comparison for the Finsler Infinity Laplacian]{Comparison Principles for the Finsler Infinity Laplacian with Applications to Absolutely Minimizing Lipschitz Extensions}
\author{Peter S.\ Morfe}
\address{%
  Max Planck Institute for Mathematics in the Sciences,
  Inselstrasse 22-26,
  04103 Leipzig, Germany
}
\email{morfe@mis.mpg.de}
\begin{abstract} This paper proves comparison principles for elliptic PDE involving the Finsler infinity Laplacian, a second-order differential operator with discontinuities in the gradient variable arising in $L^{\infty}$-variational problems and tug-of-war games.  The core of the paper consists in proving generalized cone comparison principles.  Among other consequences, these results imply that, for any Finsler norm $\varphi$ in $\mathbb{R}^{d}$, a function $u$ is a $\varphi$-absolutely minimizing Lipschitz extension if and only if it is a viscosity solution of the $\varphi$-infinity Laplace equation, settling a longstanding question in the $L^{\infty}$-calculus of variations. The proofs combine new geometric constructions with classical notions from convex analysis. \end{abstract}

\maketitle

\section{Introduction}

The aim of this paper is to prove comparison principles for elliptic PDE involving the so-called Finsler infinity Laplacian.  Recall that a function $\varphi : \mathbb{R}^{d} \to [0,+\infty)$ is called a \emph{Finsler norm} if it is convex, positively one-homogeneous, and positive away from zero.  The Finsler infinity Laplacian associated with $\varphi$ is then the (multivalued) differential operator defined by 
	\begin{equation*}
		\langle D^{2}u \cdot \partial \varphi^{*}(Du), \partial \varphi^{*}(Du) \rangle = \left\{ \langle D^{2}u \cdot q, q \rangle \, \mid \, q \in \partial \varphi^{*}(Du) \right\}.
	\end{equation*}
Here $\varphi^{*}$ denotes the dual norm of $\varphi$ and $\partial \varphi^{*}$, its subdifferential.  
PDE involving the Finsler infinity Laplacian arise in a number of contexts, particularly the study of $L^{\infty}$-variational problems and stochastic tug-of-war games.

Delaying the discussion of variational problems to Sections \ref{S: lipschitz extension problem} and \ref{S: eigenvalue intro}, a typical example relevant to tug-of-war games is the following Dirichlet problem with zeroth-order term and variable diffusivity:
	\begin{equation} \label{E: naive}
		u - c(x) \langle D^{2}u \cdot \partial \varphi^{*}(Du), \partial \varphi^{*}(Du) \rangle = f(x) \quad \text{in} \, \, U, \quad u = g \quad \text{on} \, \, \partial U.
	\end{equation}
This PDE arises in the diffusive scaling limit of a two-player tug-of-war game in which, at each step of the game, the players choose their next jump from the ball $\{\varphi < 1 \}$; see Section \ref{S: intro tug-of-war} for more details.

Since the subdifferential $\partial \varphi^{*}$ is multivalued in general, the notion of equation needs to be relaxed somehow.  In the viscosity solutions literature, the standard approach involves treating the Finsler infinity Laplacian as a discontinuous operator, and hence reformulating the equation via two differential inequalities
	\begin{equation} \label{E: finsler dirichlet with zeroth order}
		u - c(x) G_{\varphi}^{*}(Du,D^{2}u) \leq f(x), \quad u - c(x) G_{*}^{\varphi}(Du,D^{2}u) \geq f(x),
	\end{equation}
where $G_{\varphi}^{*}$ and $G_{*}^{\varphi}$ are the respective upper and lower semicontinuous envelopes
	\begin{align}
		G_{\varphi}^{*}(p,X) &= \max \left\{ \langle X q, q \rangle \, \mid \, q \in \partial \varphi^{*}(p) \right\}, \label{E: upper def operator} \\
		G_{*}^{\varphi}(p,X) &= \min \left\{ \langle X q, q \rangle \, \mid \, q \in \partial \varphi^{*}(p) \right\}. \label{E: lower def operator}
	\end{align}
Note that the coefficients of the PDE are ``continuous," that is, the envelopes $G_{\varphi}^{*}$ and $G_{*}^{\varphi}$ are equal as functions in $(\mathbb{R}^{d} \setminus \{0\}) \times \mathcal{S}_{d}$, if and only if $\partial \varphi^{*}$ is single-valued away from the origin.\footnote{It is necessary to treat the origin separately since, in general, $\partial \varphi^{*}(0) = \{\varphi \leq 1\}$ so $\partial \varphi^{*}$ is always multivalued at the origin.  For this same reason, the envelopes $G_{\varphi}^{*}$ and $G_{*}^{\varphi}$ always differ on $\{0\} \times \mathcal{S}^{d}$.} 

While the pair of inequalities \eqref{E: finsler dirichlet with zeroth order} constitute a natural replacement for the na\"{i}ve equation \eqref{E: naive}, standard results on viscosity solutions of elliptic equations do not apply directly to such pairs.

To date, comparison principles for PDE involving the (discontinuous) Finsler infinity Laplacian have been proved, for particular choices of Finsler norm $\varphi$ or under restrictions on the dimension $d$, by Crandall, Gunnarson, and Wang \cite{crandall_gunnarsson_wang} and Morfe and Souganidis \cite{m_souganidis}.  In this paper, the question is treated in full generality.

Where the Dirichlet problem \eqref{E: finsler dirichlet with zeroth order} is concerned, the results of this paper imply a comparison principle, along with uniqueness and H\"{o}lder regularity of solutions.

	\begin{theorem} \label{T: comparison dirichlet} Let $\varphi$ be a Finsler norm in $\mathbb{R}^{d}$ and assume that $c: \mathbb{R}^{d} \to [0,+\infty)$ is such that $\sqrt{c}$ is uniformly Lipschitz continuous.  Given any bounded open set $U \subseteq \mathbb{R}^{d}$ and any bounded continuous function $f : U \to \mathbb{R}$, if $u \in USC(\overline{U})$ and $v \in LSC(\overline{U})$ are such that
		\begin{gather*}
			u - c(x) G_{\varphi}^{*}(Du,D^{2}u) \leq f(x) \quad \text{in} \, \, U, \quad v - c(x) G_{*}^{\varphi}(Dv,D^{2}v) \geq f(x) \quad \text{in} \, \, U,  \quad \text{and} \\
			u \leq v \quad \text{on} \, \, \partial U,
		\end{gather*}
	then
		\begin{equation*}
			u \leq v \quad \text{in} \, \, U.
		\end{equation*}
	Further, for any $g \in C(\partial U)$, there is a unique  $u \in C(\overline{U})$ satisfying \eqref{E: finsler dirichlet with zeroth order} with $u = g$ on $\partial U$, and $u \in C^{\alpha}_{\text{loc}}(U)$ for any $\alpha \in (0,1)$.
	\end{theorem}
	
See Section \ref{S: dirichlet existence} below for the proof.

At a technical level, Theorem \ref{T: comparison dirichlet} and the other comparison theorems in this work follow from a family of \emph{cone comparison principles}.  Roughly speaking, these cone comparison results permit meaningful information to be gleaned from local maxima of functions
	\begin{equation*}
		u(x) - h(\varphi(x)),
	\end{equation*}
where $u$ is a subsolution of an elliptic differential inequality involving $G_{\varphi}^{*}$ and $h$ is a smooth nondecreasing function, as well as local maxima of functions
	\begin{equation*}
		u(x) - v(y) - h(\varphi(x - y)),
	\end{equation*}
where $u$ and $h$ are as before and $v$ is a supersolution of an inequality involving $G_{*}^{\varphi}$.  It is worth noting that when $\varphi$ is $C^{2}$ away from the origin, quite a lot can be done utilizing the fact that $\varphi$ is $\varphi$-infinity harmonic and $h(\varphi)$ is a $C^{2}$ test function; this is an observation of \cite{m_souganidis}.  In this paper, techniques are developed to overcome the fact that $h(\varphi)$ will be far from smooth in general. 

It should be noted that the case when $h \equiv \text{Id}$ is already of considerable interest.  As explained next, that case corresponds to a fundamental question in the theory of $L^{\infty}$-variational problems.
	
\subsection{Minimal Lipschitz Extensions} \label{S: lipschitz extension problem}  Much of the interest in the Finsler infinity Laplacian comes from its connection with the minimal Lipschitz extension problem, which will be reviewed now before stating one of this paper's main results.  

The problem can be described as follows: given a bounded open set $U \subseteq \mathbb{R}^{d}$ and a uniformly Lipschitz function $g : \partial U \to \mathbb{R}$, find a function $u$, which equals $g$ on $\partial U$, so as to minimize the $\varphi$-Lipschitz seminorm:
	\begin{equation} \label{E: minimal lipschitz extension}
		\min \left\{ \text{Lip}_{\varphi}(v; \overline{U}) \, \mid \, v : \overline{U} \to \mathbb{R}, \, \, v = g \, \, \text{on} \, \, \partial U \right\}.
	\end{equation}
Here the $\varphi$-Lipschitz seminorm is defined using $\varphi$ as a (nonsymmetric) distance:
	\begin{equation*}
		\text{Lip}_{\varphi}(v; \overline{U}) = \inf \left\{ C > 0 \, \mid \, v(x) - v(y) \leq C \varphi(x - y) \, \, \text{for all} \, \, x, y \in \overline{U} \right\}.
	\end{equation*}

Unfortunately, the problem \eqref{E: minimal lipschitz extension} has many minimizers in general.  Nonetheless, it turns out that there is a distinguished minimizer, the so-called \emph{absolutely minimizing Lipschitz extension} ($\varphi$-AMLE).  Precisely, this is a function $u \in W^{1,\infty}(U)$ such that, for any open subset $V \subseteq U$,
	\begin{equation} \label{E: lipschitz min}
		\text{Lip}_{\varphi}(u; \overline{V}) = \min \left\{  \text{Lip}_{\varphi}(v; \overline{V}) \, \mid \, v : \overline{V} \to \mathbb{R}, \, \, v = u \, \, \text{on} \, \, \partial V \right\}.
	\end{equation}
In other words, $u$ is absolutely minimizing if it solves the minimal Lipschitz extension problem on every subdomain of $U$.  For a given boundary condition $g$, \eqref{E: minimal lipschitz extension} has a unique $\varphi$-AMLE, though this is far from obvious \emph{a priori}.

It turns out that the $\varphi$-AMLE is characterized by another property, namely,\footnote{Indeed, when $V$ is convex, the identity $\|\varphi^{*}(Dv)\|_{L^{\infty}(V)} = \text{Lip}_{\varphi}(v;\overline{V})$ holds according to a standard argument.  The $\varphi$-AMLE minimizes both quantities even if $V$ is not convex.}
	\begin{equation} \label{E: linfty min}
		\|\varphi^{*}(Du)\|_{L^{\infty}(V)} = \min \left\{ \|\varphi^{*}(Dv)\|_{L^{\infty}(V)} \, \mid \, v : \overline{V} \to \mathbb{R}, \, \, v = u \, \, \text{on} \, \, \partial V \right\}.
	\end{equation}
Since this optimization problem is reminiscent of those appearing in the standard calculus of variations, yet with the integrability exponent $p \to +\infty$, the study of such problems is accordingly referred to as the \emph{$L^{\infty}$-calculus of variations.}

At the time of writing, it is well-known that the $\varphi$-AMLE property is characterized by another condition, referred to as \emph{comparison with cones}.  Toward that end, it will be useful to give definitions:

	\begin{definition} \label{D: cone comparison} Given an open set $U \subseteq \mathbb{R}^{d}$, an upper semicontinuous function $u \in USC(U)$ is said to satisfy \emph{comparison with cones from above} if, for any open $V \subset \subset U$, any $v \in \mathbb{R}^{d} \setminus V$, and any $\lambda > 0$,
		\begin{equation*}
			\max \left\{u(x) - \lambda \varphi(x - v) \, \mid \, x \in \overline{V} \right\} = \max \left\{ u(x) - \lambda \varphi(x - v) \, \mid \, x \in \partial V \right\}.
		\end{equation*}
	In this case, one writes $u \in CCA_{\varphi}(U)$.
		
Analogously, a lower semicontinuous function $v \in LSC(U)$ is said to satisfy \emph{comparison with cones from below} if, for any open $V \subset \subset U$, any $v \in \mathbb{R}^{d} \setminus V$, and any $\lambda > 0$, 
	\begin{equation*}
		\min \left\{v(x) + \lambda \varphi(v - x) \, \mid \, x \in \overline{V} \right\} = \min \left\{ v(x) + \lambda \varphi(v - x) \, \mid \, x \in \partial V \right\}.
	\end{equation*}
In this case, one writes $v \in CCB_{\varphi}(U)$. \end{definition}

If $\varphi$ equals the Euclidean norm, then a result of Crandall, Evans, and Gariepy \cite{crandall_evans_gariepy} asserts that the AMLE is characterized by comparison with cones.  This was extended to arbitrary symmetric norms by Aronsson, Crandall, and Juutinen \cite{aronsson_crandall_juutinen}, while the general nonsymmetric case follows from results of Armstrong, Crandall, Julin, and Smart \cite{armstrong_crandall_julin_smart}.  The precise equivalences are recalled next:

	\begin{theorem*}[Theorem 4.1 in \cite{aronsson_crandall_juutinen} and Theorem 4.8 in \cite{armstrong_crandall_julin_smart}] \label{T: equivalent conditions} Fix an open set $U \subseteq \mathbb{R}^{d}$.  Given any continuous function $u : U \to \mathbb{R}$, the following are equivalent:
		\begin{itemize}
			\item[(i)] $u \in W^{1,\infty}_{\text{loc}}(U)$ and $u$ is an absolute minimizer of $\text{Lip}_{\varphi}(\cdot \, ; \overline{V})$ (i.e., \eqref{E: lipschitz min} holds) for any open $V \subset \subset U$;
			\item[(ii)] $u \in W^{1,\infty}_{\text{loc}}(U)$ and $u$ is an absolute minimizer of $\|\varphi^{*}(D(\cdot))\|_{L^{\infty}(V)}$ (i.e., \eqref{E: linfty min} holds) for any open $V \subset \subset U$; and
			\item[(iii)] $u \in CCA_{\varphi}(U) \cap CCB_{\varphi}(U)$.
		\end{itemize}\end{theorem*}
		
Using comparison with cones, it is possible to show that if $u$ is a $\varphi$-AMLE in $U$, then $u$ is $\varphi$-infinity harmonic.  More precisely, given an open set $U \subseteq \mathbb{R}^{d}$, let $\text{Sub}_{\varphi}(U)$ and $\text{Super}_{\varphi}(U)$ be the spaces of $\varphi$-infinity sub- and superharmonic functions in $U$:
	\begin{align}
		\text{Sub}_{\varphi}(U) &= \{u \in USC(U) \, \mid \, -G_{\varphi}^{*}(Du,D^{2}u) \leq 0 \, \, \text{in the viscosity sense in} \, \, U\}, \nonumber \\
		\text{Super}_{\varphi}(U) &= \{v \in LSC(U) \, \mid \, -G_{*}^{\varphi}(Dv,D^{2}v) \geq 0 \, \, \text{in the viscosity sense in} \, \, U\}. \label{E: supersolution class}
	\end{align}
Due to the works \cite{crandall_evans_gariepy, aronsson_crandall_juutinen, armstrong_crandall_julin_smart}, it is now well-known that functions that satisfy comparison with cones from above or below are respectively $\varphi$-infinity sub- or superharmonic:\footnote{$CCA_{\varphi}(U)$ and $CCB_{\varphi}(U)$ are also characterized by the so-called absolutely sub- and superminimizing properties, basically one-sided versions of (i) and (ii) in the previous theorem, see \cite{armstrong_crandall_julin_smart}.}
	\begin{equation} \label{E: first inclusion}
		CCA_{\varphi}(U) \subseteq \text{Sub}_{\varphi}(U), \quad CCB_{\varphi}(U) \subseteq \text{Super}_{\varphi}(U).
	\end{equation}
This leads naturally to the question: is every $\varphi$-infinity harmonic function a $\varphi$-AMLE?  

This question was posed twenty years ago in \cite{aronsson_crandall_juutinen}, where it was observed that the answer is yes if $\partial \varphi^{*}$ is single-valued (away from the origin).  The aforementioned works \cite{crandall_gunnarsson_wang} and \cite{m_souganidis} extended this to the discontinuous case under some additional assumptions on $\varphi$.  As one of the main applications of the present work, the next theorem resolves this question in full generality.

	\begin{theorem} \label{T: cone comparison} Let $\varphi$ be a Finsler norm in $\mathbb{R}^{d}$.  Then $\text{Sub}_{\varphi}(U) \subseteq CCA_{\varphi}(U)$ and $\text{Super}_{\varphi}(U) \subseteq CCB_{\varphi}(U)$ for any open set $U \subseteq \mathbb{R}^{d}$.   \end{theorem}

Combined with the complementary inclusions \eqref{E: first inclusion}, the theorem implies that a function $u$ is a $\varphi$-AMLE if and only if it is $\varphi$-infinity harmonic:
	
	\begin{corollary} In addition to the equivalences (i), (ii), and (iii) stated above, $\varphi$-AMLE are characterized by the following condition:
		\begin{itemize}
			\item[(iv)] $u \in \text{Sub}_{\varphi}(U) \cap \text{Super}_{\varphi}(U)$.
		\end{itemize}
	\end{corollary}

See Section \ref{S: intro discussion of proof} for a discussion of the proof.
	
Finally, in addition to a PDE characterization of $\varphi$-absolutely minimizing functions, Theorem \ref{T: cone comparison} can be combined with previous results of Armstrong and Smart \cite{armstrong_smart_easy_proof} and \cite{armstrong_crandall_julin_smart} to establish a comparison result for the Dirichlet problem for the Finsler infinity Laplacian.  Indeed, it suffices to recall the following comparison theorem for functions in $CCA_{\varphi}(U)$ and $CCB_{\varphi}(U)$:

	\begin{theorem*}[Main result of \cite{armstrong_smart_easy_proof} and Theorem 2.1 in \cite{armstrong_crandall_julin_smart}] \label{T: old comparison} Let $\varphi$ be a Finsler norm.  Given any bounded open set $U \subseteq \mathbb{R}^{d}$, if $u \in USC(U)$ and $v \in LSC(U)$ satisfy
		\begin{equation*}
			u \in CCA_{\varphi}(U) \quad \text{and} \quad v \in CCB_{\varphi}(U),
		\end{equation*}
	then
		\begin{equation} \label{E: comparison of AMLE}
			\max\left\{ u(x) - v(x) \, \mid \, x \in \overline{U} \right\} = \max \left\{ u(x) - v(x) \, \mid \, x \in \partial U\right\}.
		\end{equation}
	In particular, given any uniformly Lipschitz $g \in C(\partial U)$, there is a unique $\varphi$-AMLE minimizer of \eqref{E: minimal lipschitz extension}.
	\end{theorem*}
	
Combining the previous two theorems, one deduces that, for any bounded open set $U \subseteq \mathbb{R}^{d}$, if $u \in USC(\overline{U})$ and $v \in LSC(\overline{U})$ satisfy
	\begin{equation*}
		- G_{\varphi}^{*}(Du,D^{2}u) \leq 0 \quad \text{in} \, \, U, \quad -G_{*}^{\varphi}(Dv,D^{2}v) \geq 0 \quad \text{in} \, \, U,
	\end{equation*}
then
	\begin{equation*}
		\max \left\{ u(x) - v(x) \, \mid \, x \in \overline{U} \right\} = \max \left\{ u(x) - v(x) \, \mid \, x \in \partial U \right\}.
	\end{equation*}
This proves that, for any $g \in C(\partial U)$, there is exactly one $u \in C(\overline{U})$ such that 
	\begin{equation*}
		-G_{\varphi}^{*}(Du,D^{2}u) \leq 0 \quad \text{and} \quad -G_{*}^{\varphi}(Du,D^{2}u) \geq 0 \quad \text{in} \, \, U, \quad u = g \quad \text{on} \, \, \partial U.
	\end{equation*}
Existence of such a $u$ is well-known; what is new here is uniqueness in the class of viscosity solutions.  Of course, if $g$ is uniformly Lipschitz on $\partial U$, this solution $u$ is the $\varphi$-AMLE discussed above.

\begin{remark} As one might expect, the approach of the present paper can be used in tandem with the proof strategy of Barles and Busca \cite{barles_busca} to derive a PDE proof of \eqref{E: comparison of AMLE}, see Remark \ref{R: barles busca} below. \end{remark}

\subsection{The Principal Infinity Eigenvalue} \label{S: eigenvalue intro} The next application of this work concerns the well-definedness of the principal infinity eigenvalue.  The problem of interest is the following PDE involving a (nonlinear) eigenvalue $\Lambda$:
		\begin{equation} \label{E: finsler infinity eigenvalue}
		\left\{ \begin{array}{r l}	
			\min\{\varphi^{*}(Du) - \Lambda u, -G_{\varphi}^{*}(Du,D^{2}u)\} \leq 0 & \text{in} \, \, U, \\
			\min\{\varphi^{*}(Du) - \Lambda u, -G_{*}^{\varphi}(Du,D^{2}u)\} \geq 0 & \text{in} \, \, U, \\
			u = 0  & \text{on} \, \, \partial U.
		\end{array} \right.
	\end{equation}

	\begin{definition} Given any bounded open set $U \subseteq \mathbb{R}^{d}$, a constant $\Lambda \in \mathbb{R}$ is said to be the \emph{principal} $\varphi$-\emph{infinity eigenvalue of} $U$ if there is at least one $u \in C(\overline{U})$ satisfying \eqref{E: finsler infinity eigenvalue} in the viscosity sense. \end{definition}
	
If $u$ is a solution of \eqref{E: finsler infinity eigenvalue}, then so is $\lambda u$ for any $\lambda > 0$.  It should be stressed that there may be other solutions, however, as was proved by Hynd, Smart, and Yu \cite{hynd_smart_yu} in the Euclidean setting. 
	
Nonetheless, the principle eigenvalue $\Lambda$ is well-defined.  Juutinen, Lindqvist, and Manfredi \cite{juutinen_lindqvist_manfredi}, Belloni, Kawohl, and Juutinen \cite{belloni_kawohl_juutinen}, and \cite{m_souganidis} established that, for certain choices of $\varphi$, the principal $\varphi$-infinity eigenvalue is determined by the $L^{\infty}$-variational formula
	\begin{equation} \label{E: variational formula}
		\Lambda_{\infty}^{\varphi}(U) = \min \left\{ \|\varphi^{*}(Dv)\|_{L^{\infty}(U)} \, \mid \, v \in W^{1,\infty}_{0}(U), \, \, \|v\|_{L^{\infty}(U)} \leq 1 \right\}.
	\end{equation}
As a consequence of the comparison results proved in this work, the proof of \cite{belloni_kawohl_juutinen} can be extended to arbitrary Finsler norms.

	\begin{theorem} \label{T: uniqueness infinity eigenvalue} Let $\varphi$ be a Finsler norm in $\mathbb{R}^{d}$.  Given any bounded open set $U \subseteq \mathbb{R}^{d}$, if there is a $\Lambda \in \mathbb{R}$ and a $u \in C(\overline{U})$ such that \eqref{E: finsler infinity eigenvalue} holds, then $\Lambda = \Lambda_{\infty}^{\varphi}(U)$ as in \eqref{E: variational formula}, $u \in W^{1,\infty}_{0}(U)$, and $\|u\|_{L^{\infty}(U)}^{-1} u$ attains the minimum in \eqref{E: variational formula}.\end{theorem}

Indeed, observe that if $u$ is a solution of the eigenvalue problem \eqref{E: finsler infinity eigenvalue}, then $u \in \text{Sup}_{\varphi}(U)$.  Revisiting the proof of \cite{belloni_kawohl_juutinen}, the one detail that is needed to extend to arbitrary $\varphi$ is precisely the inclusion $\text{Sup}_{\varphi}(U) \subseteq CCB_{\varphi}(U)$.  Since that is provided by Theorem \ref{T: cone comparison} above, Theorem \ref{T: uniqueness infinity eigenvalue} follows.
		
\subsection{Tug-of-war Games} \label{S: intro tug-of-war} In addition to $L^{\infty}$-variational problems, the Finsler infinity Laplace equation arises naturally in the analysis of so-called stochastic tug-of-war games.

For instance, the solution of the Dirichlet problem \eqref{E: finsler dirichlet with zeroth order} describes the asymptotics of a certain two-player game with running cost $f$ and random times determined by $c$.  The description of this game can be found in Section \ref{S: tug of war}.  On a purely analytical level, after diffusive rescaling, the value function $u^{\epsilon}$ of the game is the solution of the following finite-difference equation
	\begin{equation} \label{E: stochastic game}
		u^{\epsilon} + \epsilon^{-2} c(x) (u - \mathcal{M}^{\varphi}_{\epsilon}u^{\epsilon}) = f(x) \quad \text{in} \, \, U_{\epsilon}, \quad u^{\epsilon} = G \quad \text{in} \, \, \overline{U} \setminus U_{\epsilon},
	\end{equation}
where $G$ is some terminal payout, $U_{\epsilon} \subseteq U$ is a suitably defined approximation of $U$ (defined in Section \ref{S: tug of war}), and $\mathcal{M}_{\epsilon}^{\varphi}$ is the finite-difference operator defined by 
	\begin{align*}
		\mathcal{M}^{\varphi}_{\epsilon}u(x) = \frac{1}{2} \sup \left\{ u(y) \, \mid \, \varphi(y - x) < \epsilon \right\} + \frac{1}{2} \inf \left\{ u(y) \, \mid \, \varphi(x -y) < \epsilon \right\}.
	\end{align*}

The study of tug-of-war games was initiated by Peres, Schramm, Sheffield, and Wilson \cite{peres_schramm_sheffield_wilson}; see the books by Parviannen \cite{parviainen} and Lewicka \cite{lewicka} for introductions to the topic.  When $\partial \varphi^{*}$ is single-valued (hence $G_{\varphi}^{*} \equiv G_{*}^{\varphi}$ away from zero), well-known techniques apply to show that $u^{\epsilon} \to u$ as $\epsilon \to 0^{+}$, where $u$ is a solution of \eqref{E: finsler dirichlet with zeroth order}.  The next theorem shows this remains true for an arbitrary choice of $\varphi$.

	\begin{theorem} \label{T: tug of war convergence} Let $\varphi$ be a Finsler norm in $\mathbb{R}^{d}$ and assume that $c : \mathbb{R}^{d} \to [0,+\infty)$ is such that $\sqrt{c}$ is uniformly Lipschitz continuous.  Given a bounded open set $U \subseteq \mathbb{R}^{d}$, $f \in C(\overline{U})$, and $g \in C(\partial U)$, let $(u^{\epsilon})_{\epsilon > 0}$ denote the solutions of \eqref{E: stochastic game} with boundary condition $G : \overline{U} \to \mathbb{R}$ an arbitrary continuous extension of $g$, and let $u$ denote the solution of \eqref{E: finsler dirichlet with zeroth order} with $u = g$ on $\partial U$.  Then $u^{\epsilon} \to u$ uniformly in $\overline{U}$ as $\epsilon \to 0^{+}$.   \end{theorem}

See Section \ref{S: tug of war} for the proof.

\subsection{Method of Proof} \label{S: intro discussion of proof}

The approach of this paper expands upon an insight from \cite{m_souganidis}.  To explain the connection, it is first useful to note that, in the proof of the inclusion $\text{Sub}_{\varphi}(U) \subseteq CCA_{\varphi}(U)$ in Theorem \ref{T: cone comparison}, it is enough to assume that $u$ is strictly $\varphi$-infinity subharmonic in the sense that, for some $\epsilon > 0$,
	\begin{align} \label{E: strict sub ack}
		\frac{1}{2} \epsilon \varphi^{*}(Du)^{2} - G_{\varphi}^{*}(Du,D^{2}u) \leq 0 \quad \text{in} \, \, U;
	\end{align}
see Lemma \ref{L: symmetries} below for this reduction.

As was already mentioned above, in the recent work \cite{m_souganidis}, it was observed that even though $G_{\varphi}^{*}$ and $G_{*}^{\varphi}$ might differ along a rather rough set, nonetheless if $\varphi$ is $C^{2}$ away from the origin, then it is possible to prove cone comparison results simply by using $\varphi$ itself as a test function and directly invoking the definition of viscosity subsolution.  Indeed, at a local maximum $x \neq 0$ of $u - \varphi$, one finds
	\begin{align*}
		\frac{1}{2} \epsilon = \frac{1}{2} \epsilon \varphi^{*}(D\varphi(x))^{2} = \frac{1}{2} \epsilon \varphi^{*}(D\varphi(x))^{2} - G_{\varphi}^{*}(D\varphi(x),D^{2}\varphi(x)) \leq 0,
	\end{align*}
where above the key observation is that $\varphi^{*}(D\varphi(x)) = 1$, which is well-known, and $G_{\varphi}^{*}(D\varphi(x),D^{2}\varphi(x)) = 0$, which is a fundamental observation of \cite[Section 3]{m_souganidis}.

The point of view taken in this work extends this insight of \cite{m_souganidis} by, in effect, expanding the class of test functions to include $\varphi$, even if it is not smooth.  Notice that, compared to an arbitrary nonsmooth function, $\varphi$ has the advantage that it is (1) convex and positively one-homogeneous and (2) $\varphi$-infinity superharmonic.  

Concerning (2), there is an important observation, which appears to be new, and helps to explain why Theorem \ref{T: cone comparison} at the very least ought to be true.  While it is well-known that $\varphi$ is itself a $\varphi$-infinity harmonic function in $\mathbb{R}^{d} \setminus \{0\}$, the next result shows that it is a supersolution in a strong sense:

	\begin{restatable}{prop}{stronglysuperharmonic} \label{P: strongly superharmonic} If $\varphi$ is a Finsler norm in $\mathbb{R}^{d}$, then 
		\begin{equation} \label{E: strongly superharmonic}
			-G_{\varphi}^{*}(D\varphi,D^{2}\varphi) \geq 0 \quad \text{in the viscosity sense in} \quad \mathbb{R}^{d} \setminus \{0\}.
		\end{equation}
	\end{restatable}

Note that this inequality is stronger than the definition of supersolution in \eqref{E: supersolution class} since it involves $G_{\varphi}^{*}$ instead of $G_{*}^{\varphi}$.  This suggests $\varphi$ should act as a barrier controlling subsolutions, as expected.  Indeed, using Proposition \ref{P: strongly superharmonic}, it is easy to prove that any $C^{2}$ $\varphi$-infinity subharmonic function is in $CCA_{\varphi}$; see Section \ref{S: preliminaries} for details.

Though Proposition \ref{P: strongly superharmonic} is a step in the right direction, nonetheless the road ahead is treacherous since, as $G_{\varphi}^{*}$ is merely upper semicontinuous, the lower bound in \eqref{E: strongly superharmonic} is highly sensitive to perturbations.  Still, as hinted already, the convexity and homogeneity of $\varphi$ provide a lot of structure to work with.    This is easiest to grasp when the unit ball $\{\varphi \leq 1\}$ is of class $C^{1,1}$ and strictly convex, as in the next result:

	\begin{restatable}{prop}{strictlyconvexgame} \label{P: strictly convex game} If $\varphi$ is a Finsler norm in $\mathbb{R}^{d}$ for which the unit ball $\{\varphi \leq 1\}$ is both of class $C^{1,1}$ and strictly convex, then there is a family of convex positively one-homogeneous functions $\{\psi_{x}\}_{x \in \mathbb{R}^{d} \setminus \{0\}}$ with the following properties:
		\begin{itemize}
			\item[(i)] $\psi_{x}$ is smooth in a neighborhood of $x$; 
			\item[(ii)] $\varphi \leq \psi_{x}$ pointwise in $\mathbb{R}^{d}$ and $\varphi(x) = \psi_{x}(x)$; and
			\item[(iii)] at the point $x$, the following identities hold:
				\begin{equation*}
					G_{*}^{\varphi}(D\psi_{x}(x),D^{2}\psi_{x}(x)) = G_{\varphi}^{*}(D\psi_{x}(x), D^{2}\psi_{x}(x)) = 0.
				\end{equation*}
		\end{itemize}
	\end{restatable}

In effect, the proposition says that if $\{\varphi \leq 1\}$ is $C^{1,1}$ and strictly convex, then $\varphi$ might as well be smooth: if $u$ is a strict subsolution as in \eqref{E: strict sub ack} and $u - \varphi$ has a local maximum at $x \neq 0$, then $u - \psi_{x}$ also has a local maximum at $x$, and then one invokes the aforementioned properties of $\psi_{x}$ to obtain
	\begin{align*}
		\frac{1}{2} \epsilon \varphi^{*}(D\varphi(x))^{2} = \frac{1}{2} \epsilon \varphi^{*}(D\psi_{x}(x))^{2} - G_{\varphi}^{*}(D\psi_{x}(x),D^{2}\psi_{x}) \leq 0,
	\end{align*}
a contradiction exactly as before. (See Section \ref{S: c11 strictly convex} for more details.)  

Of course, asking for both $C^{1,1}$ regularity and strict convexity of $\{\varphi \leq 1\}$ is too restrictive.\footnote{It is convenient to specify the regularity of the unit ball $\{\varphi \leq 1\}$ rather than $\varphi$ itself since the latter is never differentiable at the origin. It is not hard to convince oneself that $\{\varphi \leq 1\}$ is a manifold with boundary of class $C^{k,\alpha}$ if and only if $\varphi \in C^{k,\alpha}_{\text{loc}}(\mathbb{R}^{d} \setminus \{0\})$, cf.\ Section \ref{S: pos hom}.}  Nonetheless, it turns out that if $\{\varphi \leq 1\}$ is only $C^{1,1}$, then this is already enough to construct certain test functions, which are less regular than those in the previous proposition, but can still be used to the same effect.  This step of the proof of Theorem \ref{T: cone comparison}, which completely addresses the $C^{1,1}$ case, relies on a new perturbation argument for nonsmooth convex test functions, very much in the spirit of the classical Alexandrov-Bakelman-Pucci argument, but requiring novel ideas to work around the discontinuities of the Finsler infinity Laplacian; see Sections \ref{S: conical test} and \ref{S: c11 setting}.

Finally, it remains to show that an arbitrary Finsler norm $\varphi$ can be approximated by $C^{1,1}$ Finsler norms in a manner that respects the discontinuities of the Finsler infinity Laplacian.  It turns out that this can, indeed, be done.  The main theorem involved in this approximation argument is stated next:

\begin{restatable}{theorem}{regularizednorm} \label{T: regularized_norm} If $\varphi$ is any Finsler norm in $\mathbb{R}^{d}$, then there is a family of Finsler norms $(\varphi_{\zeta})_{\zeta > 0} \subseteq C^{1,1}_{\text{loc}}(\mathbb{R}^{d} \setminus \{0\})$ along with constants $\kappa(\varphi), \delta(\varphi) > 0$ such that
	\begin{equation*}
		\lim_{\zeta \to 0^{+}} \varphi_{\zeta} = \varphi \quad \text{locally uniformly in} \, \, \mathbb{R}^{d}
	\end{equation*} 
and, for any $\zeta \in (0,1)$, the following properties hold:
		\begin{itemize}
			\item[(i)] $\{\varphi_{\zeta} \leq 1\}$ satisfies a uniform interior ball condition with radius $\zeta$; 
			\item[(ii)] for any $p \in \mathbb{R}^{d} \setminus \{0\}$, 
				\begin{equation} \label{E: subdifferential version of infinity harmonic}
					\partial \varphi_{\zeta}^{*}(p) = \partial \varphi^{*}(p) + \frac{\zeta p}{\|p\|};
				\end{equation}
			\item[(iii)] $\varphi_{\zeta} \leq \varphi$ pointwise in $\mathbb{R}^{d}$;
			\item[(iv)] the following gradient bound holds:
				\begin{equation*}
			\varphi^{*}(D\varphi_{\zeta}(x)) \geq \kappa(\varphi) \quad \text{for each} \quad x \in \mathbb{R}^{d} \setminus \{0\}; \quad \text{and}
				\end{equation*} 
			\item[(v)] the following nondegeneracy condition holds:
				\begin{equation} \label{E: nondegeneracy condition zeta}
					\overline{B}_{\delta(\varphi)}(0) \subseteq \{\varphi_{\zeta} \leq 1\} \subseteq \overline{B}_{\delta(\varphi)^{-1}}(0).
				\end{equation}
		\end{itemize}
\end{restatable}

Where the Finsler infinity Laplacian is concerned, the key ingredient is \eqref{E: subdifferential version of infinity harmonic}.  As is argued below, this should be interpreted as implying that $\varphi_{\zeta}$ is \emph{almost} strongly $\varphi$-superharmonic in the sense of Proposition \ref{P: strongly superharmonic}.  Alternatively, notice that $G_{\varphi}^{*}$ can be rewritten as follows:
	\begin{equation*}
		G_{\varphi}^{*}(p,X) = \max\left\{ \left \langle X \left[q - \frac{\zeta p}{\|p\|}\right], q - \frac{\zeta p}{\|p\|} \right \rangle \, \mid \, q \in \partial \varphi_{\zeta}^{*}(p) \right\} =: \mathcal{G}_{\varphi_{\zeta}}^{*,\zeta}(p,X).
	\end{equation*}
This shows that $G_{\varphi}^{*}$ can be regarded as a perturbation of the operator $G_{\varphi_{\zeta}}^{*}$, which, in accordance with the previous discussion, is easier to analyze due to the $C^{1,1}$ regularity of $\{ \varphi_{\zeta} \leq 1 \}$.\footnote{Of course, an analogous formula also pertains to the lower envelope $G_{*}^{\varphi}$.}  This permits results from the $C^{1,1}$ setting to be lifted, \emph{mutatis mutandis}, to the general case; see Section \ref{S: c11 reduction} for precise statements.

\subsection{Extensions and Open Problems} The comparison principles proved in this work concern only elliptic PDE in bounded domains.  

The extension to parabolic PDE in bounded domains is immediate: the required changes in the proofs are only cosmetic.  In particular, this establishes that $\varphi$-infinity sub- and supercaloric functions, that is, those satisfying the respective inequalities
	\begin{equation*}
		u_{t} - G_{\varphi}^{*}(Du,D^{2}u) \leq 0 \quad \text{and} \quad u_{t} - G_{*}^{\varphi}(Du,D^{2}u) \geq 0
	\end{equation*}
are characterized by a ``curved" version of comparison with cones.  Namely, a function $u$ is $\varphi$-infinity subcaloric in a space-time open set $U \subseteq \mathbb{R}^{d} \times \mathbb{R}$ if and only if, for any parabolic cylinder $V = \Omega \times (t_{1},t_{2}) \subset \subset U$, any $v \notin \Omega$, and $s < t_{1}$,
	\begin{align*}
		&\max \left\{ u(x,t) + (t - s)^{-\frac{1}{2}} \exp \left( - \frac{\varphi(x - v)^{2}}{4(t - s)} \right) \, \mid \, (x,t) \in \overline{V} \right\} \\
		&\quad = \max \left\{ u(x,t) + (t - s)^{-\frac{1}{2}} \exp \left( - \frac{\varphi(x - v)^{2}}{4(t - s)} \right) \, \mid \, (x,t) \in \partial_{p} V \right\}.
	\end{align*}
Above $\partial_{p} V$ denotes the parabolic boundary.  Such results were first proved by Crandall and Wang \cite{crandall_wang} and Juutinen and Kawohl \cite{juutinen_kawohl} in the case of the Euclidean norm. 

On the other hand, the extension to unbounded domains, while likely only technical, is nonetheless nontrivial.  The reason is, generally speaking, the strategy for proving comparison principles in unbounded domains involves adding a penalization term in the variable-doubling argument, hence analyzing local maxima of functions such as
	\begin{equation*}
		u(x) - v(y) - \frac{\psi(x - y)^{4}}{4 \zeta} - \beta \|x\|^{2}.
	\end{equation*}
In the present setting, however, it is not at all clear what the perturbation term ($\beta \|x\|^{2}$ in the example above) should be.  

One way to understand why this is problematic is, in effect, in the example above, $u$ has been replaced by $u(x) - \beta \|x\|^{2}$, which, for a continuous equation, solves approximately the same equation as $u$ when $\beta$ is small.  This is no longer true in the analysis of discontinuous equations, unless one has very good control over the function $\psi$ as in \cite{m_souganidis}.  The strategy developed in this work simply does not yield such good control over (the analogue of) $\psi$, hence the extension to unbounded domains is nontrivial.

Finally, while this work proves the cone comparison principle for the PDE associated with the Lipschitz extension problem, the proof does not carry over immediately to more general $L^{\infty}$-variational problems, in particular, when the Hamiltonian is not positively homogeneous.  This remains an interesting open problem.

\subsection{Related Literature} This work belongs to both the literatures on (1) the infinity Laplacian and $L^{\infty}$-variational problems and (2) comparison principles for second-order elliptic PDE with discontinuities in the gradient variable.  Accordingly, this literature review will be separated into two parts.

\subsubsection{Infinity Laplacian and absolutely minimizing functions}  The general theory of $L^{\infty}$-variational problems was initiated by Aronsson in the series of works \cite{aronsson_minimization_1,aronsson_minimization_2,aronsson_minimization_3}, the notion of absolutely minimizing Lipschitz extension being introduced in the Euclidean context in \cite{aronsson_lipschitz}.  Jensen proved the uniqueness of such extensions by establishing a comparison principle for infinity harmonic functions \cite{jensen}.  Other proofs are now known, for instance, the more widely applicable proof in \cite{barles_busca} and the elementary proof in \cite{armstrong_smart_easy_proof}.

The fact that the comparison-with-cones properties are equivalent to infinity harmonicity was first proved by Crandall, Evans, and Gariepy \cite{crandall_evans_gariepy}, again in the context of the Euclidean norm.  Together with Jensen's work, this established the equivalence of the AMLE, comparison-with-cones, and infinity harmonicity properties in the Euclidean setting.

The extension of such results to other (symmetric) norms was pursued in \cite{aronsson_crandall_juutinen}.  They proved, in particular, that comparison with cones characterizes AMLE.  However, it needs to be emphasized that they left open the question of a PDE characterization.

Many works have treated PDE characterizations of absolute minimizers, not only for AMLE but also other $L^{\infty}$-variational problems.  A summary of the relevant literature can be found in \cite{armstrong_crandall_julin_smart}.  It should be emphasized that, with the exception of \cite{crandall_gunnarsson_wang} and \cite{m_souganidis}, these results concern settings where the coefficients of the PDE are continuous.\footnote{With the usual caveat that there may be a discontinuity where the gradient variable vanishes.}

Prior to \cite{m_souganidis}, the paper \cite{crandall_gunnarsson_wang} was the only one to prove the PDE characterization of AMLE in a setting in which the Finsler infinity Laplacian is discontinuous.  Therein comparison was proved conditional on the construction of special ``shielding" norms.  The existence of shielding norms was then proved in the case when $\varphi$ equals the $\ell^{1}$ and $\ell^{\infty}$ norm, or, in $d = 2$, in the case when the unit ball $\{\varphi \leq 1\}$ contains only finitely many line segments.

Among other contributions, the paper \cite{m_souganidis} proved that the PDE characterization for $\varphi$-AMLE holds under any one of the following conditions on $\varphi$: (i) $\{\varphi \leq 1\}$ is of class $C^{2}$, (ii) $\{\varphi \leq 1\}$ is polyhedral, or (iii) the dimension $d = 2$.\footnote{See that reference for an expository account explaining why the PDE characterization also holds if (iv) $\{\varphi \leq 1\}$ is strictly convex, which was already observed in \cite{aronsson_crandall_juutinen}.}  Since the proof strategy builds on ideas from the literature on comparison principles for level-set PDE, there will be more to say about this work in the next subsection.  For now, as already hinted at above, it is worth emphasizing that the most important aspect of \cite{m_souganidis} for the purposes of this paper is the analysis of the $C^{2}$ setting, see Section 3 therein.  

Finally, building on the ideas of \cite{armstrong_smart_easy_proof}, the previously mentioned work \cite{armstrong_crandall_julin_smart} proved various characterizations of AMLE, and even the comparison principle \eqref{E: comparison of AMLE}, in the more general setting in which the functional $\|\varphi^{*}(Du)\|_{L^{\infty}(U)}$ in \eqref{E: linfty min} is replaced by the more general one $\|H(Du)\|_{L^{\infty}(U)}$ under minimal assumptions on $H$.  

Both foundational works \cite{aronsson_crandall_juutinen} and \cite{armstrong_crandall_julin_smart} raised the question of a PDE characterization of absolute minimizers.  In this work, that question is closed in the specific context of AMLE and the minimal Lipschitz extension problem \eqref{E: minimal lipschitz extension}.  There is reason to hope that the ideas laid out in this paper can be extended to $L^{\infty}$-variational problems with general Hamiltonians $H$ as in \cite{armstrong_crandall_julin_smart}.  However, the positive homogeneity of Finsler norms (especially the positive zero-homogeneity of $\partial \varphi^{*}$) is used in a nontrivial way throughout.  As such, the geometry of the problems involved changes considerably when the homogeneity assumption is removed and new ideas will be needed.

\subsubsection{Comparison for equations with gradient discontinuities} Within the theory of viscosity solutions, the genesis for methods treating equations with discontinuities in the gradient variable by appropriately modifying the standard test functions goes back to the work of Evans and Spruck \cite{evans_spruck} and Chen, Giga, and Goto \cite{chen_giga_goto} on the level-set formulation of mean curvature flow.  In that setting, the origin is the only discontinuity of the operator, and \cite{evans_spruck} circumvented it by using a test function whose derivatives avoid the discontinuity.

Ishii and Souganidis \cite{ishii_souganidis} subsequently adapted this technique to treat equations with coefficients that blow-up when the gradient is zero.  This extends the level-set theory to cover, for instance, flow of convex sets by Gaussian curvature.  Ohnuma and Sato \cite{ohnuma_sato_p-laplace} further adapted this approach to develop a viscosity solution theory for $p$-Laplace-type equations with $p < 2$.

The extension to operators with more complicated discontinuities in the gradient variable is more challenging.  Until recently, the most general and systematic approach was introduced by Ishii \cite{ishii}, again in the setting of level-set PDE, improving previous work by Gurtin, Soner, and Souganidis \cite{gurtin_soner_souganidis} and Ohnuma and Sato \cite{ohnuma_sato}.  This work treats the case of operators that are positively one-homogeneous in the gradient variable and for which the set of discontinuous directions forms a smooth submanifold of the sphere.  Inspired by  \cite{ishii}, the work \cite{m_souganidis} extended the method to operators in which the set of discontinuous directions need not be smooth, but instead need only be, in a suitable sense, compatible with the geometry of some Finsler norm.  While the Finsler infinity Laplacian is always compatible with its associated Finsler norm, there is a restriction in \cite{ishii} and \cite{m_souganidis}, namely, the method requires the existence of certain so-called shielding norms, which are hard to construct.

Compared to the approach of \cite{m_souganidis}, the comparison proofs presented here are somewhat more cumbersome, yet also more local.  The existence of a shielding norm in \cite{m_souganidis} is a global problem: it is equivalent to the existence of an open convex set with $C^{2}$ boundary and boundary curvature (second fundamental form) satisfying certain linear constraints determined by the normal vector.\footnote{See \cite[Definition 2]{m_souganidis}.}  By contrast, the construction of conical test functions introduced in Section \ref{S: conical test} is local.  For instance, in the simple setting covered by Proposition \ref{P: strictly convex game} above, to any point $x \in \mathbb{R}^{d} \setminus \{0\}$, there is an associated test function $\psi_{x}$, the construction of which is based solely on the geometry of the level sets of $\varphi$ near $x$.  In principle, if one could glue together the conical test functions suitably, then one would obtain a shielding norm --- but the difficult issue of \emph{how} to perform such a gluing is precisely what makes the local approach taken here attractive. 

Finally, there are some qualitative similarities between the strategy employed in this paper and the techniques developed in connection with Hamilton-Jacobi equations on junctions and stratified domains.  First, it is worth emphasizing that the techniques in this paper have the flavor of nonsmooth and convex analysis, in a manner that is at least philosophically similar to recent work on Hamilton-Jacobi equations in stratified domains as carried out, for instance, by Jerhaoui and Zidani \cite{jerhaoui_zidani}.  Put one way, whereas in those works, it is the spatial domain that is stratified, in the present work, it is instead the gradient variable that takes values in a stratified space.  Where such Hamilton-Jacobi equations are concerned, there is also some similarity to the fundamental contribution of Lions and Souganidis \cite{lions_souganidis}.  In that paper and this one, a blow-up argument is used to prove what amounts to a boundary maximum principle-type lemma: the reader can compare \cite[Lemma 3.1]{lions_souganidis} to Proposition \ref{P: key mountain part 1} below.  In view of these similarities, there is reason to hope that ideas from this paper will be applicable more broadly.

\subsection{Organization of the Paper}  This paper is divided into four parts:
	\begin{itemize}
		\item[Part 1.] Two Tools: $C^{1,1}$ Approximation and Conical Test Functions
		\item[Part 2.] Cone Comparison Principles for $C^{1,1}$ Finsler Norms
		\item[Part 3.] The General Case and Applications
		\item[Part 4.] Appendices
	\end{itemize}  

\subsubsection{Part 1} The first part is intended to efficiently introduce the main new ideas of the paper.  The hope is this can be readable and informative, even if the reader chooses not to consult the other parts.  It consists of Sections \ref{S: preliminaries}-\ref{S: conical test}.  

Section \ref{S: preliminaries} first reviews preliminary definitions, then goes on to present a warm-up example.  Specifically, this warm-up consists of an easy proof that $C^{2}$ $\varphi$-infinity subharmonic (resp.\ superharmonic) functions are in $CCA_{\varphi}$ (resp.\ $CCB_{\varphi}$) for an arbitrary Finsler norm $\varphi$.

Section \ref{S: regularized} presents the $C^{1,1}$ approximation theorem, Theorem \ref{T: regularized_norm}, stated already above.  The section motivates the approximation by (1) showing that the approximating norms are ``almost" strongly $\varphi$-superharmonic and (2) demonstrating that the cone comparison principle (Theorem \ref{T: cone comparison}) holds for general Finsler norms $\varphi$ provided certain generalized cone comparison results hold for  $C^{1,1}$ Finsler norms.  (The proof of the latter is taken up in Part 2.)

Section \ref{S: conical test} introduces the conical test functions that will be used throughout the remainder of the paper.  To motivate the construction, they are first discussed in the case when the ball $\{\varphi \leq 1\}$ is both of class $C^{1,1}$ and strictly convex, as in Proposition \ref{P: strictly convex game}.  As suggested already above, this leads to an easy proof of the cone comparison principle for such norms, which entirely avoids sup-convolutions and the maximum principle for semicontinuous functions.

\subsubsection{Part 2} The second part is the most technically demanding.  It establishes cone comparison-type results when the ball $\{\varphi \leq 1\}$ is of class $C^{1,1}$.  The main difficulty lies in testing subsolutions against the conical test functions introduced in Section \ref{S: conical test}, which is nontrivial because the latter are convex but usually not smooth.

Section \ref{S: pos hom} prepares the ground by presenting some needed regularity properties of nonnegative, positively one-homogeneous convex functions.  It can be skimmed on a first reading.

Section \ref{S: c11 setting} proves cone comparison-type results for $C^{1,1}$ Finsler norms.  The proof involves testing with the conical test functions introduced in Section \ref{S: conical test}.  To do this, what can be regarded as a version of Evans' perturbed test function method is applied: the conical test function is replaced by a suitable smooth perturbation.

Section \ref{S: key mountain} proves a key claim of Section \ref{S: c11 setting}, namely that useful information can be extracted from the perturbed test functions as the perturbation parameter goes to zero.  Here tools from convex analysis are used to tie together the geometric constructions developed earlier.

The proof of the cone comparison principle, Theorem \ref{T: cone comparison}, is already complete at the conclusion of Part 2.  It is worth reiterating that, up to this point of the paper, the maximum principle for semicontinuous functions is never used.

\subsubsection{Part 3} The third part of the paper extends the results of Part 2 to other equations involving the Finsler infinity Laplacian.  

Section \ref{S: single variable} clears the ground by reformulating the results of Part 2 so as to apply to any Finsler norm, and also extending from the cone $\varphi$ to ``curved cones" $h(\varphi)$.

The main step, which comprises Section \ref{S: elliptic bounded domains}, consists in extending from single-variable cone comparison results, involving maxima of functions $u(x) - h(\varphi(x))$, to double-variable results, involving maxima of functions $u(x) - v(y) - h(\varphi(x - y))$.  This is not exactly trivial, but it follows by combining the results of Part 2 with a version of the maximum principle for semicontinuous functions\footnote{Sometimes called Jensen's Lemma, Ishii's Lemma, or the Jensen-Ishii Lemma, see \cite{user,chen_giga_goto,silvestre_imbert}.} from \cite{ishii}.  In particular, this is the first time in the paper that this maximum principle is used.

Applications of the comparison principle are discussed in Section \ref{S: applications}.  Specifically, the comparison principle for the Dirichlet problem \eqref{E: naive} is proved (Theorem \ref{T: comparison dirichlet}), as well as the approximation by tug-of-war (Theorem \ref{T: tug of war convergence}).

\subsubsection{Part 4} The final part of the paper consists of appendices.  Appendix \ref{A: technical results} provides proofs of miscellaneous technical results used elsewhere in the paper.  Appendix \ref{A: perturbed test functions} gives a complete construction of the perturbed test functions from Section \ref{S: c11 setting}.

\subsection{Acknowledgements} This work was completed through the support of NSF grant DMS-2202715.

\subsection{Notation}  In $\mathbb{R}^{d}$, the Euclidean inner product between two vectors $v, w \in \mathbb{R}^{d}$ is denoted by $\langle v, w \rangle$, and $\|\cdot\|$ denotes the associated Euclidean norm.  

Given any $r > 0$ and $x_{*} \in \mathbb{R}^{d}$, the Euclidean open and closed balls centered at $x_{*}$ of radius $r$ and the sphere of radius $r$ are denoted as follows:
	\begin{gather*}
		B_{r}(x_{*}) = \{x \in \mathbb{R}^{d} \, \mid \, \|x - x_{*}\| < r\}, \quad \overline{B}_{r}(x_{*})= \{x \in \mathbb{R}^{d} \, \mid \, \|x - x_{*}\| \leq r\}, \\
	 \partial B_{r}(x_{*}) = \{x \in \mathbb{R}^{d} \, \mid \, \|x - x_{*}\| = r\}.
	\end{gather*}
The unit sphere in $\mathbb{R}^{d}$ is denoted by $S^{d-1}$, that is,
	\begin{equation*}
		S^{d-1} = \{e \in \mathbb{R}^{d} \, \mid \, \|e\| = 1\}.
	\end{equation*}
Given an $e \in S^{d-1}$, the associated halfspace is denoted by $H_{e}$:
	\begin{equation*}
		H_{e} = \{x \in \mathbb{R}^{d} \, \mid \, \langle x,e \rangle \geq 0\}.
	\end{equation*}

\textit{Set Operations:} Given a set $A \subseteq \mathbb{R}^{d}$, the closure of $A$ with respect to the Euclidean norm topology is denoted by $\overline{A}$.  The interior of $A$ with respect to this topology is denoted by $\text{int}(A)$, and the boundary by $\partial A$.

Given an open set $U \subseteq \mathbb{R}^{d}$, the notation $V \subset \subset U$ means that $\overline{V} \subseteq U$ and $\overline{V}$ is compact.

If $A \subseteq \mathbb{R}^{d}$, then $A^{\perp}$ is the linear subspace orthogonal to $A$, that is,
	\begin{equation*}
		A^{\perp} = \bigcap_{x \in A} \{v \in \mathbb{R}^{d} \, \mid \, \langle v, x \rangle = 0\}.
	\end{equation*} 
For any $v \in \mathbb{R}^{d}$, the sum $v + A$ is the set of all translates
	\begin{equation*}
		v + A = \{v + x \, \mid \, x \in A\}.
	\end{equation*}
The cone $\mathcal{C}(A)$ generated by $A$ is the set
	\begin{equation*}
		\mathcal{C}(A) = \{tx \, \mid \, x \in A, \, \, t \geq 0\}.
	\end{equation*}
	
For any nonempty set $A \subseteq \mathbb{R}^{d}$, the distance function $\text{dist}(\cdot,A) : \mathbb{R}^{d} \to [0,+\infty)$ is defined by 
	\begin{equation*}
		\text{dist}(x,A) = \min \left\{ \|x - y\| \, \mid \, y \in A \right\}.
	\end{equation*}

\textit{Matrices:} The space of symmetric $d \times d$ matrices is denoted by $\mathcal{S}_{d}$.  	Given a symmetric matrix $X \in \mathcal{S}_{d}$, the associated quadratic form is denoted by $Q_{X}$, that is,
	\begin{equation*}
		Q_{X}(v) = \langle X v, v \rangle \quad \text{for each} \quad v \in \mathbb{R}^{d}.
	\end{equation*}
The partial order $\leq$ on $\mathcal{S}_{d}$ is defined as usual at the level of quadratic forms: given $X, Y \in \mathcal{S}_{d}$, one has
	\begin{equation*}
		X \leq Y \quad \iff \quad Q_{X}(v) \leq Q_{Y}(v) \quad \text{for each} \quad v \in \mathbb{R}^{d}.
	\end{equation*}
	
Given vectors $v_{1},v_{2} \in \mathbb{R}^{d}$, the tensor product $v_{1} \otimes v_{2}$ is the linear operator on $\mathbb{R}^{d}$ defined by 
	\begin{equation*}
		(v_{1} \otimes v_{2}) w = \langle v_{2}, w \rangle v_{1} \quad \text{for each} \quad w \in \mathbb{R}^{d}.
	\end{equation*}
For a single vector $v$, the shorthand notation $v^{\otimes 2} = v \otimes v$ is also used.

\textit{Functions:} Given a set $E \subseteq \mathbb{R}^{d}$, $USC(E)$ denotes the set of all upper semicontinuous functions in $E$, while $LSC(E)$ denotes the lower semicontinuous functions.

\tableofcontents

%The main interest of this work is in viscosity sub- and super-solutions of the Finsler infinity Laplace equation
%	\begin{equation} \label{E: infinity subharmonic}
%		- G_{\varphi}^{*}(Du,D^{2}u) \leq 0.
%	\end{equation}
%
%\begin{theorem} \label{T: main} Fix an open set $U \subseteq \mathbb{R}^{d}$ and a Finsler norm $\varphi$ in $\mathbb{R}^{d}$.  If $u : U \to \mathbb{R} \cup \{-\infty\}$ is upper semi-continuous and $-G_{\varphi}^{*}(Du,D^{2}u) \leq 0$ holds in the viscosity sense in $U$, then $u \in CCA(U)$. \end{theorem}
%
%The analogous statement holds if $u$ is instead a lower semi-continuous superharmonic function, a fact which is readily deduced upon observing that then $-u$ is subharmonic.
%
\part{Two Tools: $C^{1,1}$ Approximation and Conical Test Functions} 

This part of the paper introduces two tools, (1) a $C^{1,1}$ approximation of Finsler norms that is appropriately adapted to the Finsler infinity Laplacian and (2) a family of conical test functions, which constitute useful barriers.  The presentation is meant to be expository and the emphasis is on motivating these tools in the context of cone comparison for Finsler infinity harmonic functions, or the proof of Theorem \ref{T: cone comparison}. 

\section{Preliminaries and a Warm-Up Example} \label{S: preliminaries}

Recall from Section \ref{S: lipschitz extension problem} that one of the main goals of this paper is to prove that if a function $u$ is $\varphi$-infinity subharmonic, that is, if $-G_{\varphi}^{*}(Du,D^{2}u) \leq 0$ in some open set $U$, then it possesses the comparison-with-cones property from above, or
	\begin{align*}
		\max \left\{ u(x) - \lambda \varphi(x - v) \, \mid \, x \in V \right\} = \max \left\{ u(x) - \lambda \varphi(x - v) \, \mid \, x \in \partial V \right\}
	\end{align*}
for any $V \subset \subset U$, $\lambda > 0$, and $v \in \mathbb{R}^{d} \setminus V$.  To clear the ground for the proof of this, it is instructive to consider the case when the subsolution $u$ is $C^{2}$.  As will be shown below, this setting is simpler since then $u$ can be regarded as a test function touching $\varphi$ and then it is only necessary to understand the viscosity-theoretic properties of $\varphi$.

	\begin{prop} \label{P: smooth subharmonic} Given an open set $U \subseteq \mathbb{R}^{d}$ and a Finsler norm $\varphi$ in $\mathbb{R}^{d}$, if $u$ is a $C^{2}$ function satisfying $-G_{\varphi}^{*}(Du,D^{2}u) \leq 0$ pointwise in $U$, then $u \in CCA_{\varphi}(U)$. \end{prop}

As mentioned already in the introduction, the previous proposition follows easily from the next result, which says that $\varphi$ is, in some sense, ``strongly" $\varphi$-infinity superharmonic.

\stronglysuperharmonic*
	
		\begin{remark} Note the difference between the inequality here and the one in \eqref{E: supersolution class}: $\varphi$ satisfies the differential inequality with the operator $G_{\varphi}^{*}$ replacing $G_{*}^{\varphi}$, hence it is a stronger inequality.  \end{remark}
		
		\begin{remark} At a high level, Proposition \ref{P: strongly superharmonic} is the reason why Theorem \ref{T: cone comparison} is plausible.  Although this proposition will not be used explicitly elsewhere in the paper, the results that follow in Sections \ref{S: conical test} and \ref{S: c11 setting} can be understood as refined versions of it; see especially Remark \ref{R: strongly superharmonic again}. 
		
		Nonetheless, even with Proposition \ref{P: strongly superharmonic} in hand, the proof of cone comparison given below is far from trivial.  The essential difficulty stems from the fact that $G_{\varphi}^{*}$ is only upper semicontinuous, and, thus, \eqref{E: strongly superharmonic} is highly unstable.\end{remark}

	\begin{remark} Although it is well-known that $\varphi$ is itself infinity harmonic,\footnote{See \cite[Section 1.4]{aronsson_crandall_juutinen} or \cite[Lemma 3.5 and Proposition 4.7]{armstrong_crandall_julin_smart}.} Proposition \ref{P: strongly superharmonic} appears to be new.  It strengthens an observation from \cite{m_souganidis}: if $\varphi$ is $C^{2}$ away from the origin (or, equivalently, if $\{\varphi \leq 1\}$ is of class $C^{2}$), then
		\begin{equation*}
			-G_{\varphi}^{*}(D\varphi(x),D^{2}\varphi(x)) = -G_{*}^{\varphi}(D\varphi(x),D^{2}\varphi(x)) = 0 \quad \text{for each} \quad x \in \mathbb{R}^{d} \setminus \{0\}.
		\end{equation*}
	\end{remark}
	
	\begin{remark} It should be noted that, in general, the inequality $-G_{*}^{\varphi}(D\varphi,D^{2}\varphi) \leq 0$ also holds.  Hence it is tempting to say that $\varphi$ is also ``strongly subharmonic."  However, this is not an interesting inequality --- note that it is satisfied by \emph{any} convex function --- consistent with the fact that the definition of $CCA_{\varphi}(U)$ (i.e., Definition \ref{D: cone comparison}) treats $\varphi$ only as a supersolution.  Indeed, notice that the relevant comparison property in the definition of $CCB_{\varphi}(U)$ involves the ``inverted" function $x \mapsto -\varphi(-x)$, not $\varphi$!\end{remark}
	
\subsection{Definition of Viscosity Solutions} To avoid any ambiguity, and since the differential operators considered here are generally discontinuous, it will be worthwhile to recall the definition of viscosity sub- and super-solutions.  

Toward that end, fix a function $\mathcal{F} : \mathbb{R}^{d} \times \mathbb{R} \times \mathbb{R}^{d} \times \mathcal{S}_{d} \to \mathbb{R}$.  

	\begin{definition} Given an open set $U \subseteq \mathbb{R}^{d}$, an upper semicontinuous function $u \in USC(U)$ is said to satisfy the differential inequality 
		\begin{equation*}
			-\mathcal{F}(x,u,Du,D^{2}u) \leq 0 \quad \text{in the viscosity sense in} \, \, U,
		\end{equation*}
	if the following condition is satisfied: for any $x_{0} \in U$, and any smooth function $\varphi$ defined in a neighborhood of $x_{0}$, if $u - \varphi$ has a local maximum at $x_{0}$, then
		\begin{equation*}
			-\mathcal{F}(x_{0},u(x_{0}),D\varphi(x_{0}),D^{2}\varphi(x_{0})) \leq 0.
		\end{equation*}
	Alternatively, in this case, $u$ is referred to as a \emph{viscosity subsolution} of the equation $-\mathcal{F}(x,w,Dw,D^{2}w) = 0$ in $U$.
	 \end{definition}
	 	 
The definition of viscosity supersolution is analogous, except local maxima are replaced by local minima:

	\begin{definition} Given an open set $U \subseteq \mathbb{R}^{d}$, a lower semicontinuous function $v \in LSC(U)$ is said to satisfy the differential inequality 
		\begin{equation*}
			-\mathcal{F}(x,v,Dv,D^{2}v) \geq 0 \quad \text{in the viscosity sense in} \, \, U,
		\end{equation*}
	if the following condition is satisfied: for any $x_{0} \in U$, and any smooth function $\varphi$ defined in a neighborhood of $x_{0}$, if $v - \varphi$ has a local minimum at $x_{0}$, then
		\begin{equation*}
			-\mathcal{F}(x_{0},v(x_{0}),D\varphi(x_{0}),D^{2}\varphi(x_{0})) \geq 0.
		\end{equation*}
	Alternatively, in this case, $v$ is referred to as a \emph{viscosity supersolution} of the equation $-\mathcal{F}(x,w,Dw,D^{2}w) = 0$ in $U$.
	 \end{definition}
	 
\begin{remark} Except where explicitly stated otherwise, all differential inequalities in this work should be interpreted to hold in the viscosity sense.  The words ``in the viscosity sense" will usually be omitted. \end{remark}

\begin{remark} In the definition of viscosity subsolution, one usually requires that the operator $\mathcal{F}$ be at least upper semicontinuous; while, for supersolutions, $\mathcal{F}$ is usually lower semicontinuous; see, for instance, \cite[Section 6]{barles}.  Proposition \ref{P: strongly superharmonic} is a notable exception to the rule.  \end{remark}

For more information on the theory of viscosity solutions, see \cite{user, primer, barles}.
		
\subsection{Properties of Finsler Norms}  Basic facts about Finsler norms will be used throughout the paper, in particular in the proof of Proposition \ref{P: strongly superharmonic}.  For the reader's convenience, the most important of these are recalled here.\footnote{There does not seem to be a standard reference for these (with proofs), but see \cite{bellettini} and \cite{m_souganidis}.}

First, to avoid any ambiguity, recall that a function $\varphi : \mathbb{R}^{d} \to [0,+\infty)$ is a Finsler norm if it is convex, positively homogeneous, and positive away from zero.  Equivalently, $\varphi$ is a Finsler norm if, for any $x,y \in \mathbb{R}^{d}$ and $\lambda > 0$, \footnote{A Finsler norm $\varphi$ is \emph{symmetric} if $\varphi(-x) = \varphi(x)$ for all $x \in \mathbb{R}^{d}$, in which case it is a norm according to the the usual definition.}
	\begin{equation*}
		\varphi(\lambda x) = \lambda \varphi(x), \quad \varphi(x + y) \leq \varphi(x) + \varphi(y) \quad \text{for each} \quad x,y \in \mathbb{R}^{d}, \, \, \lambda > 0,
	\end{equation*}
and
	\begin{equation*}
		\min \left\{ \varphi(q) \, \mid \, q \in \partial B_{1}(0) \right\} > 0.
	\end{equation*}

Given a Finsler norm $\varphi$, the dual norm $\varphi^{*}$ is defined by 
	\begin{equation*}
		\varphi^{*}(p) = \sup \left\{ \frac{\langle p, q \rangle}{\varphi(q)} \, \mid \, q \in \mathbb{R}^{d} \setminus \{0\} \right\} = \max \left\{ \langle p, q \rangle \, \mid \, q \in \{\varphi \leq 1\} \right\}.
	\end{equation*}
It is not hard to see that $\varphi^{*}$ is also a Finsler norm and $(\varphi^{*})^{*} = \varphi$.  

Since it is precisely the subdifferential $\partial \varphi^{*}$ of $\varphi^{*}$ that appears in the definitions \eqref{E: upper def operator} and \eqref{E: lower def operator} of $G_{\varphi}^{*}$ and $G_{*}^{\varphi}$, it will be essential to understand how $\varphi^{*}$ relates to $\varphi$.  Specifically, in the paper, certain basic duality identities, described next, will play a recurring role.

Toward that end, first, observe that, by definition of $\varphi^{*}$,
	\begin{equation} \label{E: euler}
		\langle p, q \rangle \leq \varphi(q) \varphi^{*}(p) \quad \text{for each} \quad p, q \in \mathbb{R}^{d}.
	\end{equation}
It turns out that the subdifferential $\partial \varphi$ consists precisely of the points of $\{\varphi^{*} = 1\}$ for which equality holds above:
	\begin{equation} \label{E: subdifferential basic identity}
		\partial \varphi(q) = \left\{p \in \{\varphi^{*} = 1\} \, \mid \, \langle p, q \rangle = \varphi(q) \right\} \quad \text{for each} \quad q \in \mathbb{R}^{d} \setminus \{0\}.
	\end{equation}
When $\varphi(q) = 1$, this says that the supporting hyperplanes to $\{\varphi \leq 1\}$ at $q$ can be conveniently parametrized by the elements of the dual ball $\{\varphi^{*} \leq 1\}$ saturating \eqref{E: euler}. 

Since $\varphi$ is itself the dual of $\varphi^{*}$, a corresponding formula describes $\partial \varphi^{*}$:
	\begin{equation} \label{E: subdifferential basic identity dual}
		\partial \varphi^{*}(p) = \left\{ q \in \{\varphi = 1\} \, \mid \, \langle p, q \rangle = \varphi^{*}(p) \right\} \quad \text{for each} \quad p \in \mathbb{R}^{d} \setminus \{0\}.
	\end{equation}
In terms of the geometry of $\varphi$, this says that $\partial \varphi^{*}(p)$ equals the boundary face of $\{\varphi \leq 1\}$ with normal vector parallel to $p$.\footnote{See Section \ref{S: faces} for a precise definition of a face of a convex set.}

Throughout the paper, it will be useful to remember that $\partial \varphi$ and $\partial \varphi^{*}$ are inverses of each other (up to normalization): for any $p, q \in \mathbb{R}^{d} \setminus \{0\}$, 
	\begin{equation} \label{E: inversion}
		\varphi^{*}(p)^{-1}p \in \partial \varphi(q) \quad \iff \quad \varphi(q)^{-1}q \in \partial \varphi^{*}(p).
	\end{equation}
This follows directly from \eqref{E: subdifferential basic identity} and \eqref{E: subdifferential basic identity dual}.  Furthermore, the subdifferential is positively zero-homogeneous:
	\begin{equation} \label{E: zero hom}
		\partial \varphi(\lambda q) = \partial \varphi(q) \quad \text{for each} \quad q \in \mathbb{R}^{d}, \, \, \lambda > 0.
	\end{equation}

\subsection{Strong Supersolution Property} This section proves Proposition \ref{P: strongly superharmonic}, asserting that the Finsler norm $\varphi$ is ``strongly" $\varphi$-infinity superharmonic.

Throughout the paper, it will be useful to exploit certain concepts from convex analysis.  Already in the proof of Proposition \ref{P: strongly superharmonic}, convex cones appear naturally.  To begin with, given a set $A \subseteq \mathbb{R}^{d}$, define the cone $\mathcal{C}(A)$ it generates by
	\begin{equation*}
		\mathcal{C}(A) = \{\lambda x \, \mid \, x \in A, \, \, \lambda \geq 0\}.
	\end{equation*}
Note that if $E \subseteq \mathbb{R}^{d}$ is convex, then $\mathcal{C}(E)$ is a convex cone, that is, 
	\begin{equation} \label{E: cone thing}
		\lambda_{1} x + \lambda_{2} y \in \mathcal{C}(E) \quad \text{for each} \quad x,y \in \mathcal{C}(E), \, \, \lambda_{1}, \lambda_{2} \geq 0.
	\end{equation}

Of particular interest in this work are the cones $\{\mathcal{C}(\partial \varphi^{*}(p)) \, \mid \, p \in \mathbb{R}^{d} \setminus \{0\}\}$ determined by $\partial \varphi^{*}$.  Since $\partial \varphi^{*}(p)$ is a convex set for any $p$, the set $\mathcal{C}(\partial \varphi^{*}(p))$ is a convex cone.  By \eqref{E: subdifferential basic identity dual} and positive homogeneity, these cones cover $\mathbb{R}^{d}$:
	\begin{equation} \label{E: covering}
		\mathbb{R}^{d} = \bigcup_{p \in \mathbb{R}^{d} \setminus \{0\}} \mathcal{C}(\partial \varphi^{*}(p)).
	\end{equation}

\begin{proof}[Proof of Proposition \ref{P: strongly superharmonic}] Fix $x_{0} \in \mathbb{R}^{d} \setminus \{0\}$ and let $\psi$ be a smooth function defined in a neighborhood of $x_{0}$.  Suppose that $\varphi - \psi$ has a local minimum at $x_{0}$.  

After a Taylor expansion at $x_{0}$, one finds
	\begin{equation} \label{E: taylor part}
		\varphi(x) \geq \varphi(x_{0}) + \langle D\psi(x_{0}), x - x_{0} \rangle + \frac{1}{2} \langle D^{2}\psi(x_{0})(x - x_{0}), x - x_{0} \rangle + o(\|x - x_{0}\|^{2}).
	\end{equation}
From this, it readily follows that $D\psi(x_{0}) \in \partial \varphi(x_{0})$, hence, by \eqref{E: subdifferential basic identity} and \eqref{E: inversion},
	\begin{gather}
		\varphi^{*}(D\psi(x_{0})) = 1, \quad \varphi^{*}(x_{0})^{-1} x_{0} \in \partial \varphi^{*}(D\psi(x_{0})), \quad \text{and} \nonumber \\
		\langle D\psi(x_{0}), q \rangle = 1 \quad \text{for each} \quad q \in \partial \varphi^{*}(D\psi(x_{0})). \label{E: useful subdifferential part 1}
	\end{gather}
Notice that this implies $x_{0} \in \mathcal{C}(\partial \varphi^{*}(D\psi(x_{0})))$.
	
Consider the cone $\mathcal{C}(\partial \varphi^{*}(D\psi(x_{0})))$ determined by the direction $D\psi(x_{0})$.  Observe that $\varphi$ is linear in $\mathcal{C}(\partial \varphi^{*}(D\psi(x_{0})))$, that is,
	\begin{equation} \label{E: linear thing}
		\varphi(x) = \langle D\psi(x_{0}), x \rangle \quad \text{for each} \quad x \in \mathcal{C}(\partial \varphi^{*}(D\psi(x_{0}))).
	\end{equation}
Indeed, if $q \in \partial \varphi^{*}(D\psi(x_{0}))$ and $\lambda \geq 0$, then \eqref{E: subdifferential basic identity dual} and \eqref{E: useful subdifferential part 1} together imply
	\begin{equation*}
		\varphi(\lambda q) = \lambda \varphi(q) = \lambda = \lambda \langle q, D\psi(x_{0}) \rangle = \langle \lambda q, D\psi(x_{0}) \rangle.
	\end{equation*}
Now \eqref{E: linear thing} follows from this and the definition of $\mathcal{C}(\partial \varphi^{*}(D\psi(x_{0})))$.

Next, notice that if $q \in \partial \varphi^{*}(D\psi(x_{0}))$ and $t > 0$, then $x_{0} + t q \in \mathcal{C}(\partial \varphi^{*}(D\psi(x_{0})))$.  This is immediate from \eqref{E: cone thing} since $\{x_{0}\} \cup \partial \varphi^{*}(D\psi(x_{0})) \subseteq \mathcal{C}(\partial \varphi^{*}(D\psi(x_{0})))$.

Finally, setting $x = x_{0} + tq$ with $t > 0$ and $q \in \partial \varphi^{*}(D\psi(x_{0}))$, the identity \eqref{E: linear thing} combines with \eqref{E: taylor part} and \eqref{E: useful subdifferential part 1} to yield
	\begin{align*}
		\varphi(x_{0}) + t = \varphi(x_{0} + t q) &\geq \varphi(x_{0}) + t\langle D\psi(x_{0}), q \rangle + \frac{t^{2}}{2} \langle D^{2}\psi(x_{0})q, q \rangle + o(t^{2}) \\
			&= \varphi(x_{0}) + t + \frac{t^{2}}{2} \langle D^{2}\psi(x_{0}) q, q \rangle + o(t^{2}).
	\end{align*}
This proves that 
	\begin{equation*}
		- \langle D^{2}\psi(x_{0}) q, q \rangle \geq 0,
	\end{equation*}
and, thus, by the arbitrariness of $q$ and the definition \eqref{E: upper def operator} of $G_{\varphi}^{*}$,
	\begin{equation*}
		- G^{*}_{\varphi}(D\psi(x_{0}),D^{2}\psi(x_{0})) = \min \left\{ -\langle D^{2}\psi(x_{0})q, q \rangle \, \mid \, q \in \partial \varphi^{*}(D\psi(x_{0})) \right\} \geq 0.
	\end{equation*}
\end{proof}

\subsection{Cone Comparison for $C^{2}$ Solutions} With Proposition \ref{P: strongly superharmonic} in hand, the proof of comparison for $C^{2}$ subharmonic functions is now relatively easily obtained.  The only other ingredient that will be used concerns approximation by strictly subharmonic functions.

	\begin{lemma} \label{L: strict subharmonic} Fix an open set $U \subseteq \mathbb{R}^{d}$, an $x_{0} \in U$, and a Finsler norm $\varphi : \mathbb{R}^{d} \to [0,+\infty)$.  If $u$ is a bounded upper semicontinuous function in $U$ such that $-G_{\varphi}^{*}(Du,D^{2}u) \leq 0$ in $U$ and $V \subset \subset U$, then there is a constant $\epsilon_{*} = \epsilon_{*}(u,V) > 0$ such that if $(u^{\epsilon})_{\epsilon > 0}$ are defined by 
		\begin{equation} \label{E: strictness}
			u^{\epsilon}(x) = u(x) + \frac{\epsilon}{2} u(x)^{2},
		\end{equation}
	then, for each $\epsilon \in (0,\epsilon_{*})$,
		\begin{equation*}
			\frac{1}{2} \epsilon \varphi^{*}(Du^{\epsilon})^{2} -G_{\varphi}^{*}(Du^{\epsilon},D^{2}u^{\epsilon})  \leq 0 \quad \text{in} \, \, V.
		\end{equation*}\end{lemma}
		
Note that above it is no loss of generality to consider bounded $\varphi$-infinity subharmonic functions.  Indeed, if $u$ is $\varphi$-infinity subharmonic in $U$, then it is bounded above in any compact subset by upper semi-continuity, and the functions $u_{K} = \max\{u,-K\}$ are thus bounded $\varphi$-infinity subharmonic functions that approximate $u$.  Furthermore, if $u_{K}$ satisfies comparison with cones from above for any $K \in \mathbb{N}$, then $u$ also has this property.

			\begin{proof} First, suppose that $u$ is $C^{2}$.  With $u^{\epsilon}$ as defined above, one differentiates to find, for all $x$ such that $1 + \epsilon u(x) > 0$,
				\begin{equation*}
					\varphi^{*}(Du^{\epsilon}(x)) = (1 + \epsilon u(x)) \varphi^{*}(Du(x)), \quad D^{2}u^{\epsilon}(x) = (1 + \epsilon u(x)) D^{2} u(x) + \epsilon Du(x) \otimes Du(x),
				\end{equation*}
			hence, by \eqref{E: subdifferential basic identity dual} and the definition \eqref{E: upper def operator},
				\begin{align*}
					-G_{\varphi}^{*}(Du^{\epsilon}(x),D^{2}u^{\epsilon}(x)) &= -\epsilon \varphi^{*}(Du(x))^{2} - (1 + \epsilon u(x)) G_{\varphi}^{*}(Du(x),D^{2}u(x)) \\
						&\leq -\epsilon (1 + \epsilon u(x))^{-2} \varphi^{*}(Du^{\epsilon}(x))^{2}.
				\end{align*}
			In particular, if $\epsilon$ is small enough,
				\begin{equation*}
					\frac{1}{2} \epsilon \varphi^{*}(Du^{\epsilon}(x))^{2} -G_{\varphi}^{*}(Du^{\epsilon}(x),D^{2}u^{\epsilon}(x)) \leq 0.
				\end{equation*}
			
			In general, even if $u$ is not $C^{2}$, one argues by passing the derivatives onto the test function and observing that, for any $\epsilon$ small enough, the transformation $s \mapsto s + \frac{\epsilon}{2} s^{2}$ restricts to a diffeomorphism in a neighborhood of the range of $u$.\footnote{Cf.~ Proof of \cite[Theorem 2.8]{evans_spruck}.}  \end{proof}

	\begin{proof}[Proof of Proposition \ref{P: smooth subharmonic}]  The proof proceeds by contradiction. Fix an open set $V \subset \subset U$.  Suppose that there is a point $x_{0} \in V$, constants $\lambda,\mu > 0$, and a vector $v \in \mathbb{R}^{d} \setminus V$ such that 
		\begin{align} 
			u(x_{0}) - \lambda \varphi(x_{0} - v) &= \max \left\{ u(x) - \lambda \varphi(x - v) \, \mid \, x \in \overline{V} \right\} \label{E: touching point 1} \\
				&> \max\left\{u(x) - \lambda \varphi(x - v) \, \mid \, x \in \partial V \right\}. \nonumber
		\end{align}
	Up to replacing $u$ by $\lambda^{-1} u$, there is no loss of generality assuming that $\lambda = 1$.  Further, it is safe to assume that $v = 0$ as otherwise one can replace $U$ by $U - v$, $x_{0}$ by $x_{0} - v$, $u$ by $u(\cdot + v)$, and so on.
	
	Let $(u^{\epsilon})_{\epsilon > 0}$ be the family defined by \eqref{E: strictness} and let $\epsilon_{*} = \epsilon_{*}(u,\mu) > 0$ be as in Lemma \ref{L: strict subharmonic}.  Choose a family $(x_{\epsilon})_{\epsilon > 0} \subseteq \overline{V}$ such that 
		\begin{equation*}
			u^{\epsilon}(x_{\epsilon}) - \varphi(x_{\epsilon}) = \max \left\{ u^{\epsilon}(x) - \varphi(x) \, \mid \, x \in \overline{V} \right\}.
		\end{equation*}
	By \eqref{E: touching point 1}, there is no loss of generality assuming (i.e., up to extraction) that there is an $x_{*} \in V$ such that $x_{\epsilon} \to x_{*}$ as $\epsilon \to 0^{+}$.  When $\epsilon$ is small enough, there holds $x_{\epsilon} \in V$, and, thus, by the lemma,
		\begin{equation*}
			\frac{1}{2} \epsilon \varphi^{*}(Du^{\epsilon}(x_{\epsilon}))^{2} -G_{\varphi}^{*}(Du^{\epsilon}(x_{\epsilon}),D^{2}u^{\epsilon}(x_{\epsilon}))  \leq 0.
		\end{equation*}
	At the same time, $\varphi - u^{\epsilon}$ has a local minimum at $x_{\epsilon}$.  Hence, by Proposition \ref{P: strongly superharmonic} and its proof,
		\begin{equation*}
			-G_{\varphi}^{*}(Du^{\epsilon}(x_{\epsilon}),D^{2}u^{\epsilon}(x_{\epsilon}) \geq 0, \quad \varphi^{*}(Du^{\epsilon}(x_{\epsilon})) = 1.
		\end{equation*}
	(Here is where smoothness of $u^{\epsilon}$ is used.) In particular,
		\begin{equation*}
			0 < \frac{1}{2} \epsilon \leq \frac{1}{2} \epsilon \varphi^{*}(Du^{\epsilon}(x_{0}))^{2} -G_{\varphi}^{*}(Du^{\epsilon}(x_{0}),D^{2}u^{\epsilon}(x_{0})) \leq 0, 
		\end{equation*}
	which is absurd. \end{proof}
	
Later in the paper, it will again be convenient to use the symmetries of the equation in conjunction with Lemma \ref{L: strict subharmonic} to reduce the proof of cone comparison to a simpler statement.  The following lemma makes this explicit:

	\begin{lemma} \label{L: symmetries} Let $\varphi$ be a Finsler norm in $\mathbb{R}^{d}$.  Then $\text{Sub}_{\varphi}(U) \subseteq CCA_{\varphi}(U)$ for all open sets $U$ if the following statement holds: for any open, bounded set $V \subseteq \mathbb{R}^{d} \setminus \{0\}$, any $\epsilon > 0$, and any $u^{\epsilon} \in USC(\overline{V})$ such that $\frac{1}{2} \epsilon \varphi^{*}(Du^{\epsilon})^{2} - G_{\varphi}^{*}(Du^{\epsilon},D^{2}u^{\epsilon}) \leq 0$ in $V$, there holds
		\begin{equation*}
			\max \left\{ u^{\epsilon}(x) - \varphi(x) \, \mid \, x \in \overline{V} \right\} = \max \left\{ u^{\epsilon}(x) - \varphi(x) \, \mid \, x \in \partial V \right\}.
		\end{equation*}
	\end{lemma}
	
%		\begin{proof} First, observe that if $u \in \text{Sub}_{\varphi}(U)$, then $u$ is upper semicontinuous, hence it is bounded above in any compact subset of $U$.  Furthermore, if $u_{K} = \max\{u,-K\}$ for $K \in \mathbb{N}$, then $u_{K} \in \text{Sub}_{\varphi}(U)$ and $u$ is bounded in any compact subset of $U$.  By Lemma \ref{L: } and the proof of Proposition \ref{P: } above, in order for $u_{K} \in CCA_{\varphi}(U)$ to hold, it suffices that the function $u_{K}^{\epsilon} \in CCA_{\varphi}(U)$ for small enough $\epsilon$.  \end{proof}
		
Lemma \ref{L: symmetries} follows readily from Lemma \ref{L: strict subharmonic}, the symmetries of the equation, and arguments similar to those presented in the proof of Proposition \ref{P: smooth subharmonic}.  Accordingly, the proof is omitted. 

\subsection{Two More Reductions} Before proceeding to introduce the tools that will be used to complete the proof of cone comparison, it will be worthwhile to note that, throughout the paper, there is no loss in focusing attention on subsolutions rather than supersolutions.  Indeed, it is possible to transform any $\varphi$-infinity superharmonic function into a function that is infinity subharmonic for the inverted norm $\underline{\varphi}$ defined below.  This easy computation is made precise in the next proposition for completeness:

	\begin{lemma} \label{L: reduction to subsolutions} Given any Finsler norm $\varphi$ in $\mathbb{R}^{d}$, let $\underline{\varphi}$ be the inverted norm, that is, the Finsler norm given by 
		\begin{equation*}
			\underline{\varphi}(q) = \varphi(-q) \quad \text{for each} \quad q \in \mathbb{R}^{d}.
		\end{equation*}
	Given any open set $U \subseteq \mathbb{R}^{d}$, if $\mathcal{N} : LSC(U) \to USC(U)$ is the map
		\begin{equation*}
			\mathcal{N}u(x) = -u(x) \quad \text{for each} \quad x \in U,
		\end{equation*}
	then
		\begin{equation*}
			\mathcal{N}(\text{Super}_{\varphi}(U)) = \text{Sub}_{\underline{\varphi}}(U), \quad \mathcal{N}(CCB_{\varphi}(U)) = CCA_{\underline{\varphi}}(U).
		\end{equation*}
	In particular, $\text{Super}_{\varphi}(U) \subseteq CCB_{\varphi}(U)$ if and only if $\text{Sub}_{\underline{\varphi}}(U) \subseteq CCA_{\underline{\varphi}}(U)$.
	\end{lemma}
	
		Since the assertions above follow from routine manipulations of the definitions involved, the details are left to the reader.
		
%Finally, note that, as already mentioned in the introduction, it is already well-known that $CCA_{\varphi}(U) \subseteq \text{Sub}_{\varphi}(U)$.  This is recalled in the next proposition:
%
%	\begin{prop} \label{P: cone comparison implies solution} Given any Finsler norm $\varphi$ in $\mathbb{R}^{d}$ and any open set $U \subseteq \mathbb{R}^{d}$, the identity $CCA_{\varphi}(U) \subseteq \text{Sub}_{\varphi}(U)$ holds. \end{prop}
%	
%		\begin{proof} A simple proof of this can be found in \cite{}, which is written in the case when $\varphi$ is symmetric (i.e., $\varphi(-q) = \varphi(q)$) but extends directly to the nonsymmetric case.  Alternatively, see \cite[Appendix A]{armstrong_crandall_julin_smart}. \end{proof}
		
%Accordingly, in the proof of Theorem \ref{T: cone comparison}, which will occupy the remainder of the first two parts of the paper, the focus will be on establishing that $\text{Sub}_{\varphi}(U) \subseteq CCA_{\varphi}(U)$.
	
	%%%%%%%%%%%%%%%%%%%%%%%%%%%%%%%%%%%%%%%%%%%%%%%%%%%%%%%%%%%%%%%%%%%%%%%%%%%%%%%%%%%%%%%%%%%%%%%%%%%%%%%%%%

\section{$C^{1,1}$ Approximation} \label{S: regularized}

This section describes a regularization procedure that will be used to approximate general Finsler norms by $C^{1,1}$ norms in a manner that is suitable for proving cone comparison principles.

In particular, given a Finsler norm $\varphi$, this section details the construction of a one-parameter family of Finsler norms $(\varphi_{\zeta})_{\zeta > 0}$ that are, in some sense, \emph{almost} $\varphi$-infinity superharmonic. 

Actually, in the next result, what amounts to a supersolution property will nonetheless be stated only in terms of the subdifferential $\partial \varphi^{*}_{\zeta}$.  This is not too surprising:\footnote{For analogous situations where strong control of the first derivative is used to analyze fully nonlinear second-order equations, see \cite[Proposition 1.2]{caffarelli_souganidis_rate} and \cite[Lemmas 4.1 and 4.2]{armstrong_smart_revisited}.} the difficulty in the equation stems from $\partial \varphi^{*}$ so there is reason to hope that strong control of the first derivatives suffices.

\regularizednorm*

\begin{remark} \label{R: lower bound} Later in the paper, it will be useful to know that, for any $\zeta \in (0,1)$, there is a lower bound
	\begin{equation*}
		\delta(\varphi)^{-1} \|x\| \leq \varphi_{\zeta}(x) \quad \text{for each} \quad x \in \mathbb{R}^{d}.
	\end{equation*}
Indeed, this follows directly from \eqref{E: nondegeneracy condition zeta} and positive homogeneity. \end{remark}

%	\begin{theorem} \label{T: regularized_norm} If $\varphi$ is any Finsler norm in $\mathbb{R}^{d}$, then there is a family of Finsler norms $(\varphi_{\zeta})_{\zeta > 0} \subseteq C^{1,1}_{\text{loc}}(\mathbb{R}^{d} \setminus \{0\})$ such that 
%		\begin{itemize}
%			\item[(i)] $\{\varphi_{\zeta} \leq 1\}$ satisfies a uniform interior ball condition with radius $\zeta$; 
%			\item[(ii)] for any $p \in \mathbb{R}^{d} \setminus \{0\}$, there holds
%				\begin{equation} \label{E: subdifferential version of infinity harmonic}
%					\partial \varphi_{\zeta}^{*}(p) = \partial \varphi^{*}(p) + \frac{\zeta p}{\|p\|};
%				\end{equation}
%			\item[(iii)] $\varphi_{\zeta} \to \varphi$ locally uniformly as $\zeta \to 0^{+}$;
%			\item[(iv)] there is a $\kappa(\varphi) > 0$ such that, for any $\zeta \in (0,1)$,
%				\begin{equation*}
%			\varphi^{*}(D\varphi_{\zeta}(x)) \geq \kappa(\varphi) \quad \text{for each} \quad x \in \mathbb{R}^{d} \setminus \{0\}; \quad \text{and}
%				\end{equation*} 
%			\item[(v)] there is a $c(\varphi) > 0$ such that, for any $\zeta \in (0,1)$,
%				\begin{equation} \label{E: weak second deriv bound}
%					\varphi_{\zeta}(x) D^{2}\varphi_{\zeta}(x) \leq c(\varphi) \zeta^{-1} \text{Id} \quad \text{for almost every} \quad x \in \mathbb{R}^{d}.
%				\end{equation}
%		\end{itemize}
%\end{theorem}

One way to appreciate the utility of the identity \eqref{E: subdifferential version of infinity harmonic} is as follows.  Observe that, for any $(p,X) \in (\mathbb{R}^{d} \setminus \{0\}) \times \mathcal{S}_{d}$,
	\begin{align} \label{E: shifted operator key identity}
		G_{\varphi}^{*}(p,X) &= \max \left\{ \langle X[q - \zeta \|p\|^{-1} p], q - \zeta \|p\|^{-1} p \rangle \, \mid \, q \in \partial \varphi_{\zeta}(p) \right\} =: \mathcal{G}_{\varphi_{\zeta}}^{*,\alpha}(p,X),
	\end{align}
The operator $\mathcal{G}_{\varphi_{\zeta}}^{*}$ defined implicitly above is effectively a perturbation of the Finsler infinity Laplacian $G^{*}_{\varphi_{\zeta}}$ associated with $\varphi_{\zeta}$.  Thus, this identity shows that, up to introducing a perturbation in the operator, $\varphi$-infinity sub- and superharmonic functions can be controlled by using the smoother norm $\varphi_{\zeta}$ as a barrier instead of $\varphi$.

As further motivation for the key identity \eqref{E: subdifferential version of infinity harmonic} of the previous theorem, consider the next result (which, however, will not be needed elsewhere): analogously to Proposition \ref{P: strongly superharmonic}, it suggests that $\varphi_{\zeta}$ is \emph{almost} $\varphi$-infinity superharmonic in a strong sense.

	\begin{prop} \label{P: approximate supersolution} There is a constant $c(\varphi) > 0$ such that $-G_{\varphi}^{*}(D\varphi_{\zeta},D^{2}\varphi_{\zeta}) \geq - c(\varphi) \zeta \varphi_{\zeta}^{-1}$ holds in $\mathbb{R}^{d} \setminus \{0\}$ with respect to convex test functions.  More precisely, if $x_{*} \in \mathbb{R}^{d} \setminus \{0\}$ and $\psi$ is a smooth convex function defined in a neighborhood of $x_{*}$ that touches $\varphi_{\zeta}$ from below at $x_{*}$, then
		\begin{equation*}
			- G_{\varphi}^{*}(D\psi(x_{*}), D^{2} \psi(x_{*})) \geq -c(\varphi) \zeta \varphi_{\zeta}(x_{*})^{-1}.
		\end{equation*}  
	\end{prop}
	
Due to the discontinuities of $G_{\varphi}^{*}$, it is not obvious how to remove the convexity restriction on the test function.  As already suggested in the previous section, since $G_{\varphi}^{*}$ is merely upper semicontinuous, lower bounds such as the one above are not stable under perturbations.  This instability will be overcome through a series of delicate constructions in what follows.

The constant $c(\varphi)$ in the previous proposition comes from a relatively basic Hessian estimate, which is stated next.  Indeed, the interior ball condition of radius $\zeta$ satisfied by $\{\varphi_{\zeta} \leq 1\}$ roughly implies that $\varphi_{\zeta} D^{2} \varphi_{\zeta} \leq \zeta^{-1} \text{Id}$ in the sense of distributions.  In fact, the estimate is not quite so clean, but it is true up to a $\zeta$-independent constant.
	
	\begin{restatable}{lemma}{basichessianbound} \label{L: basic hessian bound} Given any Finsler norm $\varphi$, there is a constant $c(\varphi) > 0$ such that if $(\varphi_{\zeta})_{\zeta > 0}$ are the approximations of Theorem \ref{T: regularized_norm}, then, for any $\zeta \in (0,1)$,
		\begin{equation} \label{E: hessian estimate kind of silly}
			\varphi_{\zeta} D^{2}\varphi_{\zeta} \leq c(\varphi) \zeta^{-1} \text{Id} \quad \text{in the distributional sense.}
		\end{equation}
	\end{restatable}
	
The proof of the lemma is deferred to Section \ref{S: hessian bound approximation} below, where it is presented after certain basic regularity results for convex positively one-homogeneous functions are established.  

	\begin{proof}[Proof of Proposition \ref{P: approximate supersolution}]  Henceforth, to lighten the notation, it will frequently be useful to denote by $Q_{X}$ the quadratic form associated with a symmetric matrix $X$, hence $Q_{X}(v) = \langle Xv, v \rangle$ for any $v \in \mathbb{R}^{d}$.  
	
	First, by Proposition \ref{P: strongly superharmonic},
		\begin{equation*}
			-Q_{D^{2}\psi(x_{*})}(q) = -\langle D^{2}\psi(x_{*}) q, q \rangle \geq 0 \quad \text{for each} \quad q \in \partial \varphi^{*}_{\zeta}(D\psi(x_{0})).
		\end{equation*}
	Yet $\psi$ is convex so $D^{2}\psi(x_{*}) \geq 0$.  Therefore, the previous bound is actually equivalent to an apparently stronger statement:
		\begin{equation} \label{E: key vanishing part 1}
			D^{2}\psi(x_{*}) q = 0 \quad \text{for each} \quad q \in \partial \varphi^{*}_{\zeta}(D\psi(x_{*})).
		\end{equation}
	
	Next, recall that \eqref{E: hessian estimate kind of silly} holds not only distributionally but also in the viscosity sense; the two are equivalent.\footnote{See \cite[Lemma 1]{alvarez_lasry_lions}, \cite[Theorem 1]{oberman}, or \cite[Proposition 5.1]{bardi_dragoni}.}  Thus, since $\psi$ touches $\varphi_{\zeta}$ from below at $x_{*}$,
		\begin{equation*}
			-D^{2}\psi(x_{*}) \geq - c(\varphi) \varphi_{\zeta}(x_{*})^{-1} \zeta^{-1}.
		\end{equation*}
	Finally, appeal to the definition of $G_{\varphi}^{*}$ and \eqref{E: shifted operator key identity}:
		\begin{align*}
			-G_{\varphi}^{*}(D\psi(x_{*}), D^{2}\psi(x_{*})) &= - \max \left\{ Q_{D^{2}\psi(x_{*})}(\bar{q}) \, \mid \, \bar{q} \in \partial \varphi(D\psi(x_{*})) \right\} \\
				&= - \max_{q \in \partial \varphi_{\zeta}(D\psi(x_{*}))} \left\{ Q_{D^{2}\psi(x_{*})}(q - \zeta \|D\psi(x_{*})\|^{-1} D\psi(x_{*}))  \right\}
		\end{align*}
	Due to \eqref{E: key vanishing part 1}, if $q \in \partial \varphi_{\zeta}(D\psi(x_{*}))$, then
		\begin{equation*}
			Q_{D^{2}\psi(x_{*})}(q - \zeta \|D\psi(x_{*})\|^{-1} D\psi(x_{*})) = \zeta^{2} \|D\psi(x_{*})\|^{-2} Q_{D^{2}\psi(x_{*})}(D\psi(x_{*}))
		\end{equation*}
	Thus, by the Hessian bound on $\psi$ at $x_{*}$,
		\begin{align*}
				-G_{\varphi}^{*}(D\psi(x_{*}), D^{2}\psi(x_{*})) &= - \zeta^{2}\|D\psi(x_{*})\|^{-2} \langle D^{2} \psi(x_{*}) D\psi(x_{*}), D\psi(x_{*}) \rangle \\
				&\geq - c(\varphi) \zeta^{-1} \varphi_{\zeta}(x_{*})^{-1} \zeta^{2} = -c(\varphi) \zeta \varphi_{\zeta}(x_{*})^{-1}.
		\end{align*}
	\end{proof}

%%%%%%%%%%%%%%%%%%%%%%%%%%%%%%%%%%%%%%%%%%%%%%%%%%%%%%%%%%%%%%%%%%%%%%%%%%%%%%%%%%%%%%%%%%%%%%%%%%%%%%%%%%%%%%%%
	
\subsection{Reduction to $C^{1,1}$ Finsler Norms} \label{S: c11 reduction} Before delving into the proof of the approximation theorem (Theorem \ref{T: regularized_norm}), here is the pay-off.  Using the approximations $(\varphi_{\zeta})_{\zeta > 0}$, it turns out that one can prove comparison with cones for  $\varphi$-infinity harmonic functions using a generalized cone comparison-type result for $C^{1,1}$ Finsler norms.  

As already suggested above, it is convenient to reformulate $\varphi$-infinity subharmonicity in terms of $\varphi_{\zeta}$.  First, recall that $Q_{X}(v) = \langle Xv, v \rangle$ for a symmetric matrix $X$ and $v \in \mathbb{R}^{d}$.  For an arbitrary Finsler norm $\varphi$ and $\alpha \in \mathbb{R}$, define the shifted operators $\mathcal{G}_{\varphi}^{*,\alpha}$ and $\mathcal{G}_{*,\alpha}^{\varphi}$ by
	\begin{align*}
		\mathcal{G}_{\varphi}^{*,\alpha}(p,X) &= \left\{ \begin{array}{r l}
					\max \left\{ Q_{X}(q - \alpha \|p\|^{-1} p) \, \mid \, q \in \partial \varphi^{*}(p) \right\}, & \text{if} \, \, p \neq 0, \\
					\limsup_{0 \neq p' \to 0} \mathcal{G}_{\varphi}^{*,\alpha}(p',X), & \text{otherwise.}
				\end{array} \right. \\ 
		\mathcal{G}_{*,\alpha}^{\varphi}(p,X) &= \left\{ \begin{array}{r l}
							\min \left\{ Q_{X}(q - \alpha \|p\|^{-1} p) \, \mid \, q \in \partial \varphi^{*}(p) \right\}, & \text{if} \, \, p \neq 0, \\
					\liminf_{0 \neq p' \to 0} \mathcal{G}_{*,\alpha}^{*}(p',X), & \text{otherwise.}
				\end{array} \right.
	\end{align*}
Notice that $G_{\varphi}^{*} \equiv \mathcal{G}_{\varphi}^{*,0}$ and $G_{*}^{\varphi} \equiv \mathcal{G}_{*,0}^{\varphi}$.  

The shifted operators provide a change of perspective.  To appreciate this, observe that, for any Finsler norm $\varphi$ and any $\zeta > 0$, the following relations hold as a consequence of \eqref{E: subdifferential version of infinity harmonic}:
	\begin{gather} \label{E: key transformation}
		G_{\varphi}^{*} \equiv \mathcal{G}_{\varphi_{\zeta}}^{*,\zeta} \quad \text{and} \quad G_{*}^{\varphi} \equiv \mathcal{G}_{*,\zeta}^{\varphi_{\zeta}}.
	\end{gather}
As hinted already above, this identity is useful since it suggests replacing $\varphi$ by $\varphi_{\zeta}$, which is easier to analyze due to its $C^{1,1}$ regularity.  

Specifically, the next result, which exploits the fact that $\{\varphi_{\zeta} \leq 1\}$ is of class $C^{1,1}$, will be used in conjunction with \eqref{E: key transformation} to show that $\varphi$-infinity subharmonic functions can be compared to $\varphi_{\zeta}$.

	\begin{prop} \label{P: intermediate cone comparison} Let $\varphi$ be a Finsler norm in $\mathbb{R}^{d}$ and let $(\varphi_{\zeta})_{\zeta > 0}$ be the regularized norms of Theorem \ref{T: regularized_norm}.
	
	There is a constant $\tilde{\Gamma}(\varphi) > 0$ such that, for any $\zeta \in (0,1)$, $\alpha \in \mathbb{R}$, $\epsilon \geq 0$, and bounded open set $W \subseteq \mathbb{R}^{d}$ with $0 \notin W$, if $u \in USC(\overline{W})$ satisfies
		\begin{equation} \label{E: perturbed equation}
			\frac{1}{2} \epsilon \varphi^{*}(Du)^{2} - \mathcal{G}_{\varphi_{\zeta}}^{*,\alpha}(Du,D^{2}u) \leq 0 \quad \text{in} \, \, W,
		\end{equation}
	then either $u - \varphi_{\zeta}$ attains its maximum on $\partial W$ or else there is a triple $(x_{*},X_{*},q_{*}) \in W \times \mathcal{S}_{d} \times \{\varphi_{\zeta} = 1\}$ such that 
		\begin{gather*}
			\frac{1}{2} \epsilon \varphi^{*}(D\varphi_{\zeta}(x_{*}))^{2} - Q_{X_{*}}(q_{*} - \alpha \|D\varphi_{\zeta}(x_{*})\|^{-1} D\varphi_{\zeta}(x_{*})) \leq 0, \\
			X_{*}q_{*} = 0, \quad \text{and} \quad 0 \leq \varphi_{\zeta}(x_{*}) X_{*} \leq \tilde{\Gamma}(\varphi) \zeta^{-1} \text{Id}.
		\end{gather*}\end{prop}  
		
In effect, the proposition says that, as far as the differential inequality \eqref{E: perturbed equation} is concerned, $\varphi_{\zeta}$ can be treated as a test function, even though it is not smooth.  The price to pay compared to, say, Proposition \ref{P: strongly superharmonic} or the $C^{2}$ case described in \cite{m_souganidis}\footnote{Cf.\ Proposition 7 in that reference and the paragraph immediately following its statement.} is the result does \emph{not} say that $G_{\varphi_{\zeta}}^{*}(D\varphi_{\zeta}(x_{*}),X_{*}) = 0$, but only that $Q_{X_{*}}(q_{*}) = 0$ for \emph{some} $q_{*} \in \partial \varphi^{*}_{\zeta}(D\varphi_{\zeta}(x_{*}))$.  
		
	\begin{proof} The proposition follows directly from the more general result, Theorem \ref{T: technical core}, presented in Section \ref{S: c11 setting} below.  The constant $\tilde{\Gamma}(\varphi)$ is given by $\tilde{\Gamma}(\varphi) = \Gamma(\delta(\varphi))$, where $\delta(\varphi)$ is the constant from Theorem \ref{T: regularized_norm}, and the function $\Gamma$ is as in the statement of Theorem \ref{T: technical core}. \end{proof}
	
Part 2 of the paper will take up the proof of generalized cone comparison principles like the one above.  Taking it for granted for now, here is the proof of cone comparison for $\varphi$-infinity sub- and superharmonic functions.
	
	\begin{proof}[Proof of Theorem \ref{T: cone comparison}] By Lemma \ref{L: symmetries}, the proof reduces to establishing that, given any bounded open set $V \subseteq \mathbb{R}^{d} \setminus \{0\}$ and any $\epsilon > 0$, if $u \in USC(\overline{V})$ satisfies 
	\begin{equation*}
		\frac{1}{2} \epsilon \varphi^{*}(Du)^{2} - G_{\varphi}^{*}(Du,D^{2}u) \leq 0 \quad \text{in} \, \, V,
	\end{equation*}
then
	\begin{equation*}
		\max \left\{ u(x) - \varphi(x) \, \mid \, x \in \overline{V} \right\} = \max \left\{ u(x) - \varphi(x) \, \mid \, x \in \partial V \right\}.
	\end{equation*}
	
It is simplest to argue by contradiction: suppose that
	\begin{equation*}
		\max \left\{ u(x) - \varphi(x) \, \mid \, x \in \overline{V} \right\} > \max \left\{ u(x) - \varphi(x) \, \mid \, x \in \partial V \right\}.
	\end{equation*}
In particular, since $u$ is upper semicontinuous, there is an $r > 0$ such that if $\overline{B}_{r}(\partial V) = \bigcup_{y \in \partial V} \overline{B}_{r}(y)$, then
	\begin{equation*}
		\max \left\{ u(x) - \varphi(x) \, \mid \, x \in \overline{V} \right\} > \max \left\{ u(x) - \varphi(x) \, \mid \, x \in \overline{B}_{r}(\partial V) \right\}.
	\end{equation*}
Note that since $0 \notin V$, it follows that $V \setminus \overline{B}_{r}(\partial V) \subseteq \mathbb{R}^{d} \setminus \overline{B}_{r}(0)$. 

Let $(\varphi_{\zeta})_{\zeta > 0}$ be the approximating norms of Theorem \eqref{T: regularized_norm}.  Since $\varphi_{\zeta} \to \varphi$ uniformly in $\overline{V}$ as $\zeta \to 0^{+}$, it follows that there is a $\zeta_{*} \in (0,1)$ such that, for any $\zeta < \zeta_{*}$,
	\begin{equation} \label{E: what we want to prove gosh}
		\max \left\{ u(x) - \varphi_{\zeta}(x) \, \mid \, x \in \overline{V} \right\} > \max \left\{ u(x) - \varphi_{\zeta}(x) \, \mid \, x \in \overline{B}_{r}(\partial V) \right\}.
	\end{equation}

Fix an arbitrary $\zeta < \zeta_{*}$.  In view of \eqref{E: key transformation}, $u$ satisfies
	\begin{equation*}
		\frac{1}{2} \epsilon \varphi^{*}(Du)^{2} - \mathcal{G}_{\varphi_{\zeta}}^{*,\zeta}(Du,D^{2}u) \leq 0 \quad \text{in} \, \, V.
	\end{equation*}
Therefore, by Proposition \ref{P: technical core eaten} applied with $W = V \setminus \overline{B}_{r}(\partial V)$, there is a triple $(x_{*},X_{*},q_{*}) \in (V \setminus \overline{B}_{r}(\partial V)) \times \mathcal{S}_{d} \times \{\varphi_{\zeta} = 1\}$ such that
	\begin{gather*}
		\frac{1}{2} \epsilon \varphi^{*}(D\varphi_{\zeta}(x_{*}))^{2} - Q_{X_{*}}(q_{*} - \zeta \|D\varphi_{\zeta}(x_{*})\|^{-1} D\varphi_{\zeta}(x_{*})) \leq 0, \\
		X_{*} q_{*} = 0, \quad \text{and} \quad 0 \leq \varphi_{\zeta}(x_{*}) X_{*} \leq c(\varphi) \zeta^{-1} \text{Id}.
	\end{gather*}
Here it is important to emphasize specifically that $x_{*} \notin \overline{B}_{r}(\partial V)$ so $\|x_{*}\| > r$. 

Due to the bound on $X_{*}$ and the fact that $q_{*}$ lies in its kernel, one can write
	\begin{align*}
		Q_{X_{*}}(q_{*} - \zeta \|D\varphi_{\zeta}(x_{*})\|^{-1} D\varphi_{\zeta}(x_{*})) &= \zeta^{2} Q_{X_{*}}(\|D\varphi_{\zeta}(x_{*})\|^{-1} D\varphi_{\zeta}(x_{*})) \\
			&\leq c(\varphi) \varphi_{\zeta}(x_{*})^{-1} \zeta.
	\end{align*}

Harvesting the above inequalities, along with the bound $\varphi^{*}(D\varphi_{\zeta}(x_{*})) \geq \kappa(\varphi)$ from Theorem \ref{T: regularized_norm}, one finds
	\begin{equation*}
		\frac{1}{2} \epsilon \kappa(\varphi)^{2} - c(\varphi) \varphi_{\zeta}(x_{*})^{-1} \zeta \leq 0.
	\end{equation*}
In view of the fact that $\|x_{*}\| > r$, this can be rearranged to give
	\begin{equation*}
		\min \left\{ \varphi_{\zeta}(x) \, \mid \, x \in \partial B_{r}(0) \right\} = \min \left\{\varphi_{\zeta}(x) \, \mid \, x \in \mathbb{R}^{d} \setminus \overline{B}_{r}(0) \right\} \leq \frac{2 c(\varphi)}{\kappa(\varphi)^{2}} \cdot \frac{\zeta}{\epsilon}
	\end{equation*}
In the limit $\zeta \to 0^{+}$, this last inequality becomes the conclusion that 
	\begin{equation*}
		\min \left\{ \varphi(x) \, \mid \, x \in \partial B_{r}(0) \right\} \leq 0,
	\end{equation*}
but this contradicts the fact that $\varphi$ is positive away from zero. \end{proof}

	\begin{remark} The reader can check that if $\{\varphi \leq 1\}$ itself already has $C^{1,1}$ regularity, then, in the above proof, it is not necessary to use the approximations $(\varphi_{\zeta})_{\zeta > 0}$.  Indeed, in this case, one can directly apply Theorem \ref{T: technical core} below to the operator $G_{\varphi}^{*} \equiv \mathcal{G}_{\varphi}^{*,0}$. \end{remark}
	
	\begin{remark} The proof above is somewhat delicate since it was necessary to mitigate the blow-up of $D^{2}\varphi_{\zeta}$ as $\zeta \to 0^{+}$.  In fact, using the results of Part \ref{Part: applications} below, one can give an alternative proof of Theorem \ref{T: cone comparison} involving $\varphi_{\zeta}$ with just a single value of $\zeta \in (0,1)$ --- there is no need to send $\zeta$ to zero.  This can be seen by applying Proposition \ref{P: doubling variables} below in conjunction with Lemma \ref{L: strict subharmonic}.  The price to pay is the proof of Proposition \ref{P: doubling variables} involves the maximum principle for semicontinuous functions, whereas the arguments of Part \ref{Part: c11 cone comparison} avoid this.    \end{remark}

%%%%%%%%%%%%%%%%%%%%%%%%%%%%%%%%%%%%%%%%%%%%%%%%%%%%%%%%%%%%%%%%%%%%%%%%%%%%%%%%%%%%%%%%%%%%%%%%%%%%%%%%%%

\subsection{Construction of $E_{\zeta} = \{\varphi_{\zeta} \leq 1\}$}  The construction of $(\varphi_{\zeta})_{\zeta > 0}$ is fundamentally geometric, leveraging the intuition that the basic object ought to be the unit ball $\{\varphi_{\zeta} \leq 1\}$ rather than $\varphi_{\zeta}$ itself.   Accordingly, the presentation will begin with the construction of the set $E_{\zeta} := \{\varphi_{\zeta} \leq 1\}$.  This is possible due to the well-known bijective correspondence between Finsler norms and their unit balls. 

Given a $\zeta > 0$, let $E_{\zeta}$ be the set of points in $\mathbb{R}^{d}$ a distance at most $\zeta$ from $\{\varphi \leq 1\}$, that is,
	\begin{equation*}
		E_{\zeta} = \bigcup_{\bar{q} \in \{\varphi \leq 1\}} \{q \in \mathbb{R}^{d} \, \mid \, \|q - \bar{q}\| \leq \zeta \} = \bigcup_{\bar{q} \in \{\varphi \leq 1\}} \overline{B}_{\zeta}(\bar{q}).
	\end{equation*}
This operation regularizes $\{\varphi \leq 1\}$ in a manner that, in some sense, preserves the qualitative features of the boundary, as shown in Figure \ref{F: preserves curvature}.  

\begin{remark} Indeed, in the convex analysis literature, the map sending $\{\varphi \leq 1\}$ to $E_{\zeta}$ provides a means of defining a weak notion of curvature for arbitrary convex sets, see, for instance,  \cite[Section 2.6]{schneider} and the references cited therein.  For the purposes of this paper, the idea to have in mind is that, roughly speaking, the curvature of $E_{\zeta}$ converges weakly to the curvature of $\{\varphi \leq 1\}$ as $\zeta \to 0^{+}$.  In terms of the norms $(\varphi_{\zeta})_{\zeta > 0}$, Section \ref{S: c11 reduction} already showed that, in some sense, the Finsler infinity Laplace operators also converge in the limit.

Finally, note that this regularization procedure, which amounts to a geometric form of sup-convolution, is well-known in the viscosity solutions literature, notably in Jensen's classical work \cite{jensen_maximum_principle}, the literature on level-set PDE, such as \cite{gao_kim}, and the ``easy" proof of uniqueness for the (Euclidean) infinity Laplacian in \cite{armstrong_smart_easy_proof}.  \end{remark}

\begin{figure} 
\begin{tikzpicture}[scale=1.5]

    \fill[fill=blue!10] (0,0.4in) arc (150:210:0.8in);
    \fill[fill=blue!10] (0,-0.4in) arc (-30:30:0.8in);
    
    \draw[dashed] (0.1299in,0.4750in) arc (30:150:0.15in);
    \draw[dashed] (0.1299in,-0.4750in) arc (-30:-150:0.15in);

    \draw[dashed] (0.1299in,0.4750in) arc (30:-30:0.95in);
    \draw[dashed] (-0.1299in,0.4750in) arc (150:210:0.95in);

\end{tikzpicture}
\hspace{0.2in}
\begin{tikzpicture}[scale=1.5]

    \fill[fill=blue!10] (-0.2in,0) -- (0,0.4in) -- (1.2in,0.4in) -- (1.4in,0) -- (1.2in,-0.4in) -- (0,-0.4in) -- (-0.2in,0);
    
    \draw[dashed] (0,0.55in) -- (1.2in,0.55in);
    \draw[dashed] (0,-0.55in) -- (1.2in,-0.55in);
    
    \draw[dashed] (1.5342in,0.0671in) arc (26.5651:-26.5651:0.15in);
    \draw[dashed] (-0.3342in,0.0671in) arc (180-26.5651:180+26.5651:0.15in);
    
    \draw[dashed] (1.5342in,0.0671in) -- (1.3342in,0.4671in);
    \draw[dashed] (1.5342in,-0.0671in) -- (1.3342in,-0.4671in);
    \draw[dashed] (-0.3342in, 0.0671in) -- (-0.1342in,0.4671in);
    \draw[dashed] (-0.3342in,-0.0671in) -- (-0.1342in,-0.4671in);
    
    \draw[dashed] (1.2in,0.55in) arc (90:26.5651:0.15in);
    \draw[dashed] (1.2in,-0.55in) arc (-90:-26.5651:0.15in);
    \draw[dashed] (0in,0.55in) arc (90:180-26.5651:0.15in);
    \draw[dashed] (0in,-0.55in) arc (270:180+26.5651:0.15in);
  	
\end{tikzpicture}
\hspace{0.2in}
\begin{tikzpicture}[scale=1.5]

    \fill[fill=blue!10] (0,0.2in) arc (180:0:0.2in) -- (0.4in,-0.2in) arc (0:-180:0.2in) -- (0,0.2in);
    
    \draw[dashed] (-0.15in,0.2in) arc (180:0:0.35in) -- (0.55in,-0.2in) arc (0:-180:0.35in) -- (-0.15in,0.2in);
  	
\end{tikzpicture}
\caption{The shaded regions are the sets $\{\varphi \leq 1\}$ for various choices of $\varphi$, while $E_{\zeta}$ is the region bounded by the dashed curves.}
\label{F: preserves curvature}
\end{figure}
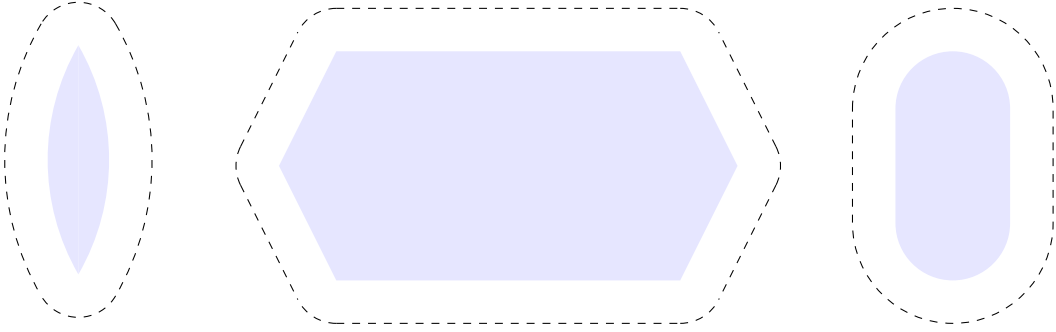

%\begin{remark} In fact, the construction and properties of $E_{\zeta}$ presented in this section are already known in the convex analysis literature, where they provide a notion of ``generalized curvatures" for convex sets; see, for instance, \cite[Section 2.6]{schneider} and \cite{hug}.\end{remark}

The next few results describe basic geometric properties of the set $E_{\zeta}$ as it relates to $\{\varphi \leq 1\}$.  These were already discovered previously in the convex analysis literature (cf.\ \cite[Section 2.6]{schneider}); the proofs are presented for the sake of completeness.

\begin{prop} \label{P: structure theorem} For any $\zeta > 0$, the set $E_{\zeta}$ has the following properties:
	\begin{itemize}
		\item[(i)] $E_{\zeta}$ is a closed convex set of class $C^{1,1}$ satisfying a uniform interior ball condition with radius $\zeta$.
		\item[(ii)] There is a continuous function $\mathcal{Q}_{\zeta} : \partial E_{\zeta} \to \{\varphi = 1\}$ such that, for any $q_{*} \in \partial E_{\zeta}$, the point $\bar{q}_{*} = \mathcal{Q}_{\zeta}(q)$ is the unique point in $\{\varphi = 1\}$ closest to $q_{*}$, that is,
				\begin{equation*}
					\{\bar{q}_{*}\} = \{\bar{q} \in \{\varphi \leq 1\} \, \mid \, \|\bar{q} - q_{*}\| = \zeta \}.
				\end{equation*}
		\item[(iii)] For any $q \in \partial E_{\zeta}$, if $\bar{q} = \mathcal{Q}_{\zeta}(q)$ as before, then 
				\begin{equation*}
					\varphi^{*}(q_{*} - \bar{q}_{*})^{-1} (q_{*} - \bar{q}_{*}) \in \partial \varphi(\bar{q}_{*}), \quad
					\langle q_{*} - \bar{q}_{*}, q_{*} \rangle = \sup \left\{ \langle q_{*} - \bar{q}_{*}, q \rangle \, \mid \, q \in E_{\zeta} \right\}.
				\end{equation*}
	\end{itemize}\end{prop}

At this stage, it will be important to note a certain characterization of convex sets with $C^{1,1}$ boundary, which will be used repeatedly throughout the paper.  To state, it, first, recall the definition of interior and exterior ball conditions: a set $A \subseteq \mathbb{R}^{d}$ is said to satisfy a \emph{uniform interior ball condition (resp.\ uniform exterior ball condition) of radius $\zeta > 0$} if, for each $x \in \partial A$, there is an $e \in S^{d-1}$ such that 
	\begin{equation*}
		\overline{B}_{\zeta}(x - \zeta e) \subseteq \overline{A} \quad \text{(resp.} \quad \overline{B}_{\zeta}(x + \zeta e) \subseteq \mathbb{R}^{d} \setminus \text{int}(A).\text{)}
	\end{equation*}
The next lemma, which is well-known, asserts that a convex set is of class $C^{1,1}$ if and only if it satisfies a uniform interior ball condition.

	\begin{lemma} \label{L: c11 set thing} A compact convex set $C \subseteq \mathbb{R}^{d}$ with nonempty interior is a submanifold of $\mathbb{R}^{d}$ with boundary of class $C^{1,1}$ if and only if there is a $\zeta > 0$ such that $C$ satisfies a uniform interior ball condition of radius $\zeta$. \end{lemma}
	
The lemma follows from a similar characterization of domains of class $C^{1,1}$, made precise, for instance, in \cite{lewicka_peres}.  Hence only a short proof will be given.  Nonetheless, here and later in the paper, the notion of supporting hyperplane will be needed.  Recall that if $A \subseteq \mathbb{R}^{d}$ and $e \in S^{d-1}$, then $A$ is said to possess a \emph{supporting hyperplane} with normal $e$ at a point $x \in \partial A$ if the following inclusion holds:
	\begin{align*}
		A \subseteq \{y \in \mathbb{R}^{d} \, \mid \, \langle y - x, e \rangle \leq 0\}.
	\end{align*}
	
		\begin{proof} The lemma follows from a well-known fact about submanifolds of $\mathbb{R}^{d}$ of class $C^{1,1}$ (cf.\ \cite{lewicka_peres}): in particular, a bounded set $A \subseteq \mathbb{R}^{d}$ is a $d$-dimensional submanifold of $\mathbb{R}^{d}$ with boundary of class $C^{1,1}$ if and only if there is a $\zeta > 0$ such that $A$ satisfies both an interior and an exterior ball condition of radius $\zeta$.  What makes the convex case different is the existence of supporting hyperplanes.  Indeed, a supporting hyperplane is, in effect, an exterior ball of radius $r = +\infty$, hence the exterior ball condition is redundant.      \end{proof}

Finally, in the analysis of the geometry of $E_{\zeta}$, it will be useful to keep in mind the following elementary observation about convex sets and their supporting hyperplanes:
	
		\begin{lemma} \label{L: separation and distance} Let $E \subseteq \mathbb{R}^{d}$ be convex and suppose that $q_{*} \notin \overline{E}$. Given any $\bar{q}_{*} \in \partial E$, the equality $\|q_{*} - \bar{q}_{*}\| = \text{dist}(q_{*},E)$ holds if and only if
			\begin{equation*}
				\langle q_{*} - \bar{q}_{*}, \bar{q}_{*} \rangle = \sup \left\{ \langle q_{*} - \bar{q}_{*}, \bar{q} \rangle \, \mid \, \bar{q} \in \overline{E} \right\}.
			\end{equation*}
		In particular, there is a unique such point $\bar{q}_{*}$. \end{lemma}
		
			\begin{proof} The ``if" direction is a consequence of orthogonality.  See \cite[Lemma 1.3.1]{schneider} for the proof of the ``only if" direction.  \end{proof}
	
		\begin{proof}[Proof of Proposition \ref{P: structure theorem}]  To see that (i) holds, first, observe that $E_{\zeta}$ is convex.  Indeed, suppose that $q_{1},q_{2} \in E_{\zeta}$ and $\lambda \in [0,1]$.  By construction, there are points $\bar{q}_{1},\bar{q}_{2} \in \{\varphi \leq 1\}$ and $\xi_{1},\xi_{2} \in \overline{B}_{\zeta}$ such that $q_{i} = \bar{q}_{i} + \xi_{i}$ for $i \in \{1,2\}$.  From this, one decomposes the convex combination $(1 - \lambda) q_{1} + \lambda q_{2}$ as
			\begin{equation*}
				(1 - \lambda) q_{1} + \lambda q_{2} = [(1 - \lambda) \bar{q}_{1} + \lambda \bar{q}_{2}] + [(1 - \lambda) \xi_{1} + \lambda \xi_{2}].
			\end{equation*}
		 Since $\{\varphi \leq 1\}$ and $\overline{B}_{\zeta}$ are both convex, it follows that $(1 - \lambda) q_{1} + \lambda q_{2} \in E_{\zeta}$, hence $E_{\zeta}$ is convex.  
		 
		 Next, here is the proof that $E_{\zeta}$ satisfies a uniform interior ball condition with radius $\zeta$, and, thus, is of class $C^{1,1}$.  By definition, if $q_{*} \in \partial E_{\zeta}$, then there is a $\bar{q}_{*} \in \{\varphi \leq 1\}$ such that $\|q_{*} - \bar{q}_{*}\| = \zeta$.  From this, one readily deduces that $\overline{B}_{\zeta}(q_{*} - \zeta e) = \overline{B}_{\zeta}(\bar{q}_{*}) \subseteq E_{\zeta}$ with $e = \|q_{*} - \bar{q}_{*}\|^{-1} (q_{*} - \bar{q}_{*})$. 
		
		Next, turn to (ii).  By Lemma \ref{L: separation and distance}, for any $x \in \mathbb{R}^{d}$, there is a unique $\bar{q}_{x} \in \{\varphi \leq 1\}$ such that $\|x - \bar{q}_{x}\| = \text{dist}(x,\{\varphi \leq 1\})$.  It is also easy to check that $\bar{q}_{x} \in \{\varphi = 1\}$ if $\varphi(x) > 1$.  In particular, the function $\mathcal{Q}_{\zeta} : \partial E_{\zeta} \to \{\varphi = 1\}$ given by setting $\mathcal{Q}_{\zeta}(q) = \bar{q}_{q}$ is well-defined.  Due to the fact that $\{\varphi \leq 1\}$ is compact, uniqueness readily implies continuity.
		
		Finally, consider (iii).  Let $\bar{q}_{*} = \mathcal{Q}_{\zeta}(q_{*})$.  A classical argument establishes that $\varphi^{*}(q_{*} - \bar{q}_{*})^{-1} (q_{*} - \bar{q}_{*}) \in \partial \varphi(\bar{q}_{*})$.  Indeed, were this not the case, it would be possible to invoke \eqref{E: subdifferential basic identity dual} and \eqref{E: inversion} to find a $\bar{q}_{**} \in \{\varphi \leq 1\}$ such that 
			\begin{equation*}
				\langle q_{*} - \bar{q}_{*}, \bar{q}_{**} \rangle > \langle q_{*} - \bar{q}_{*}, \bar{q}_{*} \rangle.
			\end{equation*}
		However, this would contradict Lemma \ref{L: separation and distance}.  
		
		It only remains to show that
			\begin{equation} \label{E: super annoying}
				\langle q_{*} - \bar{q}_{*}, q_{*} \rangle = \sup \left\{ \langle q_{*} - \bar{q}_{*}, q \rangle \, \mid \, q \in \partial E_{\zeta} \right\}.
			\end{equation}
		To see this, first, note that if $N_{\partial E_{\zeta}}(q_{*}) \in S^{d-1}$ is the outward normal vector to $\partial E_{\zeta}$ at $q_{*}$, then, by convexity, $N_{\partial E_{\zeta}}(q_{*})$ determines a supporting hyperplane:
			\begin{equation} \label{E: separation}
				\langle N_{\partial E_{\zeta}}(q_{*}), q_{*} \rangle = \sup \left\{ \langle N_{\partial E_{\zeta}}, q \rangle \, \mid \, q \in E_{\zeta} \right\}.
			\end{equation}
		At the same time, as observed already, there holds $\overline{B}_{\zeta}(\bar{q}_{*}) \subseteq E_{\zeta}$ and $q_{*} \in \partial B_{\zeta}(\bar{q}_{*}) \cap \partial E_{\zeta}$, i.e., the ball $\overline{B}_{\zeta}(\bar{q}_{*})$ is tangent to $E_{\zeta}$ from the inside.  It necessarily follows that $N_{\partial E_{\zeta}}(q_{*})$ equals the normal vector to $\partial B_{\zeta}(\bar{q}_{*})$, that is, $N_{\partial E_{\zeta}}(q_{*}) = \|q_{*} - \bar{q}_{*}\|^{-1} (q_{*} - \bar{q}_{*})$.  Combining this with \eqref{E: separation} yields \eqref{E: super annoying}.  \end{proof}

In the proof of Theorem \ref{T: regularized_norm}, it will be useful to know that the map $\mathcal{Q}_{\zeta}$, which is not invertible in general, nonetheless has an explicit multivalued inverse.

	\begin{prop} \label{P: multivalued inverse} Let $\mathcal{Q}_{\zeta}$ be the map defined in Proposition \ref{P: structure theorem}.  Then, for any $\bar{q} \in \{\varphi = 1\}$,
		\begin{equation*}
			\{q \in \partial E_{\zeta} \, \mid \, \mathcal{Q}_{\zeta}(q) = \bar{q}\} = \bar{q} + \left\{ \frac{\zeta p}{\|p\|} \, \mid \, p \in \partial \varphi(\bar{q}) \right\}
		\end{equation*}\end{prop}
		
		\begin{proof} Suppose that $q \in \mathcal{Q}_{\zeta}^{-1}(\bar{q})$.  By Proposition \ref{P: structure theorem}, (iii), the vector $p = \varphi^{*}(q - \bar{q})^{-1}(q - \bar{q})$  is an element of $\partial \varphi(\bar{q})$.  At the same time, $\|q - \bar{q}\| = \zeta$.  Thus, since $\varphi^{*}(q - \bar{q}) > 0$, it follows that $q = \bar{q} + \varphi^{*}(q - \bar{q}) p = \bar{q} + \zeta \|p\|^{-1} p$.  This proves one of the necessary inclusions.
		
		 Conversely, given a $p \in \partial \varphi(\bar{q})$, it remains to show that the vector $q = \bar{q} + \zeta \|p\|^{-1} p$ is an element of the set $\{\mathcal{Q}_{\zeta} = \bar{q}\}$.  It is clear that $\|q - \bar{q}\| = \zeta$ so $q \in E_{\zeta}$.  At the same time, 
		 	\begin{align*}
				\langle q - \bar{q}, \bar{q} \rangle = \frac{\zeta}{\|p\|} \langle p, \bar{q} \rangle
				&= \frac{\zeta}{\|p\|} \sup \left\{ \langle p, \bar{q}_{1} \rangle \, \mid \, \bar{q}_{1} \in \{\varphi \leq 1\} \right\} \\
				&= \sup \left\{ \langle q - \bar{q}, \bar{q}_{1} \rangle \, \mid \, \bar{q}_{1} \in \{\varphi \leq 1\} \right\}.
			\end{align*}
		Thus, by Lemma \ref{L: separation and distance}, $\|q - \bar{q}\| = \inf \left\{ \|q - \bar{q}_{1}\| \, \mid \, \bar{q}_{1} \in \{\varphi \leq 1\}\right\}$.  This proves $q \in \partial E_{\zeta}$ and $\mathcal{Q}_{\zeta}(q) = \bar{q}$.  \end{proof}

\subsection{Construction of $\varphi_{\zeta}$} \label{S: construction of minkowski gauge zeta} Given $E_{\zeta}$ as defined in the previous section, let $\varphi_{\zeta} : \mathbb{R}^{d} \to [0,+\infty)$ be the associated Minkowski gauge, that is, the unique Finsler norm such that $\{\varphi_{\zeta} \leq 1\} = E_{\zeta}$.  Recall that the gauge is determined by the following formula:
	\begin{equation} \label{E: minkowski}
		\varphi_{\zeta}(x) = \inf \left\{ \alpha > 0 \, \mid \, \alpha^{-1} x \in E_{\zeta} \right\}.
	\end{equation}
See, for instance, \cite[Chapter 5]{reed_simon} or \cite[Section 1.2]{brezis} for a proof that \eqref{E: minkowski} defines a Finsler norm; uniqueness follows immediately from positive homogeneity.

	\begin{remark} \label{R: ball regularization} At this stage, it is worth considering what happens when $\varphi(q) = \alpha^{-1} \|q\|$ for some constant $\alpha > 0$.  In this case, the unit ball $\{\varphi \leq 1\}$ is nothing but the ball $\overline{B}_{\alpha}(0)$, and then it is easy to verify that $E_{\zeta} = \overline{B}_{\alpha + \zeta}(0)$ for any $\zeta > 0$.  In particular, this proves that
		\begin{equation*}
			\varphi_{\zeta} = \frac{1}{\alpha + \zeta} \|\cdot\| \quad \text{if} \quad \varphi = \frac{1}{\alpha} \|\cdot\|.
		\end{equation*}   \end{remark}

The key ingredient in Theorem \ref{T: regularized_norm} is the subject of the next proposition:

		\begin{prop} \label{P: subdifferential characterization} If $p \in \mathbb{R}^{d} \setminus \{0\}$, then
		\begin{equation*}
			\partial \varphi_{\zeta}^{*}(p) = \partial \varphi^{*}(p) +  \frac{\zeta p}{\|p\|}.
		\end{equation*}
	\end{prop}
	
In the proof, it will be useful to recall that $\partial \varphi_{\zeta}^{*}$ has a geometric representation:
	\begin{equation}
		\partial \varphi_{\zeta}^{*}(p) = \left\{ q \in \partial E_{\zeta} \, \mid \, \langle q, p \rangle = \max_{q' \in \partial E_{\zeta}} \langle q', p \rangle \right\}. \label{E: key subdifferential identity}
	\end{equation}
This is nothing but a recapitulation of \eqref{E: subdifferential basic identity dual}.

		\begin{proof}  First, note that if $q \in \partial \varphi_{\zeta}^{*}(p)$, then \eqref{E: key subdifferential identity} implies that, at the level of the normal vector $N_{\partial E_{\zeta}}(q)$ to $\partial E_{\zeta}$ at $q$, there holds 
				\begin{equation*}
					\frac{p}{\|p\|} = N_{\partial E_{\zeta}}(q) = \frac{D\varphi_{\zeta}(q)}{\|D\varphi_{\zeta}(q)\|}.
				\end{equation*}
			Thus, by Proposition \ref{P: structure theorem}, (iii), 
				\begin{equation} \label{E: tedious subdifferential computations}
					q - \mathcal{Q}_{\zeta}(q) = \frac{\zeta p}{\|p\|}, \quad \varphi^{*}(p)^{-1} p \in \partial \varphi(\mathcal{Q}_{\zeta}(q)).
				\end{equation}
			By duality, the second identity can be rewritten as $\mathcal{Q}_{\zeta}(q) \in \partial \varphi^{*}(p)$.  Indeed, by \eqref{E: inversion} and zero-homogeneity,
				\begin{align*}
					\mathcal{Q}_{\zeta}(q) &\in \partial \varphi^{*}(\varphi^{*}(p)^{-1} p) = \partial \varphi^{*}(p).
				\end{align*}
			Since $q$ is an arbitrary element of $\partial \varphi_{\zeta}^{*}(p)$, this yields
				\begin{equation*}
					\partial \varphi_{\zeta}^{*}(p) \subseteq \partial \varphi^{*}(p) + \frac{\zeta p}{\|p\|}.
				\end{equation*}
			
			Conversely, if $\bar{q} \in \partial \varphi^{*}(p)$, then $\varphi^{*}(p)^{-1}p \in \partial \varphi(\bar{q})$ and Proposition \ref{P: multivalued inverse} implies that
				\begin{equation*}
					\bar{q} + \frac{\zeta p}{\|p\|} \in \partial E_{\zeta} \cap \mathcal{Q}_{\zeta}^{-1}(\bar{q}).
				\end{equation*}
			By Proposition \ref{P: structure theorem}, (iii) and \eqref{E: key subdifferential identity}, the inclusion $\bar{q} + \frac{\zeta p}{\|p\|} \in \partial \varphi^{*}_{\zeta}(p)$ holds.  Since $\bar{q}$ was arbitrary, this gives the remaining inclusion:
				\begin{equation*}
					\partial \varphi^{*}(p) + \frac{\zeta p}{\|p\|} \subseteq \partial \varphi_{\zeta}^{*}(p).
				\end{equation*}
			\end{proof}
			
To conclude the proof of Theorem \ref{T: regularized_norm}, it remains to verify the remaining properties of $\varphi_{\zeta}$.  That is accomplished in the following two propositions.

	\begin{prop} \label{P: proof of approx theorem 1} $\varphi_{\zeta} \to \varphi$ locally uniformly in $\mathbb{R}^{d}$ as $\zeta \to 0^{+}$.  Furthermore, there are constants $C(\varphi), \mu(\varphi) > 0$ such that, for any $\zeta \in (0,1)$,
		\begin{gather} 
			\mu(\varphi) \leq \|D\varphi_{\zeta}(x)\| \leq C(\varphi) \quad \text{for each} \quad x \in \mathbb{R}^{d} \setminus \{0\}, \quad \text{and} \label{E: bounded below gradient} \\
			\{\varphi_{\zeta} \leq 1\} \subseteq (1 + C(\varphi) \zeta) \{\varphi \leq 1\}. \nonumber
		\end{gather} \end{prop}
	
		\begin{proof} First, recall that $\{\varphi \leq 1\} \subseteq \{\varphi_{\zeta} \leq 1\}$ for any $\zeta > 0$.  This implies $\varphi_{\zeta} \leq \varphi$ by positive one-homogeneity.  
		
		Define $C(\varphi) = \max\{\varphi(y) \, \mid \, \|y\| = 1\}$, and note that 
			\begin{equation*}
				B_{1}(0) \subseteq \{\varphi \leq C(\varphi)\} = C(\varphi) \{\varphi \leq 1\}.
			\end{equation*}
		Thus, for any $\zeta > 0$, there holds $E_{\zeta} \subseteq \{\varphi \leq 1\} + \overline{B}_{\zeta}(0) \subseteq (1 + C(\varphi)\zeta) \{\varphi \leq 1\}$.  Since $\varphi_{\zeta}$ is positively one-homogeneous, it follows that, for any $y \in \mathbb{R}^{d}$ and any $\zeta \in (0,1)$,
		\begin{equation*}
			(1 + C(\varphi)\zeta)^{-1} \min \left\{ \varphi(y) \, \mid \, \|y\| = 1\right\}\|x\| \leq \varphi_{\zeta}(x) \leq \varphi(x) \leq C(\varphi) \|x\|.
		\end{equation*}
	From this and the identity $\varphi_{\zeta}(x) = \langle D\varphi_{\zeta}(x), x \rangle$, one deduces that
		\begin{equation*}
			(1 + C(\varphi) \zeta)^{-1} \min \left\{ \varphi(y) \, \mid \, \|y\| = 1 \right\} \leq \|D\varphi_{\zeta}(x)\| \quad \text{for each} \quad x \in \mathbb{R}^{d} \setminus \{0\}.
		\end{equation*}
	Similarly, since $\langle D\varphi_{\zeta}(x), y \rangle \leq \varphi_{\zeta}(y)$ for any $x,y \in \mathbb{R}^{d} \setminus \{0\}$, 
		\begin{equation*}
			\|D\varphi_{\zeta}(x)\| \leq C(\varphi) \quad \text{for each} \quad x \in \mathbb{R}^{d} \setminus \{0\}.
		\end{equation*}
	This proves \eqref{E: bounded below gradient} holds, and it shows that $(\varphi_{\zeta})_{\zeta > 0}$ is uniformly Lipschitz.
	
	By the Arzel\`{a}-Ascoli Theorem, up to passing to a subsequence, there is a Finsler norm $\tilde{\varphi}$ such that $\varphi_{\zeta} \to \tilde{\varphi}$ locally uniformly as $\zeta \to 0^{+}$.  The previous considerations show that, for any $\zeta \in (0,1)$,
		\begin{equation*}
			\{\varphi \leq 1\} \subseteq \{\tilde{\varphi} \leq 1\} \subseteq (1 + C(\varphi) \zeta) \{\varphi \leq 1\}.
		\end{equation*}
	Therefore, $\{\varphi \leq 1\} = \{\tilde{\varphi} \leq 1\}$, which means $\tilde{\varphi} = \varphi$.  Since the subsequence was arbitrary, this proves $\varphi_{\zeta} \to \varphi$ locally uniformly. \end{proof}

	\begin{prop} \label{P: proof of approx theorem 2} There is a $\kappa(\varphi) > 0$ such that, for any $\zeta \in (0,1)$,
		\begin{equation*}
			\varphi^{*}(D\varphi_{\zeta}(x)) \geq \kappa(\varphi) \quad \text{for each} \quad x \in \mathbb{R}^{d} \setminus \{0\}.
		\end{equation*}
	\end{prop}
	
		\begin{proof} Since $\varphi^{*}$ is a Finsler norm, there is a constant $m(\varphi^{*}) > 0$ such that
			\begin{equation*}
				\varphi^{*}(p) \geq m(\varphi^{*}) \|p\| \quad \text{for each} \quad p \in \mathbb{R}^{d}.
			\end{equation*}
		Thus, the desired conclusion follows from Proposition \ref{P: proof of approx theorem 1}.
		\end{proof}

Finally, Theorem \ref{T: regularized_norm} is proved by combining the previous propositions.

	\begin{proof}[Proof of Theorem \ref{T: regularized_norm}]  Concerning (i), recall that $E_{\zeta} = \{\varphi_{\zeta} \leq 1\}$ by definition.  Thus, this is immediate from Proposition \ref{P: structure theorem}.
	
	Next, statements (ii), (iii), and (iv) were proved in Propositions \ref{P: subdifferential characterization}, \ref{P: proof of approx theorem 1}, and \ref{P: proof of approx theorem 2}, respectively.  
	
	Finally, concerning (v), it is clear that $0 \in \{\varphi < 1\} \subseteq \{\varphi_{\zeta} < 1\}$, hence there is a $\delta > 0$ such that $\overline{B}_{\delta}(0) \subseteq \{\varphi_{\zeta} \leq 1\}$ for all $\zeta > 0$.  By boundedness, there is an $R > 0$ such that $\{\varphi \leq 1\} \subseteq \overline{B}_{R}(0)$, hence Proposition \ref{P: proof of approx theorem 1} implies that $\{\varphi_{\zeta} \leq 1\} \subseteq B_{CR}(0)$ for some $C > 0$ independent of $\zeta \in (0,1)$.  Thus, up to decreasing $\delta$, one has $\overline{B}_{\delta}(0) \subseteq \{\varphi_{\zeta} \leq 1\} \subseteq \overline{B}_{\delta^{-1}}(0)$ for all $\zeta \in (0,1)$. \end{proof}

\section{Conical Test Functions} \label{S: conical test}

This section describes the test functions used in the proofs of $C^{1,1}$ cone comparison principles, such as the one encountered in Section \ref{S: c11 reduction}.  The test functions in question are positively one-homogeneous, nonnegative convex functions --- hence their domains are cones --- so they will be referred to henceforth as \emph{conical test functions}.

\subsection{Faces of Convex Sets} \label{S: faces} The construction of conical test functions is based on the notion of a \emph{face} of a convex set.  That concept will be reviewed here.  For more information, see the monographs by Rockafellar \cite{rockafellar} or Schneider \cite{schneider}.

Before stating the definition, it is necessary to first define the relative interior and boundary of a convex set.  Observe that if $C \subseteq \mathbb{R}^{d}$ is a convex set, then there is a unique smallest linear subspace $V_{C} \subseteq \mathbb{R}^{d}$ such that 
	\begin{equation*}
		C \subseteq x + V_{C} \quad \text{for each} \quad x \in C.
	\end{equation*}
The dimension of $C$ is then defined to coincide with that of $V_{C}$, $\text{dim}(C) := \text{dim}(V_{C})$.

	\begin{definition} Given a convex set $C \subseteq \mathbb{R}^{d}$, the \emph{relative interior} $\text{rint}(C)$ is the interior of $C$ relative to the subspace $V_{C}$, that is, $x \in \text{rint}(C)$ if and only if there is an $r > 0$ such that
		\begin{equation*}
			B_{r}(x) \cap (x + V_{C}) \subseteq C.
		\end{equation*}
	The \emph{relative boundary} $\text{bdry}(C)$ of $C$ is the boundary of $C$ relative to $V_{C}$:
		\begin{equation*}
			\text{bdry}(C) = \overline{C} \setminus \text{rint}(C).
		\end{equation*}
	\end{definition}

Roughly speaking, the faces of a convex set $C$ are the convex subsets that are maximal with respect to the operation $\text{rint}(\cdot)$.  The definition given next will be useful for us, but is different than that in \cite{rockafellar}; however, see Theorem 18.1 in that reference for the equivalence of the two formulations.

	\begin{definition} \label{D: face} Given a convex set $C \subseteq \mathbb{R}^{d}$, a convex subset $F \subseteq C$ is called a \emph{face} of $C$ if it has the following property: for any other convex subset $F' \subseteq C$, if $\text{rint}(F') \cap F$ is nonempty, then $F' \subseteq F$.  \end{definition}
	
It is worth noting that the (relative interiors of) faces of a convex set form a partition: that is, if $C \subseteq \mathbb{R}^{d}$ is convex, then
	\begin{equation} \label{E: face partition}
		C = \bigcup \left\{ \text{rint}(F) \, \mid \, F \, \, \text{face of} \, \, C \right\}
	\end{equation}
and $\text{rint}(F_{1}) \cap \text{rint}(F_{2}) = \emptyset$ if $F_{1}$ and $F_{2}$ are distinct faces of $C$.  See \cite[Theorem 18.2]{rockafellar} for more details. 
	
The easiest example of faces are the so-called \emph{exposed faces}.  These are the faces determined through linear optimization.

	\begin{definition}  Given a closed convex set $C \subseteq \mathbb{R}^{d}$, a subset $F \subseteq C$ is called an \emph{exposed face} if there is a $p \in \mathbb{R}^d$ such that 
		\begin{equation*}
			F = \left\{x \in C \, \mid \, \langle x, p \rangle = \max_{y \in C} \, \langle y,p \rangle \right\}.
		\end{equation*}
	\end{definition}

It is not hard to show that an exposed face is, in fact, a face.  In particular, if $\varphi$ is a Finsler norm, then, by \eqref{E: subdifferential basic identity dual}, the subdifferential $\partial \varphi^{*}(p)$ is an exposed face of $\{\varphi \leq 1\}$ for any $p \in \mathbb{R}^{d}$.  Clearly, every exposed face of $\{\varphi \leq 1\}$ takes this form.

It is important to note that, in general, a face $F$ of a convex set $C$ need \emph{not} be an exposed face.  This is actually very relevant to the approach taken in this paper --- non-exposed faces cause difficulties that significantly complicate the proofs, as indicated in Remark \ref{R: exposed face conundrum}.  An example of a face that is not exposed is described next.

\begin{remark} Consider the convex set $C \subseteq \mathbb{R}^{2}$ given by 
	\begin{equation*}
		C = ([0,1] \times [0,1]) \cup \left\{(x,y) \in \mathbb{R}^{2} \, \mid \,  1 \leq x \leq 2, \, \, 0 \leq y \leq \sqrt{1 - (x - 1)^{2}} \right\}.
	\end{equation*}
The segment $[0,1] \times \{1\}$ is an exposed face.  At the same time, the singleton $\{(1,1)\}$ is a face, but it is not exposed (since $\partial C$ is differentiable there). \end{remark}

\begin{remark} \label{R: exposed face stuff} More generally, given any Finsler norm $\varphi$, if $\{\varphi \leq 1\}$ is of class $C^{1}$ but not strictly convex --- so it contains at least one nontrivial exposed face --- then it necessarily has a face that is not exposed.  Indeed, for any $p \in \mathbb{R}^{d} \setminus \{0\}$, the set $\partial \varphi^{*}(p)$ is a face of $\{\varphi \leq 1\}$.  If $\# \partial \varphi^{*}(p) > 1$, then the differentiability of the boundary $\{\varphi = 1\}$ implies that every nontrivial exposed face of $\partial \varphi^{*}(p)$ is a non-exposed face of $\{\varphi \leq 1\}$. \end{remark}

\subsection{A Motivating Example} \label{S: c11 strictly convex} In order to motivate the conical test functions defined below, first, consider the following example of their use.  

Specifically, this subsection revisits the case of Finsler norms $\varphi$ for which the ball $\{\varphi \leq 1\}$ is $C^{1,1}$ and strictly convex.\footnote{Recall that a convex set $C \subseteq \mathbb{R}^{d}$ is \emph{strictly convex} if $\partial C$ contains no line segments.}  In this case, the strict convexity of $\{\varphi \leq 1\}$ implies that $\partial \varphi^{*}$ is single-valued away from the origin and, thus, 
	\begin{equation*}
		G_{\varphi}^{*} \equiv G^{\varphi}_{*} \quad \text{in} \, \, (\mathbb{R}^{d} \setminus \{0\}) \times \mathcal{S}_{d}
	\end{equation*}
so the $\varphi$-infinity Laplacian is continuous, at least where the gradient is nonzero.  While in this setting the comparison principle is well-understood,\footnote{See the discussion in \cite{m_souganidis}, particularly Proposition 1 therein.} this section presents a different, elementary approach.

	\begin{prop} \label{P: strictly convex c11 thing} If $\varphi$ is a Finsler norm in $\mathbb{R}^{d}$ for which $\{\varphi \leq 1\}$ is both $C^{1,1}$ and strictly convex, then $\text{Sub}_{\varphi}(U) = CCA_{\varphi}(U)$ for any open set $U \subseteq \mathbb{R}^{d}$. \end{prop}  

To reiterate, the above result is already known, as discussed, for instance, in \cite{aronsson_crandall_juutinen} and \cite{m_souganidis}.  Nonetheless, since $\varphi$ is only assumed to be $C^{1,1}$ and not $C^{2}$, the known proof involves some form of the maximum principle for semicontinuous functions, see \cite[Appendix A]{user} or \cite[Proof of Theorem 6]{m_souganidis}.  By contrast, the argument detailed next only uses elementary calculus and convex analysis.

As in the last section, the interior ball condition afforded by $C^{1,1}$ regularity plays an essential role.   In view of Lemma \ref{L: c11 set thing}, since the set $\{\varphi \leq 1\}$ is of class $C^{1,1}$, there is a radius $\zeta > 0$ such that, given any  $q \in \{\varphi = 1\}$, if $e \in S^{d-1}$ is the outward normal vector to $\{\varphi \leq 1\}$ at $q$, then	
	\begin{equation} \label{E: geometric form of touching ack}
		\overline{B}_{r}(q - \zeta e) \subseteq \{\varphi \leq 1\}, \quad \partial B_{r}(q - \zeta e) \cap \{\varphi \leq 1\} = \{q\}.
	\end{equation}
Hence although $\{\varphi \leq 1\}$ need not be smooth at $q$, it can be touched from the inside at $q$ by a convex set that \emph{is}.  

At the level of the norm $\varphi$, this implies a similar statement.  In particular, there is a function $\psi_{e,q,\zeta}$ that is smooth in a neighborhood of the ray $\mathcal{C}(\{q\})$ through $q$ and for which there holds
	\begin{gather}
		\varphi \leq \psi_{e,q,\zeta} \quad \text{in} \, \, \mathbb{R}^{d}, \label{E: psi above} \quad \text{and}\\ 
		\varphi(x) = \psi_{e,q,\zeta}(x) \quad \text{for each} \quad x \in \mathcal{C}(\{q\}). \label{E: psi equal}
	\end{gather}
Specifically, let $\psi_{e,q,\zeta} : \mathbb{R}^{d} \to [0,+\infty]$ be the Minkowski gauge of the ball $\overline{B}_{\zeta}(q - \zeta e)$, that is, the convex function defined by 
	\begin{equation*}
		\psi_{e,q,\zeta}(x) = \inf \left\{ \alpha > 0 \, \mid \, \alpha^{-1}x \in \overline{B}_{\zeta}(q - \zeta e) \right\},
	\end{equation*}
taken with the convention that $\inf \emptyset = +\infty$.  See Figure \ref{F: conical test function figure} for a depiction of this construction.

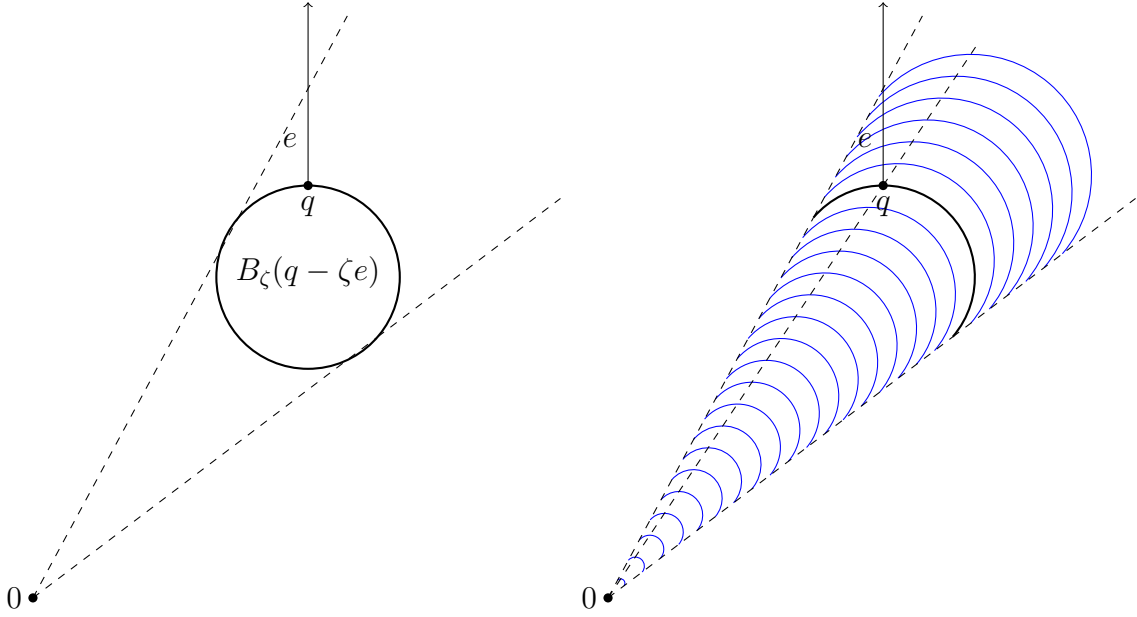
\begin{figure}
    \begin{tikzpicture}[scale=2.4]
    
    	%origin
    	\fill (0in,0in) circle (0.01in);
	\draw[anchor=east] (0in,0in) node {$0$};
    	
    	%circle
    	\draw[thick] (0.6in,0.7in) circle (0.2in);
    	
    	%base point
    	\fill (0.6in,0.9in) circle (0.01in);
    	\draw[anchor=north] (0.6in,0.9in) node {$q$};
    	\draw[anchor=south] (0.6in,0.65in) node {$B_{\zeta}(q - \zeta e)$};
    	
    	%arrow on top of circle
    	\draw[->] (0.6in,0.9in) -- (0.6in,1.3in);
    	\draw[anchor = east] (0.6in,1.0in) node {$e$};
    	
    	%construction points
    	%	\fill (0.4481in,0.8302in) circle (0.01in);	%construction points
    	\draw[dashed] (0,0) -- (0.4481in,0.8302in);
    	%	\fill (0.7519in, 0.5698in) circle (0.01in);	%construction points
    	\draw[dashed] (0,0) -- (0.7519in, 0.5698in);
    	\draw[dashed] (0.4481in,0.8302in) -- (0.6856in,1.2702in);
    	\draw[dashed] (0.7519in, 0.5698in) -- (1.1504in, 0.8718in);
    	
    \end{tikzpicture}
    \begin{tikzpicture}[scale=2.4]
    
    	%origin
    	\fill (0in,0in) circle (0.01in);
	\draw[anchor=east] (0in,0in) node {$0$};
    	
    	%main arc
    	\draw[thick] (0.4481in,0.8302in) arc (139.3999:-40.6001:0.2in);
    	
    	%arcs ahead
    	%distance to next point = 0.05in
    	%	\fill (0.4718in,0.8742in) circle (0.01in);	%construction points
    	%	\fill (0.7918in,0.6000in) circle (0.01in);	%construction points
    	\draw[blue] (0.4718in,0.8742in) arc (139.3999:-40.6001:0.2106in);
    	
    	%distance to next point = 0.10in
    	%	\fill (0.4956in, 0.9182in) circle (0.01in);	%construction points
    	%	\fill (0.8316in,0.6302in) circle (0.01in);	%construction points
    	\draw[blue] (0.4956in, 0.9182in) arc (139.3999:-40.6001:0.2212in);
    	
    	%distance to next point = 0.15in
    	%	\fill (0.5193in,0.9622in) circle (0.01in);	%construction points
    	%	\fill (0.8715in,0.6604in) circle (0.01in);	%construction points
    	\draw[blue] (0.5193in,0.9622in) arc (139.3999:-40.6001:0.2318in);
    	
    	%distance to next point = 0.20in
    	%	\fill (0.5431in,1.0062in) circle (0.01in);	%construction points
    	%	\fill (0.9113in,0.6906in) circle (0.01in);	%construction points
    	\draw[blue] (0.5431in,1.0062in) arc (139.3999:-40.6001:0.2424in);
    	
    	%distance to next point = 0.25in
    	%	\fill (0.5668in,1.0502in) circle (0.01in);	%construction points
    	%	\fill (0.9512in,0.7208in) circle (0.01in);	%construction points
    	\draw[blue] (0.5668in,1.0502in) arc (139.3999:-40.6001:0.2530in);
    	
    	%distance to next point = 0.30in
    	%	\fill (0.5906in,1.0942in) circle (0.01in);	%construction points
    	%	\fill (0.9910in,0.7510in) circle (0.01in);	%construction points
    	\draw[blue] (0.5906in,1.0942in) arc (139.3999:-40.6001:0.2636in);
    	
    	%arcs behind
    	
    	%distance to next point = -0.05in
    	%	\fill (0.4244in,0.7862in) circle (0.01in);	%construction points
    	%	\fill (0.7120in,0.5396in) circle (0.01in);	%construction points
    	\draw[blue] (0.4244in,0.7862in) arc (139.3999:-40.6001:0.1894in);
    	
    	%distance to next point = -0.10in
    	%	\fill (0.4006in,0.7422in) circle (0.01in);	%construction points
    	%	\fill (0.6722in,0.5094in) circle (0.01in);	%construction points
    	\draw[blue] (0.4006in,0.7422in) arc (139.3999:-40.6001:0.1788in);
    	
    	%distance to next point = -0.15in
    	%	\fill (0.3769in,0.6982in) circle (0.01in);	%construction points
    	%	\fill (0.6323in,0.4792in) circle (0.01in);	%construction points
    	\draw[blue] (0.3769in,0.6982in) arc (139.3999:-40.6001:0.1682in);
    	
    	%distance to next point = -0.20in
    	%	\fill (0.3531in,0.6542in) circle (0.01in);	%construction points
    	%	\fill (0.5925in,0.4490in) circle (0.01in);	%construction points
    	\draw[blue] (0.3531in,0.6542in) arc (139.3999:-40.6001:0.1576in);
    	
    	%distance to next point = -0.25in
    	%	\fill (0.3294in,0.6102in) circle (0.01in);	%construction points
    	%	\fill (0.5526in,0.4188in) circle (0.01in);	%construction points
    	\draw[blue] (0.3294in,0.6102in) arc (139.3999:-40.6001:0.1470in);
    	
    	%distance to next point = -0.30in
    	%	\fill (0.3056in,0.5662in) circle (0.01in);	%construction points
    	%	\fill (0.5128in,0.3886in) circle (0.01in);	%construction points
    	\draw[blue] (0.3056in,0.5662in) arc (139.3999:-40.6001:0.1364in);
    
    	%distance to next point = -0.35in
    	%	\fill (0.2819in,0.5222in) circle (0.01in);	%construction points
    	%	\fill (0.4729in,0.3584in) circle (0.01in);	%construction points
    	\draw[blue] (0.2819in,0.5222in) arc (139.3999:-40.6001:0.1258in);
    	
    	%distance to next point = -0.40in
    	%	\fill (0.2581in,0.4782in) circle (0.01in);	%construction points
    	%	\fill (0.4331in,0.3282in) circle (0.01in);	%construction points
    	\draw[blue] (0.2581in,0.4782in) arc (139.3999:-40.6001:0.1152in);
    	
    	%distance to next point = -0.45in
    	%	\fill (0.2344in,0.4342in) circle (0.01in);	%construction points
    	%	\fill (0.3932in,0.2980in) circle (0.01in);	%construction points
    	\draw[blue] (0.2344in,0.4342in) arc (139.3999:-40.6001:0.1046in);
    	
    	%distance to next point = -0.50in
    	%	\fill (0.2106in,0.3902in) circle (0.01in);	%construction points
    	%	\fill (0.3534in,0.2678in) circle (0.01in);	%construction points
    	\draw[blue] (0.2106in,0.3902in) arc (139.3999:-40.6001:0.0940in);
    
    	%distance to next point = -0.55in
    	%	\fill (0.1869in,0.3462in) circle (0.01in);	%construction points
    	%	\fill (0.3135in,0.2376in) circle (0.01in);	%construction points
    	\draw[blue] (0.1869in,0.3462in) arc (139.3999:-40.6001:0.0834in);
    	
    	%distance to next point = -0.60in
    	%	\fill (0.1631in,0.3022in) circle (0.01in);	%construction points
    	%	\fill (0.2737in,0.2074in) circle (0.01in);	%construction points
    	\draw[blue] (0.1631in,0.3022in) arc (139.3999:-40.6001:0.0728in);
    	
    	%distance to next point = -0.65in
    	%	\fill (0.1394in,0.2582in) circle (0.01in);	%construction points
    	%	\fill (0.2338in,0.1772in) circle (0.01in);	%construction points
    	\draw[blue] (0.1394in,0.2582in) arc (139.3999:-40.6001:0.0622in);
    
    	%distance to next point = -0.70in
    	%	\fill (0.1156in,0.2142in) circle (0.01in);	%construction points
    	%	\fill (0.1940in,0.1470in) circle (0.01in);	%construction points
    	\draw[blue] (0.1156in,0.2142in) arc (139.3999:-40.6001:0.0516in);
    	
    	%distance to next point = -0.75in
    	%	\fill (0.0919in,0.1702in) circle (0.01in);	%construction points
    	%	\fill (0.1541in,0.1168in) circle (0.01in);	%construction points
    	\draw[blue] (0.0919in,0.1702in) arc (139.3999:-40.6001:0.0410in);
    	
    	%distance to next point = -0.80in
    	%	\fill (0.0681in,0.1262in) circle (0.01in);	%construction points
    	%	\fill (0.1143in,0.0866in) circle (0.01in);	%construction points
    	\draw[blue] (0.0681in,0.1262in) arc (139.3999:-40.6001:0.0304in);
    	
    	%distance to next point = -0.85in
    	%	\fill (0.0444in,0.0822in) circle (0.01in);	%construction points
    	%	\fill (0.0744in,0.0564in) circle (0.01in);	%construction points
    	\draw[blue] (0.0444in,0.0822in) arc (139.3999:-40.6001:0.0198in);

    	%distance to next point = -0.90in
    	%	\fill (0.0206in,0.0382in) circle (0.01in);	%construction points
    	%	\fill (0.0346in,0.0262in) circle (0.01in);	%construction points
    	\draw[blue] (0.0206in,0.0382in) arc (139.3999:-40.6001:0.0092in);

    	%base point
    	\fill (0.6in,0.9in) circle (0.01in);
    	\draw[anchor=north] (0.6in,0.9in) node {$q$};
    	%	\draw[anchor=south] (0.6in,0.65in) node {$B_{\zeta}(q - \zeta e)$};
    	
    	%arrow on top of circle
    	\draw[->] (0.6in,0.9in) -- (0.6in,1.3in);
    	\draw[anchor = east] (0.6in,1.0in) node {$e$};
    	
    	%construction points
    	%	\fill (0.4481in,0.8302in) circle (0.01in);	%construction points
    	\draw[dashed] (0,0) -- (0.4481in,0.8302in);
    	%	\fill (0.7519in, 0.5698in) circle (0.01in);	%construction points
    	\draw[dashed] (0,0) -- (0.7519in, 0.5698in);
    	\draw[dashed] (0.4481in,0.8302in) -- (0.6856in,1.2702in);
    	\draw[dashed] (0.7519in, 0.5698in) -- (1.1504in, 0.8718in);
    	
    	%line through q
    	\draw[dashed] (0,0) -- (0.6in,0.9in);
    	\draw[dashed] (0.6in,0.9in) -- (0.8108in,1.2162in);
    	
    \end{tikzpicture}
    \caption{The ball $B_{\zeta}(q - \zeta e)$ is shown on the left.  The level sets of the Minkowski gauge $\psi_{e,q,\zeta}$ are shown on the right.  Note that $\psi_{e,q,\zeta}$ is infinite outside of the cone indicated by the dashed lines.}
     \label{F: conical test function figure}
\end{figure}

As in the construction of $\varphi_{\zeta}$ in Section \ref{S: construction of minkowski gauge zeta}, the function $\psi_{e,q,\zeta}$ is convex and positively one-homogeneous.  Note that \eqref{E: psi above} and \eqref{E: psi equal} follow immediately from \eqref{E: geometric form of touching ack}, the definition of $\psi_{e,q,\zeta}$, and positive one-homogeneity. 

While the domain $\{\psi_{e,q,\zeta} < +\infty\}$ may be a proper subset of $\mathbb{R}^{d}$, the function $\psi_{e,q,\zeta}$ is nonetheless smooth close to $q$.

	\begin{lemma} \label{L: smoothness of the ball thing} The function $\psi_{e,q,\zeta}$ is smooth in a neighborhood of the ray $\mathcal{C}(\{q\})$. \end{lemma}
	
\begin{proof} The lemma follows, in effect, from the fact that, at least in a small neighborhood of $q$, the level set $\{\psi_{e,q,\zeta} = 1\}$ equals a sphere; see Proposition \ref{P: smooth case test function} below for a detailed proof of a more general statement. \end{proof}

Since $q$ was arbitrary, the lemma implies that at any $x \in \mathbb{R}^{d} \setminus \{0\}$, it is always possible to find a nice convex function that touches the norm $\varphi$ from above at $x$, a fact that leads to an easy proof of comparison with cones.  In other words, what has been done so far already very nearly proves Proposition \ref{P: strictly convex game}, which was stated in the introduction and is reproduced here for the reader's convenience:

	\strictlyconvexgame*
	
	\begin{proof} Let $x \in \mathbb{R}^{d} \setminus \{0\}$.  Define $q = \varphi(x)^{-1} x$ so that $q \in \{\varphi = 1\}$ and $x \in \mathcal{C}(\{q\})$, and let $e = \|D\varphi(q)\|^{-1} D\varphi(q)$.  Define $\psi_{x}$ by $\psi_{x}(y) = \psi_{e,q,\zeta}(y)$, where $\psi_{e,q,\zeta}$ is the positively homogeneous convex function defined in the preceding discussion.  Condition (i) holds by \eqref{E: psi above} and \eqref{E: psi equal}, and condition (ii), by Lemma \ref{L: smoothness of the ball thing}.  Therefore, it only remains to check (iii).
	
	Since $\{\varphi \leq 1\}$ is strictly convex, the proof of (iii) amounts to a bit of a tautology.  By the positive one-homogeneity of $\psi_{e,q,\zeta}$, the identity $D^{2}\psi_{e,q,\zeta}(q)q = 0$ holds, and, thus, again by homogeneity,
			\begin{equation*}
				D^{2}\psi_{e,q,\zeta}(x) q = \psi_{e,q,\zeta}(x)^{-1} D^{2} \psi_{e,q,\zeta}(q) q = 0.
			\end{equation*}
	At the same time, since $\{\varphi \leq 1\}$ is strictly convex, \eqref{E: subdifferential basic identity dual} implies that the set function $\partial \varphi^{*}$ is single-valued away from the origin.  Taken together with the fact that $D\psi_{e,q,\zeta}(x) = D\varphi(x)$ by \eqref{E: psi above} and \eqref{E: psi equal}, this yields
		\begin{equation*}
				\partial \varphi^{*}(D\psi_{e,q,\zeta}(x)) = \{D\varphi^{*}(D\psi_{e,q,\zeta}(x))\} = \{D\varphi^{*}(D\varphi(x))\} = \{q\}.
		\end{equation*}
Invoking the definition of $G_{\varphi}^{*}$ and $G_{*}^{\varphi}$, one concludes
		\begin{equation*}
			G_{\varphi}^{*}(D\psi_{e,q,\zeta}(x),D^{2}\psi_{e,q,\zeta}(x)) = G_{*}^{\varphi}(D\psi_{e,q,\zeta}(x),D^{2}\psi_{e,q,\zeta}(x)) = \langle D^{2} \psi_{e,q,\zeta}(x) q, q \rangle = 0.
		\end{equation*}
      \end{proof}
      
As already indicated in the introduction, Proposition \ref{P: strictly convex game} readily leads to an easy proof of cone comparison, as shown next.
			
		\begin{proof}[Proof of Proposition \ref{P: strictly convex c11 thing}]   By Lemma \ref{L: symmetries}, to prove that $\text{Sub}_{\varphi}(U) \subseteq CCA_{\varphi}(U)$, it suffices to show that if $V$ is a bounded open set in $\mathbb{R}^{d} \setminus \{0\}$, $\epsilon > 0$, and $u$ is an upper semicontinuous function in $\overline{V}$ such that $ \frac{1}{2} \epsilon \varphi^{*}(Du)^{2} -G_{\varphi}^{*}(Du,D^{2}u) \leq 0$ in $V$, then
			\begin{equation*}
				\max \{u(x) - \varphi(x) \, \mid \, x \in \overline{V} \} = \max \{ u(x) - \varphi(x) \, \mid \, x \in \partial V\}.
			\end{equation*} 
			
	The proof proceeds by contradiction: suppose that there is an $x \in V$ such that
		\begin{equation} \label{E: ack too much}
			u(x) - \varphi(x) =  \max \{ u(y) - \varphi(y) \, \mid \, y \in \partial V\}.
		\end{equation}
	Since $x \in V \subseteq \mathbb{R}^{d} \setminus \{0\}$, Proposition \ref{P: strictly convex game} implies there is a function $\psi_{x}$, which is smooth in a neighborhood of $x$, such that 
		\begin{gather}
			\varphi \leq \psi_{x} \quad \text{in} \, \, \mathbb{R}^{d}, \quad \varphi(x) = \psi_{x}(x), \label{E: need to do this all the time}\\
			\text{and} \quad G_{\varphi}^{*}(D\psi_{x}(x),D^{2}\psi_{x}(x)) = G_{*}^{\varphi}(D\psi_{x}(x),D^{2}\psi_{x}(x)) = 0. \nonumber
		\end{gather}
	
	Combining \eqref{E: ack too much} and \eqref{E: need to do this all the time}, one deduces that 
		\begin{align*}
			u \leq (u(x) - \varphi(x)) + \varphi \leq (u(x) - \varphi(x)) + \psi_{x} \quad \text{in} \, \, \overline{V}, \\
			u(x) = (u(x) - \varphi(x)) + \varphi(x) = (u(x) - \varphi(x)) + \psi_{x}(x). 
		\end{align*}
	In particular, $u - \psi_{x}$ has a local maximum at $x$.  By the subsolution property of $u$, this implies
		\begin{equation*}
			\frac{1}{2} \epsilon \varphi^{*}(D\varphi(x)) = \frac{1}{2} \epsilon \varphi^{*}(D\psi_{x}(x))^{2} - G_{\varphi}^{*}(D\psi_{x}(x),D^{2}\psi_{x}(x)) \leq 0.
		\end{equation*}
	Yet \eqref{E: subdifferential basic identity} implies that $\varphi^{*}(D\varphi(x)) = 1$, so this leads to the inequality $0 < \frac{1}{2} \epsilon \leq 0$, which is absurd.  \end{proof}
	
Finally, note that Proposition \ref{P: strongly superharmonic} has not been used at all here.  Nonetheless, as explained in the next remark, the arguments above should be understood as a stronger version of that proposition.

\begin{remark} \label{R: strongly superharmonic again} Notice that the test functions $\{\psi_{x}\}_{x \in \mathbb{R}^{d} \setminus \{0\}}$ can be used to prove that $-G_{\varphi}^{*}(D\varphi,D^{2}\varphi) \geq 0$ holds in the viscosity sense in $\mathbb{R}^{d} \setminus \{0\}$.  Indeed, if $x \in \mathbb{R}^{d} \setminus \{0\}$ and $\psi$ is a smooth test function such that $\varphi - \psi$ has a local minimum at $x$, then $\psi_{x} - \psi$ also has a local minimum at $x$, hence, by the second derivative test and ellipticity,
	\begin{equation*}
		-G_{\varphi}^{*}(D\psi(x),D^{2}\psi(x)) \geq - G_{\varphi}^{*}(D\psi_{x}(x),D^{2}\psi_{x}(x)) = 0.
	\end{equation*}
In this sense, Proposition \ref{P: strictly convex game} is a stronger version of Proposition \ref{P: strongly superharmonic} in the special case when $\{\varphi \leq 1\}$ is strictly convex and $C^{1,1}$. \end{remark}
		
\subsection{Conical Test Functions} \label{S: general conical test} Extending the argument of the previous subsection to more general Finsler norms --- and, in particular, the discontinuous case where  $G_{\varphi}^{*} \not\equiv G_{*}^{\varphi}$ in $(\mathbb{R}^{d} \setminus \{0\}) \times \mathcal{S}_{d}$ --- will require some work.  The first step is to generate more test functions.  

As in the last subsection, it is convenient to retain the assumption that $\{\varphi \leq 1\}$ is of class $C^{1,1}$.  Hence, in particular, throughout the remainder of this section, fix a $\zeta > 0$ and assume that $\{\varphi \leq 1\}$ satisfies an interior ball condition of radius $\zeta$.  

Given a closed convex set $F \subseteq \mathbb{R}^{d}$, a direction $e \in S^{d-1}$, and a positive number $\nu > 0$, define another convex set $F_{(e,\nu)}$ by 
	\begin{equation} \label{E: fancy convolution}
		F_{(e,\nu)} = \bigcup_{x \in F} \overline{B}_{\nu}(x - \nu e). 
	\end{equation}
Note that $F_{(e,\nu)}$ is closed and convex, and it has nonempty interior if $F$ is nonempty.  Correspondingly, let $\psi_{e,F,\nu} : \mathbb{R}^{d} \to [0,+\infty]$ be the Minkowski gauge of $F_{(e,\nu)}$, that is,\footnote{Here $\inf \emptyset = +\infty$ as before.}
	\begin{equation} \label{E: test function definition}
		\psi_{e,F,\nu}(x) = \inf \left\{ \alpha > 0 \, \mid \, \alpha^{-1}x \in F_{(e,\nu)} \right\}.
	\end{equation}
This extends the previous definition, where $F$ was a singleton.

The utility of these test functions is demonstrated in the two propositions that follow.  In the first of the two, $F$ is taken to be one of the exposed faces of $\{\varphi \leq 1\}$, hence a set of the form $\partial \varphi^{*}(p)$ for some $p \in \mathbb{R}^{d} \setminus \{0\}$.  

	\begin{prop} \label{P: touching above test function part} Suppose that $\varphi$ is a Finsler norm in $\mathbb{R}^{d}$ and $\{\varphi \leq 1\}$ satisfies a uniform interior ball condition of radius $\zeta$.  For any $p \in \mathbb{R}^{d} \setminus \{0\}$ and any $\nu \in (0,\zeta]$, if $e = \|p\|^{-1} p$, then
		\begin{gather}
			\varphi \leq \psi_{e, \partial \varphi^{*}(p), \nu} \quad \text{in} \, \, \mathbb{R}^{d}, \nonumber \\
			\varphi(x) = \psi_{e,\partial \varphi^{*}(p),\nu}(x) \quad \text{for each} \quad x \in \mathcal{C}(\partial \varphi^{*}(p)). \label{E: touching}
		\end{gather}
	Further, if $\nu < \zeta$, then the contact set equals $\mathcal{C}(\partial \varphi^{*}(p))$, i.e.,
		\begin{equation} \label{E: contact set cone part}
			\{\varphi = \psi_{e,\partial \varphi^{*}(p),\nu}\} = \{\psi_{e,\partial \varphi^{*}(p),\zeta} = \psi_{e,\partial \varphi^{*}(p),\nu} < +\infty \} = \mathcal{C}(\partial \varphi^{*}(p)).
		\end{equation}
	\end{prop}
	
The relatively routine proof is deferred to the end of this section.
	
The next result holds more generally provided $F$ is contained in a $(d - 1)$-dimensional affine subspace that avoids the origin, which, in view of \eqref{E: subdifferential basic identity dual}, is certainly the case if $F = \partial \varphi^{*}(p)$ for some $p \neq 0$.  It is stated in this more general form for clarity.
	
	\begin{prop} \label{P: smooth case test function}  Suppose that $F \subseteq \mathbb{R}^{d}$ is a closed convex set, $e \in S^{d-1}$, and $\nu > 0$.  If there is a constant $c > 0$ such that 
		\begin{equation} \label{E: hyperplane assumption}
			F \subseteq \{x \in \mathbb{R}^{d} \, \mid \, \langle x,e \rangle = c\}
		\end{equation}
	and if $q_{*} \in \text{rint}(F)$, then $\psi_{e,F,\nu}$ is smooth in a neighborhood of the ray $\mathcal{C}(\{q_{*}\})$ and
		\begin{equation*}
			D^{2} \psi_{e,F,\nu}(q_{*}) q = 0 \quad \text{for each} \quad q \in F.
		\end{equation*} \end{prop}
		
	Once again, the proof is deferred to the end of the section.  (The impatient reader may be relieved to learn that Proposition \ref{P: smooth case test function} will not be used in later sections and, in particular, is not needed for the proof of any of the theorems.)
		
\subsection{Application to Cone Comparison} Here is the application of the previous two propositions to cone comparison, at a first glance.  This is only a partial result, which is limited by the lack of (global) smoothness of the test functions.  It is nonetheless presented here to motivate the construction so far, setting the scene for the complete treatment in Section \ref{S: c11 setting}.  

	\begin{prop} \label{P: conical test functions not smooth} Suppose that $\varphi$ is a Finsler norm in $\mathbb{R}^{d}$ for which the unit ball $\{\varphi \leq 1\}$ is of class $C^{1,1}$.  Fix an open set $U \subseteq \mathbb{R}^{d}$ and an $\epsilon > 0$, and suppose that $u \in USC(U)$ satisfies
		\begin{equation*}
			\frac{1}{2} \epsilon \varphi^{*}(Du)^{2} -G^{*}_{\varphi}(Du,D^{2}u)  \leq 0 \quad \text{in} \, \, U.
		\end{equation*}  
	
	If there is a bounded open set $V \subseteq U \setminus \{0\}$ and an $x_{*} \in V$ such that
		\begin{equation*}
			u(x_{*}) - \varphi(x_{*}) = \max\{u(y) - \varphi(y) \, \mid \, y \in \overline{V} \},
		\end{equation*}
	then the point $\varphi(x_{*})^{-1}x_{*}$ does not belong to the relative interior of any of the exposed faces of $\varphi$, that is,
		\begin{equation} \label{E: relative interior issue}
			x_{*} \notin \bigcup_{p \in \mathbb{R}^{d} \setminus \{0\}} \text{rint}(\mathcal{C}(\partial \varphi^{*}(p))).
		\end{equation}\end{prop}

	\begin{proof} Let $q_{*} = \varphi(x_{*})^{-1} x_{*}$ and fix a $\zeta > 0$ such that $\{\varphi \leq 1\}$ satisfies a interior ball condition of radius $\zeta$.  Argue by contradiction: assume there is a $p \in \mathbb{R}^{d} \setminus \{0\}$ such that
	\begin{equation*}
		q_{*} \in \text{rint}(\partial \varphi^{*}(p)),
	\end{equation*}
which is equivalent to $x_{*} \in \text{rint}(\mathcal{C}(\partial \varphi^{*}(p)))$.
Setting $e = \|p\|^{-1} p$ and applying \eqref{E: touching}, the test function $\psi := \psi_{e,\partial \varphi^{*}(p),\zeta}$ is smooth in a neighborhood of $x$ and $u - \psi$ has a local maximum at $x$.  Thus, by Proposition \ref{P: smooth case test function}, the subsolution property of $u$, and \eqref{E: subdifferential basic identity},
	\begin{align*}
		0 &\geq \frac{1}{2} \epsilon \varphi^{*}(D\psi(x))^{2} -G_{\varphi}^{*}(D\psi(x),D^{2}\psi(x))  \\
			&=  \frac{1}{2} \epsilon \varphi^{*}(D\varphi(x))^{2} - \max \left\{ \langle D^{2}\psi(x) q, q \rangle \, \mid \, q \in \partial \varphi^{*}(p) \right\} = \frac{1}{2} \epsilon,
	\end{align*}
which contradicts the fact that $\epsilon > 0$. \end{proof}

\subsection{Extending the Argument} The conclusion \eqref{E: relative interior issue} just slightly misses the mark.  If the relative interior $\text{rint}(\mathcal{C}(\partial \varphi^{*}(p)))$ could be replaced by $\mathcal{C}(\partial \varphi^{*}(p))$ in Proposition \ref{P: smooth case test function}, then the proof of cone comparison for $C^{1,1}$ norms would be complete, as \eqref{E: covering} asserts that this family of cones covers $\mathbb{R}^{d}$.

Unfortunately, there is no easy way to improve Proposition \ref{P: smooth case test function}.  If the relative boundary of $\partial \varphi^{*}(p)$ is nonempty, then $\text{bdry}(\mathcal{C}(\partial \varphi^{*}(p))$ is nonempty and, as will become clear shortly, $\psi$ is not twice differentiable at any of those points: the second derivative jumps!  Hence $\psi$ cannot be na\"{i}vely treated as a test function.

\begin{remark} \label{R: exposed face conundrum} If the relative boundary of $\partial \varphi^{*}(p)$ is nonempty, then one might be tempted to ``iterate:" if the contact point $x_{*}$ is such that $q_{*} := \varphi(x_{*})^{-1} x_{*} \in \text{bdry}(\partial \varphi^{*}(p))$, then one is tempted to search for a $p' \neq p$ such that $q_{*} \in \text{rint}(\partial \varphi^{*}(p'))$.  If this were possible, then the appropriate thing to do would be to work with the conical test function associated with $\partial \varphi^{*}(p')$ rather than $\partial \varphi^{*}(p)$.  However, this is not possible in general.  While it is true that there is a unique face $F$ of $\partial \varphi^{*}(p)$ such that $q_{*} \in \text{rint}(F)$, as already observed in Remark \ref{R: exposed face stuff}, the $C^{1,1}$ regularity of $\{\varphi \leq 1\}$ implies that $F$ is \emph{not} an exposed face of $\{\varphi \leq 1\}$, or, in other words, $F \notin \{\partial \varphi^{*}(p') \, \mid \, p' \in \mathbb{R}^{d} \setminus \{0\}\}$. \end{remark}

Even though conical test functions are not smooth in general, there is reason for hope.  For any $p \in \mathbb{R}^{d} \setminus \{0\}$ and any $\nu \in (0,\zeta]$, the conical test function $\psi = \psi_{e,\partial \varphi^{*}(p),\zeta}$ has two advantageous properties:
	\begin{itemize}
		\item[(i)] Even though $\psi$ may not be smooth at relative boundary points, the differential inequality $-G_{\varphi}^{*}(D\psi,D^{2}\psi) \geq 0$ holds in the viscosity sense at each point of $\mathcal{C}(\partial \varphi^{*}(p))$.  More precisely, if $\tilde{\psi}$ is a smooth test function touching $\psi$ from below at some point $x \in \mathcal{C}(\partial \varphi^{*}(p))$, then\footnote{This can be proved by generalizing the argument of Proposition \ref{P: strongly superharmonic} suitably to arbitrary positively one-homogeneous convex functions.  The details are omitted.}
			\begin{equation*}
				-G_{\varphi}^{*}(D\tilde{\psi}(x),D^{2}\tilde{\psi}(x)) \geq 0.
			\end{equation*}
		This indicates that nothing is lost passing from $\varphi$ to the conical test function $\psi$: on the contact set $\{\psi = \varphi\}$, the relevant viscosity-thereotic properties (Proposition \ref{P: strongly superharmonic}) are retained, even at singular points.
		\item[(ii)] The ``shape" of $\psi$ is very simple.  More precisely, the jumps of $D^{2}\psi$ are relatively well-behaved.  Since $\psi$ is positively one-homogeneous, this is easiest to comprehend by studying the level set $\{\psi = 1\}$ of $\psi$.  It turns out that it is easy to write the level set as a graph in a very concrete way; this is explained in the next subsection and will be used extensively in Section \ref{S: c11 setting}.
	\end{itemize}
As will be justified below, points (i) and (ii) show that, in passing from an arbitrary $C^{1,1}$ Finsler norm to a conical test function, there is a reduction in complexity.  At the level of the geometry, one passes from the arbitrary $C^{1,1}$ surface $\{\varphi = 1\}$ to one that can be written explicitly as a graph, while, where the PDE is concerned, the local structure is retained.  In effect, the approach of this paper is to treat conical test functions as ``flexible building blocks" for building up cone comparison principles.

%Via a smoothing argument, it is possible to extend Proposition \ref{P : }, and thus to prove cone comparison in the $C^{1,1}$ case, via the following result:
%
%	\begin{prop} Let $\varphi$ be a Finsler norm in $\mathbb{R}^{d}$ and assume that $\{\varphi \leq 1\}$ is $C^{1,1}$, in particular, there is a $\zeta > 0$ such that $\{\varphi \leq 1\}$ satisfies a uniform interior ball condition of radius $\zeta$.  Suppose that $U \subseteq \mathbb{R}^{d}$ is open and $V \subseteq \mathbb{R}^{d} \setminus \{0\}$ is a bounded open set such that $\overline{V} \subseteq U$.  
%	
%	If $u : U \to \mathbb{R} \cup \{-\infty\}$ is an upper semi-continuous function in $U$ and there is an $x_{0} \in V$ such that 
%		\begin{equation*}
%			u(x_{0}) - \varphi(x_{0}) = \max \{u(x) - \varphi(x) \, \mid \, x \in \overline{V} \} > \max\{u(x) - \varphi(x) \, \mid \, x \in \partial V\},
%		\end{equation*}
%	then there are points $q_{1},q_{2} \in \partial \varphi^{*}(D\varphi(x_{0}))$; a face $F$ of $\partial \varphi^{*}(D\varphi(x_{0}))$; and a symmetric matrix $X$ such that 
%		\begin{equation*}
%			- 
%		\end{equation*}  \end{prop}

\subsection{Graphical Representation of $F_{(e,\nu)}$} \label{S: graphical representation} In the previous subsection, it was suggested that the conical test functions, which are not smooth in their domains, nonetheless are quite amenable to geometric arguments.  This subsection presents an alternative graphical description of $F_{(e,\nu)}$, which makes the structure of the jumps of the Hessian tractable. 

Suppose, as in Proposition \ref{P: smooth case test function}, that $F \subseteq \mathbb{R}^{d}$ is a convex set for which there is an $e \in S^{d-1}$ and a $c > 0$ such that 
	\begin{equation*}
		F \subseteq \left\{x \in \mathbb{R}^{d} \, \mid \, \langle x,e \rangle = c \right\}.
	\end{equation*}
Define $G \subseteq \{e\}^{\perp}$ by projecting $F$ orthogonally, that is.,
	\begin{equation*}
		G = \left\{x' \in \{e\}^{\perp} \, \mid \, x' + c e \in F \right\} = F - ce.
	\end{equation*}
Here and in Section \ref{S: c11 setting} below, to lighten the notation, $\{e\}^{\perp}$ will be identified with $\mathbb{R}^{d-1}$; the prime notation, such as $x'$, will be reserved for points in $\mathbb{R}^{d-1}$.  

As before, define $F_{(e,\nu)} \subseteq \mathbb{R}^{d}$ according to \eqref{E: fancy convolution}.

Let $\text{dist}(\cdot,G) : \mathbb{R}^{d-1} \to [0,+\infty)$ denote the distance function to $G$, i.e.,
	\begin{equation*}
		\text{dist}(x',G) = \inf \left\{ \|x' - y'\| \, \mid \, y' \in G \right\}.
	\end{equation*}
Next, for a given $\nu > 0$, define $g_{\nu} : \mathbb{R}^{d-1} \to \mathbb{R}$ by
	\begin{equation} \label{E: graphical}
		g_{\nu}(x') = c - \nu + \sqrt{\nu^{2} - \text{dist}(x',G)^{2}}.
	\end{equation}
It is straightforward to check that the epigraph of $g_{\nu}$ parametrizes the ``upper half" of $F_{(e,\nu)}$ in the following sense:
	\begin{align*}
		F_{(e,\nu)} \cap \left\{x \in \mathbb{R}^{d} \, \mid \, \langle x, e \rangle \geq c - \nu \right\} &= \left\{x' + se \, \mid \, c - \nu \leq s \leq g_{\nu}(x'), \, \, \text{dist}(x',G) \leq \nu \right\}, \\
		\partial F_{(e,\nu)} \cap \left\{ x \in \mathbb{R}^{d} \, \mid \, \langle x,e \rangle \geq c - \nu \right\} &= \{x' + g_{\nu}(x')e \, \mid \, \text{dist}(x',G) \leq \nu\},
	\end{align*}
In particular,
	\begin{align*}
		F = \{x' + g_{\nu}(x')e \, \mid \, x' \in G\} = G + c e.
	\end{align*}

The representation of $F_{(e,\nu)}$ in terms of $g_{\nu}$ is convenient for a simple reason.  It makes the analysis of the singular points of $\partial F_{(e,\nu)}$ tractable; see Figure \ref{F: distance picture} for a plot of the gradient $D\text{dist}(\cdot,G)$ of the distance function to $G$ and the discontinuities of the Hessian.  In particular, when it comes time to improve Proposition \ref{P: conical test functions not smooth}, it will become clear that a manageable way to achieve this is through a perturbation argument involving mollification of $g_{\nu}$.

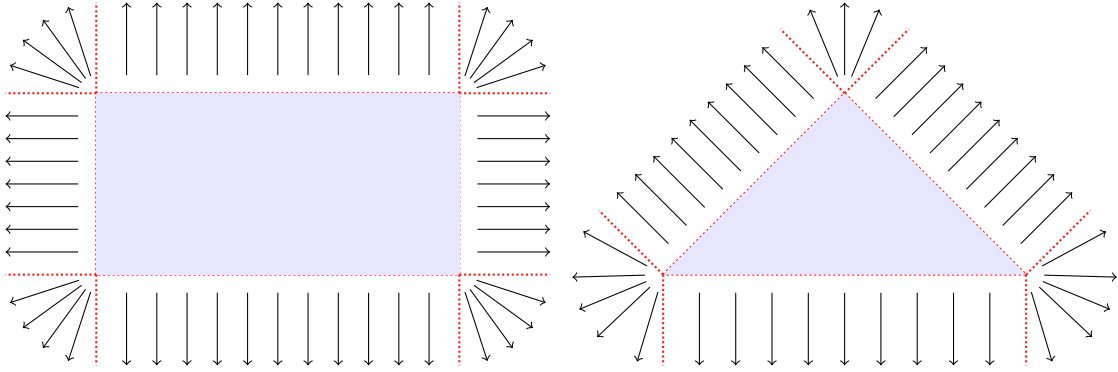
\begin{figure}
	\begin{tikzpicture}[scale=1.9]
   
    \draw[thick,color=red!90,densely dotted] (-0.5in,0) -- (-0.5in,0.5in) -- (0.5in,0.5in) -- (0.5in,0) -- (-0.5in,0);
    \fill[fill=blue!10] (-0.5in,0) -- (-0.5in,0.5in) -- (0.5in,0.5in) -- (0.5in,0) -- (-0.5in,0);
	
    \draw[->] (-0.55in,0.4375in) -- (-0.75in,0.4375in);
    \draw[->] (-0.55in,0.375in) -- (-0.75in,0.375in);
    \draw[->] (-0.55in,0.3125in) -- (-0.75in,0.3125in);
    \draw[->] (-0.55in,0.25in) -- (-0.75in,0.25in);
    \draw[->] (-0.55in,0.1875in) -- (-0.75in,0.1875in);
    \draw[->] (-0.55in,0.125in) -- (-0.75in,0.125in);
    \draw[->] (-0.55in,0.0625in) -- (-0.75in,0.0625in);
    
    \draw[->] (0.416667in,0.55in) -- (0.416667in,0.75in);
    \draw[->] (0.333333in,0.55in) -- (0.333333in,0.75in);
    \draw[->] (0.25in,0.55in) -- (0.25in,0.75in);
    \draw[->] (0.166667in,0.55in) -- (0.166667in,0.75in);
    \draw[->] (0.0833333in,0.55in) -- (0.0833333in,0.75in);
    \draw[->] (0in,0.55in) -- (0in,0.75in);
    \draw[->] (-0.0833333in,0.55in) -- (-0.0833333in,0.75in);
    \draw[->] (-0.166667in,0.55in) -- (-0.166667in,0.75in);
    \draw[->] (-0.25in,0.55in) -- (-0.25in,0.75in);
    \draw[->] (-0.333333in,0.55in) -- (-0.333333in,0.75in);
    \draw[->] (-0.416667in,0.55in) -- (-0.416667in,0.75in);
    
    \draw[->] (0.55in,0.0625in) -- (0.75in,0.0625in);
    \draw[->] (0.55in,0.125in) -- (0.75in,0.125in);
    \draw[->] (0.55in,0.1875in) -- (0.75in,0.1875in);
    \draw[->] (0.55in,0.25in) -- (0.75in,0.25in);
    \draw[->] (0.55in,0.3125in) -- (0.75in,0.3125in);
    \draw[->] (0.55in,0.375in) -- (0.75in,0.375in);
    \draw[->] (0.55in,0.4375in) -- (0.75in,0.4375in);
    
    \draw[->] (-0.416667in,-0.05in) -- (-0.416667in,-0.25in);
    \draw[->] (-0.333333in,-0.05in) -- (-0.333333in,-0.25in);
    \draw[->] (-0.25in,-0.05in) -- (-0.25in,-0.25in);
    \draw[->] (-0.166667in,-0.05in) -- (-0.166667in,-0.25in);
    \draw[->] (-0.0833333in,-0.05in) -- (-0.0833333in,-0.25in);
    \draw[->] (0in,-0.05in) -- (0in,-0.25in);
    \draw[->] (0.0833333in,-0.05in) -- (0.0833333in,-0.25in);
    \draw[->] (0.166667in,-0.05in) -- (0.166667in,-0.25in);
    \draw[->] (0.25in,-0.05in) -- (0.25in,-0.25in);
    \draw[->] (0.333333in,-0.05in) -- (0.333333in,-0.25in);
    \draw[->] (0.416667in,-0.05in) -- (0.416667in,-0.25in);

    \draw[->] (-0.547553in,-0.0154508in) -- (-0.737764in,-0.0772542in);
    \draw[->] (-0.540451in,-0.0293893in) -- (-0.702254in,-0.146946in);
    \draw[->] (-0.529389in,-0.0404508in) -- (-0.646946in,-0.202254in);
    \draw[->] (-0.515451in,-0.0475528in) -- (-0.577254in,-0.237764in);
    \draw[thick,color=red!90,densely dotted] (-0.5in,0in) -- (-0.75in,0in);
    \draw[thick,color=red!90,densely dotted] (-0.5in,0in) -- (-0.5in,-0.25in);
    
    \draw[->] (-0.547553in,0.515451in) -- (-0.737764in,0.577254in);
    \draw[->] (-0.540451in,0.529389in) -- (-0.702254in,0.646946in);
    \draw[->] (-0.529389in,0.540451in) -- (-0.646946in,0.702254in);
    \draw[->] (-0.515451in,0.547553in) -- (-0.577254in,0.737764in);
    \draw[thick,color=red!90,densely dotted] (-0.5in,0.5in) -- (-0.75in,0.5in);
    \draw[thick,color=red!90,densely dotted] (-0.5in,0.5in) -- (-0.5in,0.75in);
    
    \draw[->] (0.515451in,0.547553in) -- (0.577254in,0.737764in);
    \draw[->] (0.529389in,0.540451in) -- (0.646946in,0.702254in);
    \draw[->] (0.540451in,0.529389in) -- (0.702254in,0.646946in);
    \draw[->] (0.547553in,0.515451in) -- (0.737764in,0.577254in);
    \draw[thick,color=red!90,densely dotted] (0.5in,0.5in) -- (0.5in,0.75in);
    \draw[thick,color=red!90,densely dotted] (0.5in,0.5in) -- (0.75in,0.5in);
    
    \draw[->] (0.547553in,-0.0154508in) -- (0.737764in,-0.0772542in);
    \draw[->] (0.540451in,-0.0293893in) -- (0.702254in,-0.146946in);
    \draw[->] (0.529389in,-0.0404508in) -- (0.646946in,-0.202254in);
    \draw[->] (0.515451in,-0.0475528in) -- (0.577254in,-0.237764in);
    \draw[thick,color=red!90,densely dotted] (0.5in,0in) -- (0.75in,0in);
    \draw[thick,color=red!90,densely dotted] (0.5in,0in) -- (0.5in,-0.25in);	
	
\end{tikzpicture}
	\begin{tikzpicture}[scale=1.9]
   
    \draw[thick,color=red!90,densely dotted] (-0.5in,0) -- (0,0.5in) -- (0.5in,0) -- (-0.5in,0);
    \fill[fill=blue!10] (-0.5in,0) -- (0,0.5in) -- (0.5in,0) -- (-0.5in,0);
	
    \draw[->] (-0.4in,-0.05in) -- (-0.4in,-0.25in);
    \draw[->] (-0.3in,-0.05in) -- (-0.3in,-0.25in);
    \draw[->] (-0.2in,-0.05in) -- (-0.2in,-0.25in);
    \draw[->] (-0.1in,-0.05in) -- (-0.1in,-0.25in);
    \draw[->] (0in,-0.05in) -- (0in,-0.25in);
    \draw[->] (0.1in,-0.05in) -- (0.1in,-0.25in);
    \draw[->] (0.2in,-0.05in) -- (0.2in,-0.25in);
    \draw[->] (0.3in,-0.05in) -- (0.3in,-0.25in);
    \draw[->] (0.4in,-0.05in) -- (0.4in,-0.25in);
    \draw[->] (-0.0853553in,0.485355in) -- (-0.226777in,0.626777in);
    \draw[->] (-0.135355in,0.435355in) -- (-0.276777in,0.576777in);
    \draw[->] (-0.185355in,0.385355in) -- (-0.326777in,0.526777in);
    \draw[->] (-0.235355in,0.335355in) -- (-0.376777in,0.476777in);
    \draw[->] (-0.285355in,0.285355in) -- (-0.426777in,0.426777in);
    \draw[->] (-0.335355in,0.235355in) -- (-0.476777in,0.376777in);
    \draw[->] (-0.385355in,0.185355in) -- (-0.526777in,0.326777in);
    \draw[->] (-0.435355in,0.135355in) -- (-0.576777in,0.276777in);
    \draw[->] (-0.485355in,0.0853553in) -- (-0.626777in,0.226777in);
    \draw[->] (0.485355in,0.0853553in) -- (0.626777in,0.226777in);
    \draw[->] (0.435355in,0.135355in) -- (0.576777in,0.276777in);
    \draw[->] (0.385355in,0.185355in) -- (0.526777in,0.326777in);
    \draw[->] (0.335355in,0.235355in) -- (0.476777in,0.376777in);
    \draw[->] (0.285355in,0.285355in) -- (0.426777in,0.426777in);
    \draw[->] (0.235355in,0.335355in) -- (0.376777in,0.476777in);
    \draw[->] (0.185355in,0.385355in) -- (0.326777in,0.526777in);
    \draw[->] (0.135355in,0.435355in) -- (0.276777in,0.576777in);
    \draw[->] (0.0853553in,0.485355in) -- (0.226777in,0.626777in);

    \draw[->] (-0.546194in,-0.0191342in) -- (-0.73097in,-0.0956709in);
    \draw[->] (-0.543986in,0.0237743in) -- (-0.719931in,0.118871in);
    \draw[->] (-0.549981in,-0.0013738in) -- (-0.749906in,-0.00686898in);
    \draw[->] (-0.536313in,-0.0343706in) -- (-0.681567in,-0.171853in);
    \draw[->] (-0.514292in,-0.0479139in) -- (-0.57146in,-0.239569in);
    \draw[thick,color=red!90,densely dotted] (-0.5in,0in) -- (-0.676777in,0.176777in);
    \draw[thick,color=red!90,densely dotted] (-0.5in,0in) -- (-0.5in,-0.25in);
    
    \draw[->] (0in,0.55in) -- (0in,0.75in);
    \draw[->] (-0.0191342in,0.546194in) -- (-0.0956709in,0.73097in);
    \draw[->] (0.0191342in,0.546194in) -- (0.0956709in,0.73097in);
    \draw[thick,color=red!90,densely dotted] (0in,0.5in) -- (-0.176777in,0.676777in);
    \draw[thick,color=red!90,densely dotted] (0in,0.5in) -- (0.176777in,0.676777in);

    \draw[->] (0.546194in,-0.0191342in) -- (0.73097in,-0.0956709in);
    \draw[->] (0.514292in,-0.0479139in) -- (0.57146in,-0.239569in);
    \draw[->] (0.536313in,-0.0343706in) -- (0.681567in,-0.171853in);
    \draw[->] (0.549981in,-0.0013738in) -- (0.749906in,-0.00686898in);
    \draw[->] (0.543986in,0.0237743in) -- (0.719931in,0.118871in);
    \draw[thick,color=red!90,densely dotted] (0.5in,0in) -- (0.5in,-0.25in);
    \draw[thick,color=red!90,densely dotted] (0.5in,0in) -- (0.676777in,0.176777in);

\end{tikzpicture}
	\caption{The filled regions show two possible choices of $G$ in $\mathbb{R}^{2}$.  The arrows indicate the gradient of $\text{dist}(\cdot,G)$, while the dashed segments delineate jumps of the Hessian.}
	\label{F: distance picture}
\end{figure}

\subsection{Proofs of Propositions \ref{P: touching above test function part} and \ref{P: smooth case test function}} This section concludes with the two remaining proofs.

\begin{proof}[Proof of Proposition \ref{P: touching above test function part}] To lighten the notation, let $F = \partial \varphi^{*}(p)$ and $e = \|p\|^{-1} p$.  By \eqref{E: subdifferential basic identity dual}, the outward normal vector to $\{\varphi \leq 1\}$ equals $e$ at each point of $F$.  Thus, since, by hypothesis, $\{\varphi \leq 1\}$ satisfies an interior ball condition of radius $\zeta$, for any $q \in F$, there holds
	\begin{equation*}
		\overline{B}_{\zeta}(q - \zeta e) \subseteq \{\varphi \leq 1\}.
	\end{equation*}
In particular,  $F_{(e,\zeta)} \subseteq \{\varphi \leq 1\}$ by definition.  After a straightforward manipulation of the definition \eqref{E: test function definition} of $\psi_{e,F,\zeta}$, this implies
	\begin{equation*}
		\varphi \leq \psi_{e,F,\zeta} \quad \text{in} \, \, \mathbb{R}^{d}.
	\end{equation*}

Next, if $q \in F$, then $q \in F_{(e,\nu)}$ and, thus, the definition of $\psi_{e,F,\zeta}$ immediately implies that $\psi_{e,F,\zeta}(q) \leq 1$.  On the other hand, since $F \subseteq \{\varphi = 1\}$ by \eqref{E: subdifferential basic identity dual}, there holds $1 = \varphi(q) \geq \psi_{e,F,\zeta}(q)$.  Hence $\varphi = \psi_{e,F,\zeta} = 1$ in $F$, and, by homogeneity, $\varphi = \psi_{e,F,\zeta}$ in $\mathcal{C}(F)$.

Finally, in case $\nu < \zeta$, there holds $\varphi \leq \psi_{e,F,\zeta} \leq \psi_{e,F,\nu}$.  Further, note that for an arbitrary $q \in \mathbb{R}^{d}$, one has 
	\begin{equation*}
		\overline{B}_{\nu}(q - \nu e) \cap \overline{B}_{\zeta}(q - \zeta e) = \{q\}.
	\end{equation*}
From this, it follows that $\partial F_{(e,\nu)} \cap \partial F_{(e,\zeta)} = F$.  Thus, given $x \in \{\psi_{e,F,\zeta} < +\infty\} \cap \{\psi_{e,F,\nu} < +\infty\}$, there holds  $\psi_{e,F,\zeta}(x) = \psi_{e,F,\nu}(x)$ if and only if $x \in F$. \end{proof}

\begin{proof}[Proof of Proposition \ref{P: smooth case test function}] We will show that there is an $r > 0$ and a smooth function $\tilde{\psi} : B_{r}(q) \to \mathbb{R}$ such that $\tilde{\psi}(x) = \psi_{e,F,\nu}(x)$ for each $x \in B_{r}(q)$.  

Let $V_{F} \subseteq \mathbb{R}^{d}$ be the smallest linear subspace such that $F \subseteq q + V_{F}$ (cf.\ Section \ref{S: faces}).  Let $\tilde{F}$ be the corresponding tube, defined analogously to $F_{(e,\nu)}$, that is,
	\begin{equation*}
		\tilde{F} = \bigcup_{y \in q + V_{F}} \overline{B}_{\nu}(y - \nu e).
	\end{equation*}
Let $\tilde{\psi}$ be the Minkowski gauge of $\tilde{F}$:
	\begin{equation*}
		\tilde{\psi}(x) = \inf \left\{ \alpha > 0 \, \mid \, \frac{x}{\alpha} \in \tilde{F} \right\}.
	\end{equation*}
Since $q \in \text{rint}(F)$ by assumption, it follows that there is an $r_{1} > 0$ such that $B_{r_{1}}(q) \cap F = B_{r_{1}}(q) \cap (q + V_{F})$ and, thus, 
	\begin{equation*}
		B_{r_{1}}(q) \cap F_{(e,\nu)} = B_{r_{1}}(q) \cap \tilde{F}.
	\end{equation*}
This implies $\psi_{e,F,\nu} = \tilde{\psi}$ in the cone $\mathcal{C}(B_{r_{1}}(q))$, and, in particular, in $B_{r}(q)$ for an appropriate choice of $r$.  
	
It only remains to show that $\tilde{\psi}$ is smooth in a neighborhood of $q$.  Since $\tilde{\psi}$ is a positively homogeneous convex function, the results of the next section, particularly Propositions \ref{P: differentiability in terms of level sets} and \ref{P: check proposition} (and their proofs), are applicable.  Hence to see that $\tilde{\psi}$ is smooth in a neighborhood of $q$, it suffices to prove that $\partial \tilde{F}$ is a smooth hypersurface in a neighborhood of $q$.  

Let $V_{F}^{\perp}$ denote the orthogonal complement of $V_{F}$.  To see that $\tilde{F}$ is smooth in a neighborhood of $q$, it suffices to establish that 
	\begin{equation} \label{E: need to prove cylinder thing}
		\{x \in \tilde{F} \, \mid \, \langle x, e \rangle \geq \langle q, e \rangle - \nu\} = (q - \nu e) + V_{F} + (\overline{B}_{\nu}(0) \cap H_{e} \cap V_{F}^{\perp}),
	\end{equation}
where, as before, $H_{e} = \{x \in \mathbb{R}^{d} \, \mid \, \langle x,e \rangle \geq 0\}$.
Indeed, the right-hand side is a cylinder, hence its boundary is a smooth hypersurface; and $q$ is an element of the boundary of the left-hand side.  Therefore, it only remains to prove \eqref{E: need to prove cylinder thing}.  

As far as \eqref{E: need to prove cylinder thing} is concerned, suppose that $x \in \tilde{F}$ and $\langle x,e \rangle \geq \langle q, e \rangle - \nu$.  Fix a $y \in q + V_{F}$ such that $x \in \overline{B}_{\nu}(y - \nu e)$.  Write $x = y - \nu e + \xi_{F} + \xi_{F}^{\perp}$ with $\xi_{F} \in V_{F}$ and $\xi_{F}^{\perp} \in V_{F}^{\perp}$ to find 
	\begin{equation*}
		\|\xi_{F}^{\perp}\| \leq \|x - (y - \nu e)\| \leq \nu, \quad \langle \xi_{F}^{\perp}, e \rangle = \langle x, e \rangle - \langle y, e \rangle + \nu = \langle x,e \rangle - \langle q, e \rangle + \nu \geq 0
	\end{equation*}
since $y - q \in V_{F}$ and $V_{F} \subseteq \langle e \rangle^{\perp}$ holds by \eqref{E: hyperplane assumption}.  This proves
	\begin{equation*}
		-(q - \nu e) + \{x \in \tilde{F} \, \mid \, \langle x, e \rangle \geq \langle q, e \rangle - \nu\} \subseteq V_{F} + (\overline{B}_{\nu}(0) \cap H_{e} \cap V_{F}^{\perp}).
	\end{equation*}

Conversely, fix $\xi_{F} \in V_{F}$ and $\xi_{F}^{\perp} \in \overline{B}_{\nu}(0) \cap H_{e} \cap V_{F}^{\perp}$.  Let $y = q + \xi_{F}$.  Observe that if $x = y - \nu e + \xi_{F}^{\perp}$, then $x \in \overline{B}_{\nu}(y - \nu e) \subseteq \tilde{F}$ and 
	\begin{equation*}
		\langle x,e \rangle = \langle y, e \rangle - \nu + \langle \xi_{F}^{\perp}, e \rangle \geq \langle q, e \rangle - \nu.
	\end{equation*}
This proves the remaining inclusion, and hence \eqref{E: need to prove cylinder thing}.
\end{proof}

\part{Cone Comparison Principles for $C^{1,1}$ Finsler Norms} \label{Part: c11 cone comparison}

This part of the paper concerns the main technical work, namely, utilizing the conical test functions to prove generalized cone comparison principles, such as the one encountered in Section \ref{S: c11 reduction}.  The main technique is an approximation procedure via smoothing, which is tailor-made so that the essential differential structure remains intact.  At a high level, it is analogous to Evans' perturbed test function method \cite{evans_perturbed}.  

\section{Properties of Positively One-Homogeneous Convex Functions} \label{S: pos hom}

The conical test functions introduced in the previous section, like Finsler norms themselves, belong to the class of positively one-homogeneous convex functions.  As in the previous part of the paper, properties of such functions --- particularly the extent to which their regularity is determined by that of their level sets --- will play an important role in what follows.  

This section presents a number of elementary (but fundamental) facts about positively homogeneous convex functions, which will be used extensively in what follows.  These facts are more-or-less intuitive and not hard to prove, hence the reader may wish to skim this section at a first reading.

\begin{remark} In what follows, it will be convenient to consider convex functions taking values in $\mathbb{R} \cup \{+\infty\}$.  Toward that end, recall that a function $f : \mathbb{R}^{d} \to \mathbb{R} \cup \{+\infty\}$ is called convex if and only if the set $\{f < + \infty\}$ is convex and $f$ restricts to a real-valued convex function therein.

Furthermore, in order to include nontrivial functions taking value $+\infty$, it is important to give a precise definition of \emph{postively one-homogeneous}.  Specifically, henceforth, a function $f : \mathbb{R}^{d} \to \mathbb{R} \cup \{+\infty\}$ is positively one-homogeneous if, for any $x \in \mathbb{R}^{d}$ and $\lambda > 0$,
	\begin{equation*}
		f(\lambda x) = \lambda f(x).
	\end{equation*}
(It is necessary to exclude $\lambda = 0$ in this definition so that the set $\{f = + \infty\}$ need not be trivial.)
\end{remark}

\subsection{Subdifferential and Regularity}  As was already mentioned above in the special case of Finsler norms, the subdifferential of a positively one-homogeneous convex function has a relatively simple geometric representation.

Recall that the subdifferential $\partial f$ of a convex function $f : \mathbb{R}^{d} \to \mathbb{R} \cup \{+\infty\}$ is the set-valued function on $\{f < + \infty\}$ determined by the formula
	\begin{equation} \label{E: def subdifferential}
		\partial f(x) = \bigcap_{y \in \mathbb{R}^{d}} \{p \in \mathbb{R}^{d} \, \mid \, f(y) \geq f(x) + \langle p, y - x \rangle\}.
	\end{equation}
The next result recalls the classical fact that a positively one-homogeneous convex function has a very simple subdifferential.

	\begin{prop} \label{P: subdifferential} If $\psi : \mathbb{R}^{d} \to [0,+\infty]$ is convex and positively one-homogeneous, then, for any $x_{*} \in \text{int}(\{\psi < +\infty\}) \setminus \{0\}$,
		\begin{equation} \label{E: subdifferential thing}
			\partial \psi(x_{*}) = \{p \in \mathbb{R}^{d} \, \mid \, \langle p, x_{*} \rangle = \psi(x_{*})\} \cap \bigcap_{x \in \mathbb{R}^{d}} \{p \in \mathbb{R}^{d} \, \mid \, \langle p, x \rangle \leq \psi(x)\}.
		\end{equation}
	\end{prop}
	
		\begin{proof}  First, suppose that $x_{*} \in \text{int}(\{\psi < +\infty\}) \setminus \{0\}$ and $p \in \partial \psi(x_{*})$.  Since $x_{*}$ is in the interior of $\{\psi < +\infty\}$, there is an $\epsilon > 0$ such that $(1 + t)x_{*} \in \{\psi < +\infty\} \setminus \{0\}$ for any $t \in (-\epsilon,\epsilon)$.  Setting $x = (1 + t) x_{*}$ for $t \in (-\epsilon,\epsilon)$, one invokes \eqref{E: def subdifferential} to obtain
			\begin{equation*}
				(1 + t) \psi(x_{*}) = \psi((1 + t)x_{*}) \geq \psi(x_{*}) + t \langle p, x_{*} \rangle.
			\end{equation*}
		Sending $t \to 0^{+}$ yields $\psi(x_{*}) \geq \langle p, x_{*} \rangle$, while sending $t \to 0^{-}$ gives $\psi(x_{*}) \leq \langle p, x_{*} \rangle$.  Thus, $\psi(x_{*}) = \langle p, x_{*} \rangle$.
		
		Revisiting \eqref{E: def subdifferential} and substituting $\psi(x_{*}) = \langle p, x_{*} \rangle$, one obtains
			\begin{equation*}
				\psi(x) \geq \langle p, x \rangle \quad \text{for each} \quad x \in \mathbb{R}^{d}.
			\end{equation*}
		This proves one of the inclusions involved in \eqref{E: subdifferential thing}.
		
		Finally, it remains to prove the opposite inclusion.  Suppose that $x_{*} \in \text{int}(\{\psi < +\infty\}) \setminus \{0\}$ and $p \in \mathbb{R}^{d}$ is such that 
			\begin{equation*}
				\langle p, x \rangle \leq \psi(x) \quad \text{for each} \quad x \in \mathbb{R}^{d}, \quad \langle p, x_{*} \rangle = \psi(x_{*}).
			\end{equation*}
		To see that $p \in \partial \psi(x_{*})$, add and subtract $\langle p,x_{*} \rangle$ to find, for any $x \in \mathbb{R}^{d}$,
			\begin{equation*}
				\psi(x) \geq \langle p, x \rangle = \langle p, x_{*} \rangle + \langle p, x - x_{*} \rangle = \psi(x_{*}) + \langle p, x - x_{*} \rangle.
			\end{equation*}
		\end{proof}
		
The next result makes precise the observation that, in the interior of the domain $\{\psi < +\infty\}$, a positively one-homogeneous convex function $\psi$ is as regular as the level set $\{\psi = 1\}$.

	\begin{prop} \label{P: differentiability in terms of level sets} Let $\psi : \mathbb{R}^{d} \to [0,+\infty]$ be a positively one-homogeneous convex function and suppose that $x_{*} \in \text{int}(\{\psi < +\infty\}) \setminus \{0\}$.  Then $\psi$ is differentiable at $x_{*}$ if and only if there is a unique supporting hyperplane to $\{\psi \leq 1\}$ through the point $\psi(x_{*})^{-1} x_{*}$.  
	
	In particular, $\{\psi = 1\}$ is a $C^{k}$ hypersurface in the interior of $\{\psi < +\infty\}$ if and only if $\psi$ is a $C^{k}$ function in the interior of $\{\psi < +\infty\}$. \end{prop}
	
		\begin{proof} First, note that if $p \in \partial \psi(x_{*})$, then 
			\begin{equation*}
				\langle p, q \rangle \leq \psi(q) \leq 1 = \langle p, \psi(x_{*})^{-1} x_{*} \rangle \quad \text{for each} \quad q \in \{\psi \leq 1\}.
			\end{equation*}
		Thus, each element of $\partial \psi(x_{*})$ determines a supporting hyperplane to $\{\psi \leq 1\}$ through $\psi(x_{*})^{-1} x_{*}$.  
		
		Conversely, suppose that $p \in \mathbb{R}^{d} \setminus \{0\}$ determines such a hyperplane, that is, $\langle p, q - \psi(x_{*})^{-1} x_{*} \rangle \leq 0$ for each $q \in \{\psi \leq 1\}$.  Replacing $p$ by $\lambda p$ for an appropriate choice of $\lambda > 0$, there is no loss of generality assuming that $\langle p, x_{*} \rangle = \psi(x_{*})$.  Since $\langle p, q \rangle \leq 1$ for any $q \in \{\psi \leq 1\}$, one readily deduces by positive-homogeneity that $\langle p, y \rangle \leq \psi(y)$ for any $y \in \{\psi < +\infty\}$.  Thus, by Proposition \ref{P: subdifferential}, $p \in \partial \psi(x_{*})$.  
		
		The previous two paragraphs show that $p \in \partial \psi(x_{*})$ holds if and only if $p$ determines a supporting hyperplane to $\{\psi \leq 1\}$ through the point $\psi(x_{*})^{-1}x_{*}$.  Finally, recall that $\psi$ is differentiable at $x_{*}$ if and only if there is a vector $p \in \mathbb{R}^{d}$ such that $\partial \psi(x_{*}) = \{p\}$.\footnote{See \cite[Theorem 9.18]{rockafellar_wets} or \cite[Theorem 1.5.15]{schneider}.}  Therefore, $\psi$ is differentiable at $x_{*}$ if and only if there is a unique supporting hyperplane to $\{\psi \leq 1\}$ through $\psi(x_{*})^{-1}x_{*}$.
		
		Finally, if $\psi$ is differentiable at a point $x$, then the positive one-homogeneity of $\psi$ implies that $\psi$ is differentiable at each point of $\mathcal{C}(\{x\})$ and
	\begin{equation*}
		D\psi(x) = D\psi(\check{x}) = \|D\psi(\check{x})\| N(\check{x}),
	\end{equation*}
where $\check{x} = \psi(x)^{-1} x$ and $N(\check{x})$ is the outward normal vector to $\{\psi \leq 1\}$ at $\check{x}$. 
At the same time, by Proposition \ref{P: subdifferential}, 
	\begin{equation*}
		1 = \psi(\check{x}) = \langle D\psi(\check{x}), \check{x} \rangle,
	\end{equation*}
which implies that
	\begin{equation*}
		\|D\psi(\check{x})\|^{-1} = \langle N(\check{x}), \check{x} \rangle.
	\end{equation*}
Therefore,
	\begin{equation} \label{E: gradient formula}
		D\psi(x) = \langle N(\check{x}), \check{x} \rangle^{-1} N(\check{x}).
	\end{equation}
From this, one readily deduces that if $\{\psi \leq 1\}$ is a $C^{k}$ hypersurface locally in the interior of $\{\psi < +\infty\}$, then $D\psi$ is $C^{k -1}$ in that set and hence $\psi$ is $C^{k}$ there.  The converse follows immediately from the implicit function theorem. \end{proof}

\subsection{Curvature Bounds} \label{S: curvature bounds} In the next sections, approximation arguments will be introduced that will require somewhat delicate information concerning the Hessians of the test functions involved.  Toward this end, it will be convenient to know that curvature bounds on the level sets of a positively homogeneous convex function can be propagated to Hessian bounds.  

The next result shows that the Hessian of a positively homogeneous convex function can be expressed in terms of the normal vector and second fundamental form (derivative of the normal vector) of one of its level sets.

	\begin{prop} \label{P: check proposition} Let $\psi : \mathbb{R}^{d} \to [0,+\infty]$ be a positively one-homogeneous convex function.  Assume that $\psi$ is locally $C^{1,1}$ in the interior of $\{\psi < +\infty\}$, and let $N : \{\psi = 1\} \to S^{d-1}$ be the (outward) normal vector to $\{\psi \leq 1\}$.  Define a map $x \mapsto \check{x}$ from $\{\psi < + \infty\}$ into $\{\psi = 1\}$ by 
		\begin{equation*}
			\check{x} = \frac{x}{\psi(x)}.
		\end{equation*}  
	Then, for every $x \in \{\psi < + \infty\}$ at which $D^{2}\psi(x)$ exists, there holds
		\begin{equation} \label{E: chain rule guy}
			D^{2}\psi(x) = \psi(x)^{-1} A(x) \cdot B(x),
		\end{equation}
	where $A$ and $B$ are the matrix fields given by 
		\begin{align*}
			A(x) &= \langle N(\check{x}),\check{x} \rangle^{-1} DN(\check{x}) - \langle N(\check{x}), \check{x} \rangle^{-2} (N(\check{x}) \otimes \check{x}) DN(\check{x}) - \langle N(\check{x}), \check{x} \rangle^{-2} N(\check{x}) \otimes N(\check{x}), \\
			B(x) &= \text{Id} - \check{x} \otimes D\psi(x).
		\end{align*}
	\end{prop}
	
		\begin{proof} The formula \eqref{E: chain rule guy} follows directly from \eqref{E: gradient formula} and the chain rule. \end{proof}

The previous expression for $D^{2}\psi$ provides very convenient Hessian bounds in terms of the curvature of the level set $\{\psi = 1\}$.

	\begin{prop} \label{P: curvature bound for pos hom} Let $\psi : \mathbb{R}^{d} \to [0,+\infty]$ be a positively one-homogeneous convex function.  Assume that $\psi$ is locally $C^{1,1}$ in the open set $\text{int}(\{\psi < +\infty\})$, and define the map $x \mapsto \check{x}$ as in Proposition \ref{P: check proposition}.  Then
		\begin{align*}
			\|D^{2}\psi(x)\| \leq \psi(x)^{-1} \left[ 2 \left( 1 + \|D\psi(x)\|^{2} \right) \left(1 + \|\check{x}\| \right) (1 + \|DN(\check{x})\|) \right] \left( 1 + \|D\psi(x)\| \|\check{x}\| \right)		
		\end{align*}
	at any point in $x \in \text{int}(\{\psi < +\infty\})$ at which $D^{2}\psi(x)$ exists. \end{prop}
	
	\begin{proof} In view of \eqref{E: gradient formula} and \eqref{E: chain rule guy}, it suffices to observe that the matrix functions $A$ and $B$ of the previous proposition are controlled by the trivial bounds
		\begin{align*}
			\|A(x)\| &\leq \|D\psi(x)\| \|DN(x)\| + \|D\psi(x)\|^{2} \|\check{x}\| \|DN(x)\| + \|D\psi(x)\|^{2}, \\
				&\leq 2(1 + \|D\psi(x)\|^{2}) (1 + \|\check{x}\|) (1 + \|DN(x)\|), \\
			\|B(x)\| &\leq 1 + \|D\psi(x)\| \|\check{x}\|.
		\end{align*}
	Thus, since $\|D^{2}\psi(x)\| \leq \psi(x)^{-1} \|A(x)\| \|B(x)\|$, the desired inequality follows. \end{proof}
	
The following trivial consequence of Proposition \ref{P: curvature bound for pos hom} will be used repeatedly in the arguments that follow:
	
	\begin{prop} \label{P: trivial hessian estimate} There is a continuous function $\Gamma : (0,1) \to (0,1)$ such that, for any $\delta \in (0,1)$, the constant $\Gamma(\delta) > 0$ has the following property: suppose that $\psi : \mathbb{R}^{d} \to [0,+\infty]$ is a positively homogeneous convex function and $\psi$ is locally $C^{1,1}$ in the interior of $\{\psi < +\infty\}$, fix $\zeta \in (0,1)$, and let $A \subseteq \{\psi = 1\} \cap \overline{B}_{\delta^{-1}}(0)$.  If for any $q \in A$, there holds
		\begin{equation*}
			\|DN(q)\| \leq \zeta^{-1} \quad \text{and} \quad \|D\psi(q)\| \leq \delta^{-1},
		\end{equation*}
	then $\psi$ satisfies the Hessian estimate
		\begin{equation*}
			\psi(x) D^{2}\psi(x) \leq \Gamma(\delta) \zeta^{-1} \text{Id} \quad \text{for almost every} \quad x \in \mathcal{C}(A).
		\end{equation*}
	\end{prop}
	
		\begin{proof} Suppose that $x \in \mathcal{C}(A) \setminus \{0\}$, hence $x \in \mathbb{R}^{d} \setminus \{0\}$ and $\psi(x)^{-1}x \in \overline{B}_{\delta^{-1}}(0)$.  If $D^{2}\psi(x)$ exists, then Proposition \ref{P: curvature bound for pos hom} applies to give, by hypothesis,
			\begin{equation*}
				\psi(x) D^{2}\psi(x) \leq 2 ( 1 + \delta^{-1}) (1 + \delta^{-1}) (1 + \zeta^{-1}) (1 + \delta^{-2}).
			\end{equation*}
		Since $\max\{\zeta,\delta\} < 1$, this yields the desired estimate with $\Gamma(\delta) = 4 (1 + \delta^{-2})^{3}$. \end{proof}
	
\subsection{Gradient Bounds} \label{S: gradient bounds} This section concludes with an easy gradient bound that will be particularly convenient when used in conjunction with the curvature bounds of the previous subsection.
 
 	\begin{prop} \label{P: very easy gradient bound} Let $\varphi$ be a Finsler norm in $\mathbb{R}^{d}$.  If there is a $\delta > 0$ such that 
		\begin{equation*}
			\overline{B}_{\delta}(0) \subseteq \{\varphi \leq 1\},
		\end{equation*}
	then $\{\varphi^{*} \leq 1\} \subseteq \overline{B}_{\delta^{-1}}(0)$.  
	
	In particular, if $\varphi$ is $C^{1}$ away from the origin and $\overline{B}_{\delta}(0) \subseteq \{\varphi \leq 1\}$, then $\|D\varphi(x)\| \leq \delta^{-1}$ for each $x \in \mathbb{R}^{d} \setminus \{0\}$. \end{prop}
	
		\begin{proof} Assume that $\overline{B}_{\delta}(0) \subseteq \{\varphi \leq 1\}$, and fix $p \in \{\varphi^{*} = 1\}$.  If $q \in \partial B_{\delta}(0)$, then $q \in \{\varphi \leq 1\}$ and, thus, by duality,
			\begin{equation*}
				\langle q, p \rangle \leq \varphi(q) \varphi^{*}(p) \leq 1.
			\end{equation*}
		It follows that $\|p\| \leq \delta^{-1}$.  This proves $\{\varphi^{*} = 1\} \subseteq \overline{B}_{\delta^{-1}}(0)$, hence also $\{\varphi^{*} \leq 1\} \subseteq \overline{B}_{\delta^{-1}}(0)$ by scaling.
	
		Finally, if $\varphi$ is $C^{1}$ away from the origin and $\overline{B}_{\delta}(0) \subseteq \{\varphi \leq 1\}$, then, for any $x \in \mathbb{R}^{d} \setminus \{0\}$, the variational formula \eqref{E: subdifferential basic identity} implies that
			\begin{equation*}
				\{D\varphi(x)\} = \partial \varphi(x) \subseteq \{\varphi^{*} = 1\}.
			\end{equation*}
		Thus, $\|D\varphi(x)\| \leq \delta^{-1}$ for every $x \in \mathbb{R}^{d} \setminus \{0\}$. \end{proof}
		
\subsection{Bounding the Hessians of the Approximations} \label{S: hessian bound approximation} To motivate the results above, here is a simple application, namely, the $O(\zeta^{-1})$ Hessian bound stated earlier.  The statement is recalled below for the convenience of the reader:  

\basichessianbound*

\begin{proof}The point is to combine combine conclusion (i) from the approximation theorem, Theorem \ref{T: regularized_norm}, with the key Hessian bound, Proposition \ref{P: trivial hessian estimate}.
		
		In particular, recall from Proposition \ref{P: proof of approx theorem 1} that $\|D\varphi_{\zeta}\|$ is uniformly bounded independently of $\zeta \in (0,1)$.  Further, as is well known, condition (i) implies that the derivative of the normal vector to $\{\varphi_{\zeta} \leq 1\}$ is bounded above by $\zeta^{-1} \text{Id}$; see \cite{lewicka_peres}.  As a result of these two bounds, Proposition \ref{P: trivial hessian estimate} implies the desired estimate. \end{proof}

\section{$C^{1,1}$ Cone Comparison Principles via Perturbation} \label{S: c11 setting}

In this section, cone comparison principles are proved in the special case when the unit ball $\{\varphi \leq 1\}$ is $C^{1,1}$, or, equivalently, $\varphi \in C^{1,1}_{\text{loc}}(\mathbb{R}^{d} \setminus \{0\})$.  This will be achieved by utilizing the conical test functions of Section \ref{S: conical test}.

As pointed out already above, difficulties arise due to the fact that the conical test functions are not globally smooth.  Thus, new ideas are necessary to show that they still yield useful information at contact points.  This is the main obstacle that will be overcome in this section.  

In order to fully justify the computations in Section \ref{S: c11 reduction}, it will be necessary to quantify the proofs.  Accordingly, the following running assumption will be used throughout this section:

\begin{assumption} \label{A: c11 assumption} $\varphi$ is a Finsler norm in $\mathbb{R}^{d}$ such that, for some $\zeta, \delta > 0$,
	\begin{gather}
		\{\varphi \leq 1\} \quad \text{satisfies an interior ball condition of radius} \, \, \zeta, \label{E: zeta interior ball} \\
		\overline{B}_{\delta}(0) \subseteq \{\varphi \leq 1\} \subseteq \overline{B}_{\delta^{-1}}(0). \label{E: nondegeneracy}
	\end{gather}
\end{assumption}

%Eventually, general comparison principles will be proved by combining the $C^{1,1}$ results proved in this section with the $C^{1,1}$ approximation theorem from Section \ref{S: regularized}.  Accordingly, the main result of this section has a rather abstract statement; see Theorem \ref{T: technical core} below.  

%For now, the reader may wish to keep in mind the following result, which is the most immediate consequence of the analysis of this section and extends Proposition \ref{P: strictly convex c11 thing}:
%
%	\begin{prop} \label{P: c11 thing} If $\{\varphi \leq 1\}$ is of class $C^{1,1}$, then $\text{Sub}_{\varphi}(U) = CCA_{\varphi}(U)$ for any open set $U \subseteq \mathbb{R}^{d}$. \end{prop}  
%	
%\begin{remark} Note that, unlike Proposition \ref{P: strictly convex c11 thing}, the assumption that $\{\varphi \leq 1\}$ is of class $C^{1,1}$ does not imply that $G_{\varphi}^{*} \equiv G_{*}^{\varphi}$ in $(\mathbb{R}^{d} \setminus \{0\}) \times \mathcal{S}_{d}$.\end{remark}
	
\subsection{Abstract Cone Comparison Principles} The generalized cone comparison principle that appeared in Section \ref{S: c11 reduction} follows from a more general result, stated next.  Toward that end, given $\alpha \in \mathbb{R}$, recall the shifted infinity Laplacian $\mathcal{G}^{*,\alpha}_{\varphi}$ defined by
	\begin{align*}
		\mathcal{G}^{*,\alpha}_{\varphi}(p,X) &= \max \left\{ Q_{X}(q - \alpha \|p\|^{-1} p) \, \mid \, q \in \partial \varphi^{*}(p) \right\} \quad \text{if} \quad (p,X) \in (\mathbb{R}^{d} \setminus \{0\}) \times \mathcal{S}_{d}, \\
		\mathcal{G}^{*,\alpha}_{\varphi}(0,X) &= \limsup_{0 \neq p \to 0} \mathcal{G}^{\alpha}_{\varphi}(p,X).
	\end{align*}
The arguments of Section \ref{S: c11 reduction} above motivated the interest in these operators.

The main result of this section concerns differential inequalities of the form
	\begin{equation} \label{E: general equation model}
		-\mathcal{H}(x,u,Du,D^{2}u,\mathcal{G}^{*,\alpha}_{\varphi}(Du,D^{2}u)) \leq 0,
	\end{equation}
where $\mathcal{H} : \mathbb{R}^{d} \times \mathbb{R} \times \mathbb{R}^{d} \times \mathcal{S}_{d} \times \mathbb{R} \to \mathbb{R}$ is a given function and $\alpha \in \mathbb{R}$ is fixed.  

The assumptions on $\mathcal{H}$ will be as follows:
	\begin{gather}
		\mathcal{H} \, \, \text{is upper semicontinuous in} \, \, \mathbb{R}^{d} \times \mathbb{R} \times \mathbb{R}^{d} \times \mathcal{S}_{d} \times \mathbb{R}, \label{Ass: continuity assumption} \\
		\mathcal{H}(x,r,p,X,a) \leq \mathcal{H}(x,s,p,Y,b) \quad \text{if} \quad r \geq s, \, \, X \leq Y, \, \, \text{and} \, \, a \leq b. \label{Ass: elliptic assumption}
	\end{gather}

%In order to apply the results that follow in the non-$C^{1,1}$ setting, it will be useful to quantify the results slightly.  Toward that end, it is convenient to quantify the $C^{1,1}$ property of $\varphi$ in terms of the radius of curvature from \eqref{E: zeta interior ball}:
%	\begin{equation} \label{E: hessian bound}
%		\varphi(x) D^{2} \varphi(x) \leq \Gamma \zeta^{-1} \text{Id} \quad \text{for almost every} \quad x \in \mathbb{R}^{d} \setminus \{0\}.
%	\end{equation}

	\begin{theorem} \label{T: technical core} Assume that $\varphi$ satisfies Assumption \ref{A: c11 assumption}.  Let $V \subseteq \mathbb{R}^{d} \setminus \{0\}$ be a bounded open set, fix $\alpha \in \mathbb{R}$, and let $\mathcal{H}$ be a function satisfying \eqref{Ass: continuity assumption} and \eqref{Ass: elliptic assumption}. 
	
	If $u \in USC(\overline{V})$ satisfies
		\begin{equation*}
			-\mathcal{H}(x,Du,D^{2}u,\mathcal{G}_{\varphi}^{*,\alpha}(Du,D^{2}u)) \leq 0 \quad \text{in} \, \, V,
		\end{equation*}
	and
		\begin{equation*}
			\max \{u(x) - \varphi(x) \, \mid \, x \in \overline{V} \} > \max \{u(x) - \varphi(x) \, \mid \, x \in \partial V \},
		\end{equation*}
	then there is a triple $(x_{*},X_{*},q_{**}) \in V \times \mathcal{S}_{d} \times \{\varphi = 1\}$ such that
		\begin{gather*}
			u(x_{*}) - \varphi(x_{*}) = \max \{u(x) - \varphi(x) \, \mid \, x \in \overline{V} \}, \\
			- \mathcal{H}(x_{*},u(x_{*}),D\varphi(x_{*}),X_{*}, Q_{X_{*}}(q_{**} - \alpha \|D\varphi(x_{*})\|^{-1} D\varphi(x_{*}))) \leq 0, \\
			X_{*} q_{**} = 0, \quad \text{and} \quad 0 \leq \varphi(x_{*}) X_{*} \leq \Gamma(\delta) \zeta^{-1} \text{Id}.
		\end{gather*}
	Above $\Gamma(\delta)$ is the constant from Proposition \ref{P: trivial hessian estimate}.
	\end{theorem}

\subsection{The Touching Argument} \label{S: touching discussion} Here is the outline of the proof of Theorem \ref{T: technical core}.  As in the statement of the theorem, suppose that $u \in USC(\overline{V})$ and $x_{0} \in V$ are such that 
	\begin{gather*}
		-\mathcal{H}(x,u,Du,D^{2}u,\mathcal{G}_{\varphi}^{*,\alpha}(Du,D^{2}u)) \leq 0 \quad \text{in} \, \, V, \\
		u(x_{0}) - \varphi(x_{0}) = \max \left\{ u(x) - \varphi(x) \, \mid \, x \in \overline{V} \right\} > \max \left\{ u(y) - \varphi(y) \, \mid \, y \in \partial V \right\}.
	\end{gather*}
Define $M = \max_{\overline{V}} (u - \varphi) = u(x_{0}) - \varphi(x_{0})$.  Since $\varphi \in C^{1,1}_{\text{loc}}(\mathbb{R}^{d} \setminus \{0\})$ and $x_{0} \in V \subseteq \mathbb{R}^{d} \setminus \{0\}$, one can define vectors $p \in \mathbb{R}^{d} \setminus \{0\}$ and $e \in S^{d-1}$ by setting
	\begin{equation*}
		p = D\varphi(x_{0}), \quad e = \|p\|^{-1} p.
	\end{equation*} 

Since the $\zeta$-interior ball condition \eqref{E: zeta interior ball} is satisfied, the family of conical test functions $\{\psi_{e,\partial \varphi^{*}(p),\nu} \, \mid \, 0 < \nu \leq \zeta\}$ all touch $u$ above at $x_{0}$.  More precisely, for any $\nu \in (0,\zeta]$, one readily verifies using Proposition \ref{P: touching above test function part} that
	\begin{gather}
		u(x_{0}) = M + \psi_{e,\partial \varphi^{*}(p),\nu}(x_{0}), \quad u \leq M + \psi_{e,\partial \varphi^{*}(p),\nu} \quad \text{in} \, \, V, \label{E: worth noting later} \\
		\{x \in \overline{V} \, \mid \, u(x) = M + \psi_{e,\partial \varphi^{*}(p),\nu}(x)\} \subseteq V. \label{E: touch strictly inside}
	\end{gather}
Since it will be used later, it is worth noting at this stage that \eqref{E: worth noting later} and \eqref{E: touch strictly inside} hold for any $\nu \in (0,\zeta]$ if and only if they hold for $\nu = \zeta$.  This is a consequence of the ordering $\psi_{\zeta} \leq \psi_{\nu}$ and the inclusion $x_{0} \in \mathcal{C}(\partial \varphi^{*}(p)) = \{\psi_{\zeta} = \psi_{\nu} < + \infty\}$.  

For convenience, since $p$ and $e$ will be fixed henceforth, the following shorthand notation will be used throughout this section:
	\begin{equation*}
		\psi_{\nu} := \psi_{e,\partial \varphi^{*}(p),\nu}.
	\end{equation*}
	
As already mentioned in Section \ref{S: conical test}, while $\psi_{\nu}$ touches $u$ from above at $x_{0}$, unfortunately $\psi_{\nu}$ is not a smooth test function in general.  Nonetheless, it still has the ``right shape," which suggests a perturbation argument.  Toward that end, in this section, a family of smooth perturbations $(\psi_{\nu}^{\epsilon})_{\epsilon \in (0,\nu)}$ will be introduced, which converge to $\psi_{\nu}$ as $\epsilon \to 0^{+}$ while retaining the main geometric properties of $\psi_{\nu}$.  

In particular, assume that there is an open set $\mathcal{U}$ such that $\mathcal{C}(\partial \varphi^{*}(p)) \subseteq \mathcal{U}$, a parameter $\epsilon_{*} \in (0,\nu)$, and nonnegative functions $(\psi_{\nu}^{\epsilon})_{\epsilon \in (0,\nu)}$ such that 
	\begin{equation*}
		\psi_{\nu}^{\epsilon} \to \psi_{\nu} \quad \text{locally uniformly in} \, \, \mathcal{U} \quad \text{as} \quad \epsilon \to 0^{+}
	\end{equation*}
and, for any $\epsilon < \epsilon_{*}$,
	\begin{equation*}
		\psi_{\nu}^{\epsilon} \quad \text{is smooth in} \, \, \mathcal{U} \quad \text{and} \quad \{\psi_{\nu}^{\epsilon} < +\infty\} = \mathcal{U}.
	\end{equation*}

Given any $\epsilon \in (0,\nu)$, fix a point $x_{\epsilon} \in \overline{V}$ such that 
	\begin{equation*}
		u(x_{\epsilon}) - \psi_{\nu}^{\epsilon}(x_{\epsilon}) = \max \left\{ u(x) - \psi_{\nu}^{\epsilon}(x) \, \mid \, x \in \overline{V} \right\}.
	\end{equation*}
Due to \eqref{E: touch strictly inside} and the compactness of $\overline{V}$, $x_{\epsilon} \in V \cap \mathcal{U}$ provided $\epsilon$ is small enough.
Thus, since $\psi_{\nu}^{\epsilon}$ is smooth in $\mathcal{U}$, the subsolution property of $u$ implies that
	\begin{equation*}
		-\mathcal{H}(x_{\epsilon},D\psi_{\nu}^{\epsilon}(x_{\epsilon}), D^{2}\psi_{\nu}^{\epsilon}(x_{\epsilon}), \mathcal{G}^{*,\alpha}_{\varphi}(D\psi_{\nu}^{\epsilon}(x_{0}), D^{2}\psi_{\nu}^{\epsilon}(x_{0}))) \leq 0.
	\end{equation*}	
By definition of $\mathcal{G}^{*,\alpha}_{\varphi}$, there is a point $q_{\epsilon} \in \partial \varphi^{*}(D\psi_{\nu}^{\epsilon}(x_{\epsilon}))$ such that
	\begin{equation} \label{E: I reference this later}
		\mathcal{G}^{*,\alpha}_{\varphi}(D\psi_{\nu}^{\epsilon}(x_{\epsilon}),D^{2}\psi_{\nu}^{\epsilon}(x_{\epsilon})) = Q_{D^{2}\psi_{\nu}^{\epsilon}(x_{\epsilon})}(q_{\epsilon} - \alpha \|D\psi_{\nu}^{\epsilon}(x_{\epsilon})\|^{-1} D\psi_{\nu}^{\epsilon}(x_{\epsilon})).
	\end{equation}
The key now is to show that the perturbation $(\psi_{\nu}^{\epsilon})_{0 < \epsilon < \nu}$ can be constructed in such a way that, along subsequences,
	\begin{equation*}
		(x_{\epsilon},q_{\epsilon},D^{2}\psi_{\nu}^{\epsilon}(x_{\epsilon})) \to (x_{*},q_{**},X_{*}),
	\end{equation*}
where the triple $(x_{*},q_{**},X_{*})$ has the property
	\begin{equation*}
		X_{*} q_{**} = 0.
	\end{equation*}
Below the perturbation $(\psi_{\nu}^{\epsilon})_{\epsilon \in (0,\nu)}$ is constructed precisely so that this is the case.

In order to avoid repetition later and to prepare for future results, it is worth observing that the previous discussion amounts to a proof of the following proposition.

	\begin{prop} \label{P: basic touching thing} Assume that $\varphi$ satisfies Assumption \eqref{A: c11 assumption}.  Let $V \subseteq \mathbb{R}^{d} \setminus \{0\}$ be a bounded open set, fix an $\alpha \in \mathbb{R}$, and let $\mathcal{H}$ be a function satisfying \eqref{Ass: continuity assumption} and \eqref{Ass: elliptic assumption}.
	
	Fix a $\nu \in (0,\zeta]$ and a $p \in \mathbb{R}^{d} \setminus \{0\}$.  Let $e = \|p\|^{-1} p$ and write $\psi_{\mu} = \psi_{e,\partial \varphi^{*}(p),\mu}$ for $\mu \in \{\nu,\zeta\}$.  Assume that there is an open set $\mathcal{U} \subseteq \mathbb{R}^{d}$, a constant $\epsilon_{*} > 0$, and a family of nonnegative functions $(\psi^{\epsilon}_{\nu})_{\epsilon \in (0,\nu)}$ such that $\mathcal{C}(\partial \varphi^{*}(p)) \subseteq \mathcal{U}$ and
		\begin{gather*}
			\psi_{\nu}^{\epsilon} \to \psi_{\nu} \quad \text{locally uniformly in} \, \, \mathcal{U} \quad \text{as} \quad \epsilon \to 0^{+}, \\
			\psi_{\nu}^{\epsilon} \quad \text{is smooth in} \, \, \mathcal{U} \quad \text{and} \quad \{\psi_{\nu}^{\epsilon} < + \infty\} = \mathcal{U} \quad \text{for each} \quad \epsilon \in (0,\epsilon_{*}).
		\end{gather*}
	
	If $u \in USC(\overline{V})$, $M \in \mathbb{R}$, and $x_{0} \in V \cap \mathcal{C}(\partial \varphi^{*}(p))$ are such that
		\begin{gather*}
			-\mathcal{H}(x,u,Du,D^{2}u,\mathcal{G}_{\varphi}^{*,\alpha}(Du,D^{2}u)) \leq 0 \quad \text{in} \, \, V, \\
			u \leq M + \psi_{\zeta} \quad \text{in} \, \, \overline{V}, \quad u(x_{0}) = M + \psi_{\zeta}(x_{0}), \\
			\{x \in \overline{V} \, \mid \, u(x) = M + \psi_{\zeta}(x)\} \subseteq V,
		\end{gather*}
	then there are sequences $(\epsilon_{j})_{j \in \mathbb{N}} \subseteq (0,+\infty)$ and $(x_{j})_{j \in \mathbb{N}} \subseteq V \cap \mathcal{U}_{\nu}$ such that $\epsilon_{j} \to 0$ as $j \to +\infty$ and, for each $j \in \mathbb{N}$,
		\begin{gather}
			u(x_{j}) - \psi_{\nu}^{\epsilon_{j}}(x_{j}) = \max \left\{ u(x) - \psi_{\nu}^{\epsilon_{j}}(x) \, \mid \, x \in \overline{V} \right\}, \label{E: touching above key mountain} \\
			-\mathcal{H}(x_{j},u(x_{j}),D\psi_{\nu}^{\epsilon_{j}}(x_{j}),D^{2}\psi_{\nu}^{\epsilon_{j}}(x_{j}), \mathcal{G}_{\varphi}^{*,\alpha}(D\psi_{\nu}^{\epsilon_{j}}(x_{j}),D^{2}\psi_{\nu}^{\epsilon_{j}}(x_{j}))) \leq 0. \label{E: subsolution part key mountain}
		\end{gather}
	\end{prop}

	\subsection{Smoothing the Test Function}  \label{S: smoothing the graph} Assume, as in the previous discussion, that $x_{0} \in \mathbb{R}^{d} \setminus \{0\}$, and let $p = D\varphi(x_{0})$ and $e = \|D\varphi(x_{0})\|^{-1} D\varphi(x_{0})$.

The question now --- and here is the main technical hurdle of this section --- is: how to perturb $\psi_{e,\partial \varphi^{*}(p),\nu}$ at the contact point so as to obtain ``better" contact points, where the viscosity subsolution property of $u$ can be invoked?  This will be analogous to the classical approach in the proof of Jensen's Lemma,\footnote{See \cite[Lemma A.3]{user} or \cite[Lemma 11.2]{primer}.} except that the Alexandrov-Bakelman-Pucci argument will be replaced by completely different techniques involving smoothing the test function.

As suggested in the previous subsection, it will be convenient to consider the entire family $\{\psi_{e,\partial \varphi^{*}(p),\nu} \, \mid \, 0 < \nu \leq \zeta\}$.  Again, the following shorthand will be used:
	\begin{equation*}
		\psi_{\nu} := \psi_{e,\partial \varphi^{*}(p),\nu}	
	\end{equation*}

To overcome the fact that $\psi_{\nu}$ may not be smooth at the contact point $x_{0}$, a family of functions $(\psi_{\nu}^{\epsilon})_{\epsilon \in (0,\nu)}$ will be constructed in such a way as to exploit the graphical representation of the surface $\{\psi_{\nu} = 1\}$ introduced in Section \ref{S: graphical representation}.  

First, following the discussion in that subsection, define the convex set $G_{p} \subseteq \{e\}^{\perp}$; the distance function $\text{dist}(\cdot,G_{p}) : \mathbb{R}^{d-1} \to [0,+\infty)$; and the function $g_{\nu}$ according to the following formulas:
	\begin{align} 
		G_{p} &= \{x' \in \{e\}^{\perp} \, \mid \, x' + \varphi^{*}(e) e \in \partial \varphi^{*}(p) \} = \partial \varphi^{*}(p) - \varphi^{*}(e)e, \nonumber \\
		\text{dist}(x',G_{p}) &= \inf \left\{ \|x' - y'\| \, \mid \, y' \in G_{p} \right\}, \nonumber \\
		g_{\nu}(x') &= \varphi^{*}(e) - \nu + \sqrt{\nu^{2} - \text{dist}(x',G_{p})^{2}}, \label{E: graphical special case}
	\end{align}
where, to reiterate, here $\mathbb{R}^{d-1}$ is identified with $\{e\}^{\perp}$.  Note that $g_{\nu}$ is concave since $\text{dist}(\cdot,G_{p})$ is convex. 

Recall from Section \ref{S: graphical representation} that if $F$ denotes the face $F = \partial \varphi^{*}(p) = \partial \varphi^{*}(e)$ and $F_{(e,\nu)}$ is defined as in \eqref{E: fancy convolution}, then
	\begin{align}
		F_{(e,\nu)} \cap \left\{x \in \mathbb{R}^{d} \, \mid \, \langle x, e \rangle \geq \varphi^{*}(e) - \nu \right\} &= \{x' + se \, \mid \, \text{dist}(x',G_{p}) \leq \nu, \label{E: bulk part needed} \\
		&\qquad \quad \varphi^{*}(e) - \nu \leq s \leq g_{\nu}(x') \, \, \}, \nonumber \\
		\partial F_{(e,\nu)} \cap \left\{ x \in \ \, \mid \, \langle x,e \rangle \geq \varphi^{*}(e) - \nu \right\} &= \{x' + g_{\nu}(x')e \, \mid \, \text{dist}(x',G_{p}) \leq \nu\}, \label{E: graph part needed} \\
		F &= \{x' + g_{\nu}(x') e \, \mid \, x' \in G_{p}\}. \nonumber
	\end{align}
Recalling that $\psi_{\nu}$ is nothing but the Minkowski gauge \eqref{E: test function definition} of $F_{(e,\nu)}$, this shows that, at least close to $\partial \varphi^{*}(p)$, the hypersurface $\{\psi_{\nu} = 1\}$ equals the graph of $g_{\nu}$.

More precisely, and in preparation for what is to come, the connection between $\psi_{\nu}$ and $g_{\nu}$ can be described in the following way: for any small enough $r > 0$, there is an open set $\mathcal{U}(r)$ such that $\mathcal{C}(\partial \varphi^{*}(p)) \subseteq \mathcal{U}(r) \subseteq \{\psi_{\nu} < + \infty\}$ and, for any $x \in \mathcal{U}(r)$,
	\begin{align} \label{E: graph correspondence}
		\psi_{\nu}(x) = 1 \quad \text{if and only if} \quad \text{dist}\left(x',G_{p}\right) < r \quad \text{and} \quad \langle x, e \rangle = g_{\nu} \left( x' \right).
	\end{align}
Note that the ``only if" follows more-or-less directly from the definition of $\psi_{\nu}$; the interest lies in the ``if" direction, which, indeed, requires $r$ to be small.  Since this statement follows from Theorem \ref{T: existence admissible perturbation} stated below, justification is deferred until later.

The connection between $\psi_{\nu}$ and $g_{\nu}$ encapsulated by \eqref{E: graph correspondence} will be exploited in the construction of the perturbations $(\psi_{\nu}^{\epsilon})_{\epsilon \in (0,\nu)}$.  In particular, $\psi_{\nu}^{\epsilon}$ will be a convex positively one-homogeneous function constructed specifically so that a version of \eqref{E: graph correspondence} holds, except with $g_{\nu}$ replaced by a smoothed function $g_{\nu}^{\epsilon}$.   

Toward that end, let $\rho : \mathbb{R}^{d - 1} \to [0,+\infty)$ be a smooth function such that 
	\begin{gather}
		\int_{\mathbb{R}^{d-1}} \rho(x') \, dx' = 1, \quad \rho(x') = 0 \quad \text{if and only if} \quad \|x'\| \geq 1, \label{E: mollifier positive integral etc} \\
		\rho(x') = \rho(y') \quad \text{if} \quad \|x'\| = \|y'\|. \label{E: symmetry mollifier}
	\end{gather}
Given $\epsilon \in (0,\nu)$, let $g^{\epsilon}_{\nu}$ be function obtained by mollifying $g_{\nu}$ against $\rho$ at scale $\epsilon$:
	\begin{equation} \label{E: mollification}
		g^{\epsilon}_{\nu}(x') = \int_{B_{1}} g_{\nu}(x' + \epsilon y') \rho(y') \, dy'.
	\end{equation}
Since $g_{\nu}$ is concave, Jensen's inequality implies that $g^{\epsilon}_{\nu}$ is concave and it lies below $g_{\nu}$.  These facts are recorded below for later use:
	\begin{equation} \label{E: useful concavity stuff}
		g_{\nu}^{\epsilon} \, \, \text{is concave} \quad \text{and} \quad g_{\nu}^{\epsilon} \leq g_{\nu} \quad \text{in} \, \, \{\text{dist}(\cdot,G_{p}) \leq \nu - \epsilon\}.
	\end{equation}  

Now that $g_{\nu}$ has been regularized, the approximation $\psi_{\nu}^{\epsilon} : \mathbb{R}^{d} \to [0,+\infty]$ of $\psi_{\nu}$ will be constructed so that $\{\psi_{\nu}^{\epsilon} = 1\}$ equals the graph of $g_{\nu}^{\epsilon}$ close to $\partial \varphi^{*}(p)$.  To make this precise, and to divorce the details of the construction from the properties needed in this section, here is a definition of the type of approximation that will be needed:

	\begin{definition} \label{D: admissible} Given a conical test function $\psi_{\nu}$ as above, a family of convex, positively one-homogeneous functions $(\psi_{\nu}^{\epsilon})_{\epsilon \in (0,\nu)}$ is said to be an \emph{admissible perturbation of $\psi_{\nu}$} if there is an open set $\mathcal{U}_{\nu} \subseteq \mathbb{R}^{d}$, a parameter $\epsilon_{*} > 0$, and a mollifier $\rho$ satisfying \eqref{E: mollifier positive integral etc} and \eqref{E: symmetry mollifier} such that (i)
		\begin{align*}
			(\psi_{\nu}^{\epsilon}, D\psi_{\nu}^{\epsilon}) \to (\psi_{\nu},D \psi_{\nu}) \quad \text{locally uniformly in} \, \, \mathcal{U}_{\nu} \quad \text{as} \quad \epsilon \to 0^{+};
		\end{align*}
	(ii) for any $\epsilon \in (0,\epsilon_{*})$,
		\begin{gather*}
			\psi_{\nu}^{\epsilon} \quad \text{is smooth in} \, \, \mathcal{U}_{\nu}, \quad \{\psi_{\nu}^{\epsilon} < +\infty\} = \overline{\mathcal{U}_{\nu}}, \\
			\text{and} \quad \psi_{\nu} \leq \psi_{\nu}^{\epsilon} \quad \text{pointwise in} \quad \mathcal{U}_{\nu};		
		\end{gather*}
	and (iii) there is an $r > 0$ such that, for any $x \in \mathcal{U}_{\nu}$ and any $\epsilon \in (0,\epsilon_{*})$, 
		\begin{align} \label{E : graph property ack ack ack}
			\psi_{\nu}^{\epsilon}(x) = 1 \quad \text{if and only if} \quad \text{dist}(x', G_{p}) < r \quad \text{and} \quad \langle x, e \rangle = g_{\nu}^{\epsilon}(x').
		\end{align}
	\end{definition} 

Notice that \eqref{E : graph property ack ack ack} mirrors the claimed property \eqref{E: graph correspondence} of $\psi_{\nu}$.  
	
	The above definition spells out the properties of the perturbation that will be used in the arguments that follow.  That still leaves the question of existence, but this is somewhat intuitive.  Indeed, a natural place to start is by considering the Minkowski gauge of the hypograph of $g_{\nu}^{\epsilon}$, which is certainly a convex set since $g_{\nu}^{\epsilon}$ is concave.  The main issue --- and for this reason the existence problem warrants a careful argument --- is the existence of an open set $\mathcal{U}_{\nu}$ for which the level surface $\{\psi_{\nu}^{\epsilon} = 1\}$ can be characterized as in \eqref{E : graph property ack ack ack}.
	
	However, since the proof of the existence of a perturbation satisfying \eqref{E : graph property ack ack ack} is not immediately relevant here, the proof is deferred to Appendix \ref{A: perturbed test functions}.  For now, the next theorem asserts that existence does hold and the graph characterization \eqref{E : graph property ack ack ack} directly extends to $\psi_{\nu}$:
	
		\begin{theorem} \label{T: existence admissible perturbation} Assume that $\varphi$ satisfies Assumption \ref{A: c11 assumption}.  Let $\psi_{\nu} = \psi_{e,\partial \varphi^{*}(p),\nu}$ be the conical test function associated to the parameters $p \in \mathbb{R}^{d} \setminus \{0\}$, $e = \|p\|^{-1} p$, and $\nu \in (0,\zeta]$.  Then, for any mollifier $\rho$ satisfying \eqref{E: mollifier positive integral etc} and \eqref{E: symmetry mollifier}, there is an admissible perturbation $(\psi_{\nu}^{\epsilon})_{\epsilon \in (0,\nu)}$ of $\psi_{\nu}$.
		
		Furthermore, if $\mathcal{U}_{\nu}$ and $r$ are as in the definition, then, for any $x \in \mathcal{U}_{\nu}$,
			\begin{equation*}
				\psi_{\nu}(x) = 1 \quad \text{if and only if} \quad \text{dist}(x',G_{p}) < r \quad \text{and} \quad \langle x', e \rangle= g_{\nu}(x').
			\end{equation*} \end{theorem}

The next result shows that touching $u$ above with $\psi^{\epsilon}_{\nu}$ in place of $\psi_{\nu}$ leads to a useful outcome in the limit $\epsilon \to 0^{+}$.  

	\begin{prop}  \label{P: key mountain} Assume that $\varphi$ satisfies Assumption \eqref{A: c11 assumption}.  Let $V \subseteq \mathbb{R}^{d} \setminus \{0\}$ be a bounded open set.
	
Fix a $p \in \mathbb{R}^{d} \setminus \{0\}$ and let $e = \|p\|^{-1}p$.  For any $\nu \in (0,\zeta]$, let $\psi_{\nu} := \psi_{e,\partial \varphi^{*}(p),\nu}$ be the associated conical test function.
	
	Suppose there is a $u \in USC(\overline{V})$, an $M \in \mathbb{R}$, and an $x_{0} \in V \cap \mathcal{C}(\partial \varphi^{*}(p))$ such that
		\begin{gather} 
			u \leq M + \psi_{\zeta} \quad \text{in} \, \, \overline{V}, \quad u(x_{0}) = M + \psi_{\zeta}(x_{0}), \label{E: basic touching ack} \\
			 \{x \in \overline{V} \, \mid \, u(x) = M + \psi_{\zeta}(x)\} \subseteq V. \label{E: key localization}
		\end{gather}
	Fix a $\nu < \zeta$; let $(\epsilon_{j})_{j \in \mathbb{N}} \subseteq (0,+\infty)$ and $(x_{j})_{j \in \mathbb{N}} \subseteq V \times \mathcal{S}_{d} \times \{\varphi = 1\}$ be sequences such that \eqref{E: touching above key mountain} holds for each $j$; and let $(\psi_{\nu}^{\epsilon})_{\epsilon \in (0,\nu)}$ be an admissible perturbation of $\psi_{\nu}$.
	
	If $((X_{j},\tilde{q}_{j}))_{j \in \mathbb{N}} \subseteq \mathcal{S}_{d} \times \{\varphi = 1\}$ is a sequence such that, for each $j \in \mathbb{N}$, 
		\begin{equation}
			X_{j} = D^{2}\psi_{\nu}^{\epsilon_{j}}(x_{j}), \quad \tilde{q}_{j} \in \partial \varphi^{*}(D\psi_{\nu}^{\epsilon_{j}}(x_{j})), \label{E: nice things key mountain}
		\end{equation}
	then the following statements hold:
		\begin{itemize}
			\item[(i)] $((x_{j},X_{j},\tilde{q}_{j}))_{j \in \mathbb{N}}$ is precompact in $V \times \mathcal{S}_{d} \times \{\varphi = 1\}$.  In particular,
				\begin{equation}
					\limsup_{j \to \infty} \varphi(x_{j}) \|X_{j}\| \leq \Gamma(\delta) \nu^{-1}. \label{E: upper bound key mountain}
				\end{equation}
			\item[(ii)] If $(x_{*},X_{*},q_{**})$ is any accumulation point of $((x_{j},X_{j},\tilde{q}_{j}))_{j \in \mathbb{N}}$, then there is a face $F_{*} \subseteq \partial \varphi^{*}(D\varphi(x_{*}))$ such that 
				\begin{equation} \label{E: key mountain conclusion}
					(a) \quad \left\{ \varphi(x_{*})^{-1}x_{*}, q_{**} \right\} \subseteq F_{*} \quad \text{and} \quad (b) \quad X_{*} q = 0 \quad \text{for each} \quad q \in F_{*}. 
			\end{equation}
		\end{itemize}
	\end{prop}

Since it is somewhat technical, the proof of Proposition \ref{P: key mountain} is deferred to the next section.  Instead, for the moment, here is how the proposition leads to a proof of Theorem \ref{T: technical core}.  

	\begin{proof}[Proof of Theorem \ref{T: technical core}] To begin the proof, fix an $x_{0} \in V$ such that 
		\begin{equation*}
			u(x_{0}) - \varphi_{\zeta}(x_{0}) = \max \left\{ u(x) - \varphi(x) \, \mid \, x \in \overline{V} \right\} > \max \left\{ u(x) - \varphi(x) \, \mid \, x \in \partial V \right\}.
		\end{equation*}
	In view of the discussion of Section \ref{S: touching discussion}, there is a $p \in \mathbb{R}^{d} \setminus \{0\}$ such that if $e = \|p\|^{-1} p$ and $\psi_{\zeta} = \psi_{e,\partial \varphi_{\zeta}^{*}(p),\zeta}$, then \eqref{E: basic touching ack} and \eqref{E: key localization} hold.  Given any $\nu \in (0,\zeta)$, Proposition \ref{P: basic touching thing} implies there are sequences $(\epsilon_{j})_{j \in \mathbb{N}} \subseteq (0,+\infty)$ and $(x_{j})_{j \in \mathbb{N}} \subseteq V$ such that \eqref{E: touching above key mountain} and \eqref{E: subsolution part key mountain} hold for each $j \in \mathbb{N}$.
	
	Define $(X_{j})_{j \in \mathbb{N}} \subseteq \mathcal{S}_{d}$ by setting $X_{j} = D^{2}\psi_{\nu}^{\epsilon_{j}}(x_{j})$.  By \eqref{E: subsolution part key mountain}, for any $j \in \mathbb{N}$, there is a $\tilde{q}_{j} \in \{\varphi = 1\}$ such that
		\begin{gather}
			- \mathcal{H}(x_{j},u(x_{j}),D\psi_{\nu}^{\epsilon_{j}}(x_{j}), X_{j}, Q_{X_{j}}(\tilde{q}_{j} - \alpha \|D\psi_{\nu}^{\epsilon_{j}}(x_{j})\|^{-1} D\psi_{\nu}^{\epsilon_{j}}(x_{j}))) \leq 0, \label{E: need another equation here ack} \\ 
			\text{and} \quad \tilde{q}_{j} \in \partial \varphi^{*}(D\psi_{\nu}^{\epsilon_{j}}(x_{j})). \nonumber
		\end{gather}
	By Proposition \ref{P: key mountain}, there is no loss of generality assuming there is a limit $(x_{*},X_{*},q_{**}) = \lim_{j \to \infty} (x_{j},X_{j},\tilde{q}_{j})$ and a face $F_{*}$ of $\partial \varphi^{*}(D\varphi(x_{*}))$ such that \eqref{E: key mountain conclusion} holds.
	
In particular, after defining $q_{*} = \varphi(x_{*})^{-1} x_{*}$, \eqref{E: key mountain conclusion} says that 
	\begin{equation} \label{E: vanishing on face part of proof use}
		\{q_{*},q_{**}\} \subseteq F_{*} \quad \text{and} \quad X_{*} q = 0 \quad \text{for each} \quad q \in F_{*}.
	\end{equation}

Note that $u(x_{*}) = \lim_{j \to \infty} u(x_{j})$ due to \eqref{E: touching above key mountain} and the local uniform convergence of $\psi^{\epsilon_{j}}_{\nu} \to \psi_{\nu}$.  Thus, since $\mathcal{H}$ is upper semicontinuous and $\psi_{\nu}$ is $C^{1}$, sending $j \to +\infty$ in \eqref{E: need another equation here ack} yields\footnote{Here is where one needs that $D\psi_{\nu}^{\epsilon} \to D\psi_{\nu}$ locally uniformly, as prescribed by Definition \ref{D: admissible}.}
	\begin{equation*}
		- \mathcal{H}(x_{*},u(x_{*}),D\psi_{\nu}(x_{*}),X_{*}, Q_{X_{*}}(q_{**} - \alpha \|D\psi_{\nu}(x_{*})\|^{-1} D\psi_{\nu}(x_{*}))) \leq 0.
	\end{equation*}
Notice that since $x_{*} \in \mathcal{C}(\partial \varphi^{*}(p))$, Proposition \ref{P: touching above test function part} implies that
	\begin{equation*}
		D\psi_{\nu}(x_{*}) = D\varphi(x_{*}).
	\end{equation*}
Thus, the previous inequality can be rewritten in the form
	\begin{equation} \label{E: need to reference in proof 1}
		-\mathcal{H}(x_{*},u(x_{*}),D\varphi(x_{*}),X_{*}, Q_{X_{*}}(q_{**} - \alpha \|D\varphi(x_{*})\|^{-1} D\varphi(x_{*}) ) ) \leq 0.
	\end{equation}

Since $x_{*} \in V \subseteq \mathbb{R}^{d} \setminus \{0\}$, the positivity of $\varphi$ away from zero implies $\varphi(x_{*}) > 0$, and, thus, by \eqref{E: upper bound key mountain},
	\begin{equation} \label{E: hessian_bound_last_minute}
		0 \leq X_{*} \leq \varphi(x_{*})^{-1} \Gamma(\delta) \nu^{-1} .
	\end{equation}
	
The combination of \eqref{E: vanishing on face part of proof use}, \eqref{E: need to reference in proof 1}, and \eqref{E: hessian_bound_last_minute} yields the desired conclusion in the limit $\nu \to \zeta^{+}$ (i.e., after possibly extracting a limit point of the triples $(x_{*},q_{**},X_{*})$). \end{proof}

\section{The Perturbation Argument} \label{S: key mountain}

This section presents the proof of Proposition \ref{P: key mountain}, the main technical ingredient in the perturbation argument presented above.  The bulk of the work consists in showing that the derivatives of the perturbed test functions really are well-behaved in the limit $\epsilon \to 0^{+}$.  This will be a consequence of the main technical step in this section, namely, Proposition \ref{P: key mountain part 1}, stated next.  

While the proof of Proposition \ref{P: key mountain part 1} is geometrically intuitive (in low dimensions), making it rigorous is greatly facilitated through the use of the so-called normal cone of convex analysis.  This will become apparent multiple times in the proof.  Toward that end, Section \ref{S: normal cones} reviews the definition and relevant properties.  The proof of Proposition \ref{P: key mountain part 1} appears in Section \ref{S: most technical mountain proof}, while the proof of Proposition \ref{P: key mountain} comes earlier in Section \ref{S: mountain proof}.

\subsection{Taming the Perturbation}  The next result distills the main intermediate step involved in the proof of Proposition \ref{P: key mountain}.  In particular, the time has come to establish that the derivatives of the perturbed test functions $\psi^{\epsilon}_{\nu}$ defined in Section \ref{S: smoothing the graph} are well-behaved in the limit $\epsilon \to 0^{+}$.

More precisely, recall what remains to be proved.  As in Section \ref{S: c11 setting}, for the remainder of the section, assume that $\varphi$ satisfies Assumption \ref{A: c11 assumption}.  Fix a $\nu \in (0,\zeta]$ and a $p \in \mathbb{R}^{d} \setminus \{0\}$, and let $e = \|p\|^{-1} p$.  The objective is to show that at a contact point $x_{0} \in \mathcal{C}(p)$ where $\psi_{\nu}$ touches the subsolution $u$ from above, if $x_{\epsilon}$ is a contact point obtained after replacing $\psi_{\nu}$ by the perturbation $\psi_{\nu}^{\epsilon}$, and if $q_{\epsilon} \in \partial \varphi^{*}(D\psi_{\nu}^{\epsilon}(x_{\epsilon}))$, then, up to extraction,
	\begin{align*}
		(x_{\epsilon},q_{\epsilon},D^{2}\psi_{\nu}^{\epsilon}(x_{\epsilon})) \to (x_{*},q_{**},X_{*}) \quad \text{as} \quad \epsilon \to 0^{+},
	\end{align*}
where
	\begin{gather*}
		X_{*} q_{**} = 0, \quad 0 \leq \varphi(x_{*}) X_{*} \leq \tilde{\Gamma}(\varphi) \nu^{-1}.
	\end{gather*}
The next result will be used to prove this.

In the statement, as usual, the representation $x = x' + \langle x,e \rangle e$ is used so that $x' \in \{e\}^{\perp}$ denotes the orthogonal projection onto $\{e\}^{\perp}$ for an arbitrary $x \in \mathbb{R}^{d}$.

	\begin{prop} \label{P: key mountain part 1} As in the discussion above, let $(\psi_{\nu},g_{\nu})$ be the conical test function together with the function parametrizing its level sets, and let $(\psi_{\nu}^{\epsilon},g_{\nu}^{\epsilon})_{0 < \epsilon < \nu}$ be the perturbed versions.
	
	(a) Fix an arbitrary sequence $(\epsilon_{j})_{j \in \mathbb{N}} \subseteq (0,+\infty)$ such that $\epsilon_{j} \to 0$ as $j \to \infty$.  If there is a sequence $(q_{j},\tilde{q}_{j})_{j \in \mathbb{N}} \subseteq \mathbb{R}^{d} \times \{\varphi \leq 1\}$ and a pair $(q_{*},q_{**}) \in \partial \varphi^{*}(p) \times \partial \varphi^{*}(p)$ such that 
		\begin{itemize}
			\item[(i)] $\lim_{j \to \infty} (q_{j},\tilde{q}_{j}) = (q_{*},q_{**})$;
			\item[(ii)] for any $j \in \mathbb{N}$, 
				\begin{gather}
					q_{j} \in \{\psi_{\nu}^{\epsilon_{j}} = 1\} \label{E: orthogonal projection part} \quad 
			\text{and} \quad \tilde{q}_{j} \in \partial \varphi^{*}(D\psi_{\nu}^{\epsilon_{j}}(q_{j})); \quad \text{and} \nonumber
				\end{gather}
			\item[(iii)] the following smallness condition is satisfied:
				\begin{align}
					\limsup_{j \to \infty} \epsilon_{j}^{-1} \text{dist} &(q_{j}',G_{p}) < +\infty, \label{E: distance nuisance}
				\end{align}
		\end{itemize}
	then there is a subsequence $(j_{k})_{k \in \mathbb{N}} \subseteq \mathbb{N}$, a vector $w'_{*} \in \{e\}^{\perp}$, and a matrix $A'_{*} \in \mathcal{S}_{d}$ such that
		\begin{align}
			-w_{*}' &= \lim_{k \to \infty} \epsilon_{j_{k}}^{-1} Dg_{\nu}^{\epsilon_{j_{k}}}(q_{j_{k}}'), \label{E: first derivative thingy} \\
			A_{*}' &= \lim_{k \to \infty} D^{2}g_{\nu}^{\epsilon_{j_{k}}}(q_{j_{k}}'). \label{E: second derivative thingy}
		\end{align}
	Furthermore,
		\begin{equation} \label{E: optimality condition directional}
			\langle q_{*}, w_{*}' \rangle = \langle q_{**}, w_{*}' \rangle = \max \left\{ \langle q, w_{*}' \rangle \, \mid \, q \in \partial \varphi^{*}(p) \right\}.
		\end{equation}
	In particular, if $F_{*}$ is the face of $\partial \varphi^{*}(p)$ given by 
		\begin{equation} \label{E: F* definition}
			F_{*} = \{q \in \partial \varphi^{*}(p) \, \mid \, \langle q, w_{*}' \rangle = \langle q_{*}, w_{*}' \rangle\},
		\end{equation}
	then $\{q_{*},q_{**}\} \subseteq F_{*}$.
	
	(b) If, in addition, $(Y_{j})_{j \in \mathbb{N}} \subseteq \mathcal{S}_{d}$ is defined by 
		\begin{equation*}
			Y_{j} = D^{2}\psi_{\nu}^{\epsilon_{j}}(q_{j}),
		\end{equation*}
	then
		\begin{equation} \label{E: hessian bounds ack ack}
			\limsup_{j \to \infty} \|Y_{j}\| \leq \Gamma(\delta)  \nu^{-1}.
		\end{equation}
	Furthermore, if $(j_{k})_{k \in \mathbb{N}}$ is a subsequence as in part (a) with corresponding vector $w_{*}'$ and face $F_{*}$, and if there is a $Y_{*} \in \mathcal{S}_{d}$ such that $Y_{j_{k}} \to Y_{*}$ as $k \to +\infty$, then
		\begin{equation} \label{E: hessian constraint}
			Y_{*}q = 0 \quad \text{for each} \quad q \in F_{*}.
		\end{equation}\end{prop}

The proof of this proposition, which is deferred until Section \ref{S: most technical mountain proof}, will begin by establishing part (a), which will then be leveraged to prove part (b).  The proof of (a), which identifies the face $F_{*}$,  amounts to an analysis of the curvature of the level set of $\psi_{\nu}$ near the face $\partial \varphi^{*}(p)$.  Here, due to the lack of smoothness, curvature is understood in terms of directional derivatives of the normal vector, evaluated along subsequences as in Lemma \ref{L: direction of approach} below.  Once the conceptual issues are clear, the explicit definition \eqref{E: graphical special case} of $g_{\nu}$ reduces the task to checking certain conditions via calculus.  

The proof of (b), which determines the limiting Hessian $Y_{*}$, is more involved: it becomes necessary to understand the curvature of the level set $\{\psi_{\nu} = 1\}$ interpreted in terms of the structure of the Hessian $D^{2}\psi_{\nu}$, hence Lemma \ref{L: direction of approach} no longer suffices.  In order to make rigorous the intuitive sense that the Hessians of the perturbed test functions should curve ``in the right directions," it turns out that arguing in terms of faces and normal cones clarifies the picture considerably.  Once these notions are brought in, the work once again reduces to straightforward calculations.

\subsection{Application to Cone Comparison} \label{S: mountain proof} The stage is almost set for the proof of Proposition \ref{P: key mountain}, which is an application of Proposition \ref{P: key mountain part 1}.  Still, one key condition needs to be checked, namely, the smallness condition imposed as a hypothesis in Proposition \ref{P: key mountain part 1}.  

The next result verifies this condition.  This is where the strict inequality $\nu < \zeta$ is important.

		\begin{prop} \label{P: smallness condition king} Let $\varphi$ satisfy Assumption \ref{A: c11 assumption} and $V \subseteq \mathbb{R}^{d} \setminus \{0\}$ be open and bounded.
		
		Fix a $p \in \mathbb{R}^{d} \setminus \{0\}$, set $e = \|p\|^{-1} p$, and let $(\psi_{\nu})_{0 < \nu \leq \zeta} = \{\psi_{e,\partial \varphi^{*}(p),\zeta} \, \mid \, 0 < \nu \leq \zeta\}$ be the conical test functions of Section \ref{S: conical test}. 
	
	Fix a $\nu \in (0,\zeta)$ and let $(\psi_{\nu}^{\epsilon})_{0 < \epsilon < \nu}$ be an admissible perturbation of $\psi_{\nu}$.
	
	Suppose that there is a $u \in USC(\overline{V})$, an $M \in \mathbb{R}$, and an $x_{0} \in V \cap \mathcal{C}(\partial \varphi^{*}(p))$ such that 
		\begin{gather}
			u \leq M + \psi_{\zeta} \quad \text{in} \, \, \overline{V}, \quad u(x_{0}) = M + \psi_{\zeta}(x_{0}), \label{E: hopefully the last time touching above} \\
			\{x \in \overline{V} \, \mid \, u(x) = M + \psi_{\zeta}(x) \} \subseteq V. \label{E: stay in the cone appendix}
		\end{gather}
	If $(x_{\epsilon})_{0 < \epsilon < \nu} \subseteq \overline{V}$ and $(q_{\epsilon})_{0 < \epsilon < \nu}$ satisfy, for each $\epsilon \in (0,\nu)$,
		\begin{equation*}
			u(x_{\epsilon}) - \psi_{\nu}^{\epsilon}(x_{\epsilon}) = \max \left\{ u(y) - \psi^{\epsilon}_{\nu}(y) \, \mid \, y \in \overline{V} \right\}, \quad q_{\epsilon} = \psi_{\nu}^{\epsilon}(x_{\epsilon})^{-1} x_{\epsilon},
		\end{equation*}
then
		\begin{equation} \label{E: smallness condition finally}
			\limsup_{\epsilon \to 0^{+}} \epsilon^{-1} \text{dist}(q_{\epsilon}',G_{p}) < +\infty.
		\end{equation} 
	(Here, as throughout this section, $q \mapsto q'$ denotes the orthogonal projection onto $\{e\}^{\perp}$.)
	\end{prop}
	
The somewhat technical proof of Proposition \ref{P: smallness condition king} is presented in Appendix \ref{A: perturbed test functions} alongside the construction of admissible perturbations.  Still, at this stage, it may be worthwhile to consider why the proposition is true.  Here is an analogous, but much simpler situation: suppose that $u : [-1,1] \to \mathbb{R}$ is upper semicontinuous and 
	\begin{equation*}
		u(x) \leq \frac{x^{2}}{2 \zeta} \quad \text{for each} \quad x \in [-1,1], \quad u(0) = 0,
	\end{equation*}
i.e., the quadratic $x \mapsto \frac{x^{2}}{2\zeta}$ touches $u$ from above at $0$. 
Fix a $\nu < \zeta$ and, for any $\epsilon > 0$, let $v_{\nu}^{\epsilon}$ be the function obtained by mollifying the quadratic $x \mapsto \frac{x^{2}}{2 \nu}$ at scale $\epsilon$, i.e., $v_{\nu}^{\epsilon}(x) = \frac{1}{2 \nu} \int_{-1}^{1} y^{2} \rho(x + \epsilon y) \, dy$ for some mollifying kernel $\rho$.  Fix points $(x_{\epsilon})_{\epsilon > 0}$ with
	\begin{equation*}
		u(x_{\epsilon}) - v_{\nu}^{\epsilon}(x_{\epsilon}) = \max \left\{ u(y) - v_{\nu}^{\epsilon}(y) \, \mid \, y \in [-1,1] \right\} \quad \text{for each} \quad \epsilon > 0.
	\end{equation*}
In view of the fact that $\frac{x^{2}}{2 \nu} - \frac{x^{2}}{2 \zeta} = C(\nu,\zeta) x^{2}$ for some $C(\nu,\zeta) > 0$, it is not hard to show that
	\begin{equation*}
		|x_{\epsilon}| \leq C \epsilon,
	\end{equation*}
where the constant depends only on $\nu$, $\zeta$, and the mollifier $\rho$.  That is, after the mollification, the new contact points cannot stray too far from the original one.

The above analogy is slightly too simple, yet Proposition \ref{P: smallness condition king} is essentially the same situation, except with the quadratic $\frac{x^{2}}{2 \nu}$ replaced by $\psi_{\nu}$ and $|\cdot|$ (the distance to the origin) replaced by $\text{dist}(\cdot,G_{p})$ (the distance to the set $G_{p} = \{g_{\zeta} = g_{\nu}\}$). 

Taking Propositions \ref{P: key mountain part 1} and \ref{P: smallness condition king} for granted for now, here is the proof of Proposition \ref{P: key mountain}, which boils down to concatenating the previous string of propositions.

\begin{proof}[Proof of Proposition \ref{P: key mountain}]  Begin with the proof of (i), that is, the claim that the sequence $((x_{j},X_{j},\tilde{q}_{j}))_{j \in \mathbb{N}}$ is precompact in $V \times \mathcal{S}_{d} \times \{\varphi = 1\}$ and \eqref{E: upper bound key mountain} holds.  Toward that end, since $\overline{V}$ and $\{\varphi = 1\}$ is compact, it suffices to show that (i) if $x_{*}$ is any limit point of $(x_{j})_{j \in \mathbb{N}}$, then $x_{*} \in V$; and (ii) the matrices $(X_{j})_{j \in \mathbb{N}}$ are bounded.  

Let $x_{*}$ be an accumulation point of $(x_{j})_{j \in \mathbb{N}}$.  Without loss of generality (passing to a subsequence if necessary), assume that $x_{*} = \lim_{j \to \infty} x_{j}$.  Observe that \eqref{E: touching above key mountain} implies
	\begin{equation*}
		\lim_{\epsilon \to 0^{+}} [u(x_{\epsilon}) - \psi_{\nu}^{\epsilon}(x_{\epsilon})] \geq \max \left\{ u(y) - \psi_{\nu}(y) \, \mid \, y \in \overline{V} \right\}.
	\end{equation*}
Thus, by upper semicontinuity of $u$ and the local uniform convergence $\psi^{\epsilon}_{\nu} \to \psi_{\nu}$,
	\begin{equation} \label{E: touching bit}
		u(x_{*}) - \psi_{\nu}(x_{*}) = \max \left\{ u(y) - \psi_{\nu}(y) \, \mid \, y \in \overline{V} \right\},
	\end{equation}
and then \eqref{E: key localization} implies that $x_{*} \in V$.

In fact, by Proposition \ref{P: touching above test function part} and the strict inequality $\nu < \zeta$, \eqref{E: touching bit} and \eqref{E: key localization} together imply that $x_{*} \in V \cap \mathcal{C}(\partial \varphi^{*}(p))$, and then $x_{*} \neq 0$ since $V \subseteq \mathbb{R}^{d} \setminus \{0\}$.  It follows that one can define $(q_{j})_{j \in \mathbb{N}} \subseteq \mathbb{R}^{d}$ by setting $q_{j} = \psi_{\nu}^{\epsilon_{j}}(x_{j})^{-1} x_{j}$ for all $j$ sufficiently large.  Defining $q_{*} = \psi_{\nu}(x_{*})^{-1} x_{*}$, the local uniform convergence $\psi_{\nu}^{\epsilon_{j}} \to \psi_{\nu}$ implies $q_{*} = \lim_{j \to \infty} q_{j}$.  Further, by Proposition \ref{P: smallness condition king}, if $q_{j}'$ denotes the component of $q_{j}$ orthogonal to $e$, then
	\begin{equation*}
		\limsup_{j \to \infty} \epsilon_{j}^{-1} \text{dist}(q_{j}',G_{p}) < +\infty.
	\end{equation*} 
Thus, the sequence $((q_{j},\tilde{q}_{j}))_{j \in \mathbb{N}}$ satisfies the assumptions (i)-(iii) of Proposition \ref{P: key mountain part 1}.  Invoking that proposition, one deduces that if $(Y_{j})_{j \in \mathbb{N}}$ is given by 
	\begin{equation*}
		Y_{j} = D^{2}\psi_{\nu}^{\epsilon_{j}}(q_{j}),
	\end{equation*}
then $(Y_{j})_{j \in \mathbb{N}}$ satisfies \eqref{E: hessian bounds ack ack}.  At the same time, observe that, by the definition \eqref{E: nice things key mountain} of $X_{j}$ and homogeneity, one can write
	\begin{equation*}
		X_{j} = \psi_{\nu}^{\epsilon_{j}}(x_{j})^{-1} Y_{j} \quad \text{for each} \quad j \in \mathbb{N}.
	\end{equation*}
Accordingly, since $\psi_{\nu}^{\epsilon_{j}}(x_{j}) \to \psi_{\nu}(x_{*})$ as $j \to \infty$, it follows that $(X_{j})_{j \in \mathbb{N}}$ is bounded, and the bound \eqref{E: hessian bounds ack ack} can be converted directly into the bound \eqref{E: upper bound key mountain}.  This completes the proof of (i).

It remains to prove (ii).  Toward that end, suppose that $(x_{*},X_{*},q_{**})$ is any accumulation point of the sequence $((x_{j},X_{j},\tilde{q}_{j}))_{j \in \mathbb{N}}$.  Up to passing to a subsequence, there is no loss of generality assuming that $(x_{*},X_{*},q_{**}) = \lim_{j \to \infty} (x_{j},X_{j},\tilde{q}_{j})$. 

As in the previous step of the proof, it is expedient to define $(q_{j})_{j \in \mathbb{N}}$ and $(Y_{j})_{j \in \mathbb{N}}$ such that, for all $j$ sufficiently large,  
	\begin{equation*}
		q_{j} = \psi_{\nu}^{\epsilon_{j}}(x_{j})^{-1} x_{j}, \quad Y_{j} = D^{2}\psi_{\nu}(q_{j}).
	\end{equation*} 
Let $q_{*} = \psi_{\nu}(x_{*})^{-1}x_{*}$ so that $q_{*} = \lim_{j \to \infty} q_{j}$.  Let $Y_{*} = \psi_{\nu}(x_{*}) X_{*}$.  Since $X_{j} \to X_{*}$, it follows that $Y_{j} \to Y_{*}$ as $j \to +\infty$.  

As before, one can invoke Proposition \ref{P: key mountain part 1}: this leads to the conclusion that there is a face $F_{*} \subseteq \partial \varphi^{*}(p)$ such that 
	\begin{equation*}
		\{q_{*},q_{**}\} \subseteq F_{*} \quad \text{and} \quad Y_{*}q = 0 \quad \text{for each} \quad q \in F_{*}.
	\end{equation*}	
Since $X_{*} = \psi_{\nu}(x_{*})^{-1} Y_{*}$, it follows that $X_{*} q = 0$ for each $q \in F_{*}$. \end{proof}

\subsection{Normal Cones}  \label{S: normal cones} With the exception of Proposition \ref{P: smallness condition king}, the proof of which is presented in Appendix \ref{A: perturbed test functions}, the only loose end remaining in this section is the proof of Proposition \ref{P: key mountain part 1}.  The notion of normal cones will be used repeatedly in the proof.  Therefore, before diving into the analysis, it will be necessary to review the definition and some relevant properties, which will significantly clarify the geometric arguments to come.  Much of this information can be found in \cite{schneider} or \cite{rockafellar_wets}.

First, let $C \subseteq \mathbb{R}^{m}$ be a convex set in some arbitrary dimension $m$ and suppose $x_{*} \in C$.  Recall that a vector $v \in \mathbb{R}^{m} \setminus \{0\}$ determines a supporting hyperplane to $C$ at a point $x_{*}$ if 
	\begin{equation*}
		C \subseteq \left\{ x \in \mathbb{R}^{m} \, \mid \, \langle x - x_{*}, v \rangle \leq 0 \right\},
	\end{equation*}
or, equivalently, if 
	\begin{equation*}
		\langle v, x_{*} \rangle = \max \left\{\langle v, x \rangle \, \mid \, x \in C \right\}.
	\end{equation*}
Thus, up to rescaling, the supporting hyperplanes to $C$ at $x_{*}$ are in one-to-one correspondence with such vectors $v$.  The normal cone is nothing but the set of all such supporting hyperplanes with this extra multiplicity.

	\begin{definition} Given a convex set $C \subseteq \mathbb{R}^{m}$ in dimension $m$, the normal cone $\mathcal{N}^{m}_{C}(x)$ to $C$ at a point $x \in C$ is the convex cone defined by 
		\begin{equation*}
			\mathcal{N}^{m}_{C}(x) = \bigcap_{y \in C} \left\{ v \in \mathbb{R}^{m} \, \mid \, \langle v, x \rangle \geq \langle v, y \rangle \right\}.
		\end{equation*}
	\end{definition}
	
The normal cone plays the same role for a convex set that the subdifferential does for a convex function.\footnote{Indeed, the normal cone equals the subdifferential of the indicator function $I_{C}^{\infty}$ defined by $I^{\infty}_{C}(x) = 0$ if $x \in C$ and $I^{\infty}_{C}(x) = +\infty$, otherwise.}  In particular, $\mathcal{N}^{m}_{C}(x)$ is nonempty at each boundary point $x \in \partial C$.
	
\begin{remark} The ambient dimension $m$ is included in the notation $\mathcal{N}^{m}_{C}$ to avoid the slight ambiguity that will arise later, when certain convex sets will be treated both as subsets of $\mathbb{R}^{d-1}$ and $\mathbb{R}^{d}$.  Note, for instance, that the normal cone to the square $[0,1] \times [0,1] \times \{0\}$ changes according to whether one restricts the normal cone to vectors $v \in \mathbb{R}^{2} \times \{0\}$, i.e., identifying $\mathbb{R}^{2} \times \{0\}$ with $\mathbb{R}^{2}$, or considers $\mathcal{N}^{3}_{[0,1] \times [0,1] \times \{0\}}$ as literally defined above. \end{remark}

The most basic example of the normal cone occurs when the set $C$ is an open set in $\mathbb{R}^{m}$ of class $C^{1}$.  In that case, if $N_{\partial C}$ is the outward normal to $C$, then, for any $x \in \partial C$, 
	\begin{equation*}
		\mathcal{N}^{m}_{C}(x) = \{\lambda N_{\partial C}(x) \, \mid \, \lambda \geq 0\}.
	\end{equation*}
	
A more interesting example, which will be needed shortly, occurs when $C = \{\varphi \leq 1\}$ for some Finsler norm $\varphi$.  The next lemma shows that the normal cone is nothing but the cone generated by the subdifferential in this case.

	\begin{lemma} \label{L: explicit normal cone} Let $\varphi$ be a Finsler norm in $\mathbb{R}^{d}$.  If  $q \in \{\varphi = 1\}$, then
		\begin{equation*}
			\mathcal{N}^{d}_{\{\varphi \leq 1\}}(q) = \mathcal{C}(\partial \varphi(q)).
		\end{equation*}\end{lemma}
		
			\begin{proof} This follows more-or-less immediately from \eqref{E: subdifferential basic identity} and the definition of the normal cone. \end{proof}

There is another analogy with the subdifferential of a convex function.  Even if the convex set in question fails to be $C^{1}$, the normal cone is always upper semicontinuous in a sense that can be made precise.  This will be useful later so it is recorded now.

	\begin{lemma} \label{L: upper semicontinuity} Let $C \subseteq \mathbb{R}^{m}$ be a convex set in dimension $m$.  The normal cone $\mathcal{N}^{m}_{C}$ is upper semicontinuous in the following sense: for any $x \in C$, if the limit superior of $\mathcal{N}^{m}_{C}$ at $x$ is defined by 
		\begin{equation*}
			\limsup_{y \to x} \mathcal{N}^{m}_{C}(y) = \overline{\bigcap_{\epsilon > 0} \bigcup_{y \in B_{\epsilon}(x)} \mathcal{N}^{m}_{C}(y)},
		\end{equation*}
	then 
		\begin{equation*}
			\limsup_{y \to x} \mathcal{N}^{m}_{C}(y) \subseteq \mathcal{N}^{m}_{C}(x) \quad \text{for each} \quad x \in C.
		\end{equation*}
	\end{lemma}
	
		\begin{proof} If $v \in \limsup_{y \to x} \mathcal{N}^{m}_{C}(y)$, then there is a sequence $((x_{n},v_{n}))_{n \in \mathbb{N}} \subseteq C \times \mathbb{R}^{d}$ such that $(x_{n},v_{n}) \to (x,v)$ as $n \to +\infty$ and $v_{N} \in \mathcal{N}^{m}_{C}(x_{N})$ for all $N$.  If $y \in C$, then the definition of the normal cone implies that
			\begin{equation*}
				\langle v_{n}, x_{n} \rangle \geq \langle v_{n}, y \rangle \quad \text{for each} \quad n \in \mathbb{N},
			\end{equation*}
		which implies that $\langle v, x \rangle \geq \langle v, y \rangle$.  Since this is true for any $y \in C$, this proves $v \in \mathcal{N}^{m}_{C}(x)$ by definition.  \end{proof}
		
For the purposes of this paper, the main utility of the normal cone lies in the next lemma.

	\begin{lemma} \label{L: direction of approach}  Fix $m \in \mathbb{N}$ and let $C \subseteq \mathbb{R}^{m}$ be a convex set. Suppose that $x \in C$ and $p \in \mathcal{N}^{m}_{C}(x)$.  If there is a sequence $((x_{n},p_{n}))_{n \in \mathbb{N}} \subseteq C \times (\mathbb{R}^{d} \setminus \{p\})$ and a point $w' \in \mathbb{R}^{d}$ such that (i)  $(x_{n},p_{n}) \to (x,p)$ as $n \to +\infty$; (ii) $p_{N} \in \mathcal{N}^{m}_{C}(x_{N})$ for each $N \in \mathbb{N}$; and
		\begin{equation*}
			\text{(iii)} \quad w' = \lim_{n \to \infty} \frac{p_{n} - p}{\|p_{n} - p\|},
		\end{equation*}
	then
		\begin{equation*}
			\langle w', x \rangle = \max \left\{ \langle w', y \rangle \, \mid \, y \in C \, \, \text{such that} \, \, p \in \mathcal{N}^{m}_{C}(y) \right\}.
		\end{equation*}
	\end{lemma}
	
See Figure \ref{F: tangent thing} for a pictorial proof.

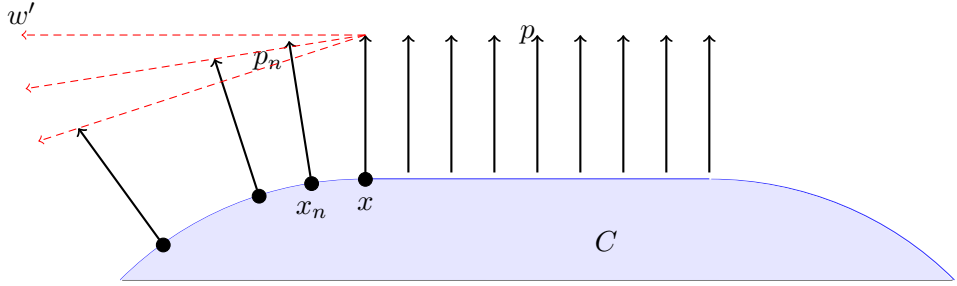
\begin{figure}
	\begin{tikzpicture}[scale=3.6]
%	\draw[color=blue!80, fill=blue!10] (0.25in,0) -- (0.75in,0) arc (0:90:0.5in);
%	\draw[color=blue!80, fill=blue!10] (-0.25in,0) -- (-0.75in,0) arc (180:90:0.5in);
%	\fill[fill=blue!10] (0.25in,0) rectangle (-0.25in,0.5in);	
	\draw[color=blue!80, thick] (-0.25in,0.5in) -- (0.25in,0.5in);
%	\draw (-0.75in,0) -- (0.75in,0);
%	\filldraw[fill=black] (-0.6036in, 0.3536in) circle (0.015in);
%	\filldraw[fill=black] (0.6036in, 0.3536in) circle (0.015in);
	\draw (-0.6036in, 0.3536in) -- (0.6036in,0.3536in);
	\draw[color = blue!80] (-0.60636in,0.3536in) arc (135:90:0.5in);
%	\fill[fill=blue!10] (-0.6036in, 0.3536in) -- (-0.25in,0.50in) -- (-0.25in,0.3536in);
%	\fill[fill=blue!10] (-0.60636in,0.3536in) arc (135:90:0.5in) -- (-0.25in,0.3536in) -- (-0.75in,0.3536in);
	\draw[color = blue!80] (0.60636in,0.3536in) arc (45:90:0.5in);
%	\fill[fill=blue!10] (0.60636in,0.3536in) arc (45:90:0.5in) -- (0.25in,0.3536in) -- (0.75in,0.3536in);
%	\fill[fill=blue!10] (0.25in,0.3536in) rectangle (-0.25in,0.5in);	
	\fill[fill=blue!10] (-0.60636in,0.3536in) arc (135:90:0.5in) -- (0.25in,0.5in) arc (90:45:0.5in) -- (-0.75in, 0.3536in);

% angles are pi/10, pi/5, and pi/20

	\draw[thick, ->] (0, 0.51in) -- (0,0.71in);
	\draw[anchor = east] (0.03,0.71in) node {\small $p$}; 
	\draw[thick,->] (0.0625in,0.51in) -- (0.0625in,0.71in);
	\draw[thick,->] (0.1250in,0.51in) -- (0.1250in,0.71in);
	\draw[thick,->] (0.1875in,0.51in) -- (0.1875in,0.71in);
	\draw[thick,->] (0.250in,0.51in) -- (0.250in,0.71in);
	
	\draw[thick,->] (-0.0625in,0.51in) -- (-0.0625in,0.71in);
	\draw[thick,->] (-0.1250in,0.51in) -- (-0.1250in,0.71in);
	\draw[thick,->] (-0.1875in,0.51in) -- (-0.1875in,0.71in);
	\draw[thick,->] (-0.250in,0.51in) -- (-0.250in,0.71in);
	
	\filldraw[fill=black] (-0.4045in,0.4755in) circle (0.010in);
	\draw[thick,->] (-0.4076in,0.4850in) -- (-0.4694in,0.6753in);
	
	\filldraw[fill=black] (-0.5439in,0.4045in) circle (0.010in);
	\draw[thick,->] (-0.5498in,0.4126in) -- (-0.6673in,0.5744in);
	
	\filldraw[fill=black] (-0.3282in,0.4938in) circle (0.010in);
	\draw[anchor = north] (-0.3282in,0.4838in) node {\small $x_{n}$}; 
	\draw[thick,->] (-0.3298in,0.5037in) -- (-0.3611in,0.7013in);
	\draw[anchor = north] (-0.3911in,0.7013in) node {\small $p_{n}$}; 
	
	\filldraw[fill=black] (-0.25in,0.5in) circle (0.010in);
	\draw[anchor = north] (-0.25in,0.49in) node {\small $x$}; 
	
%	\draw[red,thin,densely dashed,->] (-0.25in,0.5in) -- (-0.75in,0.5in);
	
	\draw[red,thin,densely dashed,->] (-0.25in,0.71in) -- (-0.75in,0.71in);
	\draw[anchor = south] (-0.75in,0.71in) node {\small $w'$}; 
	\draw[red,thin,densely dashed,->] (-0.25in,0.71in) -- (-0.7439in,0.6319in);
	\draw[red,thin,densely dashed,->] (-0.25in,0.71in) -- (-0.7255in,0.5555in);
	
	\draw[anchor = north] (0.1in, 0.44in) node {\small $C$}; 
	
\end{tikzpicture}
	\caption{The vector $w'$ describes the direction from which the sequence of support vectors $(p_{n})_{n \in \mathbb{N}}$ approach the limit $p$.  Notice that the point $x$ has the largest $w'$ component among all points on the corresponding face of $C$.}
	\label{F: tangent thing}
\end{figure}
	
	The proof leans on a monotonicity formula, which is worth highlighting in its own right: if $C$ is any convex set, then 
		\begin{equation} \label{E: monotonicity formula}
			\langle p_{2} - p_{1}, x_{2} - x_{1} \rangle \geq 0 \quad \text{for each} \quad x_{1},x_{2} \in C, \, \, (p_{1},p_{2}) \in \mathcal{N}^{m}_{C}(x_{1}) \times \mathcal{N}^{m}_{C}(x_{2}).
		\end{equation}
	This monotonicity formula captures the intuition that a convex set curves only ``inwards," and is a weak expression of the fact that, in the smooth setting, the principal curvatures are all nonnegative.\footnote{As is well known, the subdifferential of a convex function has a similar monotonicity property.} \footnote{To see that the monotonicity formula holds, observe that if $x_{1},x_{2} \in C$ and $(p_{1},p_{2}) \in \mathcal{N}^{m}_{C}(x_{1}) \times \mathcal{N}^{m}_{C}(x_{2})$, then the definition of the normal cone directly implies that
		\begin{equation*}
			\langle p_{2} - p_{1}, x_{2} - x_{1} \rangle = \langle p_{2},x_{2} - x_{1} \rangle + \langle p_{1}, x_{1} - x_{2} \rangle \geq 0.
		\end{equation*}}
	
		\begin{proof}[Proof of Lemma \ref{L: direction of approach}] Suppose that $y \in C$ and $p \in \mathcal{N}^{m}_{C}(y)$.  By the monotonicity formula \eqref{E: monotonicity formula},
			\begin{equation*}
				\langle p_{n} - p, x_{n} - y \rangle \geq 0.
			\end{equation*}
		Dividing by $\|p_{n} - p\|$ is harmless here, hence, after doing so and sending $n \to +\infty$, one finds
			\begin{equation*}
				\langle w', x - y\rangle \geq 0.
			\end{equation*}
		Since $y$ was arbitrary and $p \in \mathcal{N}^{m}_{C}(x)$, the desired conclusion is now immediate.   \end{proof}

\subsection{Proof of Proposition \ref{P: key mountain part 1}} \label{S: most technical mountain proof} Fnally, the section concludes with the proof that the derivatives of the perturbed test functions are well-behaved in the limit $\epsilon \to 0^{+}$.  The proof of Proposition \ref{P: key mountain part 1} will be presented in two parts, corresponding to parts (a) and (b) in the statement.

In the proof of part (a), the optimality condition \eqref{E: optimality condition directional} in the definition of the face $F_{*}$ will arise from an application of Lemma \ref{L: direction of approach}.  In particular, upon setting $\hat{p}_{j} = \|D\psi^{\epsilon_{j}}_{\nu}(q_{j})\|^{-1} D\psi^{\epsilon_{j}}_{\nu}(q_{j})$, the proof of \eqref{E: optimality condition directional} proceeds by establishing that there is a $w_{*}' \in \{e\}^{\perp} \setminus \{0\}$ such that \eqref{E: first derivative thingy} holds and\footnote{Actually, it will also be necessary to consider the possibility that $w_{*}' = 0$, but this presents no additional difficulties.}
	\begin{equation} \label{E: directional limit key mountain}
		\frac{w_{*}'}{\|w_{*}'\|} = \lim_{j \to \infty} \frac{\hat{p}_{j} - e}{\|\hat{p}_{j} - e\|}.
	\end{equation}
Once $w_{*}'$ is identified, the relation \eqref{E: optimality condition directional} will follow after an application of Lemma \ref{L: direction of approach}.

To prove that it is possible to find $w_{*}'$ satisfying both \eqref{E: first derivative thingy} and \eqref{E: directional limit key mountain}, it is useful to write the derivatives of $\psi_{\nu}^{\epsilon_{j}}$ somewhat explicitly, or, more precisely, to write those of $g_{\nu}^{\epsilon_{j}}$ explicitly.  Recall that $g_{\nu}^{\epsilon}$ is given by \eqref{E: mollification} and $g_{\nu}$, by \eqref{E: graphical special case}.  Thus, the first two derivatives of $g_{\nu}^{\epsilon}$ are explicitly determined by
	\begin{align}
		Dg_{\nu}^{\epsilon}(x') &= -\int_{B_{1}} \frac{\text{dist}(x' + \epsilon y',G_{p})}{\sqrt{\nu^{2} - \text{dist}(x' + \epsilon y',G_{p})^{2}}} D\text{dist}(x' + \epsilon y',G_{p}) \, \rho(y) \, dy, \label{E: first derivative key mountain} \\
		D^{2}g_{\nu}^{\epsilon}(x') &= - \int_{B_{1}} \frac{\text{dist}(x' + \epsilon y',G_{p})}{\sqrt{\nu^{2} - \text{dist}(x' + \epsilon y',G_{p})^{2}}} D^{2} \text{dist}(x' + \epsilon y',G_{p}) \, \rho(y) dy \label{E: second derivative key mountain} \\
		&\qquad - \int_{B_{1}} \frac{\nu^{2}}{(\nu^{2} - \text{dist}(x' + \epsilon y',G_{p})^{2})^{3/2}} D\text{dist}(x' + \epsilon y',G_{p})^{\otimes 2} \, \rho(y) dy. \nonumber
	\end{align}
Since $G_{p}$ is convex, the distance function $\text{dist}(\cdot,G_{p})$ is itself convex; hence there are enough (weak) derivatives available to justify these formulas.

To appreciate the role played by the derivatives of $g_{\nu}^{\epsilon}$ here, recall that since the level set $\{\psi_{\nu}^{\epsilon} = 1\}$ is contained in the graph of $g_{\nu}^{\epsilon}$ by \eqref{E : graph property ack ack ack}, it follows that the outward normal vector to this surface is given, for any $q \in \{\psi_{\nu}^{\epsilon} = 1\}$, by the formula
	\begin{align} \label{E: classical normal vector formula}
		\frac{D \psi^{\epsilon}_{\nu}(q)}{\|D\psi^{\epsilon}_{\nu}(q)\|} = - \frac{Dg_{\nu}^{\epsilon}(q')}{\sqrt{1 + \|Dg_{\nu}^{\epsilon}(q')\|^{2}}} + \frac{e}{\sqrt{1 + \|Dg_{\nu}^{\epsilon}(q')\|^{2}}}
	\end{align}
and the second fundamental form (matrix of principal curvatures) has a related expression
	\begin{align} \label{E: graph curvature}
		&\frac{(\text{Id} - \|D\psi_{\nu}^{\epsilon}(q)\|^{-2} D\psi_{\nu}^{\epsilon}(q)^{\otimes 2}) D^{2}\psi_{\nu}^{\epsilon}(q) (\text{Id} - \|D\psi_{\nu}^{\epsilon}(q)\|^{-2} D\psi_{\nu}^{\epsilon}(q)^{\otimes 2})}{\|D\psi_{\nu}^{\epsilon}(q)\|} =\\
			 &\qquad \left( \text{Id} - \frac{Dg_{\nu}^{\epsilon}(q')^{\otimes 2}}{1 + \|Dg_{\nu}^{\epsilon}(q')\|^{2}} \right) D^{2}g_{\nu}^{\epsilon}(q') \left(\text{Id} - \frac{Dg_{\nu}^{\epsilon}(q')^{\otimes 2}}{1 + \|Dg_{\nu}^{\epsilon}(q')\|^{2}} \right). \nonumber
	\end{align}
As before, here $q'$ denotes the orthogonal projection of $q$ onto $\{e\}^{\perp}$. 

From the representations above, a straightforward computation shows that \eqref{E: first derivative thingy} implies \eqref{E: directional limit key mountain} provided $w_{*}' \neq 0$. 

As will be shown shortly, the preceding discussion more-or-less amounts to a proof of part (a) of the proposition.  Before proceeding to the proof of part (b), though, it will be necessary to extract some information concerning the ``shapes" of possible accumulation points of the family $(g_{\nu}^{\epsilon})_{0 < \epsilon < \nu}$.  That is the purpose of the next lemma.

	\begin{lemma} \label{L: intermediate key lemma} Let $G \subseteq \mathbb{R}^{d - 1}$ be a convex set, fix a $\nu > 0$, and define $g : \{\text{dist}(\cdot,G) \leq \nu\} \to [0,+\infty)$ and $g^{\epsilon} : \{\text{dist}(\cdot,G) \leq \nu - \epsilon\} \to [0,+\infty)$ by 
		\begin{equation*}
			g(x) = \sqrt{\nu^{2} - \text{dist}(x',G)^{2}}, \quad g^{\epsilon}(x') = \int_{\mathbb{R}^{d-1}} g(x' + \epsilon y') \rho(y') \, dy'.
		\end{equation*} 
	If $(\epsilon_{j})_{j \in \mathbb{N}} \subseteq (0,+\infty)$ is a sequence converging to zero and $(q_{j}')_{j \in \mathbb{N}} \subseteq \mathbb{R}^{d-1}$ satisfies
		\begin{equation} \label{E: nice convergence}
			\limsup_{j \to \infty} \epsilon_{j}^{-1} \text{dist}(q_{j}',G) < +\infty,
		\end{equation}
	then the blow-up sequence $(d_{j})_{j \in \mathbb{N}}$ defined by 
		\begin{equation*}
			d_{j}(y') = \epsilon_{j}^{-1} \text{dist}(q'_{j} + \epsilon_{j} y',G) = \text{dist}(y', \epsilon_{j}^{-1} (G - q'_{j}))
		\end{equation*}
	has a subsequence $(d_{j_{k}})_{k \in \mathbb{N}}$ that converges locally uniformly to a convex function $d_{*} : \mathbb{R}^{d-1} \to [0,+\infty)$, which is $C^{1}$ in the set $\{d_{*} > 0\}$.  Furthermore, the following formulas hold:
		\begin{align}
			\lim_{k \to \infty} \epsilon_{j_{k}}^{-1} Dg^{\epsilon_{j_{k}}}(q_{j_{k}}') &= - \nu^{-1} \int_{B_{1}} d_{*}(y') Dd_{*}(y') \, \rho(y') dy', \label{E: direction of approach} \\
			\lim_{k \to \infty} D^{2}g^{\epsilon_{j_{k}}}(q_{j_{k}}') &= - \nu^{-1} \int_{B_{1}} \left( d_{*}(y') D^{2}d_{*}(y') + Dd_{*}(y') \otimes Dd_{*}(y') \right) \rho(y') \, dy'. \label{E: curvature}
		\end{align}
	In addition, if the subsequence $(q'_{j_{k}})_{k \in \mathbb{N}}$ converges to some point $q'_{*} \in G$, then 
		\begin{equation} \label{E: key normal cone part}
			Dd_{*}(y') \in \mathcal{N}^{d-1}_{G}(q'_{*}) \quad \text{for each} \quad y' \in \{d_{*} > 0\}.
		\end{equation}  \end{lemma} 
		
\begin{remark} Actually, if $q_{j_{k}}' \to q_{*}'$ as in the final statement of the lemma, then one can further say that there is a shift $v' \in \mathbb{R}^{d-1}$ such that $d_{*}$ is given by 
	\begin{equation*}
		d_{*}(y') = \text{dist}(y', v' + \mathcal{T}_{G}(q_{*}')) = \text{dist}(y' - v', \mathcal{T}_{G}(q_{*}'),
	\end{equation*}
where $\mathcal{T}_{G}(q_{*}')$ is the tangent cone to $G$ at $q_{*}'$.\footnote{See \cite{rockafellar_wets} for the definition of tangent cones.}  Though it helps in understanding the geometric meaning of the lemma, this observation will not be needed in what follows.  \end{remark}
		
			\begin{proof} Begin with a priori estimates: it is well-known that distance functions are uniformly Lipschitz with constant one, that is,
				\begin{equation*}
					\{d_{j}\}_{j \in \mathbb{N}} \subseteq W^{1,\infty}_{\text{loc}}(\mathbb{R}^{d-1}), \quad \|Dd_{j}\|_{L^{\infty}(\mathbb{R}^{d-1})} = 1 \quad \text{for each} \quad j \in \mathbb{N}.
				\end{equation*}
			At the same time, by hypothesis,
				\begin{equation*}
					\limsup_{j \to \infty} d_{j}(0) = \limsup_{j \to \infty} \epsilon_{j}^{-1} \text{dist}(q_{j}',G) < +\infty.
				\end{equation*}
			This establishes that $\{d_{j}\}_{j \in \mathbb{N}}$ is locally bounded in $\mathbb{R}^{d-1}$.
			
			It is convenient to obtain slightly more information about $\{d_{j}\}_{j \in \mathbb{N}}$.  First, observe that it is possible to write
				\begin{equation*}
					\frac{1}{2} d_{j}(y')^{2} = \min \left\{ \frac{1}{2} \|y' - \tilde{y}'\|^{2} \, \mid \, \tilde{y}' \in \epsilon_{j}^{-1}(G - x_{j}') \right\}.
				\end{equation*}
			Thus, since this last expression has the form of an inf-convolution,\footnote{See \cite[Section 11]{primer} or \cite[Appendix]{user}.} 
				\begin{equation*}
					-D^{2} \left\{ 2^{-1} d_{j}^{2}\right\} \geq -\text{Id} \quad \text{in the viscosity sense in} \, \, \mathbb{R}^{d-1}.
				\end{equation*}
			As already mentioned in the proof of Proposition \ref{P: approximate supersolution} above, such a viscosity inequality is equivalent to its distributional version, hence $D^{2}\{2^{-1}d_{j}^{2}\} \leq \text{Id}$ holds distributionally.  At the same time, since $G$ is convex, $d_{j}$ is convex for any $j$.  In particular, $2^{-1} d_{j}^{2}$ is also convex for all $j$.  It follows that
				\begin{equation*}
					0 \leq D^{2} \left\{ 2^{-1} d_{j}^{2} \right\} \leq \text{Id} \quad \text{in the sense of distributions in} \, \, \mathbb{R}^{d-1}.
				\end{equation*}
			Thus, $2^{-1} d_{j}^{2} \in W^{2,\infty}_{\text{loc}}(\mathbb{R}^{d-1})$ for any $j$ and the second derivative, which is \emph{a priori} just a Radon measure, is represented by a function in $L^{\infty}(\mathbb{R}^{d-1})$.  In particular, there is no abuse of notation writing it as
				\begin{equation*}
					D^{2}\{2^{-1} d_{j}^{2}\} = d_{j} D^{2} d_{j} + Dd_{j} \otimes Dd_{j}
				\end{equation*}
			since this formula holds almost everywhere; and the positivity of $Dd_{j} \otimes Dd_{j}$ implies
				\begin{equation*}
					0 \leq d_{j} D^{2} d_{j} \leq \text{Id} \quad \text{almost everywhere in} \, \, \mathbb{R}^{d-1}.
				\end{equation*}			
			
			At last, the previous arguments show that there is a subsequence $(j_{k})_{k \in \mathbb{N}} \subseteq \mathbb{N}$ and a nonnegative convex function $d_{*}$ such that
				\begin{gather*}
					d_{j_{k}} \to d_{*} \quad \text{locally uniformly in} \, \, \mathbb{R}^{d-1}, \\
					Dd_{j_{k}} \to Dd_{*} \quad \text{locally uniformly in} \, \, \{d_{*} > 0\},  \\
					d_{j_{k}} D^{2} d_{j_{k}} \overset{*}{\rightharpoonup} d_{*} D^{2}d_{*} \quad \text{weakly in} \, \, L^{\infty}(\mathbb{R}^{d - 1}).
				\end{gather*}
			In view of \eqref{E: first derivative key mountain} and \eqref{E: second derivative key mountain}, these convergences imply \eqref{E: direction of approach} and \eqref{E: curvature}. 
			
			It remains to prove that if $q_{j_{k}}' \to q_{*}'$ as $k \to +\infty$, then $Dd_{*}$ takes values in $\mathcal{N}^{d-1}_{G}(q'_{*})$.  Recall that if $q' \in \mathbb{R}^{d-1} \setminus \overline{G}$, then $D\text{dist}(q', G) = \|q' - [q']\|^{-1}(q' - [q'])$, where $[q']$ is the (unique) closest point in $G$ to $x'$.\footnote{See \cite[Section 1.2]{schneider} and \cite[Example 2.25]{rockafellar_wets}.}  In particular, $[q'] \in \partial G$ and $\|q' - [q']\|^{-1}(q' - [q'])$ is the normal vector to a supporting hyperplane to $G$ at $[q']$.  Thus,
				\begin{equation*}
					D\text{dist}(q',G) = \|q' - [q']\|^{-1}(q' - [q']) \in \mathcal{N}^{d-1}_{G}([q']).
				\end{equation*}
			After rescaling, this becomes
				\begin{equation*}
					Dd_{j}(y') \in \mathcal{N}^{d-1}_{G}([q_{j}' + \epsilon_{j} y']).
				\end{equation*}
			Clearly, $[q_{j}' + \epsilon y'] \to [q_{*}'] = q_{*}'$ as $\epsilon \to 0$.  Therefore, by the upper semicontinuity of $\mathcal{N}_{G}^{d-1}$ (i.e., Lemma \ref{L: upper semicontinuity}), for any $y' \in \{d_{*} > 0\}$,
				\begin{equation*}
					Dd_{*}(y') = \lim_{j \to \infty} Dd_{j_{k}}(y') \in \limsup_{q' \to q'_{*}} \mathcal{N}_{G}^{d-1}(q') \subseteq \mathcal{N}^{d-1}_{G}(q_{*}').
				\end{equation*}
			This proves that $Dd_{*}$ maps $\{d_{*} > 0\}$ into $\mathcal{N}^{d-1}_{G}(q'_{*})$. \end{proof}

At this point, the stage is set for the proof of the first part of Proposition \ref{P: key mountain part 1}.  

%:
	\begin{proof}[Proof of Proposition \ref{P: key mountain part 1}, (a)]  Apply Lemma \ref{L: intermediate key lemma} with $G$ equal to $G_{p} = \partial \varphi^{*}(p) - \varphi^{*}(e) e$ (using the identification of $\{e\}^{\perp}$ with $\mathbb{R}^{d-1}$ to regard this as a convex set in $\mathbb{R}^{d-1}$) and the sequence $(q_{j}')_{j \in \mathbb{N}}$ determined as usual by projecting $(q_{j})_{j \in \mathbb{N}}$ to $\{e\}^{\perp}$, see \eqref{E: orthogonal projection part}.  The lemma asserts that there is a convex function $d_{*} : \mathbb{R}^{d-1} \to [0,+\infty)$ such that 
		\begin{equation} \label{E: gradient constraint}
			Dd_{*}(y') \in \mathcal{N}^{d-1}_{G_{p}}(q_{*}') \quad \text{for each} \quad y' \in \{d_{*} > 0\}.
		\end{equation}
	Further, note that the function $g$ in the statement of the lemma coincides with $g_{\nu}$ up to a constant, hence the conclusions concerning its derivatives imply the desired conclusions \eqref{E: first derivative thingy} and \eqref{E: second derivative thingy}.  In particular, there is a subsequence $(j_{k})_{k \in \mathbb{N}} \subseteq \mathbb{N}$ such that
		\begin{equation} \label{E: need this gradient business later}
			- \lim_{k \to \infty} \epsilon_{j_{k}}^{-1} Dg_{\nu}^{\epsilon_{j_{k}}}(q_{j_{k}}') = \nu^{-1} \int_{B_{1}} d_{*}(y') Dd_{*}(y') \, \rho(y') dy'
		\end{equation}
	As in the discussion preceding this proof, see \eqref{E: first derivative thingy}, denote this vector by $w_{*}'$.
	
	It remains to show that
		\begin{equation*}
			\langle w_{*}', q_{*} \rangle = \langle w_{*}', q_{**} \rangle = \max \left\{ \langle w_{*}', q \rangle \, \mid \, q \in \partial \varphi^{*}(p) \right\}.
		\end{equation*}
	This is trivial if $w_{*}' = 0$; hence, in the rest of the proof, assume that $w_{*}' \neq 0$.  
	
	Start with $q_{**}$, the easier of the two: this part is a straightforward application of Lemma \ref{L: direction of approach}.  Consider the convex set $\{\varphi \leq 1\} \subseteq \mathbb{R}^{d}$.  For any $j \in \mathbb{N}$, one knows that $\tilde{q}_{j} \in \partial \varphi^{*}(D\psi_{\nu}^{\epsilon_{j}}(q_{j}))$ by hypothesis and, thus, by inversion \eqref{E: inversion},
		\begin{equation*}
			\varphi^{*}(D\psi_{\nu}^{\epsilon_{j}}(q_{j}))^{-1} D\psi_{\nu}^{\epsilon_{j}}(q_{j}) \in \partial \varphi(\tilde{q}_{j}).
		\end{equation*}
	From this and Lemma \ref{L: explicit normal cone}, it follows that if one defines the sequence $(\hat{p}_{j})_{j \in \mathbb{N}}$ by letting $\hat{p}_{j} = \|D\psi_{\nu}^{\epsilon_{j}}(q_{j})\|^{-1} D\psi_{\nu}^{\epsilon_{j}}(q_{j})$, then
		\begin{equation*}
			\hat{p}_{j} \in \mathcal{N}^{d}_{\{\varphi \leq 1\}}(\tilde{q}_{j}) \quad \text{for each} \quad j \in \mathbb{N}.
		\end{equation*}
	Furthermore, since $q_{j} \to q_{*}$ as $j \to +\infty$ and $q_{*} \in \partial \varphi^{*}(p)$, Proposition \ref{P: touching above test function part} says that $\varphi$ touches $\psi_{\nu}$ from below at $x_{*}$.  Thus,
		\begin{equation*}
			\lim_{j \to \infty} D\psi_{\nu}^{\epsilon_{j}}(q_{j}) = D\psi_{\nu}(q_{*}) = D\varphi(q_{*}) \overset{\eqref{E: subdifferential basic identity}}{=} \varphi^{*}(p)^{-1} p,
		\end{equation*}
	It follows that $\hat{p}_{j} \to e$ as $j \to +\infty$.  At the same time, by the definition \eqref{E: first derivative thingy} of $w_{*}'$ and the representation formula \eqref{E: classical normal vector formula}, 
		\begin{equation*}
			\lim_{k \to \infty} \frac{\hat{p}_{j_{k}} - e}{\|\hat{p}_{j_{k}} - e\|} = \frac{w_{*}'}{\|w_{*}'\|}.
		\end{equation*}
	At last, by Lemma \ref{L: direction of approach} applied to $(\tilde{q}_{j_{k}},\hat{p}_{j_{k}})_{k \in \mathbb{N}}$, plus another application of Lemma \ref{L: explicit normal cone},
		\begin{align*}
			\langle w_{*}', q_{**} \rangle &= \max \left\{ \langle w_{*}', q \rangle \, \mid \, q \in \{\varphi \leq 1\} \, \, \text{such that} \, \, e \in \mathcal{N}^{d}_{\{\varphi \leq 1\}}(q) \right\} \\
				&= \max \left\{ \langle w_{*}', q \rangle \, \mid \, q \in \partial \varphi^{*}(e) \right\} \overset{\eqref{E: zero hom}}{=} \max \left\{ \langle w_{*}', q \rangle \, \mid \, q \in \partial \varphi^{*}(p) \right\}.
		\end{align*}
%	Finally, recall that $\partial \varphi_{\zeta}^{*}(p) = \partial \varphi^{*}(p) + \zeta e$ by Theorem \ref{T: regularized_norm}.  Therefore, setting $q_{**} = \bar{q}_{**} + \zeta e$ and recalling that $\langle w_{*}', e \rangle = 0$, one concludes
%		\begin{equation*}
%			\langle w_{*}', q_{**} \rangle = \max \left\{ \langle w_{*}', q \rangle \, \mid \, q \in \partial \varphi^{*}(p) \right\}.
%		\end{equation*}
		
	It remains to consider $q_{*}$.  This boils down to yet another application of Lemma \ref{L: direction of approach}, except for the (nontrivial) complication that the convex set in question varies with $j$.  
	
	For any $j$, consider the convex set $C_{j} = \{\psi_{\nu}^{\epsilon_{j}} \leq 1\} \subseteq \mathbb{R}^{d}$.  By assumption, $q_{j} \in \partial C_{j}$ for any $j$ and $q_{j} \to q_{*}$ as $j \to +\infty$.  By the discussion preceding Lemma \ref{L: explicit normal cone},		
	\begin{equation*}
			\hat{p}_{j} = \|D\psi_{\nu}^{\epsilon_{j}}(q_{j})\|^{-1} D\psi_{\nu}^{\epsilon_{j}}(q_{j}) \in \mathcal{N}^{d}_{C_{j}}(q_{j}) \quad \text{for each} \quad j \in \mathbb{N}.
		\end{equation*}
	At the same time, if the set $[\partial \varphi^{*}(p)]^{(\mu)}$ is defined for $\mu > 0$ by
		\begin{equation*}
			[\partial \varphi^{*}(p)]^{(\mu)} = \left\{q \in \partial \varphi^{*}(p) \, \mid \, \text{dist}(q, \text{bdry}(\partial \varphi^{*}(p))) \geq \mu \right\}, 
		\end{equation*}
	then it is not difficult to show that, for any $\epsilon < \mu$,
		\begin{equation*}
			\|D\psi_{\nu}^{\epsilon}(q)\|^{-1} D\psi_{\nu}^{\epsilon}(q) = e \quad \text{for each} \quad q \in [\partial \varphi^{*}(p)]^{(\mu)};
		\end{equation*}
	see Proposition \ref{P: flatness prop} in Appendix \ref{App: flatness} for the proof.  In particular, for any $\mu > 0$, there is a $J(\mu) \in \mathbb{N}$ such that if $j \geq J(\mu)$, then
		\begin{equation*}
			e \in \mathcal{N}^{d}_{C_{j}}(q) \quad \text{for each} \quad q \in [\partial \varphi^{*}(p)]^{(\mu)}.
		\end{equation*}
	Thus, by the monotonicity formula \eqref{E: monotonicity formula}, for any $j \geq J(\mu)$,
		\begin{equation*}
			\langle \hat{p}_{j} - e, q_{j} - q \rangle \geq 0 \quad \text{for each} \quad q \in [\partial \varphi^{*}(p)]^{(\mu)}.
		\end{equation*}
	Dividing by $\|\hat{p}_{j} - e\|$ and sending $j \to +\infty$, this becomes
		\begin{equation*}
			\langle w_{*}', q_{*} - q \rangle \geq 0 \quad \text{for each} \quad q \in [\partial \varphi^{*}(p)]^{(\mu)}.
		\end{equation*}
	In the limit $\mu \to 0^{+}$, this yields the remaining half of \eqref{E: optimality condition directional}.      \end{proof}
	
To complete the proof of Proposition \ref{P: key mountain part 1}, it only remains to prove part (b) concerning the asymptotic behavior of the Hessians.  Here is where the explicit representations \eqref{E: first derivative thingy} and \eqref{E: second derivative thingy} and the gradient constraint \eqref{E: key normal cone part} are used.

To appreciate the significance of the matrix $A_{*}'$ obtained in the limit \eqref{E: second derivative thingy}, first, notice that it can be expressed in the form
	\begin{align*}
		A_{*}' &= \lim_{j \to \infty} D^{2}g_{\nu}^{\epsilon_{j}}(q_{j}') \\
			&= - \nu^{-1} \int_{B_{1}} \left( d_{*}(y') D^{2}d_{*}(y') + Dd_{*}(y') \otimes Dd_{*}(y') \right) \rho(y') \, dy'.
	\end{align*} 
Further, the formula \eqref{E: graph curvature} implies that
			\begin{align*}
		&(\text{Id} - e \otimes e) Y_{*} (\text{Id} - e \otimes e) = \lim_{j \to \infty} (\text{Id} - \hat{p}_{j}^{\otimes 2}) Y_{j} (\text{Id} - \hat{p}_{j}^{\otimes 2}) \\
			&\quad = \lim_{j \to \infty} \|D\psi_{\nu}^{\epsilon_{j}}(q_{j})\| \left(\text{Id} - \frac{Dg_{\nu}^{\epsilon_{j}}(q_{j}')^{\otimes 2}}{1 + \|Dg_{\nu}^{\epsilon_{j}}(q_{j}')\|^{2}} \right) D^{2}g_{\nu}^{\epsilon_{j}}(q_{j}') \left( \text{Id} - \frac{Dg_{\nu}^{\epsilon_{j}}(q_{j}')^{\otimes 2}}{1 + \|Dg_{\nu}^{\epsilon_{j}}(q_{j}')\|^{2}} \right) \\
			&\quad = \|D\psi_{\nu}(q_{*})\| A_{*}'
	\end{align*}
since $Dg_{\nu}^{\epsilon_{j}}(q_{j}') \to Dg_{\nu}(q_{*}') = 0$ as $j \to +\infty$.  Thus, the constraints \eqref{E: hessian constraint} on $Y_{*}$ will follow from an analysis of the integral formula for $A_{*}$, which amounts to studying the ``shape" of the limiting distance function $d_{*}$.

Finally, before entering into the proof of part (b), there is still one more convex analytic consideration that will facilitate the proof, basically an elementary lemma.
	
\begin{lemma} \label{L: cone integral} Given any dimension $m \in \mathbb{N}$ and any closed convex cone $\mathcal{N} \subseteq \mathbb{R}^{m}$, let $\mu$ be a finite positive Borel measure on $\mathcal{N}$ and let $\mathcal{N}'$ be a face of $\mathcal{N}$.  If $\int_{\mathcal{N}} \|v\| \, \mu(dv) < +\infty$ and the mean of $\mu$ is in $\mathcal{N}'$, that is,
		\begin{equation*}
			\int_{\mathcal{N}} v \, \mu(dv) \in \mathcal{N}',
		\end{equation*}
then $\mu$ is supported on $\mathcal{N}'$ (i.e., $\mu(\mathcal{N} \setminus \mathcal{N}') = 0$). \end{lemma}

For the reader's convenience, a proof of the lemma can be found in Appendix \ref{A: technical results}.

\begin{proof}[Proof of Proposition \ref{P: key mountain part 1}, (b)]  The Hessian bounds \eqref{E: hessian bounds ack ack} follow from the estimates in Section \ref{S: pos hom}.  Indeed, at the limit point $q_{*}$, one knows that $D\psi_{\nu}(q_{*}) = D\varphi(q_{*})$, and control over $D\varphi$ is furnished by the $\delta$-nondegeneracy condition \eqref{E: nondegeneracy}.  Similarly, control over the curvature of the level sets of the sequence $\{\psi^{\epsilon_{j}}_{\nu}\}_{j \in \mathbb{N}}$ can be gleaned from the explicit formula for the second derivative of $g_{\nu}^{\epsilon_{j}}$.  For the full details, the reader is referred to Lemma \ref{L: key hessian control} in Appendix \ref{A: technical results}. 

That only leaves the more interesting statement, namely, the Hessian constraint \eqref{E: hessian constraint}.  Toward that end, it suffices to establish that 
	\begin{equation} \label{E: A constant part}
		A_{*}'( q - q_{*}) = 0 \quad \text{for each} \quad q \in F_{*}.
	\end{equation}
To see why, assume for the moment that \eqref{E: A constant part} holds.  Note that $Y_{*}q_{*} = 0$ since, by positive homogeneity,
	\begin{align*}
		Y_{*}q_{*} = \lim_{k \to \infty} D^{2}\psi_{\nu}^{\epsilon_{j_{k}}}(q_{j_{k}}) q_{j_{k}} = \lim_{k \to \infty} (0) = 0.
	\end{align*} Therefore, for any $q \in F_{*}$,
	\begin{align*}
		\langle Y_{*} q, q \rangle &= \langle Y_{*}(q - q_{*}), q - q_{*} \rangle.
	\end{align*}
At the same time, since $F_{*} \subseteq \partial \varphi^{*}(p) = \partial \varphi^{*}(e)$, the representation \eqref{E: subdifferential basic identity dual} implies that $q - q_{*} = (\text{Id} - e \otimes e) (q - q_{*})$ for any $q \in F_{*}$.  Thus, by the discussion preceding Lemma \ref{L: cone integral},
	\begin{align*}
		\langle Y_{*}q,q \rangle &= \langle (\text{Id} - e \otimes e) Y_{*} (\text{Id} - e \otimes e) (q - q_{*}), q - q_{*} \rangle \\
			&= \|D\psi_{\nu}(q_{*})\| \langle A_{*}'(q - q_{*}), q - q_{*} \rangle = 0.
	\end{align*}
Since $Y_{*} \geq 0$ by convexity, the identity $\langle Y_{*}q,q \rangle = 0$ actually implies $Y_{*}q = 0$.  Therefore, \eqref{E: A constant part} implies \eqref{E: hessian constraint} as claimed.

It only remains to prove \eqref{E: A constant part}.  For the rest of the proof, $G_{p}$ is treated as a subset of $\mathbb{R}^{d-1}$ so that $\mathcal{N}_{G_{p}}^{d-1}$ makes sense.  In what follows, given a vector $w' \in \mathcal{N}^{d-1}_{G_{p}}(q_{*}')$, let $\tilde{\mathcal{N}}(w')$ denote the face of $\mathcal{N}^{d-1}_{G_{p}}(q_{*}')$ such that 
	\begin{equation*}
		w' \in \text{rint}(\tilde{\mathcal{N}}(w')).
	\end{equation*}
By \eqref{E: face partition}, there is a unique such face, hence the map $w' \mapsto \tilde{\mathcal{N}}(w')$ is well-defined.

To ease the exposition, the proof will be presented arguing by cases.

\textit{Case: $w_{*}' = 0$ and $\tilde{\mathcal{N}}(0) = \{0\}$.} 

This case is particularly simple.  By \eqref{E: direction of approach},
	\begin{equation*}
		\int_{B_{1} \cap \{d_{*} > 0\}} d_{*}(y') Dd(y') \rho(y') \, dy = \int_{B_{1}} d_{*}(y') Dd(y') \rho(y') \, dy = w_{*}' = 0.
	\end{equation*}
At the same time, since $Dd_{*} \in \mathcal{N}^{d-1}_{G_{p}}(q_{*}')$ pointwise by \eqref{E: gradient constraint}, this last integral has the form described in Lemma \ref{L: cone integral}.  Applying that lemma with $\mathcal{N} = \mathcal{N}^{d-1}_{G_{p}}(q_{*}')$, $\mathcal{N}' = \{0\}$, and $\mu$ the measure on $\mathcal{N}^{d-1}_{G_{p}}(q_{*}')$ given by 
	\begin{equation*}
		\int_{\mathcal{N}^{d-1}_{G_{p}}(q_{*}')} f(v) \, \mu(dv) = \int_{B_{1} \cap \{d_{*} > 0\}} f(d_{*}(y') Dd_{*}(y')) \rho(y') \, dy',
	\end{equation*} 
one finds
	\begin{equation}
		Dd_{*}(y') = 0 \quad \text{for each} \quad y' \in \{d_{*} > 0\} \cap B_{1}.
	\end{equation}
It follows that $D^{2}d_{*}(y') = 0$ for a.e.\ $y' \in B_{1}$, and, thus, $A_{*}' = 0$ by \eqref{E: curvature}.  Hence \eqref{E: A constant part} holds trivially in this case.

\textit{Case: $w_{*}' = 0$ and $\tilde{\mathcal{N}}(0) \neq \{0\}$.}  

First, it is useful to note that $\tilde{\mathcal{N}}(0)$ has a special form: specifically, $\tilde{\mathcal{N}}(0)$ is what is referred to in \cite{rockafellar} as the \emph{lineality} of $\mathcal{N}^{d-1}_{G_{p}}(q_{*}')$, that is, the maximal linear subspace of $\mathbb{R}^{d-1}$ contained in $\mathcal{N}^{d-1}_{G_{p}}(q_{*}')$.  At the same time, by the definition of normal cone, this is equivalent to saying that $\tilde{\mathcal{N}}(0)$ is the maximal subspace of $\mathbb{R}^{d-1}$ such that 
	\begin{equation} \label{E: vanishing thing}
		G_{p} - q_{*}' \subseteq \tilde{\mathcal{N}}(0)^{\perp}.
	\end{equation}

Next, applying Lemma \ref{L: cone integral} as in the last case, observe that the inclusion $w_{*}' = 0 \in \tilde{\mathcal{N}}(0)$ implies that
	\begin{equation*}
		Dd_{*}(y') \in \tilde{\mathcal{N}}(0) \quad \text{for each} \quad y' \in B_{1} \cap \{d_{*} > 0\}.
	\end{equation*}
Since $\tilde{\mathcal{N}}(0)$ is a linear space, the last inclusion can be differentiated to find
	\begin{equation*}
		D^{2}d_{*}(y')(\mathbb{R}^{d-1}) \subseteq \tilde{\mathcal{N}}(0) \quad \text{for a.e.} \quad y' \in B_{1} \cap \{d_{*} > 0\}.
	\end{equation*}
Integrating over $B_{1} \cap \{d_{*} > 0\}$, one finds from \eqref{E: curvature} that $A_{*}'$ has the same property, that is, the range of $A_{*}'$ is contained in $\tilde{\mathcal{N}}(0)$.  Hence, since $A_{*}'$ is a symmetric matrix,
	\begin{equation} \label{E: first kernel condition}
		\tilde{\mathcal{N}}(0)^{\perp} \subseteq \text{Ker}(A_{*}'),
	\end{equation}
where $\text{Ker}(A_{*}')$ denotes the kernel of $A_{*}'$.

Combining \eqref{E: vanishing thing} and \eqref{E: first kernel condition}, one deduces that, for any $q' \in G_{p}$,
	\begin{equation*}
		A_{*}'(q' - q_{*}') = 0.
	\end{equation*}
Since $\partial \varphi^{*}(p) = G_{p} + \varphi^{*}(e) e$, this implies that
	\begin{equation*}
		A_{*}'(q - q_{*}) = 0 \quad \text{for each} \quad q \in \partial \varphi^{*}(p),
	\end{equation*}
which is nothing but \eqref{E: A constant part}.

\textit{Case: $w_{*}' \neq 0$.}  

This case is very similar to the previous one in so far as the use of the map $w' \mapsto \tilde{\mathcal{N}}(w')$ is concerned.  By definition of $F_{*}$ (i.e., \eqref{E: F* definition}), one knows that
	\begin{equation} \label{E: linear relation}
		\langle w_{*}', q \rangle = \langle w_{*}', q_{*} \rangle \quad \text{for each} \quad q \in F_{*}. 
	\end{equation}
In particular,
	\begin{equation*}
		F_{*} - q_{*} \subseteq \{ w_{*}' \}^{\perp}.
	\end{equation*}
In fact, this can be upgraded to
	\begin{equation} \label{E: perpendicular part}
		F_{*} - q_{*} \subseteq \tilde{\mathcal{N}}(w_{*}')^{\perp}.
	\end{equation}

Indeed, to see that \eqref{E: perpendicular part} holds, suppose that $q \in F_{*}$ and let $q'$ denote its orthogonal projection to $\{e\}^{\perp}$ as usual.  By \eqref{E: linear relation}, 
	\begin{equation*}
		w_{*}' \in \{w' \in \mathcal{N}^{d-1}_{G_{p}}(q_{*}') \, \mid \, \langle w', q \rangle = \langle w', q_{*} \rangle \} =: \mathcal{N}''(q).
	\end{equation*}
At the same time, $\mathcal{N}''(q)$ is an exposed face of $\mathcal{N}^{d-1}_{G_{p}}(q_{*}')$ since, by the definition of normal cone and the fact that $\mathbb{R}^{d-1}$ is being identified with $\{e\}^{\perp}$,
	\begin{equation*}
		\langle w', q \rangle = \langle w', q' \rangle \leq \langle w', q'_{*} \rangle = \langle w', q_{*} \rangle \quad \text{for each} \quad w' \in \mathcal{N}^{d-1}_{G_{p}}(q_{*}').
	\end{equation*}
Therefore, since $w_{*}' \in \text{rint}(\tilde{\mathcal{N}}(w_{*}')) \cap \mathcal{N}''(q)$, the defining property of faces (Definition \ref{D: face}) implies
	\begin{equation*}
		\tilde{\mathcal{N}}(w_{*}') \subseteq \mathcal{N}''(q).
	\end{equation*}
Since $q$ was arbitrary, this leads to the inclusion $\tilde{\mathcal{N}}(w_{*}') \subseteq \bigcap_{q \in F_{*}} \mathcal{N}''(q)$, which is equivalent to \eqref{E: perpendicular part}.

Finally, once again, Lemma \ref{L: cone integral} implies that
	\begin{equation*}
		Dd_{*}(y') \in \tilde{\mathcal{N}}(w_{*}') \quad \text{for each} \quad y' \in B_{1} \cap \{d_{*} > 0\},
	\end{equation*}
and hence
	\begin{equation*}
		D^{2}d_{*}(y')(\mathbb{R}^{d-1}) \subseteq \text{span}(\tilde{\mathcal{N}}(w_{*}')) \quad \text{for a.e.} \quad y' \in B_{1} \cap \{d_{*} > 0\},
	\end{equation*}
where $\text{span}(\tilde{\mathcal{N}}(w_{*}'))$ denotes the linear span of $\tilde{\mathcal{N}}(w_{*}')$.
In particular, combining this last observation with \eqref{E: perpendicular part}, one concludes as in the last step that
	\begin{equation*}
		F_{*} - q_{*} \subseteq \tilde{\mathcal{N}}(w_{*}')^{\perp} \subseteq \text{Ker}(A_{*}'),
	\end{equation*}
which is precisely \eqref{E: A constant part}.  The proof of \eqref{E: hessian constraint} is complete. \end{proof}

\part{The General Case and Applications} \label{Part: applications}

In this final part of the paper, the cone-type comparison results proved above are used to establish general comparison principles.  This is achieved by doubling variables.  The applications mentioned already in the introduction are treated in the final section, Section \ref{S: applications}.  
		
\section{General Cone Comparison Principles} \label{S: single variable}

%The stage is set for the proof of cone comparison principles for an arbitrary Finsler norm $\varphi$.  This will be achieved by combining the $C^{1,1}$ cone comparison principle, particularly Theorem \ref{T: technical core}, with the approximation theorem  of Section \ref{S: regularized}. 
%
%The bridge between the two theorems is provided by the following identities, repeated here for the reader's convenience: for each $p \in \mathbb{R}^{d} \setminus \{0\}$ and $X \in \mathcal{S}_{d}$, 
%	\begin{align}
%		\partial \varphi_{\zeta}(p) &= \partial \varphi(p) + \frac{\zeta p}{\|p\|}, \label{E: shifted subdifferential} \\
%		G_{\varphi}^{*}(p,X) &= \mathcal{G}_{\varphi_{\zeta}}^{*,\zeta}(p,X). \label{E: shifted laplace}
%	\end{align}
%Here $\varphi_{\zeta}$ is the regularized norm of Theorem \ref{T: technical core} and $\mathcal{G}_{\varphi_{\zeta}}^{*,\zeta}$ is defined in \eqref{E: }.
%

Before delving into comparison proofs involving variable-doubling arguments, the next proposition recapitulates the $C^{1,1}$ results proved above in the context of a general Finsler norm $\varphi$.  In effect, it shows that $(\varphi_{\zeta})_{\zeta > 0}$ provides a good class of test functions for elliptic PDE involving the operators $G_{\varphi}^{*}$ and $G_{*}^{\varphi}$.  

\begin{prop} \label{P: technical core eaten} Let $\varphi$ be a Finsler norm in $\mathbb{R}^{d}$ and let $(\varphi_{\zeta})_{\zeta > 0}$ be the regularized norms of Theorem \ref{T: regularized_norm}.  Let $V \subseteq \mathbb{R}^{d} \setminus \{0\}$ be a bounded open set, and let $\mathcal{H}$ be a function satisfying \eqref{Ass: continuity assumption} and \eqref{Ass: elliptic assumption}.
	
	If $u \in USC(\overline{V})$ satisfies
		\begin{equation*}
			-\mathcal{H}(x,Du,D^{2}u,G_{\varphi}^{*}(Du,D^{2}u)) \leq 0 \quad \text{in} \, \, V,
		\end{equation*}
	and if there is a $\zeta \in (0,1)$ such that 
		\begin{equation*}
			\max \{u(x) - \varphi_{\zeta}(x) \, \mid \, x \in \overline{V} \} > \max \{u(x) - \varphi_{\zeta}(x) \, \mid \, x \in \partial V \},
		\end{equation*}
	then there is a triple $(x_{*},X_{*},q_{**}) \in V \times \mathcal{S}_{d} \times \{\varphi_{\zeta} = 1\}$ such that
		\begin{gather*}
			u(x_{*}) - \varphi_{\zeta}(x_{*}) = \max \{u(x) - \varphi_{\zeta}(x) \, \mid \, x \in \overline{V} \}, \\
			- \mathcal{H}(x_{*},D\varphi_{\zeta}(x_{*}),X_{*}, Q_{X_{*}}(q_{**} - \zeta \|D\varphi_{\zeta}(x_{*})\|^{-1} D\varphi_{\zeta}(x_{*}))) \leq 0, \\
			X_{*} q_{**} = 0, \quad \text{and} \quad 0 \leq \varphi_{\zeta}(x_{*}) X_{*} \leq \tilde{\Gamma}(\varphi) \zeta^{-1} \text{Id},
		\end{gather*}
	where $\tilde{\Gamma}(\varphi) > 0$ depends only on $\varphi$.
	\end{prop}
	
\begin{proof} This is an application of Theorem \ref{T: technical core}.  Indeed, reasoning as in Section \ref{S: c11 reduction}, one observes that
	\begin{equation*}
		-\mathcal{H}(x,Du,D^{2}u,\mathcal{G}_{\varphi_{\zeta}}^{\zeta}(Du,D^{2}u)) \leq 0 \quad \text{in} \, \, V.
	\end{equation*}
Furthermore, $\varphi_{\zeta}$ satisfies \eqref{E: zeta interior ball} with radius $\zeta$ and \eqref{E: nondegeneracy} with constant $\delta = \delta(\varphi)$ by Theorem \ref{T: regularized_norm}.  Taken together with the other hypotheses, these observations establish that Theorem \ref{T: technical core} applies, and the theorem directly yields the desired conclusion with $\tilde{\Gamma}(\varphi) := \Gamma(\delta(\varphi))$.
\end{proof}

\subsection{Curved Cones} The cone comparison principle proved so far only treats $\varphi$ itself as a test function.  In various contexts, it useful to know that these results extend to ``curved" cones of the form $h(\varphi)$ for some increasing function $h$.

\begin{prop} \label{P: technical core eaten curved} Let $\varphi$ be a Finsler norm in $\mathbb{R}^{d}$ and let $(\varphi_{\zeta})_{\zeta > 0}$ be the regularized norms of Theorem \ref{T: regularized_norm}.  Let $V \subseteq \mathbb{R}^{d} \setminus \{0\}$ be a bounded open set, and let $\mathcal{H}$ be a function satisfying \eqref{Ass: continuity assumption} and \eqref{Ass: elliptic assumption}.  Suppose that $h : \mathbb{R} \to \mathbb{R}$ is smooth and increasing.
	
	If $u$ is an upper semicontinuous function in $\overline{V}$ satisfying
		\begin{equation*}
			-\mathcal{H}(x,Du,D^{2}u,G_{\varphi}^{*}(Du,D^{2}u)) \leq 0 \quad \text{in} \, \, V,
		\end{equation*}
	and if there is a point $x_{0} \in V$ and a $\zeta \in (0,1)$ such that 
		\begin{equation*}
			u(x_{0}) - h(\varphi_{\zeta}(x_{0})) = \max \{u(x) - h(\varphi_{\zeta}(x)) \, \mid \, x \in \overline{V} \} > \max \{u(x) - h(\varphi_{\zeta}(x)) \, \mid \, x \in \partial V \},
		\end{equation*}
	then there is a triple $(x_{*},X_{*},q_{**}) \in V \times \mathcal{S}_{d} \times \{\varphi_{\zeta} = 1\}$ such that
		\begin{gather}
			u(x_{*}) - h(\varphi_{\zeta}(x_{*})) = \max \{u(x) - h(\varphi_{\zeta}(x)) \, \mid \, x \in \overline{V} \}, \label{E: equality thing we need ack} \\
			- \mathcal{H}(x_{*},h'(u(x_{*}))D\varphi_{\zeta}(x_{*}),X_{*}^{h}, Q_{X^{h}_{*}}(q_{**} - \zeta \|D\varphi_{\zeta}(x_{*})\|^{-1} D\varphi_{\zeta}(x_{*}))) \leq 0, \nonumber \\
			X_{*} q_{**} = 0, \quad \text{and} \quad 0 \leq \varphi_{\zeta}(x_{*}) X_{*} \leq c(\varphi) \zeta^{-1} \text{Id}, \nonumber
		\end{gather}
		where $X^{h}_{*}$ is derived from $h$, $X_{*}$, and $x_{*}$ via the rule
		\begin{equation*}
			X^{h}_{*} = h'(u(x_{*})) X_{*} + h''(u(x_{*})) D\varphi_{\zeta}(x_{*}) \otimes D\varphi_{\zeta}(x_{*}).
		\end{equation*}
	\end{prop}

The proof of the proposition amounts to writing down the differential inequality governing the function $v = h^{-1}(u)$, as explained next.

%Now that Proposition \ref{P: technical core eaten curved} is available, Theorem \ref{T: } follows in a manner similar to the previous section.
%
%	\begin{proof}[Proof of Theorem \ref{T: }]  Fix $\beta, k \in \mathbb{R}$, suppose that $U \subseteq \mathbb{R}^{d}$ is open, and $u : U \to \mathbb{R} \cup \{-\infty\}$ is an upper semi-continuous function in $U$ such that 
%		\begin{align*}
%			-G_{\varphi}^{*}(Du,D^{2}u) - \beta \varphi^{*}(Du) \leq k \quad \text{in} \, \, U.
%		\end{align*}
%	The claim is: given an open $V \Subset U$, if $\gamma$ is the solution of the ODE 
%		\begin{equation*}
%			-\gamma'' - \beta \gamma' = k \quad \text{in} \, \, (0,\infty), \quad \gamma'(0) = 0,
%		\end{equation*}
%	then, for any constant $M \in \mathbb{R}$, there holds
%		\begin{equation*}
%			\max \left\{ u(x) - \gamma(\varphi(x - x_{0})) \, \mid \, x \in \overline{V} \right\} = \max \left\{ u(x) - \gamma(\varphi(x - x_{0})) \, \mid \, x \in \partial V \right\}.
%		\end{equation*}
%\end{proof}
	
		\begin{proof} By hypothesis, there is an $r > 0$ such that if $\overline{B}_{r}(\partial V) = \bigcup_{y \in \partial V} \overline{B}_{r}(y)$, then
			\begin{equation*}
				\max \{u(x) - h(\varphi_{\zeta}(x)) \, \mid \, x \in \overline{V} \} > \max \{u(x) - h(\varphi_{\zeta}(x)) \, \mid \, x \in \overline{B}_{r}(\partial V) \}.
			\end{equation*}
		To begin with, there is no loss of generality assuming that $h'$ is positive and $h(\mathbb{R}) = \mathbb{R}$; otherwise, one can perturb $h$ by setting $h_{\epsilon}(u) = h(u) + \epsilon u$, and then recover the desired result in the limit $\epsilon \to 0^{+}$.
		
		Up to subtracting a constant from $u$, which does not change the equation, one can assume that
			\begin{equation*}
				\max\left\{ u(x) - h(\varphi_{\zeta}(x)) \, \mid \, x \in \overline{V} \right\} = 0,
			\end{equation*}
		hence $u \leq h(\varphi_{\zeta})$ holds in $\overline{V}$ with equality at some point in $V$.
		
		In particular, if $v = h^{-1}(u)$, then
			\begin{equation*}
				\max \left\{ v(x) - \varphi_{\zeta}(x) \, \mid \, x \in \overline{V} \right\} = 0 > \max \left\{v(x) - \varphi_{\zeta}(x) \, \mid \, x \in \overline{B}_{r}(\partial V) \right\}.
			\end{equation*}
		
		A straightforward computation shows that $v$ satisfies
			\begin{equation*}
				-\mathcal{H}(x,h'(u(x)) Dv, \mathcal{X}(v,Dv,D^{2}v), G_{\varphi}^{*}(h'(u(x))Dv, \mathcal{X}(u(x),Dv,D^{2}v))) \leq 0 \quad \text{in} \, \, V,
			\end{equation*}
		where $\mathcal{X} : \mathbb{R} \times \mathbb{R}^{d} \times \mathcal{S}_{d} \to \mathcal{S}_{d}$ is the map
			\begin{equation*}
				\mathcal{X}(u,p,X) = h'(u) X + h''(u) p \otimes p.
			\end{equation*}
		This differential inequality fits into the framework of Proposition \ref{P: technical core eaten}.  Therefore, by that proposition, there is a triple $(x_{*},X_{*},q_{**}) \in V \times \mathcal{S}_{d} \times \{\varphi = 1\}$ such that 
			\begin{gather}
				v(x_{*}) - \varphi_{\zeta}(x_{*}) = 0 = \max \left\{ v(x) - \varphi(x) \, \mid x \in \overline{V} \right\}, \label{E: monotonicity use here ack} \\
				- \mathcal{H}(x_{*}, h'(u(x_{*})) D\varphi_{\zeta}(x_{*}), X^{h}_{*}, Q_{X^{h}_{*}}(q_{**} - \zeta \|D\varphi_{\zeta}(x_{*})\|^{-1} D\varphi_{\zeta}(x_{*}))) \leq 0, \label{E: equation for the umpteenth time} \\
				X_{*} q^{\epsilon}_{**} = 0, \quad \text{and} \quad 0 \leq \varphi_{\zeta}(x_{*}) X_{*} \leq c(\varphi) \zeta^{-1} \text{Id}. \label{E: another bound}
			\end{gather}
		where $X^{h}_{*} = \mathcal{X}(u(x_{*}), D\varphi_{\zeta}(x_{*}),X_{*})$. 
	Note that \eqref{E: monotonicity use here ack} implies \eqref{E: equality thing we need ack} since $h$ is monotone.  Therefore, the proof is complete. \end{proof}

\section{Comparison Results in Bounded Domains} \label{S: elliptic bounded domains}

In this section, general comparison principles are proved by combining the cone comparison results of the previous sections with the standard uniqueness machinery from the theory of viscosity solutions.

\subsection{The Dirichlet Problem}  Consider the following generalization of \eqref{E: naive}, namely, 
	\begin{align} \label{E: dirichlet with fancy h}
		\left\{ \begin{array}{r l}
			- \mathcal{H}(x,u,Du,D^{2}u,\mathcal{G}_{\varphi}^{*}(Du,D^{2}u)) \leq 0 & \text{in} \, \, U, \\
			- \mathcal{H}(x,u,Du,D^{2}u, \mathcal{G}_{*}^{\varphi}(Du,D^{2}u)) \geq 0 & \text{in} \, \, U, \\
			u = g & \text{on} \, \, \partial U,
		\end{array} \right.
	\end{align}
where the function $\mathcal{H} : \mathbb{R}^{d} \times \mathbb{R} \times \mathbb{R}^{d} \times \mathcal{S}_{d} \times \mathbb{R} \to \mathbb{R}$ is a continuous function satisfying the following two assumptions.  First, assume there is a constant $\lambda = \lambda(\mathcal{H}) > 0$ such that, for any $(x,p) \in \mathbb{R}^{d} \times \mathbb{R}^{d}$, $s,r, a,b \in \mathbb{R}$, and $X,Y \in \mathcal{S}_{d}$,
	\begin{align} 
		\lambda |s - r| \leq \mathcal{H}(x,s,p,Y,b) - \mathcal{H}(x,r,p,X,a) \quad \text{if} \, \, s \leq r, \, \, Y \geq X, \, \, b \geq a; \label{E: strict monotonicity operator}
	\end{align}
Next, assume there is a $\zeta = \zeta(\mathcal{H}) \in (0,1)$ such that if $X, Y \in \mathcal{S}_{d}$ satisfy, for some constants $C, T > 0$,
	\begin{align*}
		-C T \left( \begin{array}{r l}
					\text{Id} & 0 \\
					0 & \text{Id}
				\end{array} \right) \leq \left( \begin{array}{r l}
										X & 0 \\
										0 & -Y
									\end{array} \right) \leq CT \left( \begin{array}{r l}
																\text{Id} & - \text{Id} \\
																- \text{Id} & \text{Id}
															\end{array} \right),
	\end{align*}
then there is a continuous function $\omega_{C} : [0,+\infty) \to [0,+\infty)$ depending only on $C$ such that $\omega_{C}(0) = 0$ and, for any $x,y \in \mathbb{R}^{d}$, $r \in \mathbb{R}$, and $q \in B_{C}(0)$, \footnote{Here $D(\varphi_{\zeta}^{2})$ denotes the derivative of the function $x \mapsto \varphi_{\zeta}(x)^{2}$.  In particular, if $\varphi_{\zeta} = \|\cdot\|$, then $D(\varphi_{\zeta}^{2})(x - y) = 2(x - y)$.}
	\begin{align}
		&\mathcal{H}(x,r, 2^{-1} T D(\varphi_{\zeta}^{2})(x - y), X, Q_{X}(q)) - \mathcal{H}(y,r, 2^{-1} T D(\varphi_{\zeta}^{2})(x - y), X, Q_{X}(q)) \label{E: regularity operator} \\
		&\qquad \qquad \qquad \leq \omega_{C}(\|x - y\| + T\|x - y\|^{2}). \nonumber
	\end{align}
For examples of functions $\mathcal{H}$ for which this last condition holds, see Remarks \ref{R: regularity assumption} and \ref{R: diffusion assumption} below.

The comparison principle for \eqref{E: dirichlet with fancy h} reads as follows:

	\begin{theorem} \label{T: penalized problem comparison} Let $\varphi$ be a Finsler norm in $\mathbb{R}^{d}$, $U \subseteq \mathbb{R}^{d}$ be a bounded open set, and $\mathcal{H}$ be a continuous function such that \eqref{E: strict monotonicity operator} and \eqref{E: regularity operator} both hold for some fixed parameters $\lambda(\mathcal{H}) > 0$ and $\zeta(\mathcal{H}) \in (0,1)$.  If $u \in USC(\overline{U})$ and $v \in LSC(\overline{U})$ are such that 
		\begin{gather*}
			-\mathcal{H}(x,u,Du,D^{2}u,\mathcal{G}_{\varphi}^{*}(Du,D^{2}u)) \leq 0 \quad \text{in} \, \, U, \\
			- \mathcal{H}(x,v,Dv,D^{2}v,\mathcal{G}_{*}^{\varphi}(Dv,D^{2}v)) \geq 0 \quad \text{in} \, \, U, \\
			\text{and} \quad \max \left\{ u(x) - v(x) \, \mid \, x \in \partial U \right\} \leq 0,
		\end{gather*}
	then 
		\begin{equation*}
			\max \left\{ u(x) - v(x) \, \mid \, x \in \overline{U} \right\} \leq 0.
		\end{equation*}\end{theorem}

The proof proceeds using the variable-doubling technique.  In particular, given $\beta > 0$, consider the function
	\begin{equation*}
		\Phi(x,y) = u(x) - v(y) - \frac{\varphi_{\zeta}(x - y)^{4}}{4 \beta},
	\end{equation*}
where $\varphi_{\zeta}$ is the regularized version of $\varphi$ from Theorem \ref{T: regularized_norm} and $\zeta = \zeta(\mathcal{H})$ is the parameter from \eqref{E: regularity operator}.  The next result asserts that, compared to the cone comparison results proved earlier (i.e., Propositions \ref{P: technical core eaten} and \ref{P: technical core eaten curved}), the addition of a second variable causes no additional difficulties.  

	\begin{prop} \label{P: doubling variables} Let $\varphi$ be a Finsler norm and $U \subseteq \mathbb{R}^{d}$ be a bounded open set.  Given a parameter $\zeta \in (0,1)$, let $\varphi_{\zeta}$ be the regularized version of $\varphi$ from Theorem \ref{T: regularized_norm}.  Let $\mathcal{H} : \mathbb{R}^{d} \times \mathbb{R} \times \mathbb{R}^{d} \times \mathcal{S}_{d} \times \mathbb{R} \to \mathbb{R}$ be a function satisfying assumptions \eqref{Ass: continuity assumption} and \eqref{Ass: elliptic assumption}.  
	
	Suppose that $u \in USC(\overline{U})$ and $v \in LSC(\overline{U})$ are such that
		\begin{align}
			-\mathcal{H}(x,u,Du,D^{2}u,G_{\varphi}^{*}(Du,D^{2}u)) &\leq 0 \quad \text{in} \, \, U, \label{E: we are doubling variables 1}\\
			-\mathcal{H}(x,u,Dv,D^{2}v,G_{*}^{\varphi}(Dv,D^{2}v)) &\geq 0 \quad \text{in} \, \, U. \label{E: we are doubling variables 2}
		\end{align}
	
	If for some $\beta > 0$ one has that
		\begin{align}
			\sup &\left\{ u(x) - v(y) - \frac{\varphi_{\zeta}(x-y)^{4}}{4 \beta} \, \mid \, (x,y) \in U \times U \right\} \label{E: nondegenerate boundary assumption} \\
				&\qquad > \max \left\{ u(x) - v(y) - \frac{\varphi_{\zeta}(x -y)^{4}}{4 \beta} \, \mid \, (x,y) \in \partial(U \times U) \right\}, \nonumber
		\end{align}
	then there are triples $(x_{*},X_{*},q^{u}_{**}), (y_{*},Y_{*},q^{v}_{**}) \in U \times \mathcal{S}_{d} \times \{\varphi_{\zeta} = 1\}$ and a constant $C(\varphi) > 0$ depending only on $\varphi$ such that, after defining $T_{*} \geq 0$ and $\mu_{*} \in \mathbb{R}^{d}$ by
		\begin{equation*}
			T_{*} = \beta^{-1} \varphi_{\zeta}(x_{*} - y_{*})^{2}, \quad \mu_{*} = \left\{ \begin{array}{r l}
			 			\beta^{-1} \varphi_{\zeta}(x_{*} - y_{*})^{3} D\varphi_{\zeta}(x_{*} - y_{*}), & \text{if} \, \, x_{*} \neq y_{*}, \\
						0, & \text{otherwise,}
					\end{array} \right.
		\end{equation*} 
	one has
		\begin{gather}
			u(x_{*}) - v(y_{*}) - \frac{\varphi_{\zeta}(x_{*} - y_{*})^{4}}{4 \beta} \qquad \qquad \qquad \qquad \qquad\qquad \label{E: big maximization} \\
			\quad \qquad \qquad \qquad = \max\left\{ u(x) - v(y) - \frac{\varphi_{\zeta}(x-y)^{4}}{4 \beta} \, \mid \, (x,y) \in \overline{U} \times \overline{U} \right\}, \nonumber \\
			-\mathcal{H}(x_{*},u(x_{*}),\mu_{*}, X_{*}, Q_{X_{*}}(q_{**}^{u} - \zeta \|\mu_{*}\|^{-1} \mu_{*})) \leq 0 \quad \text{if} \, \, x_{*} \neq y_{*}, \label{E: big equation 1} \\
			-\mathcal{H}(x_{*},u(x_{*}),0,0,0) \leq 0 \quad \text{if} \, \, x_{*} = y_{*}, \label{E: big equation 1 star}, \\
			-\mathcal{H}(y_{*},u(y_{*}),\mu_{*},Y_{*}, Q_{Y_{*}}(q^{v}_{**} - \zeta \|\mu_{*}\|^{-1} \mu_{*})) \geq 0 \quad \text{if} \, \, x_{*} \neq y_{*}, \label{E: big equation 2} \\
			-\mathcal{H}(y_{*},u(y_{*}),0,0,0) \geq 0 \quad \text{if} \, \, x_{*} = y_{*}, \label{E: big equation 2 star} \\
			X_{*} (q^{u}_{**} - q^{v}_{**}) = Y_{*}(q^{u}_{**} - q^{v}_{**}) = 0, \label{E: big killing}, \\
			- C(\varphi) T_{*} \zeta^{-1} \left( \begin{array}{c c}
															\text{Id} & 0 \\
															0 & \text{Id}
														\end{array} \right) \leq \left( \begin{array}{c c}
		X_{*} & 0 \\
		0 & -Y_{*}
	\end{array} \right) \leq C(\varphi) T_{*} \zeta^{-1} \left( \begin{array}{c c}
																	\text{Id} & - \text{Id} \\
																	-\text{Id} & \text{Id}
																\end{array} \right), \label{E: crazy matrix business}
		\end{gather}
	\end{prop}
	
\begin{remark} \label{R: big killing crazy} Notice that if, as in the previous proposition, there is a matrix $Y_{*} \in \mathcal{S}_{d}$ and vectors $q_{**}^{u}, q_{**}^{v} \in \mathbb{R}^{d}$ such that $Y_{*}q_{**}^{u} = Y_{*}q_{**}^{v}$, then, for any $\xi \in \mathbb{R}^{d}$,
	\begin{equation*}
		Q_{Y_{*}}(q^{v}_{**} + \xi) = Q_{Y_{*}}(q_{**}^{u} + \xi).
	\end{equation*}
In particular, in this scenario, \eqref{E: big equation 2} is equivalent to
	\begin{equation*}
		-\mathcal{H}(y_{*},u(y_{*}),\mu_{*},Y_{*},Q_{Y_{*}}(q^{u}_{**} - \zeta \|\mu_{*}\|^{-1} \mu_{*}) \leq 0.
	\end{equation*}
As will become clear in the proof of Theorem \ref{T: penalized problem comparison} below, this means the discontinuity of $G_{\varphi}^{*}$ and $G_{*}^{\varphi}$ has effectively been circumvented, and the rest of the variable-doubling argument can be conducted as in the case of equations with continuous coefficients.
\end{remark}

Before going into the proof of Proposition \ref{P: doubling variables}, here are some remarks on the regularity condition \eqref{E: regularity operator}:

	\begin{remark} \label{R: regularity assumption} Note that \eqref{E: regularity operator} is a generalization of the standard assumption guaranteeing comparison in the viscosity theory for second-order elliptic operators, see, for instance, \cite[Section 3]{user}.  Indeed, \eqref{E: regularity operator} becomes the assumption in that reference if $\varphi_{\zeta} = \alpha^{-1} \|\cdot\|$ for some $\alpha > 0$ (cf. Remark \ref{R: ball regularization}).

	In particular, as usual, if $\mathcal{H}$ has the following ``separable" form
		\begin{equation*}
			\mathcal{H}(x,s,p,X,a) = \mathcal{H}_{0}(s,p,X,a) - f(x),
		\end{equation*}
	where $\mathcal{H}_{0}$ is monotone increasing in the third and fourth arguments and $f$ is uniformly continuous, then \eqref{E: regularity operator} holds.   \end{remark}
	
	\begin{remark} \label{R: diffusion assumption} The computation of \cite[Example 3.6]{user} applies just as well in this setting.  Thus, for instance, if $\sigma : \mathbb{R}^{d} \to [0,+\infty)$ is uniformly Lipschitz, that is, if there is a $K > 0$ such that, for any $x,y \in \mathbb{R}^{d}$,
		\begin{equation*}
			|\sigma(x) - \sigma(y)| \leq K \|x - y\|,
		\end{equation*}
	then the operator $\mathcal{H}_{\sigma}$ given by 
		\begin{equation*}
			\mathcal{H}_{\sigma}(x,r,p,X,a) = -u + \sigma(x)^{2} a
		\end{equation*}
	satisfies \eqref{E: regularity operator}.  Hence, setting $c(x) = \sigma(x)^{2}$, Theorem \ref{T: penalized problem comparison} applies to the problem \eqref{E: finsler dirichlet with zeroth order} presented in the introduction. \end{remark}

Taking Proposition \ref{P: doubling variables} for granted for the moment, here is the proof of Theorem \ref{T: penalized problem comparison}. 

	\begin{proof}[Proof of Theorem \ref{T: penalized problem comparison}]  Argue by contradiction, that is, assume that
		\begin{equation} \label{E: running assumption ack}
			\max\left\{ u(x) - v(x) \, \mid \, x \in \overline{U} \right\} > 0 \geq \max \left\{ u(x) - v(x) \, \mid \, x \in \partial U \right\}.
		\end{equation}
	
	Let $\zeta = \zeta(\mathcal{H}) \in (0,1)$ be as in the hypotheses and let $\varphi_{\zeta}$ be the regularized norm of Theorem \ref{T: regularized_norm}.
	
	Let $\beta > 0$.  Assume that there is a $\beta_{*} \in (0,1)$ such that, for any $\beta \in (0,\beta_{*})$,		\begin{align}
			\sup &\left\{ u(x) - v(y) - \frac{\varphi_{\zeta}(x - y)^{4}}{4 \beta} \, \mid \, (x,y) \in U \times U \right \} \label{E: ack ack ack need to prove ack ack ack} \\
			&\qquad > \max \left\{ u(x) - v(y) - \frac{\varphi_{\zeta}(x - y)^{4}}{4 \beta} \, \mid \, (x,y) \in \partial (U \times U) \right\}. \nonumber
		\end{align}
	The existence of $\beta_{*}$ follows from \eqref{E: running assumption ack}, as will be verified at the end of this proof.

	Taking $\beta < \beta_{*}$ and applying Proposition \ref{P: doubling variables}, one obtains two triples 
		\begin{equation*}
			(x_{*},X_{*},q^{u}_{**}), (y_{*},Y_{*},q^{v}_{**}) \in U \times \mathcal{S}_{d} \times \{\varphi_{\zeta} = 1\}
		\end{equation*} 
	such that the conditions \eqref{E: big maximization}-\eqref{E: crazy matrix business} all hold.  As in the statement of the proposition, it is convenient to let $T_{*} = \beta^{-1} \varphi_{\zeta}(x_{*} - y_{*})^{2}$.  The crux of the proof lies in establishing the following inequality		
		\begin{equation} \label{E: what we are proving now finally at last}
			\lambda (u(x_{*}) - v(y_{*})) \leq \omega_{C}(\|x_{*} - y_{*}\| + T_{*} \|x_{*} -y_{*}\|^{2}),
		\end{equation}
	where $\lambda = \lambda(\mathcal{H})$ is as in the hypotheses and $\omega_{C}$ is as in \eqref{E: regularity operator} for some well-chosen $C$.
	
	If $x_{*} = y_{*}$, then this follows immediately from \eqref{E: big equation 1 star}, \eqref{E: big equation 2 star}, and \eqref{E: strict monotonicity operator}.  Indeed, here in the application of \eqref{E: strict monotonicity operator}, one needs to observe that, by \eqref{E: running assumption ack} and the choice of the point $(x_{*},y_{*})$,
		\begin{align}
			0 &< \sup \left\{ u(x) - v(x) \, \mid \, x \in \overline{U} \right\} \label{E: useful string of inequalities ack} \\
			&\leq \sup \left\{ u(x) - v(y) - \frac{\varphi_{\zeta}(x - y)^{4}}{4 \beta} \, \mid \, x,y \in \overline{U} \times \overline{U} \right\} \leq u(x_{*}) - v(y_{*}). \nonumber
		\end{align}
	Thus, if $x_{*} = y_{*}$, one invokes \eqref{E: strict monotonicity operator} and \eqref{E: regularity operator} (with an arbitrary $C$) to find
		\begin{align*}
			\lambda (u(x_{*}) - v(y_{*})) &\leq \mathcal{H}(x_{*},v(y_{*}),0,0,0) - \mathcal{H}(x_{*},u(x_{*}),0,0,0) \\
			&\leq \mathcal{H}(x_{*}, v(y_{*}),0,0,0) - \mathcal{H}(y_{*},v(y_{*}),0,0,0) \\
			&\leq \omega_{C}(\|x_{*} - y_{*}\| + T_{*} \|x_{*} - y_{*}\|^{2}).
		\end{align*}
	Hence, in what follows, assume $x_{*} \neq y_{*}$.
	
	As explained in Remark \ref{R: big killing crazy}, when $x_{*} \neq y_{*}$, \eqref{E: big equation 2} combines with \eqref{E: big killing} to yield the following inequality
		\begin{equation} \label{E: key thing we need ack}
			-\mathcal{H}(y_{*},v(y_{*}),\mu_{*},Y_{*},Q_{Y_{*}}(q^{u}_{**} - \zeta \|\mu_{*}\|^{-1} \mu_{*}) \geq 0,
		\end{equation}
	where $\mu_{*} = \beta^{-1} \varphi_{\zeta}(x_{*} - y_{*})^{3} D\varphi_{\zeta}(x_{*} - y_{*})$.  At this stage, the goal is to apply \eqref{E: strict monotonicity operator} and \eqref{E: regularity operator} to obtain \eqref{E: what we are proving now finally at last}.
	
Toward that end, note that, by \eqref{E: crazy matrix business}, $X_{*}$ and $Y_{*}$ satisfy the following inequality
	\begin{equation*}
		- C(\varphi) T_{*} \left( \begin{array}{c c}
							\text{Id} & 0\\
							0 & \text{Id}
						\end{array} \right) \leq \left( \begin{array}{c c}
													X_{*} & 0 \\
													0 & - Y_{*}
											\end{array} \right) \leq C(\varphi) T_{*} \left( \begin{array}{c c}
																	\text{Id} & - \text{Id} \\
																	- \text{Id} & \text{Id}
															\end{array} \right).
	\end{equation*}
Furthermore, observe that since $\{q_{**}^{u},q_{**}^{v}\} \subseteq \{\varphi_{\zeta} = 1\}$, there holds
	\begin{equation*}
		\left\| q_{**}^{u} - \zeta \|\mu_{*}\|^{-1} \mu_{*} \right\|, \left\| q_{**}^{v} - \zeta \|\mu_{*}\|^{-1} \mu_{*} \right\| \leq \max \left\{ \|q\| \, \mid \, \varphi_{\zeta}(q) = 1 \right\} + 1.
	\end{equation*}
Thus, if $C = C(\varphi) +  \max \{ \|q\| \, \mid \, \varphi_{\zeta}(q) = 1\} + 1$, then one can invoke \eqref{E: regularity operator} to find
	\begin{align*}
		&\mathcal{H}(x_{*},v(y_{*}), \mu_{*},X_{*},Q_{X_{*}}(q_{u}^{**} - \zeta \|\mu_{*}\|^{-1} \mu)) -\mathcal{H}(y_{*},v(y_{*}),\mu_{*},Y_{*},Q_{Y_{*}}(q^{u}_{**} - \zeta \|\mu_{*}\|^{-1} \mu_{*}) \\
		&\qquad \qquad \qquad \leq \omega_{C}(\|x_{*} - y_{*}\| + T_{*} \|x_{*} -y_{*}\|^{2}).
	\end{align*}

Adding this last expression to \eqref{E: big equation 1} and invoking \eqref{E: key thing we need ack} and \eqref{E: strict monotonicity operator}, one obtains
	\begin{align*}
		\lambda (u(x_{*}) - v(y_{*})) &\leq \mathcal{H}(x_{*}, v(y_{*}), \mu_{*}, X_{*}, Q_{X_{*}}(q_{**}^{u} - \zeta \|\mu_{*}\|^{-1} \mu_{*})) \\
			&\quad - \mathcal{H}(x_{*},u(x_{*}), \mu_{*},X_{*}, Q_{X_{*}}(q_{**}^{u} - \zeta \|\mu_{*}\|^{-1} \mu_{*})) \\
			&\leq \omega_{C}(\|x_{*} - y_{*}\| + T_{*} \|x_{*} - y_{*}\|^{2}) \\
			&\quad + \mathcal{H}(y_{*},v(y_{*}),\mu_{*},Y_{*},Q_{Y_{*}}(q_{**}^{u} - \zeta \|\mu_{*}\|^{-1} \mu_{*})) \\
			&\leq \omega_{C}(\|x_{*} - y_{*}\| + T_{*} \|x_{*} - y_{*}\|^{2}).
	\end{align*} 
Thus, \eqref{E: what we are proving now finally at last} holds, as claimed.  

Here is the conclusion of the proof: by \eqref{E: what we are proving now finally at last} and \eqref{E: useful string of inequalities ack}, one deduces that
	\begin{align*}
		\sup \left\{ u(x) - v(x) \, \mid \, x \in \overline{U} \right\} &\leq u(x_{*}) - v(y_{*}) \leq \frac{1}{\lambda} \omega_{C}(\|x_{*} - y_{*}\| + T_{*} \|x_{*} - y_{*}\|^{2}).
	\end{align*}
In view of the running assumption \eqref{E: running assumption ack}, this leads to a contradiction provided 
		\begin{align*}
			\lim_{\beta \to 0^{+}} \left( \|x_{*} - y_{*}\| + T_{*} \|x_{*} - y_{*}\|^{2} \right) = 0.
		\end{align*}
	Toward that end, by definition of $T_{*}$, it suffices to establish that
		\begin{align} \label{E: vanishing part of proof comparison}
			\lim_{\beta \to 0^{+}} \frac{\varphi_{\zeta}(x_{*} - y_{*})^{4}}{4 \beta} = 0.
		\end{align}
	Indeed, notice that, by definition of $T_{*}$, 
		\begin{equation*}
			T_{*} \|x_{*} - y_{*}\|^{2} \leq \max \left\{ \|q\|^{2} \, \mid \, \varphi_{\zeta}(q) = 1 \right\} \cdot \frac{\varphi_{\zeta}(x_{*} - y_{*})^{4}}{\beta}.
		\end{equation*}
	Thus, the proof is complete as soon as \eqref{E: ack ack ack need to prove ack ack ack} and \eqref{E: vanishing part of proof comparison} are established.  Both are fairly standard claims; nonetheless, the details are provided next for completeness.
	
	In order to avoid confusion in what follows, make the dependence on $\beta$ explicit by setting $(x_{*}^{\beta},y_{*}^{\beta}) = (x_{*},y_{*})$.  Since $\overline{U}$ is compact, there is no great loss of generality in assuming that the limits $x_{\infty} = \lim_{\beta \to 0^{+}} x_{*}^{\beta}$ and $y_{\infty} = \lim_{\beta \to 0^{+}} y_{*}^{\beta}$ exist; otherwise, pass to a subsequence.  In view of the computation \eqref{E: useful string of inequalities ack}, 
		\begin{align*}
			\frac{\varphi_{\zeta}(x_{*}^{\beta} - y_{*}^{\beta})^{4}}{4 \beta} &\leq u(x_{*}^{\beta}) - v(y_{*}^{\beta}) \\
			&\leq \max \left\{ u(x) \, \mid \, x \in \overline{U} \right\} - \inf \left\{ v(y) \, \mid \, y \in \overline{U} \right\} < +\infty.
		\end{align*}
	It follows that $x_{\infty} = y_{\infty}$. 
	
	Since $u$ is upper semicontinuous and $v$, lower semicontinuous, one finds
		\begin{align*}
			u(x_{\infty}) - v(x_{\infty}) &\geq \limsup_{\beta \to 0^{+}} u(x_{*}^{\beta}) - \liminf_{\beta \to 0^{+}} v(y_{*}^{\beta}) \\
				&\geq \limsup_{\beta \to 0^{+}} \left[ u(x_{*}^{\beta}) - v(y_{*}^{\beta}) - \frac{\varphi_{\zeta}(x_{*}^{\beta} - y_{*}^{\beta})^{4}}{4 \beta} \right].
		\end{align*}
	At the same time, 		
		\begin{equation*}
			u(x_{\infty}) - v(x_{\infty}) \leq \sup\{u(x) - v(x) \, \mid \, x \in U \} \leq \liminf_{\beta \to 0^{+}} \left[ u(x_{*}^{\beta}) - v(y_{*}^{\beta}) - \frac{\varphi_{\zeta}(x_{*}^{\beta} - y_{*}^{\beta})^{4}}{4 \beta} \right].
		\end{equation*}
	This proves $u(x_{\infty}) = \limsup_{\beta \to 0^{+}} u(x_{*}^{\beta})$ and $v(x_{\infty}) = \liminf_{\beta \to 0^{+}} v(y_{*}^{\beta})$.  Therefore, after revisiting the last inequality, one finds
		\begin{align*}
			u(x_{\infty}) - v(x_{\infty}) &\leq \limsup_{\beta \to 0^{+}} u(x_{*}^{\beta}) - \liminf_{\beta \to 0^{+}} v(x_{*}^{\beta}) - \limsup_{\beta \to 0^{+}} \frac{\varphi_{\zeta}(x_{*}^{\beta} - y_{*}^{\beta})^{4}}{4 \beta} \\
				&= u(x_{\infty}) - v(x_{\infty}) - \limsup_{\beta \to 0^{+}} \frac{\varphi_{\zeta}(x_{*}^{\beta}) - y_{*}^{\beta})^{4}}{4 \beta}.
		\end{align*} 
	This proves $(4 \beta)^{-1} \varphi_{\zeta}(x_{*}^{\beta} - y_{*}^{\beta})^{4} \to 0$ as $\epsilon \to 0$, as claimed.  Furthermore, the argument showed that 
		\begin{equation*}
			\lim_{\beta \to 0^{+}} \sup \left\{ u(x) - v(y) - \frac{\varphi_{\zeta}(x - y)^{4}}{4 \beta} \, \mid \, x,y \in \overline{U} \times \overline{U} \right\} = \sup \{u(x) - v(x) \, \mid \, x \in \overline{U}\}.
		\end{equation*}
	
	Finally, to establish \eqref{E: ack ack ack need to prove ack ack ack}, one repeats the previous arguments to show that
		\begin{align*}
			&\lim_{\beta \to 0^{+}} \sup \left\{ u(x) - v(y) - \frac{\varphi_{\zeta}(x - y)^{4}}{4 \beta} \, \mid \, x,y \in \partial (U \times U) \right\} \\
			&\qquad \qquad \qquad = \sup \left\{ u(x) - v(x) \, \mid \, x \in \partial U \right\}.
		\end{align*}
	Combining this with the previous limit, one invokes \eqref{E: running assumption ack} to deduce that \eqref{E: ack ack ack need to prove ack ack ack} holds. \end{proof}
	
Before entering into the proof of Proposition \ref{P: doubling variables}, here is a final remark concerning one of its first implications.
	
\begin{remark} \label{R: barles busca} As is readily verified, Proposition \ref{P: doubling variables} can be combined with the classical result of \cite{barles_busca} to give an alternative proof of the comparison principle for $\varphi$-infinity harmonic functions, or \eqref{E: comparison of AMLE}. \end{remark}

\subsection{Proof of Proposition \ref{P: doubling variables}} The proof of Proposition \ref{P: doubling variables} proceeds by combining the ingredients developed in the previous sections with the uniqueness machinery from the theory of viscosity solutions.  In particular, a very useful lemma from \cite{ishii} will be used to control the Hessian.

Specifically, the following version of the maximum principle for semicontinuous functions (or the Jensen-Ishii Lemma)\footnote{See \cite[Theorem 3.2]{user}, \cite[Theorem on Sums]{primer}, or \cite[Lemmas 3.23 and 3.30]{silvestre_imbert}.} will be used:

\begin{prop}[Lemma 1 in \cite{ishii}] \label{P: jensen ishii lemma} Fix $\alpha \in \mathbb{R}$, let $U \subseteq \mathbb{R}^{d}$ be a bounded open set, and suppose that $u \in USC(\overline{U})$ and $v \in LSC(\overline{U})$ satisfy
	\begin{align*}
		-\mathcal{H}(x,u,Du,D^{2}u,\mathcal{G}_{\varphi}^{*,\alpha}(Du,D^{2}u)) \leq 0 \quad \text{in} \, \, U, \\
		-\mathcal{H}(x,u,Du,D^{2}u,\mathcal{G}_{*,\alpha}^{\varphi}(Du,D^{2}u)) \geq 0 \quad \text{in} \, \, U.
	\end{align*}  
Let $\xi_{0} \in \mathbb{R}^{d}$ and suppose that $\psi : \mathbb{R}^{d} \to \mathbb{R} \cup \{+\infty\}$ is a convex function which is finite and smooth in a neighborhood of $\xi_{0}$.  If there are points $(x_{0},y_{0}) \in U$ such that $\xi_{0} = x_{0} - y_{0}$ and 
	\begin{equation*}
		\sup \left\{ u(x) - v(y) - \psi(x - y) \, \mid \, (x,y) \in U \times U\right\} = u(x_{0}) - v(y_{0}) - \psi(x_{0} - y_{0}),
	\end{equation*}
then there are matrices $X, Y \in \mathcal{S}_{d}$ such that 
	\begin{align*}
		- \mathcal{H}(x_{0},u(x_{0}),D\psi(\xi_{0}),X,\mathcal{G}_{\varphi}^{*,\alpha}(D\psi(\xi_{0}),X)) &\leq 0, \\
		-\mathcal{H}(y_{0},v(y_{0}),D\psi(\xi_{0}),Y,\mathcal{G}_{*,\alpha}^{\varphi}(D\psi(\xi_{0}),Y)) &\geq 0, 
	\end{align*}
and
	\begin{equation*}
		-3 \left( \begin{array}{c c}
			D^{2}\psi(\xi_{0}) & 0 \\
			0 & D^{2}\psi(\xi_{0}) 
		\end{array} \right) \leq \left( \begin{array}{c c}
								X & 0 \\
								0 & - Y
							\end{array} \right) \leq 3 \left( \begin{array}{c c}
													D^{2}\psi(\xi_{0}) & -D^{2}\psi(\xi_{0}) \\
													-D^{2}\psi(\xi_{0}) & D^{2}\psi(\xi_{0}) 
												\end{array} \right).
	\end{equation*}
\end{prop} 

The idea of the proof of Proposition \ref{P: doubling variables} is now more-or-less straightforward.  It starts by replacing $\varphi_{\zeta}$ by $\psi_{\zeta}$, an appropriately chosen conical test function.  Next, the perturbation argument is used, replacing $\psi_{\zeta}$ by $\psi_{\zeta}^{\epsilon}$.   Applying the maximum principle for semicontinuous functions to $\psi_{\zeta}^{\epsilon}$ yields matrices $(X_{\epsilon},Y_{\epsilon})$ with good properties, which, in particular, lead to the key constraint \eqref{E: big killing} in the limit $\epsilon \to 0^{+}$.

As in Section \ref{S: c11 setting}, it is convenient to first consider the case when the unit ball $\{\varphi \leq 1\}$ is $C^{1,1}$.  First, here is a variable-doubled analogue of Proposition \ref{P: basic touching thing}:

	\begin{prop} \label{P: generate contact points double variable} Let $\varphi$ be a Finsler norm in $\mathbb{R}^{d}$ satisfying \eqref{E: zeta interior ball} and \eqref{E: nondegeneracy} for some constants $\zeta, \delta > 0$; let $U \subseteq \mathbb{R}^{d}$ be a bounded open set; and fix a $p \in \mathbb{R}^{d} \setminus \{0\}$.  Define $e = \|p\|^{-1} p$.  
	
	Fix a $\beta > 0$ and a $\nu \in (0,\zeta)$.  For $\mu \in \{\nu,\zeta\}$, let $\psi_{\mu}$ be the conical test function $\psi_{\mu} := \psi_{e,\partial \varphi^{*}(p),\mu}$ of Section \ref{S: conical test}, and let $(\psi_{\nu}^{\epsilon})_{0 < \epsilon < \nu}$ be an admissible perturbation of $\psi_{\nu}$.
	
	If $u \in \text{USC}(\overline{U})$ and $v \in \text{LSC}(\overline{U})$ are such that 
		\begin{align}
			-\mathcal{H}(x,u,Du,D^{2}u,\mathcal{G}_{\varphi}^{*,\alpha}(Du,D^{2}u)) &\leq 0 \quad \text{in} \, \, U, \label{E: doubled equation part one} \\
			-\mathcal{H}(x,u,Du,D^{2}u,\mathcal{G}_{*,\alpha}^{\varphi}(Du,D^{2}u)) &\geq 0 \quad \text{in} \, \, U, \label{E: doubled equation part two}
		\end{align}
	and
		\begin{align}
			M &:= \max \left\{ u(x) - v(y) - \frac{\psi_{\zeta}(x - y)^{4}}{4 \beta} \, \mid \, (x,y) \in \overline{U} \times \overline{U} \right\} \label{E: maximum point doubled equation} \\
			&\qquad > \max \left\{ u(x) - v(y) - \frac{\psi_{\zeta}(x-y)^{4}}{4 \beta} \, \mid \, (x,y) \in \partial (U \times U) \right\}, \nonumber \\
			\{(x,y) &\, \mid \, u(x) - v(y) = M + \frac{\psi_{\zeta}(x - y)^{4}}{4 \beta} \} \label{E: nondegeneracy assumption for doubled cone} \\
			&\qquad \subseteq \{(x,y) \in U \times U \, \, \mid \, x - y \in \mathcal{C}(\partial \varphi^{*}(p)), \, \, x \neq y\}, \nonumber
		\end{align}
	then, for any $\nu \in (0,\zeta)$, there are sequences $(\epsilon_{j})_{j \in \mathbb{N}} \subseteq (0,+\infty)$, $((x_{j},y_{j}))_{j \in \mathbb{N}} \subseteq U \times U$, and $((p_{j},X_{j},Y_{j}))_{j \in \mathbb{N}} \subseteq \mathbb{R}^{d} \times \mathcal{S}_{d} \times \mathcal{S}_{d}$ such that
		\begin{gather}
			u(x_{j}) - v(y_{j}) - \frac{\psi_{\nu}^{\epsilon_{j}}(x_{j} - y_{j})^{4}}{4 \beta} = \max \left\{ u(x) - v(y) - \frac{\psi_{\nu}^{\epsilon_{j}}(x - y)^{4}}{4 \beta} \, \mid \, (x,y) \in \overline{U} \times \overline{U} \right\}, \label{E: the double variable maximum thing} \\
			- \mathcal{H}(x_{j},u(x_{j}),\mu_{j}, X_{j}, \mathcal{G}_{\varphi}^{*,\alpha}(\mu_{j},X_{j})) \leq 0, \quad - \mathcal{H}(y_{j},v(y_{j}),\mu_{j},Y_{j}, \mathcal{G}_{*,\alpha}^{\varphi}(\mu_{j},Y_{j})) \geq 0, \label{E: doubled equation} \\
			\mu_{j} = \beta^{-1} \psi_{\nu}^{\epsilon_{j}}(x_{j} - y_{j})^{3} D\psi_{\nu}^{\epsilon_{j}}(x_{j} - y_{j}), \label{E: maximum first derivative thing double variable} \\
			- 3 \left( \begin{array}{r l}
					\bar{A}_{j}(x_{j} - y_{j}) & 0 \\
					0 & \bar{A}_{j}(x_{j} - y_{j})
				\end{array} \right) \leq \left( \begin{array}{r l}
										X_{j} & 0 \\
										0 & -Y_{j} 
									\end{array} \right) \leq 3 \left( \begin{array}{r l}
															\bar{A}_{j}(x_{j} - y_{j}) & -\bar{A}_{j}(x_{j} - y_{j}) \\
															-\bar{A}_{j}(x_{j} - y_{j}) & \bar{A}_{j}(x_{j} - y_{j})
														\end{array} \right), \label{E: key doubled matrix inequality for later}
		\end{gather}
	where $\bar{A}$ is the matrix function
		\begin{align}
			\bar{A}_{j}(\xi) = \left\{ \begin{array}{r l}
							\beta^{-1} \psi_{\nu}^{\epsilon_{j}}(\xi)^{3} D^{2}\psi_{\nu}^{\epsilon_{j}}(\xi) + 3 \beta^{-1} \psi_{\nu}^{\epsilon_{j}}(\xi)^{2} D\psi_{\nu}^{\epsilon_{j}}(\xi)^{\otimes 2}, & \text{if} \, \, \xi \neq 0, \\
							0, & \text{otherwise.}
						\end{array} \right. \label{E: matrix function stuff}
		\end{align}
	\end{prop} 
	
		\begin{proof} Argue as in Proposition \ref{P: basic touching thing} to obtain sequences $(\epsilon_{j})_{j \in \mathbb{N}}$ and $((x_{j},y_{j}))_{j \in \mathbb{N}}$ such that \eqref{E: the double variable maximum thing} holds.  Then apply Proposition \ref{P: jensen ishii lemma} to obtain a sequence $((X_{j},Y_{j}))_{j \in \mathbb{N}}$ such that \eqref{E: doubled equation} and \eqref{E: key doubled matrix inequality for later} hold. \end{proof}
		
Finally, here is the proof of Proposition \ref{P: doubling variables}.

\begin{proof}[Proof of Proposition \ref{P: doubling variables}]  The proof boils down to an application of Propositions \ref{P: key mountain} and \ref{P: generate contact points double variable}.  In order to apply Proposition \ref{P: generate contact points double variable} directly, though, it would be necessary to know that the maximum of $u(x) - v(y) - (4 \beta)^{-1} \varphi_{\zeta}(x - y)^{4}$ occurs away from the diagonal.  Of course, this may not be true in general.  Thus, the proof begins with the (ultimately trivial) case when the maximum is on the diagonal before proceeding to the more interesting situation.

\textit{Case: maximum on the diagonal.} 

Suppose that there is a point $x_{0} \in U$ such that
	\begin{equation*}
		u(x_{0}) - v(x_{0}) = \sup \left\{ u(x) - v(y) - \frac{\varphi_{\zeta}(x-y)^{4}}{4 \beta} \, \mid \, (x,y) \in U \times U \right\}.
	\end{equation*}
This implies that 
	\begin{align*}
		u(x_{0}) - v(x_{0}) &= \max \left\{ u(x) - v(x_{0}) - \frac{\varphi_{\zeta}(x - x_{0})^{4}}{4 \beta} \, \mid \, x \in U \right\}, \\
		u(x_{0}) - v(y) &= \max \left\{ u(x_{0}) - v(y) - \frac{\varphi_{\zeta}(x_{0} - y)^{4}}{4 \beta} \, \mid \, y \in U \right\}.
	\end{align*}
At the same time, since $\varphi_{\zeta}$ sastisifes \eqref{E: nondegeneracy condition zeta},
	\begin{align*}
		\frac{\varphi_{\zeta}(\xi)^{4}}{4 \beta} &\leq \frac{\|\xi\|^{4}}{4 \beta \delta} \quad \text{with equality if} \quad \xi = 0,
	\end{align*}
Therefore,
	\begin{align*}
		u(x_{0}) - v(x_{0}) &= \sup \left\{ u(x) - v(x_{0}) - \frac{\|x - x_{0}\|^{4}}{4 \beta \delta} \, \mid \, x \in U \right\}, \\
		u(x_{0}) - v(x_{0}) &= \sup \left\{ u(x_{0}) - v(y) - \frac{\|x_{0} - y\|^{4}}{4 \beta \delta} \, \mid \, y \in U \right\},
	\end{align*}
and, thus, since $\|\cdot\|^{4}$ is smooth and vanishing to second order at zero,
	\begin{align*}
		- \mathcal{H}(x_{0},u(x_{0}),0, 0, 0) \leq 0, \quad -\mathcal{H}(y_{0},u(y_{0}),0,0,0) \geq 0.
	\end{align*}
This completes the proof with $(x_{*},y_{*}) = (x_{0},y_{0})$, $(X_{*},Y_{*}) = (0,0)$, and $q^{u}_{**}, q^{v}_{**}$ any arbitrary points in $\{\varphi = 1\}$, their values being completely irrelevant in this case.
	
\textit{Case: all maxima off the diagonal.} 

To complete the proof, it suffices to proceed under the assumption that there is an $r > 0$ such that
	\begin{equation} \label{E: trivial assumption almost done}
		\left\{(x,y) \, \mid \, u(x) - v(y) = M + \frac{\varphi_{\zeta}(x - y)^{4}}{4 \beta} \right\} \subseteq \{(x,y) \in U \times U \, \mid \, \|x - y\| > r\},
	\end{equation}
where $M$ is the maximum
	\begin{equation*}
		M = \max\left\{ u(x) - v(y) - \frac{\varphi_{\zeta}(x - y)^{4}}{4 \beta} \, \mid \, (x,y) \in \overline{U} \times \overline{U} \right\}.
	\end{equation*}

Due to \eqref{E: nondegenerate boundary assumption} and \eqref{E: trivial assumption almost done}, there are points $(x_{0},y_{0}) \in U \times U$ such that 
	\begin{equation*}
		u(x_{0}) - v(y_{0}) = M + \frac{\varphi_{\zeta}(x_{0} - y_{0})^{4}}{4 \beta}, \quad \|x_{0} - y_{0}\| > r.
	\end{equation*}
Let $\xi_{0} = x_{0} - y_{0}$.  Since $\|\xi_{0}\| > r$, the vector $p = D\varphi_{\zeta}(\xi_{0})$ is well-defined, as is the direction $e = \|p\|^{-1} p$.  

For $\nu \in (0,\zeta]$, let $\psi_{\nu} = \psi_{e,\partial \varphi_{\zeta}^{*}(p),\nu}$ be the conical test function associated with the triple $(e,\partial \varphi_{\zeta}^{*}(p),\nu)$.  By Proposition \ref{P: touching above test function part}, the following statements all hold:
	\begin{gather}
		u(x_{0}) - v(y_{0}) = M + \frac{\psi_{\nu}(x_{0} - y_{0})^{4}}{4 \beta}, \nonumber \\
		u(x) - v(y) \leq M + \frac{\psi_{\nu}(x - y)^{4}}{4 \beta} \quad \text{for each} \quad (x,y) \in \overline{U} \times \overline{U}, \nonumber \\
		\left\{(x,y) \, \mid \, u(x) - v(y) = M + \frac{\psi_{\nu}(x - y)^{4}}{4 \beta} \right\} \subseteq \{(x,y) \in U \times U \, \mid \, x - y \in \mathcal{C}(\partial \varphi^{*}_{\nu}(p)), \, \, x \neq y \}. \label{E: contact condition ack ack ack}
	\end{gather}
In particular, the hypotheses of Proposition \ref{P: generate contact points double variable} are satisfied.

Fix a $\nu \in (0,\zeta)$; for instance, $\nu = 2^{-1} \zeta$, although the exact value is immaterial to the proof.  By Proposition \ref{P: generate contact points double variable}, there are sequences 
	\begin{equation*}
		(\epsilon_{j})_{j \in \mathbb{N}} \subseteq (0,+\infty), \quad ((x_{j},y_{j}))_{j \in \mathbb{N}} \subseteq U \times U, \quad ((\mu_{j},X_{j},Y_{j}))_{j \in \mathbb{N}} \subseteq \mathbb{R}^{d} \times \mathcal{S}_{d} \times \mathcal{S}_{d} 
	\end{equation*}
such that \eqref{E: the double variable maximum thing}-\eqref{E: key doubled matrix inequality for later} all hold with this value of $\nu$.  

In what follows, it will be useful to define $p_{j} = D\psi_{\nu}^{\epsilon_{j}}(x_{j} - y_{j})$ and $A_{j} = \bar{A}_{j}(x_{j} - y_{j})$, where $\bar{A}_{j}$ is the matrix function defined in \eqref{E: matrix function stuff}.  Notice that if $(H_{j})_{j \in \mathbb{N}}$ is the sequence $H_{j} = D^{2}\psi_{\nu}^{\epsilon_{j}}(x_{j} - y_{j})$, then
	\begin{equation*}
		A_{j} = \beta^{-1} \psi_{\nu}^{\epsilon_{j}}(x_{j} - y_{j})^{3} H_{j} + 3 \beta^{-1} \psi_{\nu}^{\epsilon_{j}}(x_{j} - y_{j})^{2} p_{j} \otimes p_{j}.	
	\end{equation*}

By compactness (up to passing to a subsequence), there are points $(x_{*},y_{*}) \in \overline{U} \times \overline{U}$ such that $(x_{j},y_{j}) \to (x_{*},y_{*})$.  Note that, by upper semicontinuity and Proposition \ref{P: touching above test function part},
	\begin{equation} \label{E: touching ack ack ack ack}
		u(x_{*}) - v(y_{*}) = M + \frac{\psi_{\nu}(x_{*} - y_{*})^{4}}{4 \beta} = M + \frac{\varphi_{\zeta}(x_{*} - y_{*})^{4}}{4 \beta}
	\end{equation}
and, thus, $\|x_{*} - y_{*}\| > r$ by assumption \eqref{E: trivial assumption almost done}.  Since $x_{*} \neq y_{*}$, \eqref{E: touching ack ack ack ack} implies $\psi_{\nu}(x_{*} - y_{*}) = \varphi_{\zeta}(x_{*} - y_{*}) < +\infty$ and, thus, $p_{j} \to D\psi_{\nu}(x_{*} - y_{*}) = D\varphi_{\zeta}(x_{*} - y_{*}) = p$ as $j \to +\infty$ by \eqref{E: contact condition ack ack ack} and Proposition \ref{P: touching above test function part}. 
Additionally, by Proposition \ref{P: key mountain}, 
	\begin{align} \label{E: useful bound again and again}
		\varphi_{\zeta}(x_{*} - y_{*}) \limsup_{j \to \infty} \|H_{j}\| = \limsup_{j \to \infty} \psi_{\nu}(x_{j} - y_{j}) \|H_{j}\| \leq  \tilde{\Gamma}(\varphi) \zeta^{-1}.
	\end{align}
In particular, up to passing to yet another subsequence, there is an $H_{*} \in \mathcal{S}_{d}$ such that $H_{j} \to H_{*}$ as $j \to +\infty$.  Note that this implies $A_{j} \to A_{*}$, where 
	\begin{equation} \label{E: kind of annoying need it later}
		A_{*} = \beta^{-1} \varphi_{\zeta}(x_{*} - y_{*})^{3} H_{*} + 3 \beta^{-1} \varphi_{\zeta}(x_{*} - y_{*})^{2} p \otimes p.
	\end{equation}
Thus, by \eqref{E: useful bound again and again}, Proposition \ref{P: very easy gradient bound}, and the assumption $\zeta < 1$, upon writing $T_{*} = \beta^{-1} \varphi_{\zeta}(x_{*} - y_{*})^{2}$, one finds
	\begin{align} \label{E: i need this one too}
		0 \leq A_{*} \leq \left( \tilde{\Gamma}(\varphi) \zeta^{-1} T_{*} + 3 T_{*} \delta(\varphi)^{-2} \right) \text{Id} \leq \left( \tilde{\Gamma}(\varphi) + 3 \delta(\varphi)^{-2} \right) \zeta^{-1} T_{*} \text{Id}.
	\end{align}

Next, fix the $y$ variable: observe that, for any $j \in \mathbb{N}$,
	\begin{align*}
		u(x_{j}) - v(y_{j}) - \frac{\psi_{\nu}^{\epsilon_{j}}(x_{j} - y_{j})^{4}}{4 \beta} = \sup \left\{ u(x) - v(y_{j}) - \frac{\psi_{\nu}^{\epsilon_{j}}(x - y_{j})^{4}}{4 \beta} \, \mid \, x \in \overline{U} \right\},
	\end{align*}
hence there is a sequence $(q^{u}_{j})_{j \in \mathbb{N}} \subseteq \{\varphi_{\zeta} = 1\}$ such that
	\begin{align*}
		- \mathcal{H}(x_{j},u(x_{j}),\mu_{j},A_{j},Q_{A_{j}}(q^{u}_{j} - \zeta \|\mu_{j}\|^{-1} \mu_{j})) \leq 0, \quad q_{j}^{u} \in \partial \varphi^{*}_{\zeta}(\mu_{j}) = \partial \varphi^{*}_{\zeta}(p_{j}).
	\end{align*}
Since $\{\varphi_{\zeta} = 1\}$ is compact, there is no loss of generality assuming there is a $q^{u}_{**} \in \{\varphi_{\zeta} = 1\}$ such that $q^{u}_{j} \to q^{u}_{**}$ as $j \to +\infty$.  By Proposition \ref{P: key mountain}, it follows that there is a face $F_{*}^{u} \subseteq \partial \varphi_{\zeta}^{*}(p)$ such that
	\begin{equation} \label{E: vanishing face part a}
		q_{**}^{u} \in F_{*}^{u} \quad \text{and} \quad H_{*} q = 0 \quad \text{for each} \quad q \in F_{*}^{u}.
	\end{equation}
	
Similarly, fix the $x$ variable: observe that, for any $j \in \mathbb{N}$,
	\begin{align*}
		u(x_{j}) - v(y_{j}) - \frac{\psi_{\nu}^{\epsilon_{j}}(x_{j} - y_{j})^{4}}{4 \beta} = \sup \left\{ u(x_{j}) - v(y) - \frac{\psi_{\nu}^{\epsilon_{j}}(x_{j} - y)^{4}}{4 \beta} \, \mid \, y \in \overline{U} \right\},
	\end{align*}
hence there is a sequence $(q^{v}_{j})_{j \in \mathbb{N}} \subseteq \{\varphi_{\zeta} = 1\}$ such that
	\begin{align*}
		- \mathcal{H}(y_{j},v(y_{j}),p_{j},-A_{j},Q_{-A_{j}}(q_{j}^{v} - \zeta \|\mu_{j}\|^{-1} \mu_{j})) \geq 0, \quad q_{j}^{v} \in \partial \varphi^{*}_{\zeta}(p_{j}).
	\end{align*}
Once again, after passing to a subsequence, there is a limit $q_{**}^{v} = \lim_{j \to \infty} q_{j}^{v}$.  Further, by Proposition \ref{P: key mountain} (applied to supersolutions; cf.\ Lemma \ref{L: reduction to subsolutions}), there is a face $F_{*}^{v} \subseteq \partial \varphi^{*}_{\zeta}(p)$ such that
	\begin{equation} \label{E: vanishing face part b}
		q_{**}^{v} \in F_{*}^{v} \quad \text{and} \quad H_{*} q = 0 \quad \text{for each} \quad q \in F_{*}^{v}.
	\end{equation}

Finally, notice that \eqref{E: key doubled matrix inequality for later} implies that $- 3 A_{j} \leq X_{j} \leq Y_{j} \leq 3 A_{j}$.  Thus, the sequence $((X_{j},Y_{j}))_{j \in \mathbb{N}}$ is pre-compact, and there is no loss of generality in assuming that there are matrices $X_{*}$ and $Y_{*}$ such that $(X_{j},Y_{j}) \to (X_{*}, Y_{*})$ as $j \to +\infty$.  By \eqref{E: key doubled matrix inequality for later} and \eqref{E: i need this one too}, the inequality \eqref{E: crazy matrix business} holds with $C(\varphi) = 3(\tilde{\Gamma}(\varphi) + 3 \delta(\varphi)^{-2})$.  To prove \eqref{E: big killing}, first, observe that, on the one hand, since $q_{**}^{u}, q_{**}^{v} \in \partial \varphi_{\zeta}^{*}(p)$, the variational formula \eqref{E: subdifferential basic identity dual} implies that
	\begin{equation*}
		(p \otimes p)(q_{**}^{u} - q_{**}^{v}) = \left[ \langle p, q_{**}^{u} \rangle - \langle p, q_{**}^{v} \rangle \right] p = [\varphi_{\zeta}^{*}(p) - \varphi_{\zeta}^{*}(p)] p = 0.
	\end{equation*} 
In view of \eqref{E: kind of annoying need it later}, \eqref{E: vanishing face part a}, and \eqref{E: vanishing face part b}, this proves $A_{*}(q_{**}^{u} - q_{**}^{v}) = 0$.  On the other hand, in the limit $j \to +\infty$, \eqref{E: key doubled matrix inequality for later} implies $-3A_{*} \leq X_{*} \leq Y_{*} \leq 3A_{*}$.  Accordingly, it is an exercise in linear algebra to convince oneself that also\footnote{Cf.\ the proof of \cite[Lemma 1]{m_souganidis}.}
	\begin{equation*}
		X_{*}(q_{**}^{u} - q_{**}^{v}) = Y_{*}(q_{**}^{u} - q_{**}^{v}) = 0.
	\end{equation*}
\end{proof}

\section{Applications} \label{S: applications}

This section treats the applications described in the introduction, namely, Theorem \ref{T: comparison dirichlet} on the Dirichlet problem \eqref{E: finsler dirichlet with zeroth order}, Theorem \ref{T: uniqueness infinity eigenvalue} on the infinity eigenvalue problem \eqref{E: finsler infinity eigenvalue}, and Theorem \ref{T: tug of war convergence} on tug-of-war games.

\subsection{Existence via Perron's Method} \label{S: dirichlet existence} This section establishes an existence result for the Dirichlet problem 
	\begin{align} \label{E: Dirichlet problem redux ax}
		\left\{ \begin{array}{r l} 
			\lambda u - c(x) G_{\varphi}^{*}(Du, D^{2}u) \leq f(x) & \text{in} \, \, U, \\
			\lambda u - c(x) G_{*}^{\varphi}(Du,D^{2}u) \geq f(x) & \text{in} \, \, U, \\
			u = g & \text{on} \, \, \partial U.
		\end{array} \right.
	\end{align}
The argument exploits the fact that $\varphi$ is convex and satisfies $\varphi^{*}(D\varphi) \geq 1$ in $\mathbb{R}^{d} \setminus \{0\}$, hence it can be used to build explicit barriers when $g$ is regular enough.

	\begin{prop} \label{P: perron} For any Finsler norm $\varphi$ in $\mathbb{R}^{d}$ and any bounded open set $U \subseteq \mathbb{R}^{d}$, if $\sqrt{c}: \mathbb{R}^{d} \to [0,+\infty)$ is uniformly Lipschitz continuous, then there is a viscosity solution $u \in C(\overline{U})$ of \eqref{E: Dirichlet problem redux ax} for any $g \in C(\partial U)$. \end{prop}
		
	The proof proceeds first by constructing solutions when the boundary datum $g$ is uniformly H\"{o}lder continuous on $\partial U$, then establishes the general case by approximation.  Toward that end, it will be useful to know that there are explicit conical barriers.
	
	\begin{lemma} \label{L: barriers dirichlet} For any Finsler norm $\varphi$, any $\alpha \in (0,1)$, and any $R > 0$, there are constants $C_{+,\alpha}(R), C_{-,\alpha}(R) > 0$ such that the functions $u_{+,\alpha}, u_{-,\alpha} : \mathbb{R}^{d} \to [0, +\infty)$ given by 
		\begin{equation*}
			u_{+,\alpha}(x) = \varphi(x)^{\alpha}, \quad u_{-,\alpha}(x) = -\varphi(-x)^{\alpha},
		\end{equation*}
	satisfy the differential inequalities
		\begin{align*}
			- G_{*}^{\varphi}(Du_{+,\alpha},D^{2}u_{+,\alpha}) &\geq C_{+,\alpha}(R) \quad \text{in} \, \, B_{R}(0) \setminus \{0\}, \\
			- G^{*}_{\varphi}(Du_{-,\alpha},D^{2}u_{-,\alpha}) &\leq -C_{-,\alpha}(R) \quad \text{in} \, \, B_{R}(0) \setminus \{0\}.
		\end{align*}\end{lemma}
		
	\begin{proof} In a manner similar to Lemma \ref{L: reduction to subsolutions} above, the arguments for $u_{+,\alpha}$ and $u_{-,\alpha}$ are symmetrical.  Therefore, only those for $u_{+,\alpha}$ will be presented here.
	
	In view of the fact that $-G_{*}^{\varphi}(D\varphi,D^{2}\varphi) \geq 0$ holds by Proposition \ref{P: strongly superharmonic}, the following inequalities also hold
		\begin{align*}
			- G_{*}^{\varphi}(Du_{+,\alpha},D^{2}u_{+,\alpha}) &\geq  - \alpha (\alpha - 1) \varphi(x)^{\alpha - 2} \varphi^{*}(D\varphi)^{2}.
		\end{align*}
	At the same time, since $\varphi^{*}(D\varphi) \geq 1$ in the viscosity sense, this implies
		\begin{align*}
			- G_{*}^{\varphi}(Du_{+,\alpha},D^{2}u_{+,\alpha}) &\geq \alpha (1 - \alpha) \varphi(x)^{\alpha - 2}.
		\end{align*}
	Finally, since $\alpha < 2$, 
		\begin{align*}
			C_{+,\alpha}(R) := \min \left\{ \alpha (1 - \alpha) \varphi(x)^{\alpha - 2} \, \mid \, x \in B_{R}(0) \setminus \{0\} \right\} > 0.
		\end{align*}
	\end{proof}
	
With the barriers constructed in the previous lemma, the proof of existence for \eqref{E: Dirichlet problem redux ax} is now routine.  Here is the proof:
	
	\begin{proof}[Proof of Proposition \ref{P: perron}] As advertised, the proof begins with the case when $g$ is more reuglar, then extends  to the general case by approximation and monotonicity.
	
	\textit{Step one: uniformly H\"{o}lder continuous boundary data.} 
	
	First, assume that there is an $\alpha \in (0,1)$ such that $g$ is uniformly H\"{o}lder continuous on $\partial U$ that is, there is a $K > 0$ such that 
		\begin{equation*}
			|g(x) - g(y)| \leq K \|x - y\|^{\alpha} \quad \text{for each} \quad x, y \in \partial U.
		\end{equation*}
	Notice that, by equivalence of norns, this implies there are constants $K_{+}$ and $K_{-}$ such that, for any $x,y \in \partial U$,
		\begin{equation*}
			|g(x) - g(y)| \leq \max\{K_{+} \varphi(x - y)^{\alpha}, K_{-} \varphi(y - x)^{\alpha}\}.
		\end{equation*}
	In particular, if $x_{0} \in \partial U$ is fixed, then, for any $x \in \partial U$, 
		\begin{equation*}
			g(x_{0}) - K_{-} \varphi(x_{0} - x)^{\alpha} \leq g(x) \leq g(x_{0}) + K_{+} \varphi(x - x_{0})^{\alpha} =: u^{-}_{x_{0}}(x).
		\end{equation*}

	Let $C > 0$ be a large constant to be determined.  Define functions $u^{+}_{x_{0}}$ and $u^{-}_{x_{0}}$ by
		\begin{align*}
			u^{\pm}_{x_{0}}(y) = g(x_{0}) \pm C K_{\pm} \varphi( \pm [x - x_{0}] )^{\alpha}.
		\end{align*}
	Due to Lemma \ref{L: barriers dirichlet}, a routine argument shows that, at least if $C > 1$ is large enough, 
		\begin{align*}
			\lambda u^{+}_{x_{0}} - G_{*}^{\varphi}(Du^{+}_{x_{0}},D^{2}u^{+}_{x_{0}}) &\geq f(x) \quad \text{in} \, \, U, \\
			\lambda u^{-}_{x_{0}} - G_{\varphi}^{*}(Du^{-}_{x_{0}},D^{2}u^{-}_{x_{0}}) &\leq f(x) \quad \text{in} \, \, U.
		\end{align*}		
	Note that $u^{+}_{x_{0}} \geq g \geq u^{-}_{x_{0}}$ also holds.  Fix such a $C$ henceforth.
	
	It only remains to apply Perron's Method.  Let $\mathcal{S}$ be the set of all subsolutions, that is,
		\begin{align*}
			\mathcal{S} = \left\{v \in USC(\overline{U}) \, \mid \, \lambda v - G_{\varphi}^{*}(Dv,D^{2}v) \leq f(x) \, \, \text{in} \, \, U, \, \, v \leq g \, \, \text{on} \, \, \partial U \right\}.
		\end{align*}
	We showed that $\mathcal{S}$ is nonempty.  Further, by the comparison principle (Theorem \ref{T: penalized problem comparison}), for any $v \in \mathcal{S}$ and any $x_{0} \in \partial U$,
		\begin{align*}
			v \leq u_{x_{0}}^{+} \quad \text{pointwise in} \, \, \overline{U}.
		\end{align*}
	This shows $\mathcal{S}$ is pointwise bounded.  In particular, let $u : \overline{U} \to \mathbb{R}$ be the function
		\begin{equation*}
			u(x) = \max \left\{ v(x) \, \mid \, v \in \mathcal{S} \right\}.
		\end{equation*}
	
A standard computation shows that the upper semicontinuous envelope of $u$ belongs to $\mathcal{S}$, from which it follows that $u \in USC(\overline{U})$ and $u \in \mathcal{S}$.  Similarly, a standard construction shows that $u$ is a supersolution in $U$, that is,
		\begin{equation*}
			\lambda u - c(x) G_{*}^{\varphi}(Du,D^{2}u) \geq f(x) \quad \text{in} \, \, U.
		\end{equation*}
	See \cite[Section 4]{user} for detailed proofs of both of these claims.
	
At the same time, for any $x_{0} \in \partial U$, $u_{x_{0}}^{-} \leq u \leq u_{x_{0}}^{+}$ holds pointwise in $\overline{U}$, and, therefore, $u = g$ on $\partial U$. 
	
	\textit{Step two: the general case.} 
	
	It remains to prove existence even when $g$ need not be uniformly H\"{o}lder continuous.  To that end, observe that there is a sequence of functions $(g_{n})_{n \in \mathbb{N}}$ with constants $(K_{n})_{n \in \mathbb{N}} \subseteq (0,+\infty)$ such that, for any $n$,
		\begin{equation*}
			g \leq g_{n + 1} \leq g_{n} \quad \text{pointwise in} \, \, \partial U,
		\end{equation*}
	and
		\begin{gather*}
			\lim_{n \to \infty} \|g_{n} - g\|_{C(\partial U)} = 0, \\
			|g_{n}(x) - g_{n}(y)| \leq K_{n} \|x - y\|^{\frac{1}{2}} \quad \text{for each} \quad x,y \in \partial U.
		\end{gather*}
	Further, it is possible to fix a $h \in C(\partial U)$ such that $h \leq g$ pointwise on $\partial U$ and $|h(x) - h(y)| \leq K_{h} \|x - y\|^{\frac{1}{2}}$ for any $x,y \in \partial U$.  
	
	For any $n$, let $u_{n}$ denote the solution of \eqref{E: Dirichlet problem redux ax} with $u_{n} = g_{n}$ on $\partial U$.  Similarly, let $u_{h}$ be the solution with $u_{h} = h$ on $\partial U$.  By the comparison principle (Theorem \ref{T: penalized problem comparison}), $u_{h} \leq u_{n + 1} \leq u_{n}$ for each $n \in \mathbb{N}$.  Thus, if $u : \overline{U} \to \mathbb{R}$ is the function given by $u(x) = \min\{u_{n}(x) \, \mid \, n \in \mathbb{N}\}$, then Dini's Theorem implies that $u \in C(\overline{U})$, and the choice of $(u_{n})_{n \in \mathbb{N}}$ implies that $u = g$ on $\partial U$.  Finally, by stability of viscosity solutions, $u$ is the viscosity solution of \eqref{E: Dirichlet problem redux ax}.  \end{proof}

Finally, the barriers of Lemma \ref{L: barriers dirichlet} also imply that the solution $u$ of \eqref{E: Dirichlet problem redux ax} is locally H\"{o}lder continuous.

	\begin{prop} Let $\varphi$, $U$, and $c$ be as in Proposition \ref{P: perron}.  Given any $\alpha \in (0,1)$ and any $V \subset \subset U$, there is a $C(V,\alpha) > 0$ such that if $u$ is the solution of \eqref{E: Dirichlet problem redux ax}, then
		\begin{equation*}
			[u]_{\alpha,V}  \leq C(V,\alpha) \max \left\{ |g(x)| \, \mid \, x \in \partial U \right\},
		\end{equation*}
	where $[u]_{\alpha,V}$ is the H\"{o}lder seminorm
		\begin{equation*}
			[u]_{\alpha,V} = \inf \left\{ K > 0 \, \mid \, |u(x) - u(y)| \leq K \|x - y\|^{\alpha} \, \, \text{for each} \, \, x, y \in V \right\}.
		\end{equation*}
	\end{prop}
	
		\begin{proof} This follows from Lemma \ref{L: barriers dirichlet} by a well-known argument (cf.\ \cite[Corollary 3.8]{jensen}). \end{proof}
		
The above propositions combine with Theorem \ref{T: penalized problem comparison} to prove Theorem \ref{T: comparison dirichlet}.

%%%%%%%%%%%%%%%%%%%%%%%%%%%%%%%%%%%%%%%%%%%%%%%%%%%%%%%%%%%%%%%%%%%%%%%%%%%%%%%%%%%%%%%%%%%%%%%%%%%%%%%%%%%%%%
	
%%%%%%%%%%%%%%%%%%%%%%%%%%%%%%%%%%%%%%%%%%%%%%%%%%%%%%%%%%%%%%%%%%%%%%%%%%%%%%%%%%%%%%%%%%%%%%%%%%%%%%%%%%%%%%

\subsection{Applications to Tug-of-War Games} \label{S: tug of war} This section describes applications of the comparison results above to the asymptotics of tug-of-war games involving Finsler norms.    In this setting, the Finsler infinity Laplacian is replaced by the finite-difference operator $\mathcal{M}^{\varphi}_{\epsilon}$ defined by\footnote{Note that the inversion in the minimum is not a typo.  As in the definition of $CCA_{\varphi}$ and $CCB_{\varphi}$, $\varphi$ is replaced by the ``inverted" norm $x \mapsto \varphi(-x)$ when passing from the maximum to the minimum.} 
	\begin{equation} \label{E: finite difference}
		\mathcal{M}_{\epsilon}^{\varphi}u(x) = \frac{1}{2} \sup \left\{ u(y) \, \mid \, \varphi(y - x) < \epsilon \right\} + \frac{1}{2} \inf \left\{ u(y) \, \mid \, \varphi(x - y) < \epsilon \right\}.
	\end{equation}
In order to prove results that were previously out of reach, instead of the standard tug-of-war game, which can be treated entirely using the cone comparison results from \cite{armstrong_smart_easy_proof} and \cite{armstrong_crandall_julin_smart}, this section treats a version of the game with discount rate and running cost, the treatment of which requires Theorem \ref{T: penalized problem comparison}; see Remark \ref{R: no simple tug of war} for further explanation on this point.

Analytically, this section considers the $\epsilon$-scale finite-difference version of the problem considered in the previous subsection:
	\begin{align} \label{E: tug of war}
		u^{\epsilon} - \epsilon^{-2} c(x) (u^{\epsilon} - \mathcal{M}_{\epsilon}^{\varphi}u^{\epsilon}) = f(x) \quad \text{in} \, \, U_{\epsilon}, \quad u^{\epsilon} = G \quad \text{in} \, \, \overline{U} \setminus U_{\epsilon},
	\end{align}
where $G \in C(\overline{U})$ and $U_{\epsilon}$ is the subdomain
	\begin{equation*}
		U_{\epsilon} = \{x \in U \, \mid \, \{y \in \mathbb{R}^{d} \, \mid \, \varphi(y - x) < \epsilon \, \, \text{or} \, \, \varphi(x - y) < \epsilon\} \subseteq U\}.
	\end{equation*}  
The aim is to prove Theorem \ref{T: tug of war convergence}, which asserts that this problem has solutions for $\epsilon > 0$ that converge to the solution of the corresponding PDE in the limit $\epsilon \to 0^{+}$. 

Probabilistically, the above equation characterizes the value function of a two-player stochastic game.  Specifically, fix a discount rate $\mu > 0$, a running cost $f \in C(\overline{U})$, and a terminal payout $G \in C(\overline{U})$.  The game consists of two players, Player I and Player II.  Their counter begins at some initial point $x_{0} \in U_{1}$.  If at time $n$ the current position of the counter is $x_{n} \in U_{1}$, a random (unbiased) coin is flipped.  If the outcome is heads, Player I chooses a new position $x_{n + 1}$ such that $\varphi(x_{n + 1} - x_{n}) < 1$; if it is tails, Player II chooses a position $x_{n + 1}$ such that $\varphi(x_{n} - x_{n + 1}) < 1$.\footnote{Hence Player I moves the counter ``forward," while Player II moves it ``backwards."}  If $x_{n + 1} \notin U_{1}$, the game ends and the payout is
	\begin{equation*}
		(1 + \mu) \sum_{m = 0}^{n} (1 + \mu)^{-m} f(x_{m}) + G(x_{n + 1}).
	\end{equation*}
Otherwise, if $x_{n + 1} \in U_{1}$, the game continues as before.  Player I seeks to maximize the expected payout, while Player II seeks to minimize it.

It turns out that the value functions of the two players, that is, the expected payout suitably optimized over the choice of strategies, coincide.  This function $u$ is characterized as the solution of the difference equation
	\begin{equation} \label{E: tug of war basic}
		\mu u + u - \mathcal{M}^{\varphi}_{\epsilon} u = f(x) \quad \text{in} \, \, U_{\epsilon}, \quad u = G \quad \text{in} \, \, \overline{U} \setminus U_{\epsilon}.
	\end{equation}
See the original paper \cite{peres_schramm_sheffield_wilson} or either of the introductory texts \cite{lewicka, parviainen} for derivations to this effect.  Under the diffusive rescaling $u^{\epsilon}(x) = \epsilon^{2} u(\epsilon^{-1} x)$, this equation becomes \eqref{E: tug of war}, hence the $\epsilon$-scale equation captures the behavior of the game in a large spatial domain and over long time scales.\footnote{The diffusivity $c$ appears in a continuous-time version of the game, where jumps occur at random Poisson times occurring with variable rate equal to $c(x_{t})$.}

While the probabilistic sketch above can be made rigorous to establish the existence of solutions of \eqref{E: tug of war basic}, analytical arguments very similar to those presented for the continuum limit \eqref{E: Dirichlet problem redux ax} also apply.  The next two lemmas are the main ingredients necessary to see this:

	\begin{lemma} \label{L: comparison discrete} Let $\varphi$ be a Finsler norm in $\mathbb{R}^{d}$ and fix a bounded open set $U \subseteq \mathbb{R}^{d}$.  Given any $\epsilon, \mu > 0$, if $u, v : \overline{U} \to \mathbb{R}$ are bounded functions such that 
		\begin{equation*}
			\mu u + c(x) (u - \mathcal{M}^{\varphi}_{\epsilon}u) \leq \mu v + c(x) (v -  \mathcal{M}_{\epsilon}^{\varphi} v) \quad \text{in} \, \, U_{\epsilon},
		\end{equation*}
	then 
		\begin{equation*}
			\sup \left\{ u(x) - v(x) \, \mid \, x \in \overline{U} \right\} = \sup \left\{ u(x) - v(x) \, \mid \, x \in \overline{U} \setminus U_{\epsilon} \right\}.
		\end{equation*}\end{lemma}

		\begin{proof} Let $B(\overline{U})$ denote the space of all bounded functions in $\overline{U}$ and let $T : B(\overline{U}) \to B(\overline{U})$ denote the map
			\begin{equation*}
				Tu(x) = \left\{ \begin{array}{r l}
					 \frac{c(x)}{\mu + c(x)} \mathcal{M}_{\epsilon}^{\varphi}u(x), & \text{if} \, \, x \in U_{\epsilon}, \\
					 0, & \text{otherwise.}
					\end{array} \right.
			\end{equation*}
		Notice that $T$ is monotone nondecreasing in the sense that if $u_{1} \leq u_{2}$ pointwise in $\overline{U}$, then $Tu_{1} \leq Tu_{2}$ pointwise in $\overline{U}$.  Furthermore, since $\mu > 0$ and $c$ is bounded above in $\overline{U}$, $T$ is a strict contraction with respect to the supremum norm.  It follows that $u - Tu \leq v - Tv$ holds pointwise in $\overline{U}$ if and only if $u \leq v$ pointwise in $\overline{U}$. \end{proof}
		
	\begin{lemma} \label{L: cone comparison discrete} Let $\varphi$ be a Finsler norm in $\mathbb{R}^{d}$ and fix a bounded open set $U \subseteq \mathbb{R}^{d}$.  If $v \notin U$, then, for any $\epsilon > 0$, the function $u(x) = \varphi(x - v)$ satisfies
		\begin{equation*}
			u - \mathcal{M}_{\epsilon}^{\varphi} u = 0 \quad \text{in} \, \, U_{\epsilon}.
		\end{equation*}
	\end{lemma}
	
		\begin{proof} This is one place where the need for asymmetry in the definitinon of $\mathcal{M}^{\varphi}_{\epsilon}$ becomes apparent.  Up to translation, one can assume $v = 0$.  Observe then that, for any $x \in U_{\epsilon}$,
			\begin{align*}
				\max \left\{ \varphi(y) \, \mid \, \varphi(y - x) \leq \epsilon \right\} &= \varphi(x) + \epsilon, \\
				\min \left\{ \varphi(y) \, \mid \, \varphi(x - y) \leq \epsilon \right\} &= \varphi(x) - \epsilon.
			\end{align*}
		This proves $\mathcal{M}_{\epsilon}^{\varphi} u(x) = u(x)$. \end{proof}
		
The previous two lemmas imply existence and uniqueness of the solution of \eqref{E: tug of war} by analogy with what was done in the previous subsection.  Where existence is concerned, one can once again argue by Perron's Method, as in the next lemma.

	\begin{lemma} \label{L: discrete existence} For any bounded open set $U \subseteq \mathbb{R}^{d}$, any bounded functions $c : U \to [0, +\infty)$, $f : U \to \mathbb{R}$, and $G : \overline{U} \to \mathbb{R}$, and any $\epsilon > 0$, there is a bounded function $u : \overline{U} \to \mathbb{R}$ such that \eqref{E: tug of war basic} holds.  \end{lemma}  
	
	\begin{proof} As in the proof of Lemma \ref{L: comparison discrete}, let $B(\overline{U})$ denote the space of all bounded functions in $\overline{U}$.  Let $\mathcal{S}$ denote the space of all subsolutions, meaning
	\begin{equation*}
		\mathcal{S} = \left\{u \in B(\overline{U}) \, \mid \, \mu u + \epsilon^{-2} c(x) (u - \mathcal{M}_{\epsilon}^{\varphi} u) \leq F(x) \, \, \text{in} \, \, U_{\epsilon}, \, \, u \leq G(x) \, \, \text{in} \, \, \overline{U} \setminus U_{\epsilon} \right\}.
	\end{equation*}
Note that the constant function $u(x) = \min\{\mu^{-1} \inf F, \inf G\}$ is in $\mathcal{S}$, hence $\mathcal{S}$ is nonempty.  Further, the function $v(x) = \max\{\mu^{-1} \sup F, \sup G\}$ is a supersolution, which, in view of Lemma \ref{L: comparison discrete}, implies that any element of $\mathcal{S}$ lies below it.

Defining $u^{\epsilon}(x) = \sup_{u \in \mathcal{S}} u(x)$, one readily deduces that $u^{\epsilon} \in \mathcal{S}_{\epsilon}$.  Furthermore, by maximality, $u^{\epsilon}$ is also a supersolution.  (If $u^{\epsilon}$ fails to be a supersolution at some point $x \in \overline{U}$, one directly shows that the function $u^{\epsilon} + \zeta 1_{\{x\}}$ is a subsolution provided $\zeta > 0$ is small enough, which would contradict the maximality of $u^{\epsilon}$.) \end{proof}

Finally, it remains to show that the solution $u^{\epsilon}$ converges to the solution of the PDE \eqref{E: Dirichlet problem redux ax} in the limit $\epsilon \to 0^{+}$.  Here the classical technique of Barles and Souganidis \cite{barles_souganidis} is applicable.  Toward that end, the following lemma is needed, which shows, in the terminology used in that reference, that the approximation \eqref{E: finite difference} is a consistent scheme for the PDE.  This computation is by now quite well-known.\footnote{Cf.\ \cite[Lemma 5.10]{chatterjee_souganidis}.}

	\begin{lemma} \label{L: consistency} Let $\varphi$ be a Finsler norm in $\mathbb{R}^{d}$.  Given any $C^{2}$ function $f$ defined in some open set $V \subseteq \mathbb{R}^{d}$, any $x_{0} \in V$, and any family $(x_{\epsilon})_{\epsilon > 0} \subseteq V$ such that $x_{\epsilon} \to x_{0}$ as $\epsilon \to 0^{+}$, 
		\begin{gather*}
			- G_{\varphi}^{*}(Df(x_{0}), D^{2} f(x_{0})) \leq \liminf_{\epsilon \to 0^{+}} \epsilon^{-2} \left[ f(x_{\epsilon}) - \mathcal{M}_{\epsilon}^{\varphi} f(x_{\epsilon}) \right], \\
			\limsup_{\epsilon \to 0^{+}} \epsilon^{-2} \left[ f(x_{\epsilon}) - \mathcal{M}_{\epsilon}^{\varphi} f(x_{\epsilon}) \right] \leq - G_{*}^{\varphi}(Df(x_{0}),D^{2}f(x_{0})).
		\end{gather*}
	\end{lemma}

		For the sake of completeness, and to provide further evidence that $\mathcal{M}^{\varphi}_{\epsilon}$ is defined correctly in \eqref{E: finite difference}, a proof of this fundamental lemma is provided in Appendix \ref{A: technical results}.
		
With the above lemmas in hand, the proof of convergence of $u^{\epsilon}$ to the solution of \eqref{E: Dirichlet problem redux ax} follows by the Barles-Souganidis approach, as sketched in the next proof:

	\begin{proof}[Proof of Theorem \ref{T: tug of war convergence}] Let $u_{*}$ and $u^{*}$ be the upper and lower half-relaxed limits of $(u^{\epsilon})_{\epsilon > 0}$ as defined in \cite[Section 6]{barles} or \cite[Proof of Theorem 3.1]{chatterjee_souganidis}.  Since the barrier arguments above are $\epsilon$-independent and $G$ is an extension of $g$, it follows that $u_{*} = u^{*} = G = g$ on $\partial U$.  Further, using Lemma \ref{L: consistency}, one readily deduces that 
		\begin{equation*}
			\lambda u^{*} - G_{\varphi}^{*}(Du,D^{2}u) \leq f(x) \quad \text{in} \, \, U, \quad \lambda u_{*} - G_{*}^{\varphi}(Du,D^{2}u) \geq f(x) \quad \text{in} \, \, U.
		\end{equation*}
	Therefore, by Theorem \ref{T: penalized problem comparison}, $u^{*} \leq u_{*}$ in $\overline{U}$.  Yet the definitions imply $u^{*} \geq u_{*}$.  Thus, equality holds $u^{*} \equiv u_{*}$, and $u^{\epsilon} \to u$ uniformly in $\overline{U}$.\end{proof}
		
\begin{remark} \label{R: no simple tug of war} It is also true that the solution $u^{\epsilon}$ of the simpler problem
	\begin{equation*}
		u^{\epsilon} - \mathcal{M}^{\varphi}_{\epsilon} u^{\epsilon} = 0 \quad \text{in} \, \, U_{\epsilon}, \quad u^{\epsilon} = G \quad \text{in} \, \, \overline{U} \setminus U_{\epsilon},
	\end{equation*}
converges to the $\varphi$-infinity harmonic function in $U$ with boundary value $G$.  However, this follows already from the cone comparison results of \cite{armstrong_smart_easy_proof} and \cite{armstrong_crandall_julin_smart}, the reason being that one can use Lemmas \ref{L: comparison discrete} and \ref{L: cone comparison discrete} to prove that the upper (resp.\ lower) half-relaxed limit of $(u^{\epsilon})_{\epsilon > 0}$ as $\epsilon \to 0^{+}$ is in $CCA_{\varphi}(U)$ (resp.\ $CCB_{\varphi}(U)$) and then invoke the comparison theorems in those papers to conclude that the upper and lower half-relaxed limits coincide.\footnote{The limiting behavior is less straightforward when there is a right-hand side, see \cite{armstrong_smart_transactions}.}

The focus here is on \eqref{E: tug of war} since it seems to be out of reach of previous work.  Specifically, unlike the Finsler infinity Laplace equation, the solutions of which are characterized by comparison with cones, the PDE \eqref{E: Dirichlet problem redux ax} considered here does not seem to possess a comparison-with-cones-type characterization. \end{remark}

\begin{remark}  The problem \eqref{E: finite difference} is just one of many tug-of-war games whose value functions can be shown to converge using the comparison results of this paper.  For another simple example, consider tug-of-war with noise,\footnote{The game is the same, except for a third ``drunk" player who chooses the next position $x_{n + 1}$ uniformly at random from the ball $B_{1}(x_{n})$; see \cite{lewicka,parviainen}.}  in which the value function is the solution of the following equation
	\begin{equation*}
		u^{\epsilon} + \epsilon^{-2} (u^{\epsilon} - p \mathcal{M}_{\epsilon}^{\varphi}u^{\epsilon} - (1 - p) \mathcal{S}_{\epsilon} u^{\epsilon}) = f(x).
	\end{equation*}
Above $p \in [0,1]$ and $\mathcal{S}_{\epsilon}$ is the following averaging operator:
	\begin{equation*}
		(\mathcal{S}_{\epsilon} u)(x) = \fint_{B_{\epsilon}(x)} u(y) \, dy.
	\end{equation*}
Since $\text{Id} - \mathcal{S}_{\epsilon}$ approximates $-\Delta$ as $\epsilon \to 0^{+}$, limits of the above equation are determined by the following differential inequalities:
	\begin{equation*}
		\lambda u - p \Delta u - (1 - p) G_{\varphi}^{*}(Du,D^{2}u) \leq f(x) \quad \text{and} \quad \lambda u - p \Delta u - (1 - p) G_{*}^{\varphi}(Du,D^{2}u) \geq f(x).  
	\end{equation*}\end{remark}

\appendix

\part{Appendices} 

\section{Technical Lemmata} \label{A: technical results}

\subsection{Integrals over Cones}  Recall that the perturbative argument in the proof of the cone comparison principle involved a property of integration over cones.  The lemma in question is proved here.  

The reader may wish to compare this result to its better-known analogue: if $\mu$ is a probability measure on a compact convex set $C$ in $\mathbb{R}^{m}$ and $\int_{C} x \, \mu(dx)$ equals an extreme point of $C$, then $\mu$ is a Dirac mass.  The lemma that follows more-or-less shows (albeit in the setting of cones) that this observation can be generalized to the case where ``extreme point" is replaced by ``face."

\begin{proof}[Proof of Lemma \ref{L: cone integral}]  Let $\text{sppt}(\mu)$ denote the support of $\mu$, that is, $v \in \text{sppt}(\mu)$ if and only if
		\begin{equation*}
			\mu(B_{\epsilon}(v)) > 0 \quad \text{for each} \quad \epsilon > 0.
		\end{equation*}
	Recall that $\mu(\mathcal{N} \setminus \text{sppt}(\mu)) = 0$.  Let $\mathcal{N}_{0}$ denote the convex cone generated by $\text{sppt}(\mu)$, that is,
		\begin{equation*}
			\mathcal{N}_{0} = \left\{ \sum_{i = 1}^{N} \alpha_{i} v_{i} \, \mid \, N \in \mathbb{N}, \, \,  v_{1},\dots,v_{N} \in \text{sppt}(\mu), \, \, \alpha_{1},\dots,\alpha_{N} \geq 0\right\}.
		\end{equation*}
	The closure $\overline{\mathcal{N}_{0}}$ is also a convex cone, and $\overline{\mathcal{N}_{0}} \subseteq \mathcal{N}$ since $\mu$ is a measure on $\mathcal{N}$.  To prove that $\mu(\mathcal{N} \setminus \mathcal{N}') = 0$, it suffices to show that $\overline{\mathcal{N}_{0}} \subseteq \mathcal{N}'$.  Toward that end, since $\mathcal{N}'$ is a face of $\mathcal{N}$, it enough to show that the mean $m_{\mu} := \int_{\mathcal{N}} v \, \mu(dv)$ is an element of the relative interior of $\overline{\mathcal{N}_{0}}$.
	
	 It is clear that $m_{\mu} \in \overline{\mathcal{N}_{0}}$ since $\mu$ is a positive measure (e.g., by simple approximation).  Let us argue by contradiction: suppose that $m_{\mu} \in \text{bdry}(\overline{\mathcal{N}_{0}})$.  It follows that there is a vector $w \in \mathbb{R}^{m}$ such that 
	 	\begin{align*}
			\langle w, m_{\mu} \rangle &= \max \left\{ \langle w, v \rangle \, \mid \, v \in \overline{\mathcal{N}_{0}} \right\}.
		\end{align*}
Note that if $v \in \text{rint}(\overline{\mathcal{N}_{0}})$, then $\langle w, v \rangle < \langle w, m_{\mu} \rangle$ necessarily holds.  At the same time, since $\overline{\mathcal{N}_{0}}$ is a cone, the maximum can only be zero: $\langle w, m_{\mu} \rangle = 0$.
	
	Fix a $v_{0} \in \text{rint}(\overline{\mathcal{N}_{0}})$.  Since $\langle w, v_{0} \rangle < 0$, the definition of $\mathcal{N}_{0}$ implies there is a $\tilde{v} \in \text{sppt}(\mu)$ such that $\langle w, \tilde{v} \rangle < 0$.  In particular, there is an $\epsilon > 0$ such that
		\begin{equation*}
			\langle w, v \rangle < 0 \quad \text{for each} \quad v \in B_{\epsilon}(\tilde{v}).
		\end{equation*}
	Thus, $\mu(\{\langle w, \cdot \rangle < 0\}) \geq \mu(B_{\epsilon}(\tilde{v})) > 0$.
	On the other hand, since $\langle w, v \rangle \leq 0$ for all $v \in \overline{\mathcal{N}_{0}}$, this leads to the observation that
		\begin{align*}
			0 = \langle w, m_{\mu} \rangle = \int_{\text{sppt}(\mu)} \langle w, v \rangle \, \mu(dv) = \int_{\{\langle w, v \rangle < 0\}} \langle w, v \rangle \, \mu(dv) < 0,
		\end{align*} 
	a patent contradiction. \end{proof}
	
\subsection{Hessian Bounds for Conical Test Functions} This appendix proves the Hessian bound that was used in the analysis of conical test functions in Section \ref{S: c11 setting}.

\begin{lemma} \label{L: key hessian control} Let $\varphi$ satisfying Assumption \ref{A: c11 assumption}.  Fix a $\nu \in (0,\zeta]$, $p \in \mathbb{R}^{d} \setminus \{0\}$, and let $e = \|p\|^{-1} p$.  Let $\psi_{\nu} = \psi_{e,\partial \varphi^{*}(p),\nu}$ be the associated conical test function, and let $(\psi_{\nu}^{\epsilon})_{0 < \nu < \epsilon}$ be an admissible perturbation.  

If there are sequences $(\epsilon_{j})_{j \in \mathbb{N}} \subseteq (0,\nu)$ and $(q_{j})_{j \in \mathbb{N}} \subseteq \mathbb{R}^{d}$ and a point $q_{*} \in \partial \varphi^{*}(p)$ such that $\psi_{\nu}^{\epsilon_{j}}(q_{j}) = 1$ for each $j \in \mathbb{N}$ and 
	\begin{equation*}
		\lim_{j \to \infty} \epsilon_{j} = 0 \quad \text{and} \quad \lim_{j \to \infty} q_{j} = q_{*},
	\end{equation*}
then
	\begin{equation*}
		\limsup_{j \to \infty} \|D^{2}\psi_{\nu}^{\epsilon_{j}}(q_{*})\| \leq \Gamma(\delta) \nu^{-1},
	\end{equation*}
where $\Gamma(\delta)$ is the constant from Proposition \ref{P: trivial hessian estimate}. \end{lemma}

\begin{proof} This is a relatively routine application of the results from Section \ref{S: pos hom}, most notably Proposition \ref{P: trivial hessian estimate}.  To begin, fix an arbitrary $\delta_{*} \in (0,\delta)$.   

First, note that since $q_{*} \in \partial \varphi^{*}(p) \subseteq \{\varphi \leq 1\}$, the nondegeneracy condition \eqref{E: nondegeneracy} implies that $\|q_{*}\| \leq \delta^{-1}$.  Since $\delta_{*} < \delta$, one can fix an $R > 0$ such that
	\begin{equation*}
		\overline{B}_{R}(q_{*}) \subseteq \overline{B}_{\delta_{*}^{-1}}(0).
	\end{equation*}

Next, note that, at the limit point $q_{*}$, one has that $D\psi_{\nu}(q_{*}) = D\varphi(q_{*})$ by Proposition \ref{P: touching above test function part} and, thus, by the nondegeneracy assumption \eqref{E: nondegeneracy} and Proposition \ref{P: very easy gradient bound}, 
	\begin{equation*}
		\|D\psi_{\nu}(q_{*})\| = \|D\varphi(q_{*})\| \leq \delta^{-1}.
	\end{equation*}
Since $D\psi_{\nu}^{\epsilon_{j}} \to D\psi_{\nu}$ locally uniformly as $j \to +\infty$, up to making $R$ smaller if necessary, one can fix a $J \in \mathbb{N}$ such that, for any $j \geq J$, 
	\begin{equation*}
		\max \left\{ \|D\psi_{\nu}^{\epsilon_{j}}(q)\| \, \mid \, q \in \overline{B}_{R}(q_{*}) \right\} \leq \delta_{*}^{-1}.
	\end{equation*}
Thus, the functions $(\psi_{\nu}^{\epsilon_{j}})_{j \in \mathbb{N}}$ satisfy a gradient bound, as required in Proposition \ref{P: trivial hessian estimate}. 

It remains to check that the level set $\{\psi_{\nu}^{\epsilon_{j}} = 1\}$ satisfies a curvature bound near $q_{*}$.  Toward that end, it will be convenient to fix a $\rho \in (0,1)$.  First, recall that the proof of Lemma \ref{L: intermediate key lemma} establishes that
	\begin{equation*}
		0 \leq \text{dist}(x',G_{p}) D^{2} \text{dist}(x',G_{p}) + D\text{dist}(x',G_{p}) \otimes D\text{dist}(x',G_{p}) \leq \text{Id}.
	\end{equation*}
Combining this with the explicit formula \eqref{E: second derivative key mountain}, one deduces that, for every $q \in \mathbb{R}^{d-1}$ such that $\text{dist}(q',G_{p}) \leq \rho \nu - \epsilon_{j}$, 
	\begin{equation*}
		\|D^{2}g_{\nu}^{\epsilon_{j}}(q')\| \leq (1 - \rho)^{-1} \nu^{-1} \max \left\{ 1, \frac{1}{1 - \rho} \right\} = (1 - \rho)^{-2} \nu^{-1}.
	\end{equation*}
At the same time, recall that the derivative of the outward normal vector $N_{j}$ to the graph of $g_{\nu}^{\epsilon_{j}}$ is given, at points $q = q' + g_{\nu}^{\epsilon_{j}}(q') e$, by the formula
	\begin{equation*}
		DN_{j}(q) = \left(\text{Id} - \frac{Dg_{\nu}^{\epsilon_{j}}(q')^{\otimes 2}}{1 + \|Dg_{\nu}^{\epsilon_{j}}(q')\|^{2}} \right) D^{2} g_{\nu}^{\epsilon_{j}}(q') \left( \text{Id} - \frac{Dg_{\nu}^{\epsilon_{j}}(q')^{\otimes 2}}{1 + \|Dg_{\nu}^{\epsilon_{j}}(q')\|^{2}} \right).
	\end{equation*}
Thus, up to making $R$ smaller and $J$ larger above, there is no loss of generality assuming that
	\begin{equation*}
		\max \left\{ \|N_{j}(q)\| \, \mid \, q \in \{\psi_{\nu}^{\epsilon_{j}} = 1\} \cap \overline{B}_{R}(q_{*}) \right\} \leq (1 - \rho)^{-2} \nu^{-1}.
	\end{equation*}
This proves that $\{\psi_{\nu}^{\epsilon_{j}}\}_{j \in \mathbb{N}}$ satisfy a uniform curvature estimate close to $q$, verifying the remaining hypothesis of Proposition \ref{P: trivial hessian estimate}.

At last, apply Proposition \ref{P: trivial hessian estimate} to conclude that, for any $j \geq J$,
	\begin{equation*}
		\max \left\{ \psi_{\nu}^{\epsilon_{j}}(q) \|D^{2}\psi_{\nu}^{\epsilon_{j}}(q)\| \, \mid \, q \in \{\psi_{\nu}^{\epsilon_{j}} = 1\} \cap \overline{B}_{R}(q_{*}) \right\} \leq \Gamma(\delta_{*}) (1 - \rho)^{-2} \nu^{-1}.
	\end{equation*}
In particular, since $q_{j} \to q_{*}$ as $j \to +\infty$ and $\psi_{\nu}^{\epsilon_{j}}(q_{j}) = 1$ for all $j$, this implies
	\begin{equation*}
		\limsup_{j \to \infty} \|Y_{j}\| = \limsup_{j \to \infty} \|D^{2}\psi_{\nu}^{\epsilon_{j}}(q_{j})\| \leq \Gamma(\delta_{*}) (1 - \rho)^{-2} \nu^{-1}.
	\end{equation*}
Finally, since $\Gamma$ is a continuous function, the desired estimate \eqref{E: hessian bounds ack ack} follows upon sending $\rho \to 0$ and $\delta_{*} \to \delta$. \end{proof}

\subsection{Finite-Difference Approximation} 

\begin{proof}[Proof of Lemma \ref{L: consistency}] Only the details for the bound involving $G_{\varphi}^{*}$ will be presented since the bound for $G_{*}^{\varphi}$ follows via a transformation as in Lemma \ref{L: reduction to subsolutions}.  

Recall that $f$ is a smooth function defined in some neighborhood of a point $x_{0} \in \mathbb{R}^{d}$ and $(x_{\epsilon})_{\epsilon > 0}$ are points such that $x_{\epsilon} \to x_{0}$ as $\epsilon \to 0^{+}$, the goal being to establish that 
	\begin{equation*}
		-G_{\varphi}^{*}(Df(x_{0}),D^{2}f(x_{0})) \leq \liminf_{\epsilon \to 0^{+}} \epsilon^{-2} \left[ f(x_{\epsilon}) - (\mathcal{M}^{\varphi}_{\epsilon}f)(x_{\epsilon}) \right].
	\end{equation*}

Given $\epsilon > 0$, define $y^{+}_{\epsilon}$ and $y_{\epsilon}^{-}$ such that $\varphi(y^{+}_{\epsilon} - x) \leq \epsilon$, $\varphi(x_{\epsilon} - y^{-}_{\epsilon}) \leq \epsilon$, and
	\begin{align*}
		f(y^{+}_{\epsilon}) = \max \left\{ f(y) \, \mid \, \varphi(y - x_{\epsilon}) \leq \epsilon \right\}, \quad f(y^{-}_{\epsilon}) = \min \left\{ f(y) \, \mid \, \varphi(x_{\epsilon} - y) \leq \epsilon \right\}.
	\end{align*}
By definition of $\mathcal{M}^{\varphi}_{\epsilon}$,
	\begin{equation} \label{E: annoying asymptotic formula}
		f(x_{\epsilon}) - (\mathcal{M}^{\varphi}_{\epsilon} f)(x_{\epsilon}) = \frac{1}{2} (f(x_{\epsilon}) - f(y^{+}_{\epsilon})) + \frac{1}{2} (f(x_{\epsilon}) - f(y^{-}_{\epsilon})).
	\end{equation}

At this stage, it is convenient to rescale: consider the function $F_{\epsilon}$ given by 
	\begin{equation*}
		F_{\epsilon}(\xi) = \epsilon^{-1} (f(x_{\epsilon} + \epsilon \xi) - f(x_{\epsilon})).
	\end{equation*}
Notice that $F_{\epsilon}$ converges locally uniformly to the linear function $\xi \mapsto \langle Df(x_{0}), \xi \rangle$ as $\epsilon \to 0^{+}$.  Thus,
	\begin{align*}
		\lim_{\epsilon \to 0^{+}} \max \left\{ F_{\epsilon}(\xi) \, \mid \, \varphi(\xi) \leq 1 \right\} &= \max \left\{ \langle Df(x_{0}), \xi \rangle \, \mid \, \varphi(\xi) \leq 1 \right\} = \varphi^{*}(Df(x_{0})) \\
		\lim_{\epsilon \to 0^{+}} \min \left\{ F_{\epsilon}(\xi) \, \mid \, \varphi(-\xi) \leq 1 \right\} &= \min \left\{ \langle Df(x_{0}), \xi \rangle \, \mid \, \varphi(-\xi) \leq 1 \right\} \\
			&= -\max \left\{ \langle Df(x_{0}), \xi \rangle \, \mid \, \varphi(\xi) \leq 1 \right\} = - \varphi^{*}(Df(x_{0})).
	\end{align*}
Note that optimizers of $F_{\epsilon}$ converge to optimizers of the limit, hence, by \eqref{E: subdifferential basic identity dual},
	\begin{equation*}
		\lim_{\epsilon \to 0^{+}} \min \left\{ \| \pm \epsilon^{-1}(y^{\pm}_{\epsilon} - x_{\epsilon}) - q \| \, \mid \, q \in \partial \varphi^{*}(Df(x_{0})) \right\} = 0. 
	\end{equation*}

After Taylor expanding the two terms in the right-hand side of \eqref{E: annoying asymptotic formula}, one concludes
	\begin{align*}
		\liminf_{\epsilon \to 0^{+}} \epsilon^{-2} \left[ f(x_{\epsilon}) - (\mathcal{M}^{\varphi}_{\epsilon} f)(x_{\epsilon}) \right]  &\geq \varphi^{*}(Df(x_{0})) - \varphi^{*}(Df(x_{0})) \\
			&\quad + 2 \min \left\{ -\frac{1}{2} \langle D^{2}f(x_{0}) q, q \rangle \, \mid \, q \in \partial \varphi^{*}(Df(x_{0})) \right\} \\
			&= - G_{\varphi}^{*}(Df(x_{0}),D^{2}f(x_{0})).
	\end{align*}	
\end{proof}

\section{Construction of the Perturbed Test Functions} \label{A: perturbed test functions}

Throughout this appendix, let $\varphi$ be a Finsler norm in $\mathbb{R}^{d}$ that satisfies the $\zeta$-interior ball condition \eqref{E: zeta interior ball} for some $\zeta > 0$.  Fix a $p \in \mathbb{R}^{d} \setminus \{0\}$ and let $e = \|p\|^{-1} p$.  As in Section \ref{S: c11 setting}, let $\psi_{\nu} = \psi_{e,\partial \varphi^{*}(p),\nu}$ be the conical test function associated with $\partial \varphi^{*}(p)$.

Towards the construction of the perturbed test functions $(\psi_{\nu}^{\epsilon})_{\epsilon \in (0,\nu)}$, let $g^{\epsilon}_{\nu}$ be the mollified function defined as in Section \ref{S: smoothing the graph}.  Define the hypograph $\mathcal{H}^{\epsilon}_{\nu}$ by
	\begin{equation} \label{E: hypograph}
		\mathcal{H}^{\epsilon}_{\nu} = \{x' + s e \, \mid \, \text{dist}(x',G_{p}) \leq \nu - \epsilon, \, \, 0 \leq s \leq g_{\nu}^{\epsilon}(x') \}.
	\end{equation}
$\mathcal{H}^{\epsilon}_{\nu}$ is convex since $g_{\nu}^{\epsilon}$ is concave and nonnegative.  Accordingly, if $\underline{\psi}^{\epsilon}_{\nu}$ is the Minkowski gauge of this set
	\begin{equation} \label{E: gauge hypograph}
		\underline{\psi}^{\epsilon}_{\nu}(x) = \inf \left\{ \alpha > 0 \, \mid \, \alpha^{-1} x \in \mathcal{H}^{\epsilon}_{\nu} \right\},
	\end{equation}
then $\underline{\psi}^{\epsilon}_{\nu}$ is a convex function and it is not hard to imagine that $\{\underline{\psi}^{\epsilon}_{\nu} = 1\}$ coincides with the graph of $g_{\nu}^{\epsilon}$ close to the flat part $\partial \varphi^{*}(p)$. 

To construct an admissible perturbation as in Definition \ref{D: admissible}, the trick is to identify a nice open cone $\mathcal{U}_{\nu}$ containing $\partial \varphi^{*}(p)$ such that the functions $(\psi_{\nu}^{\epsilon})_{\epsilon \in (0,\nu)}$ defined by 
	\begin{align} \label{E: admissible perturbation}
		\psi^{\epsilon}_{\nu}(q) = \left\{ \begin{array}{r l}
									\underline{\psi}^{\epsilon}_{\nu}(q), & \text{if} \, \, q \in \overline{\mathcal{U}_{\nu}}, \\
									+ \infty, & \text{otherwise,}
								\end{array} \right.
	\end{align}
have the desired properties.  That is the approach taken below. 

\subsection{Construction of an Admissible Perturbation} For any $r \in (0,\nu]$, define open sets $U_{p}(r) \subseteq \mathbb{R}^{d-1}$ and $\mathcal{U}_{p}(r) \subseteq \mathbb{R}^{d}$ and the hypersurface $\mathcal{S}_{p}(r) \subseteq \mathbb{R}^{d}$ by 
	\begin{gather*}
		U_{p}(r) = \{q' \in \mathbb{R}^{d-1} \, \mid \, \text{dist}(q', G_{p}) < r\}, \quad \mathcal{S}_{p}(r) = \{q' + g_{\nu}(q') e \, \mid \, q' \in U_{p}(r) \}, \\
		\mathcal{U}_{p}(r) = \{t q \, \mid \, t > 0, \, \, q \in \mathcal{S}_{p}(r) \}.
	\end{gather*}
Note that these sets decrease in $r$ with respect to set inclusion.
	
	\begin{prop} \label{P: a cone construction} There are parameters $r_{*} \in (0,2^{-1} \nu)$ and $\epsilon_{*} \in (0,\nu)$ such that, for any $\epsilon \in (0,\epsilon_{*})$,
		\begin{itemize}
			\item[(i)] $\overline{\mathcal{U}_{\nu}} \subseteq \{\underline{\psi}_{\nu}^{\epsilon} < +\infty\}$.
			\item[(ii)] If $x \in \mathcal{U}_{p}(r_{*})$, then
				\begin{equation*}
					\underline{\psi}_{\nu}^{\epsilon}(x) = 1 \quad \text{if and only if} \quad \text{dist} \left( x', G_{p} \right) < 2 r_{*} \quad \text{and} \quad \langle x, e \rangle = g^{\epsilon}_{\nu} \left(x'\right).
				\end{equation*}
		\end{itemize}
	Furthermore, $\overline{\mathcal{U}_{p}(r_{*})}$ is convex.
	\end{prop}
	
Taking the proposition for granted for now, here is the proof of Theorem \ref{T: existence admissible perturbation} on the existence of admissible perturbations:

	\begin{proof}[Proof of Theorem \ref{T: existence admissible perturbation}]  Let $\epsilon_{*}$ and $r_{*}$ be the parameters of Proposition \ref{P: a cone construction} above, and let $\mathcal{U}_{\nu} = \mathcal{U}_{p}(r_{*})$ and $r = 2r_{*}$.
	
	For $\epsilon \in (0,\epsilon_{*})$, define $\psi_{\nu}^{\epsilon}$ as in \eqref{E: admissible perturbation}.  Otherwise, if $\epsilon \geq \epsilon_{*}$, simply set $\psi_{\nu}^{\epsilon}(x) = \|x\|$ for all $x \in \mathbb{R}^{d}$.  Since $\overline{\mathcal{U}_{\nu}}$ is convex and $(\underline{\psi}^{\epsilon}_{\nu})_{\epsilon \in (0,\nu)}$ are convex functions, the functions so defined are convex as well.  Similarly, they are positively one-homogeneous.
	
	Proposition \ref{P: a cone construction} shows that condition (iii) in the definition of admissible perturbation (Definition \ref{D: admissible}) is satisfied with the parameters $\epsilon_{*}$ and $r$.
	
	Trivially, $\{\psi_{\nu}^{\epsilon} < + \infty\} = \overline{\mathcal{U}_{\nu}}$ for any $\epsilon \in (0,\epsilon_{*})$ and these functions are smooth in $\mathcal{U}_{\nu}^{\epsilon}$ since $\{\psi_{\nu}^{\epsilon} = 1\}$ is a smooth graph for any $\epsilon < \epsilon_{*}$.  
	
	To conclude the proof of the existence of an admissible perturbation, it remains to show that $\psi_{\nu} \leq \psi_{\nu}^{\epsilon}$ for any $\epsilon \in (0,\epsilon_{*})$ and that $\psi_{\nu}^{\epsilon} \to \psi_{\nu}$ in $C^{1}_{\text{loc}}(\mathcal{U}_{\nu})$ as $\epsilon \to 0^{+}$.  The first follows from the fact that $g_{\nu}^{\epsilon} \leq g_{\nu}$ pointwise for any $\epsilon$, while the second follows readily from the fact that $(g_{\nu}^{\epsilon},Dg_{\nu}^{\epsilon}) \to (g_{\nu},Dg_{\nu})$ as $\epsilon \to 0^{+}$.
	
	Finally, since $\mathcal{U}_{\nu}$ is open and $\psi_{\nu}^{\epsilon} \to \psi_{\nu}$ locally uniformly as $\epsilon \to 0^{+}$, the characterization of $\{\psi_{\nu} = 1\} \cap \mathcal{U}_{\nu}$ in the final statement of the theorem follows directly from that for $(\psi_{\nu}^{\epsilon})_{\epsilon \in (0,\epsilon_{*})}$. \end{proof}

\subsection{Proof of Proposition \ref{P: a cone construction}} To begin with, it will be useful to recall that $g_{\nu}$ and $g^{\epsilon}_{\nu}$ are defined by 
	\begin{equation} \label{E: recall g annoying definition}
		g_{\nu}(x') = \varphi^{*}(e) - \nu + \sqrt{\nu^{2} - \text{dist}(y',G_{p})^{2}}, \quad g_{\nu}^{\epsilon}(x') = \int_{B_{1}} g_{\nu}(x' + \epsilon y') \rho(y') \, dy',
	\end{equation}
where above $\rho$ is a smooth nonnegative function supported in $B_{1}$ with integral equal to one.  Inspection of the proof of Lemma \ref{L: intermediate key lemma} shows that the following result holds:

	\begin{prop} \label{P: gradient vanishing ack ack} $g_{\nu}$ is $C^{1}$ in the open set $\{x' \in \mathbb{R}^{d-1} \, \mid \, \text{dist}(x',G_{p}) < \nu\}$.  In particular, $Dg_{\nu}(x') = 0$ if $x' \in G_{p}$.  \end{prop}
	
Recall that if $x' \in G_{p}$, then $x' + g_{\nu}(x')e \in \partial \varphi^{*}(p)$ and 
	\begin{equation*}
		\left \langle x' + g_{\nu}(x')e, e \right \rangle = \varphi^{*}(e).
	\end{equation*}
Thus, by continuity and Proposition \ref{P: gradient vanishing ack ack}, there are constants $r_{*}^{(1)}, \epsilon_{*}^{(1)} > 0$ such that, for each $\epsilon \in (0,\epsilon_{*}^{(1)})$ and each $x' \in U_{p}(r_{*}^{(1)})$,
	\begin{equation} \label{E: key positivity observation}
		\min \left\{ g_{\nu}^{\epsilon}(x'), \left \langle x' + g^{\epsilon}_{\nu}(x')e, \frac{e - Dg_{\nu}^{\epsilon}(x')}{\sqrt{1 + \|Dg_{\nu}^{\epsilon}(x')\|^{2}}} \right \rangle \right\} \geq \frac{1}{2} \varphi^{*}(e) > 0.
	\end{equation}
	
In the proof of the proposition, it will be useful to study the properties of the maps $\Phi$ and $(\Phi^{\epsilon})_{\epsilon \in (0,\nu)}$ defined as follows:
	\begin{gather*}
		\Phi : (0,+\infty) \times U_{p}(\nu) \to \mathbb{R}^{d}, \quad \Phi(t,x') = t[x' + g_{\nu}(x') e], \\
		\Phi^{\epsilon} : (0,+\infty) \times U_{p}(\nu - \epsilon) \to \mathbb{R}^{d}, \quad \Phi^{\epsilon}(t,x') = t[x' + g_{\nu}^{\epsilon}(x')e].
	\end{gather*}
The next result shows that these maps are homeomorphisms onto their respective ranges.

	\begin{prop} \label{P: homeo part} If $r \in (0,r_{*}^{(1)})$ and $\epsilon \in (0,\epsilon_{*}^{(1)})$, then the maps $\Phi$ and $\Phi^{\epsilon}$ restrict to homeomorphisms of $(0,+\infty) \times U_{p}(r)$. \end{prop}
	
		\begin{proof} For convenience, set $\Phi^{0} = \Phi$ and $g_{\nu}^{0} = g_{\nu}$ so there is no need to treat $\Phi$ and $\Phi^{\epsilon}$ separately.  Suppose that $\epsilon \in [0,\epsilon_{*}^{(1)})$ and $r \in (0,r_{*}^{(1)})$.
		
		It is clear that $\Phi^{\epsilon}$ is continuous.  Further, the assumption \eqref{E: key positivity observation} implies that, for each $(t,x') \in U_{p}(r)$, the vector $x' + g_{\nu}^{\epsilon}(x') e$ crosses the surface $\mathcal{S}_{p}(r)$ transversely.  It follows that $\partial_{t} \Phi^{\epsilon}(t,x')$ is not in the span of $\{\partial_{x_{1}'}\Phi^{\epsilon}(t,x'),\dots,\partial_{x_{d-1}'} \Phi^{\epsilon}(t,x')\}$.  Since these last vectors are readily shown to be linearly independent, this proves $D\Phi^{\epsilon}_{\nu}(t,x')$ has full rank, hence $\Phi^{\epsilon}$ maps open sets to open sets by the inverse function theorem.  To prove that it is a homeomorphism, it thus only remains to show that $\Phi^{\epsilon}$ restricts to an injective function in $(0,+\infty) \times U_{p}(r)$. 
		
		To see this, suppose that $\Phi(t_{1},x_{1}') = \Phi(t_{2},x_{2}')$ for some pair of points $(t_{1},x_{1}), (t_{2},x_{2}) \in (0,+\infty) \times U_{p}(r)$.  By definition, this implies
			\begin{equation} \label{E: injectivity argument}
				t_{1} [x'_{1} + g^{\epsilon}_{\nu}(x_{1}')e] = t_{2} [x_{2}' + g^{\epsilon}_{\nu}(x_{2}')e].
			\end{equation}
		For convenience, denote the outward normal vectors to the graph of $g^{\epsilon}_{\nu}$ at $x_{1}'$ and $x_{2}'$ respectively by $n_{1}$ and $n_{2}$, i.e.,
			\begin{equation*}
				n_{i} = \frac{e - Dg_{\nu}^{\epsilon}(x_{i}')}{\sqrt{1 + \|Dg_{\nu}^{\epsilon}(x_{i}')\|^{2}}} \quad \text{for} \quad i \in \{1,2\}.
			\end{equation*}
		At the level of the normal vectors, \eqref{E: injectivity argument} leads to	
			\begin{equation} \label{E: convexity guy}
				t_{1} \langle x_{1}' + g^{\epsilon}_{\nu}(x_{1}')e, n_{i} \rangle = t_{2} \langle x_{2}' + g^{\epsilon}_{\nu}(x_{2}')e, n_{i} \rangle \quad \text{for each} \quad i \in \{1,2\}.
			\end{equation}
		At the same time, since $g_{\nu}^{\epsilon}$ is concave, its hypograph is convex.  Thus, by convexity, for each $i, j \in \{1,2\}$, 
			\begin{equation*}
				\langle x_{i}' + g^{\epsilon}_{\nu}(x_{i}')e, n_{i} \rangle \geq \langle x_{j}' + g^{\epsilon}_{\nu}(x_{j}')e, n_{i} \rangle.
			\end{equation*}
		Combining this with \eqref{E: key positivity observation} and \eqref{E: convexity guy}, one deduces that $t_{j} - t_{i}\geq 0$ for any $i, j$, or, in other words, $t_{1} = t_{2}$.  
		
		Since $t_{1} = t_{2}$, the identity \eqref{E: injectivity argument} becomes 
			\begin{equation*}
				x_{1}' + g^{\epsilon}_{\nu}(x_{1}')e = x_{2}' + g^{\epsilon}_{\nu}(x_{2}') e.
			\end{equation*}
		Projecting to $\{e\}^{\perp}$, this yields $x_{1}' = x_{2}'$.  Therefore, $(t_{1},x_{1}') = (t_{2},x_{2}')$, proving $\Phi^{\epsilon}$ restricts to an injective map on $(0,+\infty) \times U_{p}(r)$. \end{proof}
		
The fact that $\mathcal{U}_{p}(r)$ is an open set follows immediately from the previous result.

	\begin{prop} If $r \in (0,r_{*}^{(1)})$, then $\mathcal{U}_{p}(r)$ is an open set. \end{prop}
	
		\begin{proof} By the previous result, the map $\Phi$ maps $(0,+\infty) \times U_{p}(r)$ homeomorphically onto $\Phi((0,+\infty) \times U_{p}(r))$.  At the same time, by definition, $\mathcal{U}_{p}(r) = \Phi((0,+\infty) \times U_{p}(r))$.  Therefore, $\mathcal{U}_{p}(r)$ is open. \end{proof}
		
The next result is the key ingredient in the proof of Proposition \ref{P: a cone construction}.

	\begin{prop}  \label{P: crazy monotonicity argument} There are parameters $r_{*}^{(2)} \in (0,2^{-1} r_{*}^{(1)})$, $\epsilon_{*}^{(2)} \in (0,\epsilon_{*}^{(1)})$, and $a_{*} \in (0,1)$ such that, given any $q \in \mathcal{S}_{p}(r_{*}^{(2)})$ and any $\epsilon \in (0,\epsilon_{*}^{(2)})$, if $f^{\epsilon}_{q} : [1 - a_{*}, 1 + a_{*}] \to \mathbb{R}$ is the function defined by
		\begin{equation*}
			f^{\epsilon}_{q}(\alpha) = \frac{\langle q, e \rangle}{\alpha} - g^{\epsilon}_{\nu}\left( \frac{q'}{\alpha} \right),
		\end{equation*}
	then $f^{\epsilon}_{q}$ is strictly decreasing in $[1 - a_{*}, 1 + a_{*}]$ and there is a unique $\alpha_{*} \in [1 - a_{*}, 1 + a_{*}]$ such that $f_{q}^{\epsilon}(\alpha_{*}) = 0$.  \end{prop}

		\begin{proof} Define the constant $D \geq 0$ by 
			\begin{equation*}
				D = \sup \left\{ \|q'\| \, \mid \, q' \in U_{p}(\nu) \right\}.
			\end{equation*}
		By Proposition \ref{P: gradient vanishing ack ack}, there is an $\epsilon_{A} \in (0,\epsilon_{*}^{(1)})$ and an $r_{*}^{(2)} \in (0,2^{-1} r_{*}^{(1)})$ such that
			\begin{equation} \label{E: gradient monotonicity annoying part}
				\|Dg_{\nu}^{\epsilon}(x')\| \leq \frac{\varphi^{*}(e)}{4 (D + 1)} \quad \text{if} \quad x' \in U_{p}(2r_{*}^{(2)}) \, \, \text{and} \, \, \epsilon \in (0,\epsilon_{A}).
			\end{equation}
		Further, since $\langle q, e \rangle = \varphi^{*}(e)$ for each $q \in \mathcal{S}_{p}(0) = \partial \varphi^{*}(p)$, up to shrinking $r_{*}^{(2)}$ further, there is no loss of generality in assuming that
			\begin{equation} \label{E: level monotonicity annoying part}
				\langle q, e \rangle \geq \frac{1}{2} \varphi^{*}(e) \quad \text{if} \quad q \in \mathcal{S}_{p}(r_{*}^{(2)}).
			\end{equation}
			
		Next, let $a_{*} = (D + 2 r_{*}^{(2)})^{-1} r_{*}^{(2)}$.  Observe that if $x' \in U_{p}(r_{*}^{(2)})$ and $\alpha \in [1 - a_{*}, 1 + a_{*}]$, then 
			\begin{equation*}
				\text{dist}(\alpha^{-1} x', G_{p}) \leq \text{dist}(x',G_{p}) + |1 - \alpha^{-1}| D \leq 2 r_{*}^{(2)} < r_{*}^{(1)}.
			\end{equation*}
		
		Finally, fix $q \in \mathcal{S}_{p}(r_{*}^{(2)})$.  By definition of $\mathcal{S}_{p}(r_{*}^{(2)})$, it is possible to write $q = q' + g_{\nu}(q') e$ for some $q' \in U_{p}(r_{*}^{(2)})$.  Thus, given any $\alpha \in [1 - a_{*}, 1 + a_{*}]$, the previous paragraph shows that the point $\alpha^{-1} q'$ is in $U_{p}(2r_{*}^{(2)}) \subseteq U_{p}(\nu)$.  In particular, the number $g_{\nu}^{\epsilon}(\alpha^{-1}q')$ is well-defined for any such $\alpha$.   Differentiating with respect to $\alpha$ and invoking \eqref{E: gradient monotonicity annoying part} and \eqref{E: level monotonicity annoying part}, one finds
			\begin{align*}
				\frac{d}{d\alpha} \left\{ f_{q}^{\epsilon}(\alpha)\right\} &= \left( - \frac{1}{\alpha^{2}} \right) \cdot \left[ \langle q, e \rangle - \left \langle Dg_{\nu}^{\epsilon} \left( \frac{q'}{\alpha} \right), q' \right \rangle \right] \\
					&\leq \left(-\frac{1}{\alpha^{2}}\right) \cdot \left[ \frac{1}{2} \varphi^{*}(e) - \left \|Dg_{\nu}^{\epsilon}\left(\frac{q'}{\alpha}\right) \right\| \cdot D \right] \\
					&\leq \left( -\frac{1}{\alpha^{2}} \right) \cdot \left( \frac{1}{4} \varphi^{*}(e) \right) < 0.
			\end{align*}
		This proves $f_{q}$ is strictly decreasing in the interval $[1 - a_{*},1 + a_{*}]$.
		
		Finally, observe that the $C^{1}$ property of $g_{\nu}$ implies there is a constant $C > 0$ depending on $r_{*}^{(2)}$ and $\nu$ such that if $q \in \mathcal{S}_{p}(r_{*}^{(2)})$ and $\epsilon \in (0,\epsilon_{*}^{(2)})$, then 
			\begin{equation*}
				f_{q}^{\epsilon}(1) = \langle q, e \rangle - g_{\nu}(q') + [g_{\nu}(q') - g_{\nu}^{\epsilon}(q')] = g_{\nu}(q') - g_{\nu}^{\epsilon}(q') \leq C \epsilon.
			\end{equation*}
		Thus, if $\epsilon_{B}, \epsilon_{*}^{(2)} \in (0,\nu)$ are given by
			\begin{equation*}
				\epsilon_{B} = \frac{\varphi^{*}(e) a_{*}}{4 (1 + a_{*})^{2} C}, \quad \epsilon_{*}^{(2)} = \min\{\epsilon_{A}, \epsilon_{B}\},
			\end{equation*}
		then one finds, for each $q \in \mathcal{S}_{p}(r_{*}^{(2)})$ and each $\epsilon \in (0,\epsilon_{*}^{(2)})$,
			\begin{equation*}
				f_{q}^{\epsilon}(1 + a_{*}) = f_{q}^{\epsilon}(1) + \int_{1}^{1 + a_{*}} \frac{d}{d\alpha}\{f_{q}^{\epsilon}\}(\alpha) \, d\alpha \leq C\epsilon - \frac{\varphi^{*}(e) a_{*}}{4(1 + a_{*})^{2}} < 0.
			\end{equation*}
		At the same time, by \eqref{E: useful concavity stuff},
			\begin{equation*}
				f_{q}^{\epsilon}(1) = g_{\nu}(q') - g_{\nu}^{\epsilon}(q') \geq 0.
			\end{equation*}
		Therefore, the intermediate value theorem implies there is an $\alpha_{*} \in [1,1 + a_{*}]$ such that $f_{q}^{\epsilon}(\alpha_{*}) = 0$.  Since $f_{q}^{\epsilon}$ is strictly decreasing, $\alpha_{*}$ is unique. \end{proof}
		
	The stage is set for the proof of Proposition \ref{P: a cone construction}.  The next lemma, which will be used in the proof, follows readily from what has been proved so far.
		
		\begin{lemma} \label{L: uniqueness time thing} Let $r_{*}^{(2)}$ and $\epsilon_{*}^{(2)}$ be the constants from Proposition \ref{P: crazy monotonicity argument}.  If $x \in \mathcal{U}_{p}(r_{*}^{(2)})$ and $\epsilon \in (0,\epsilon_{*}^{(2)})$, then there is a unique $t > 0$ such that 
			\begin{equation*}
				\text{dist}(t^{-1} x, G_{p}) < 2 r_{*}^{(2)} \quad \text{and} \quad \frac{\langle x,e \rangle}{t} = g_{\nu}^{\epsilon} \left( \frac{x'}{t} \right).
			\end{equation*}\end{lemma}
			
				\begin{proof} If $x \in \mathcal{U}_{p}(r_{*}^{(2)})$, then there is a $q \in \mathcal{S}_{p}(r_{*}^{(2)})$ and an $a > 0$ such that $x = a q$.  Therefore, it suffices to prove the result when $q \in \mathcal{S}_{p}(r_{*}^{(2)})$. 
				
				Suppose that $q \in \mathcal{S}_{p}(r_{*}^{(2)})$ and $\epsilon \in (0,\epsilon_{*}^{(2)})$.  By Proposition \ref{P: crazy monotonicity argument}, there is an $\alpha > 0$ such that 
					\begin{equation*}
						\text{dist}(\alpha^{-1} q', G_{p}) < 2r_{*}^{(2)} \quad \text{and} \quad \frac{\langle q, e \rangle}{\alpha} = g_{\nu}^{\epsilon} \left( \frac{q'}{\alpha} \right).
					\end{equation*}
				It remains to show that $\alpha$ is unique. 
				
			To establish uniqueness, suppose that $\beta > 0$ is another candidate:
				\begin{equation*}
					\text{dist}(\beta^{-1} q', G_{p}) < 2r_{*}^{(2)} \quad \text{and} \quad \frac{\langle q, e \rangle}{\beta} = g_{\nu}^{\epsilon} \left( \frac{q'}{\beta} \right).
				\end{equation*}
			Define $(t_{1},x_{1}'), (t_{2},x_{2}') \in (0,+\infty) \times U_{p}(2r_{*}^{(2)})$ by
				\begin{equation*}
					x_{1}' = \alpha^{-1} q', \quad x_{2}' = \beta^{-1} q', \quad t_{1} = \alpha, \quad t_{2} = \beta.
				\end{equation*}
			By construction, 
				\begin{equation*}
					q = \Phi^{\epsilon}(t_{1},x_{1}') = \Phi^{\epsilon}(t_{2},x_{2}').
				\end{equation*}
			Therefore, by Proposition \ref{P: homeo part}, $t_{1} = t_{2}$ and $x_{1}' = x_{2}'$.  In particular, $\alpha = \beta$. \end{proof}
		
It only remains to prove Proposition \ref{P: a cone construction}.  Toward that end, recall that, for any $\epsilon \in (0,\nu)$, $\underline{\psi}_{\nu}^{\epsilon}$ is the Minkowski gauge of the hypograph $\mathcal{H}^{\epsilon}_{\nu}$; see \eqref{E: hypograph} and \eqref{E: gauge hypograph}.  In the proof, it remains to show that $\mathcal{U}_{p}(r)$ is convex for suitable choices of $r$, as asserted in the next proposition.

	\begin{prop} \label{P: dreaded convexity} If $r \in (0,r_{*}^{(1)})$, then $\overline{\mathcal{U}_{p}(r)}$ is a closed convex cone. \end{prop}
	
The proof of convexity is tedious, hence it is deferred until after the proof of Proposition \ref{P: a cone construction}.

\begin{proof}[Proof of Proposition \ref{P: a cone construction}]  Let $r_{*} = r_{*}^{(2)}$ and $\epsilon_{*} = \epsilon_{*}^{(2)}$, where $r_{*}^{(2)}$ and $\epsilon_{*}^{(2)}$ are the parameters identified in Proposition \ref{P: crazy monotonicity argument}.

	Fix $\epsilon \in (0,\epsilon_{*})$.  Suppose that $x \in \mathcal{U}_{p}(r_{*})$ and $\alpha > 0$.
	
	Since $x \in \mathcal{U}_{p}(r_{*})$, Lemma \ref{L: uniqueness time thing} implies there is a unique $\beta > 0$ such that 
		\begin{equation*}
			\text{dist}(\beta^{-1} x', G_{p}) < 2 r_{*} \quad \text{and} \quad \frac{\langle x,e \rangle}{\beta} = g_{\nu}^{\epsilon} \left( \frac{x'}{\beta} \right).
		\end{equation*}
	On the one hand, this says that $\beta^{-1}x \in \partial \mathcal{H}_{\nu}^{\epsilon}$, hence, as a consequence of the definition \eqref{E: gauge hypograph} of $\underline{\psi}_{\nu}^{\epsilon}$, one deduces that $\underline{\psi}_{\nu}^{\epsilon}(x) \leq \beta$.  On the other hand, the definition \eqref{E: gauge hypograph} implies that $\underline{\psi}_{\nu}^{\epsilon}(x)^{-1} x \in \mathcal{H}^{\epsilon}_{\nu}$.  Consider the normal vector $n$ to $\mathcal{H}_{\nu}^{\epsilon}$ at $\beta^{-1} x$.  By the convexity of $\mathcal{H}^{\epsilon}_{\nu}$,
		\begin{equation*}
			\langle \beta^{-1} x, n \rangle = \max \left\{ \langle q, n \rangle \, \mid \, q \in \mathcal{H}_{\nu}^{\epsilon} \right\} \geq \langle \underline{\psi}_{\nu}^{\epsilon}(x)^{-1} x, n \rangle.
		\end{equation*}
	At the same time, by \eqref{E: key positivity observation}, $\langle \beta^{-1}x, n \rangle > 0$ holds.  Hence the previous inequality directly implies that $\beta \leq \underline{\psi}_{\nu}^{\epsilon}(x)$.  This proves $\beta = \underline{\psi}_{\nu}^{\epsilon}(x)$.  In particular, $\underline{\psi}_{\nu}(x) = 1$ if and only if $\beta = 1$, completing the proof. \end{proof}

Finally, here is the proof that $\mathcal{U}_{p}(r)$ is convex.  While convexity (for small $r$) is very intuitive, the proof is admittedly somewhat cumbersome.
	
		\begin{proof}[Proof of Proposition \ref{P: dreaded convexity}] The definition immediately implies that $\overline{\mathcal{U}_{p}(r)}$ is a cone, hence it only remains to show it is convex.  Toward that end, since $\mathcal{U}_{p}(r)$ is open, it suffices to prove that, for any $x_{*} \in \partial \mathcal{U}_{p}(r)$, there is a $v \in \mathbb{R}^{d} \setminus \{0\}$ such that 
			\begin{align} \label{E: maximum thing ack}
				\langle v, x_{*} \rangle = \max \left\{ \langle v, x \rangle \, \mid \, x \in \overline{\mathcal{U}_{p}(r)} \right\}.
			\end{align}
		Replacing $x_{*}$ above by $\frac{1}{\alpha} x_{*}$ for some $\alpha > 0$ does not change \eqref{E: maximum thing ack} since $\overline{\mathcal{U}_{p}(r)}$ is a cone.\footnote{Indeed, the maximum is zero if it is finite.}  Thus, up to replacing $x_{*}$ by $\frac{1}{\alpha} x_{*}$ for some suitable $\alpha > 0$, it is enough to show that if $q_{*}' \in \partial U_{p}(r)$, then there is a $v \in \mathbb{R}^{d} \setminus \{0\}$ such that 
			\begin{align} \label{E: maximum principle type convex analysis argument}
				\langle v, q_{*}' + g_{\nu}(q_{*}') e \rangle = \max \left\{ \langle v, q' + g_{\nu}(q')e \rangle \, \mid \, q' \in \overline{U_{p}(r)} \right\}.
			\end{align}
			
		Toward that end, fix $q_{*}' \in \partial U_{p}(r)$.  Recall that since $G_{p}$ is convex, the set $U_{p}(r)$ is convex.   Therefore, there is an $n \in \{e\}^{\perp} \setminus \{0\}$ such that 
			\begin{align} \label{E: another optimization ack}
				\langle n, q_{*}' \rangle = \max \left\{ \langle n, q' \rangle \, \mid \,q' \in \partial U_{p}(r) \right\}.
			\end{align}
		In fact, in view of the formula \eqref{E: recall g annoying definition} for $g_{\nu}$ in terms of $\text{dist}(\cdot,G_{p})$, \eqref{E: another optimization ack} holds if and only if $n = - \alpha Dg_{\nu}(q_{*}')$ for some $\alpha > 0$.  Hence, for definiteness, let $\alpha = 1$ and $n = -Dg_{\nu}(q_{*}')$ from now on.
		
		Let $v = n + u e$, where $u$ is chosen such that 
			\begin{equation*}
				0 = \langle v, q_{*}' + g_{\nu}(q_{*}')e \rangle = \langle n, q_{*}' \rangle + u g_{\nu}(q_{*}'), \quad \text{hence} \quad u = - \frac{\langle n, q_{*}'\rangle}{g_{\nu}(q_{*}')}.
			\end{equation*}
		(In view of \eqref{E: key positivity observation}, there is no division by zero in the definition of $u$.)  As a first step towards \eqref{E: maximum principle type convex analysis argument}, define $\ell_{v} : \overline{U_{p}(r)} \to \mathbb{R}$ by $\ell_{v}(q') = \langle v, q' + g_{\nu}(q') e \rangle$ so that $\ell_{v}$ is precisely the function being maximized in \eqref{E: maximum principle type convex analysis argument}.  Observe that, by the definition of $v$, 
			\begin{equation} \label{E: crazy annoying concavity argument ack}
				\ell_{v}(q') = \langle n, q' \rangle + g_{\nu}(q') \langle v, e \rangle = \langle n, q' \rangle + u g_{\nu}(q').
			\end{equation} 
Further, if $q' \in \partial U_{p}(r)$, then the definition of $g_{\nu}$ implies $g_{\nu}(q') = g_{\nu}(q_{*}')$, hence
			\begin{align*}
				\ell_{v}(q') = \langle n, q' \rangle + u g_{\nu}(q') \leq \langle n, q_{*}' \rangle + u g_{\nu}(q_{*}') = \ell_{v}(q_{*}').
			\end{align*}
		This proves 
			\begin{equation*}
				\ell_{v}(q_{*}') = \max \left\{ \ell_{v}(q') \, \mid \, q' \in \partial U_{p}(r) \right\}.
			\end{equation*}
			
		Finally, it remains to show that $\ell_{v}$ achieves its maximum in $\overline{U_{p}(r)}$ on the boundary.  To see this, first note that if $u \leq 0$, then \eqref{E: crazy annoying concavity argument ack} implies that $\ell_{v}$ is convex, and then, by convexity, $\ell_{v}$ necessarily attains its maximum on the boundary.    
		
		On the other hand, if $u > 0$, then $\ell_{v}$ is concave.  To see that $\ell_{v}$ is nonetheless still maximized at $q_{*}'$, fix $q' \in U_{p}(r)$ and consider the function $H_{q'} : [0,1] \to \mathbb{R}$ given by 
	\begin{equation*}
		H_{q'}(t) = \ell_{v}((1 - t) q_{*}' + t q').
	\end{equation*}
The goal now is to show that $H_{q'}$ is maximized at $t = 0$.  Toward that end, observe that since $\ell_{v}$ is concave in $U_{p}(r)$, the function $H_{q'}$ is concave in $[0,1]$.  Thus, since the derivative of a concave function is decreasing, it suffices to show that the derivative of $H_{q'}$ is negative at zero.

Indeed, differentiating $H_{q'}$ at zero leads to
	\begin{align*}
		\frac{d}{dt} \left\{ H_{q'}(t) \right\} \restriction_{t = 0} &= \langle n, q' - q_{*}' \rangle + u \langle Dg_{\nu}(q_{*}'), q' - q_{*}' \rangle = \langle n, q' - q_{*}' \rangle \left(1 - u \right).
	\end{align*}
Observe that, by \eqref{E: key positivity observation}, 
	\begin{align*}
		u = - \frac{\langle n, q_{*}' \rangle}{g_{\nu}(q_{*}')} = \frac{\langle Dg_{\nu}(q_{*}'), q_{*}' \rangle}{g_{\nu}(q_{*}')} < 1
	\end{align*}
so $1 - u > 0$. 
At the same time, by the choice of $n$,
	\begin{align*}
		U_{p}(r) \subseteq \{x' \in \mathbb{R}^{d-1} \, \mid \, \langle n, x' \rangle < \langle n, q_{*}' \rangle \}.
	\end{align*}
In particular, $\langle n, q' - q_{*}' \rangle < 0$ and then the previous computations imply that one has $\frac{d}{dt}\{H_{q'}(t)\} \restriction_{t = 0} < 0$, as claimed. \end{proof}

\subsection{Flatness in the Perturbation Argument} \label{App: flatness} This appendix establishes that, in passing from $\psi_{\nu}$ to the perturbation $\psi_{\nu}^{\epsilon}$, some flatness of the face $\partial \varphi^{*}(p)$ is preserved.  The precise result is stated next.

	\begin{prop} \label{P: flatness prop} Let $\varphi : \mathbb{R}^{d} \to [0,+\infty)$ be a Finsler norm satisfying the $\zeta$-interior ball condition \eqref{E: zeta interior ball} for some $\zeta > 0$.  Given a $p \in \mathbb{R}^{d} \setminus \{0\}$ and a $\nu \in (0,\zeta)$, let $e = \|p\|^{-1} p$; $\psi_{\nu} = \psi_{e,\partial \varphi^{*}(p),\nu}$ be the associated conical test function; and let $(\psi_{\nu}^{\epsilon})_{\epsilon \in (0,\nu)}$ be an admissible perturbation. 
	
	  Given any $\mu > 0$, if $[\partial \varphi^{*}(p)]^{(\mu)}$ is the set of points in $\partial \varphi^{*}(p)$ of distance at least $\mu$ from the (relative) boundary, that is,
	  	\begin{equation*}
			[\partial \varphi^{*}(p)]^{(\mu)} = \{q \in \partial \varphi^{*}(p) \, \mid \, \text{dist}(q,\text{bdry}(\partial \varphi^{*}(p))) \geq \mu\},
		\end{equation*}
	then, for any $\epsilon \in (0,\mu)$, 
		\begin{equation*}
			\|D\psi_{\nu}^{\epsilon}(q)\|^{-1} D\psi_{\nu}^{\epsilon}(q) = e \quad \text{for each} \quad q \in [\partial \varphi^{*}(p)]^{(\mu)}.
		\end{equation*}
	\end{prop}
	
In the proof of Proposition \ref{P: flatness prop}, it will be useful to note that if one defines the set $[G_{p}]^{(\mu)}$ via $[G_{p}]^{(\mu)} = \{q' \in G_{p} \, \mid \, \text{dist}(q',\text{bdry}(G_{p}))\geq\mu\}$, then 
	\begin{equation*}
		[\partial \varphi^{*}(p)]^{(\mu)} = [G_{p}]^{(\mu)} + \varphi^{*}(e) e.
	\end{equation*}
Thus, in view of the identity \eqref{E: classical normal vector formula} relating the normal vector $\|D\psi_{\nu}^{\epsilon}\|^{-1} D\psi_{\nu}^{\epsilon}$ to the gradient $Dg_{\nu}^{\epsilon}$, Proposition \ref{P: flatness prop} follows immediately from the next lemma.

	\begin{lemma} \label{L: flatness lemma} Under the assumptions of Proposition \ref{P: flatness prop}, if $g_{\nu}$ is defined by \eqref{E: graphical special case} and $(g^{\epsilon}_{\nu})_{\epsilon \in (0,\nu)}$ are the mollified functions defined by \eqref{E: mollification}, then 
		\begin{equation*}
			Dg_{\nu}^{\epsilon}(q') = 0 \quad \text{for each} \quad q' \in [G_{p}]^{(\mu)} \quad \text{and} \quad \epsilon < \mu.
		\end{equation*}
	\end{lemma}

To get a sense of why the lemma is true, consider the case when $\partial \varphi^{*}(p)$ is $(d - 1)$-dimensional, hence $G_{p}$ can be regarded as a convex subset of $\mathbb{R}^{d-1}$ with nonempty interior.  In this case, $q' \in [G_{p}]^{(\mu)}$ if and only if $\overline{B}_{\mu}(q') \cap \mathbb{R}^{d-1} \subseteq G_{p}$.  Thus,
	\begin{equation*}
		\text{dist}(\cdot,G_{p}) \equiv 0 \quad \text{in} \, \, \overline{B}_{\mu}(q') \cap \mathbb{R}^{d-1},
	\end{equation*}
which implies that
	\begin{equation*}
		Dg_{\nu}^{\epsilon}(q') = \int_{B_{1}} Dg_{\nu}(q' + \epsilon y') \rho(y') \, dy = 0 \quad \text{if} \, \, \epsilon < \mu.
	\end{equation*}

	\begin{proof}[Proof of Lemma \ref{L: flatness lemma}] Fix $q' \in [G_{p}]^{(\mu)}$ and let $k = \text{dim}(G_{p})$.  After an affine change of coordinates, there is no loss of generality assuming that $q' = 0$ and $G_{p} \subseteq \mathbb{R}^{k} \times \{0\} \subseteq \mathbb{R}^{d-1}$. 

In what follows, for a point $x' \in \mathbb{R}^{d-1}$, write $x' = (\check{x}', \underline{x}')$ with $\check{x}' \in \mathbb{R}^{k} \times \{0\}$ and $\underline{x}' \in \{0\} \times \mathbb{R}^{d-1 - k}$.  

Since $G_{p}$ has dimension $k$, the inequality $\text{dist}(0, \text{bdry}(G)) \geq \mu$ is equivalent to the requirement that the closed $k$-dimensional ball $\overline{B}_{\mu}(0) \cap (\mathbb{R}^{k} \times \{0\})$ is contained in $G_{p}$. Thus, by comparing to $\text{dist}(\cdot,\mathbb{R}^{k} \times \{0\})$, one deduces that
	\begin{equation*}
		\text{dist}(x', G_{p}) = \text{dist}(x',\mathbb{R}^{k} \times \{0\}) = \|\underline{x}'\| \quad \text{for each} \quad x' \in B_{\mu}(0).
	\end{equation*}
After differentiating, this becomes, for $k < d - 1$,
	\begin{equation*}
		D\text{dist}(x',G_{p}) = D\text{dist}(x',\mathbb{R}^{k} \times \{0\}) = \frac{\underline{x}'}{\|\underline{x}'\|}\quad \text{for each} \quad x' \in B_{\mu}(0) \setminus (\mathbb{R}^{k} \times \{0\})
	\end{equation*}
and, for $k = d - 1$,
	\begin{equation*}
		D\text{dist}(x',G_{p}) = 0 \quad \text{for each} \quad x' \in B_{\mu}(0).
	\end{equation*}
Applying these identities in the formulas \eqref{E: graphical special case} and \eqref{E: mollification}, changing variables via the map $T(\check{x}',\underline{x}') = (\check{x}',-\underline{x}')$, and recalling that $\epsilon < \mu$ and \eqref{E: symmetry mollifier} holds, one finds
	\begin{align*}
		Dg_{\nu}^{\epsilon}(0) &= - \epsilon^{-(d-1)} \int_{B_{\epsilon}(0)} \frac{\text{dist}(x',G)}{\sqrt{\nu^{2} - \text{dist}(x',G)^{2}}} D\text{dist}(x',G) \rho(\epsilon^{-1}x') \, dx' \\
					&= - \epsilon^{-(d-1)} \int_{B_{\epsilon}(0)} \frac{\text{dist}(T(x'),\mathbb{R}^{k} \times \{0\})}{\sqrt{\nu^{2} - \text{dist}(T(x'),\mathbb{R}^{k} \times \{0\})^{2}}} D \text{dist}(T(x'),\mathbb{R}^{k} \times \{0\}) \rho(\epsilon^{-1}x') \, dx' \\
					&= \epsilon^{-(d-1)} \int_{B_{\epsilon}(0)} \frac{\text{dist}(x',\mathbb{R}^{k} \times \{0\})}{\sqrt{\nu^{2} - \text{dist}(x',\mathbb{R}^{k} \times \{0\})^{2}}} Dd(x', \mathbb{R}^{k} \times \{0\}) \rho(\epsilon^{-1}x') \, dx' \\
					&= - Dg_{\nu}^{\epsilon}(0).
	\end{align*}
This proves $Dg_{\nu}^{\epsilon}(0) = 0$.  \end{proof}
	
\subsection{Verifying the Smallness Condition}  This subsection proves Proposition \ref{P: smallness condition king}.  The proof proceeds as follows.  First, recall that $M + \psi_{\nu}^{\epsilon} \geq M + \psi_{\nu} \geq u$ in $V$.  Hence one proceeds to lower $\psi_{\nu}^{\epsilon}$, subtracting a constant until it first touches $u$ at some point $x_{\epsilon}$.  That is, $x_{\epsilon}$ is nothing more than a point where $M + \psi_{\nu}^{\epsilon} - C_{\epsilon}$ touches $u$ from above, where $C_{\epsilon}$ is defined by 
	\begin{equation*}
		C_{\epsilon} = \inf \left\{ C > 0 \, \mid \, M + \psi^{\epsilon}_{\nu}(x) - C < u(x) \, \, \text{for some} \, \, x \in \overline{B}_{R}(x_{0}) \right\}.
	\end{equation*}
Let $q_{\epsilon} = \psi_{\nu}^{\epsilon}(x_{\epsilon})^{-1} x_{\epsilon}$ be the associated point on the surface $\{\psi_{\nu}^{\epsilon} = 1\}$.  

To establish that the projection $q_{\epsilon}'$ is, in fact, a distance $O(\epsilon)$ away from $G_{p}$, the strategy is to show that $C_{\epsilon} = O(\epsilon^{2})$, while such a bound would fail if $q_{\epsilon}'$ were too far away.  More precisely, the proof involves the following two lemmas:

	\begin{lemma} \label{L: basic bound} If $x_{*} \in \mathcal{C}(\partial \varphi^{*}(p)) \setminus \{0\}$ and $\nu \in (0,\zeta]$, then there is a constant $c(\nu,x_{*}) > 0$ such that, for all small enough $\epsilon > 0$,
		\begin{equation} \label{E: parabola estimate thing}
			\psi_{\nu}^{\epsilon}(x_{*}) \leq \psi_{\nu}(x_{*})[ 1 + c(\nu,x_{*}) \epsilon^{2}].
		\end{equation}
	\end{lemma}
	
	\begin{prop} \label{P: scarier bound} Fix a $\nu \in (0,\zeta)$.  There are constants $A(\nu,\zeta) > 0$ and $\delta(\nu,\zeta) \in (0,\nu)$ such that, for any $\epsilon \in (0,\nu)$, if $x \in \{\psi^{\epsilon}_{\nu} < +\infty\} \cap \{\psi_{\zeta} < +\infty\}$ and the points $q_{\nu}$ and $q_{\zeta}$ defined by $q _{\nu}= \psi_{\nu}^{\epsilon}(x)^{-1}x$ and $q_{\zeta} = \psi_{\zeta}(x)^{-1} x$ satisfy
		\begin{equation} \label{E: smallness condition necessary assumption}
			\max\left\{ \text{dist}(q'_{\nu},G_{p}), \text{dist}(q'_{\zeta},G_{p}) \right\} \leq \delta(\nu,\zeta),
		\end{equation}
	then
		\begin{equation} \label{E: main thing appendix smallness}
			\psi^{\epsilon}_{\nu}(x) \geq \psi_{\zeta}(x)[1 + A(\nu,\zeta) \text{dist}(q'_{\nu},G_{p})^{2}].
		\end{equation}\end{prop}
	
	Before going into the proofs of Lemma \ref{L: basic bound} and Proposition \ref{P: scarier bound}, here is how these two results combine to yield an estimate on $\text{dist}(q_{\epsilon}',G_{p})$.
	
	\begin{proof}[Proof of Proposition \ref{P: smallness condition king}]  To begin the proof, fix a sequence $(\epsilon_{j})_{j \in \mathbb{N}} \subseteq (0,+\infty)$ such that
		\begin{equation*}
			\lim_{j \to \infty} \epsilon_{j}^{-1} \text{dist}(q_{\epsilon_{j}}',G_{p}) = \limsup_{\epsilon \to 0^{+}} \epsilon^{-1} \text{dist}(q_{\epsilon}',G_{p}), \quad \lim_{j \to \infty} \epsilon_{j} = 0.
		\end{equation*}
		
	Since $\overline{V}$ is compact, there is no loss of generality in assuming that the limit $x_{*} = \lim_{j \to \infty} x_{\epsilon_{j}}$ exists.  In particular, since $u$ is upper semicontinuous, a classical argument shows that
		\begin{equation*}
			u(x_{*}) - \psi_{\nu}(x_{*}) = \max \left\{ u(y) - \psi_{\nu}(y) \, \mid \, y \in \overline{V} \right\}.
		\end{equation*}
	Since $\psi_{\nu} \geq \psi_{\zeta}$ and $\psi_{\nu}(x_{0}) = \psi_{\zeta}(x_{0})$ by Proposition \ref{P: touching above test function part}, 
		\begin{equation*}
			u(x_{*}) = M + \psi_{\zeta}(x_{*}) = M + \psi_{\nu}(x_{*}).
		\end{equation*}  
	Thus, by \eqref{E: stay in the cone appendix}, $x_{*} \in V \cap \mathcal{C}(\partial \varphi^{*}(p))$ and $0 < \psi_{\nu}(x_{*}) < +\infty$.  Further, since $\psi_{\nu}^{\epsilon} \to \psi_{\nu}$ locally uniformly in a neighborhood of $\mathcal{C}(\partial \varphi^{*}(p))$, the limit $q_{*} = \lim_{j \to \infty} q_{\epsilon_{j}}$ exists. Since $x_{*} \in \mathcal{C}(\partial \varphi^{*}(p))$, there holds $q_{*} \in \partial \varphi^{*}(p)$. 
	
	The considerations of the previous paragraph show that $q_{\epsilon_{j}}' \to q_{*}'$ as $j \to +\infty$ and $q'_{*} \in G_{p}$.  In particular, $\text{dist}(q_{\epsilon_{j}}',G_{p}) \to 0$ as $j \to +\infty$, and, up to passing to another subsequence, there is no loss of generality in assuming that
		\begin{equation*}
			\sup \left\{ \text{dist}(q_{\epsilon_{j}}',G_{p}) \, \mid \, j \in \mathbb{N} \right\} \leq \delta,
		\end{equation*}
	where $\delta = \delta(\nu,\zeta) > 0$ is the constant in Proposition \ref{P: scarier bound}.
	
Similarly, if $\tilde{q}_{\epsilon_{j}} = \psi_{\zeta}(x_{\epsilon_{j}})^{-1} x_{\epsilon_{j}}$, then Proposition \ref{P: touching above test function part} implies that
	\begin{equation*}
		\lim_{j \to \infty} \tilde{q}_{\epsilon_{j}} = \lim_{j \to \infty} \frac{x_{\epsilon_{j}}}{\psi_{\zeta}(x_{\epsilon_{j}})} = \frac{x_{*}}{\psi_{\zeta}(x_{*})} = \frac{x_{*}}{\psi_{\nu}(x_{*})} = q_{*} \in \partial \varphi^{*}(p),
	\end{equation*}
In particular, $\text{dist}(\tilde{q}_{\epsilon_{j}}',G_{p}) \to 0$ as $j \to +\infty$, and, thus, there is no loss of generality in assuming once more that
	\begin{equation*}
		\sup \left\{ \text{dist}(\tilde{q}_{\epsilon_{j}}',G_{p}) \, \mid \, j \in \mathbb{N} \right\} \leq \delta.
	\end{equation*}
In particular, $x_{\epsilon_{j}}$ satisfies the assumptions of Proposition \ref{P: scarier bound} for every $j \in \mathbb{N}$. 
	
	On the one hand, $x_{\epsilon_{j}}$ is a point where $M + \psi_{\nu}^{\epsilon_{j}} - u$ is minimized in $\overline{V}$, that is,
		\begin{equation*}
			M + \psi_{\nu}^{\epsilon_{j}}(x_{\epsilon_{j}}) - u(x_{\epsilon_{j}}) = \min \left\{ M + \psi_{\nu}^{\epsilon_{j}}(x) - u(x) \, \mid \, x \in \overline{V} \right\}.
		\end{equation*}
	Using Lemma \ref{L: basic bound} and \eqref{E: hopefully the last time touching above}, this implies
		\begin{equation*}
			M + \psi_{\nu}^{\epsilon_{j}}(x_{\epsilon_{j}}) - u(x_{\epsilon_{j}}) \leq M + \psi_{\nu}^{\epsilon_{j}}(x_{0}) - u(x_{0}) \leq c(\nu,x_{0}) \epsilon^{2}_{j} \psi_{\nu}(x_{0}).
		\end{equation*}
	On the other hand, by Proposition \ref{P: scarier bound} and \eqref{E: hopefully the last time touching above},
		\begin{align*}
			\psi_{\nu}^{\epsilon_{j}}(x_{\epsilon_{j}}) &\geq \psi_{\zeta}(x_{\epsilon_{j}}) \left[1 + A(\nu,\zeta) \text{dist}(q_{\epsilon_{j}}',G)^{2}\right] \\
				&\geq u(x_{\epsilon_{j}}) - M + A(\nu,\zeta) \psi_{\zeta}(x_{\epsilon_{j}}) \text{dist}(q_{\epsilon_{j}}',G)^{2}.
		\end{align*}
	In particular, combining the two estimates leads to the bound
		\begin{equation*}
			 A(\nu,\zeta) \psi_{\zeta}(x_{\epsilon_{j}}) \text{dist}(q_{\epsilon_{j}}',G)^{2} \leq c(\nu,x_{0}) \epsilon_{j}^{2} \psi_{\nu}(x_{0})
		\end{equation*}
	Since $x_{\epsilon_{j}} \to x_{*}$ and $x_{0},x_{*} \in \mathcal{C}(\partial \varphi^{*}(p)) \setminus \{0\}$, Proposition \ref{P: touching above test function part} implies $\psi_{\nu}(x_{0}) = \varphi(x_{0})$ and $\psi_{\zeta}(x_{*}) = \varphi(x_{*}) > 0$, hence this last inequality leads to
		\begin{equation*}
			\limsup_{\epsilon \to 0^{+}} \epsilon^{-1} \text{dist}(q'_{\epsilon},G_{p}) = \lim_{j \to \infty} \epsilon_{j}^{-1} \text{dist}(q_{\epsilon_{j}}',G_{p}) \leq  \sqrt{ \frac{c(\nu,x_{0}) \varphi(x_{0})}{A(\nu,\zeta) \varphi(x_{*})} } < +\infty.
		\end{equation*}
	\end{proof}
	
\begin{proof}[Proof of Lemma \ref{L: basic bound}]  Fix $x_{*} \in \mathcal{C}(\partial \varphi^{*}(p)) \setminus \{0\}$.  Define $q_{*} = \psi_{\nu}(x_{*})^{-1} x_{*}$.  Since $q_{*} \in \mathcal{C}(\partial \varphi^{*}(p)) \cap \{\psi_{\nu} = 1\}$, the inclusion $q_{*} \in \partial \varphi^{*}(p)$ holds and, thus,
	\begin{equation*}
		q_{*} = q_{*}' + g_{\nu}(q_{*}') e = q_{*}' + \varphi^{*}(e) e.
	\end{equation*}
In the proof, there are two cases to consider: (i) $q_{*}' \in \text{int}(G_{p})$ and (ii) $q_{*}' \in \partial G_{p}$.  (Here $\text{int}(G_{p})$ and $\partial G_{p}$ are the interior and boundary of $G_{p}$ relative to the topology of $\mathbb{R}^{d-1}$.)

In case (i), since $q_{*}' \in \text{int}(G_{p})$, there is a $\mu > 0$ such that $\overline{B}_{\mu}(q_{*}') \cap \mathbb{R}^{d-1} \subseteq G_{p}$.  Thus, if $\epsilon < \mu/2$, then $g^{\epsilon}_{\nu}(q') = \varphi^{*}(e)$ for each $q' \in \overline{B}_{\mu/2}(q_{*}') \cap \mathbb{R}^{d-1}$.  In particular, for any $\epsilon < \mu/2$,
	\begin{equation*}
		\psi_{\nu}^{\epsilon}(q_{*}) = 1 = \psi_{\nu}(q_{*}),
	\end{equation*}
which implies $\psi_{\nu}^{\epsilon}(x_{*}) = \psi_{\nu}(x_{*})$ for all $\epsilon \in (0,\mu/2)$, stronger than needed in \eqref{E: parabola estimate thing}.

To complete the proof, it remains to consider case (ii).  In this case, $q_{*}' \in \partial G_{p}$ so the convexity of $G_{p}$ implies the existence of a $p' \in \mathbb{R}^{d-1}$ such that 
	\begin{equation} \label{E: supporting hyperplane finally}
		\langle p', q' \rangle \leq \langle p', q'_{*} \rangle \quad \text{for each} \quad q' \in G_{p}.
	\end{equation}

Considering $g_{\nu}^{\epsilon}$ and Taylor expanding the function $r \mapsto \sqrt{\nu^{2} - r}$ near $r = 0$, one finds, for all sufficiently small $\epsilon$,
	\begin{align*}
		g_{\nu}^{\epsilon}(q_{*}') &= \varphi^{*}(e) - \nu + \int_{B_{1}} \sqrt{\nu^{2} - \text{dist}(q_{*}' + \epsilon y', G_{p})^{2}} \rho(y') \, dy' \\
			&\geq \varphi^{*}(e) - \frac{1}{4} \nu^{-1} \int_{B_{1}} \text{dist}(q_{*}' + \epsilon y',G_{p})^{2} \rho(y') \, dy' \\
			&= \varphi^{*}(e) - \frac{1}{4} \nu^{-1} \epsilon^{2} \int_{B_{1}} \left[ \epsilon^{-1} \text{dist}(q_{*}' + \epsilon y', G_{p}) \right]^{2} \rho(y') \, dy'.
	\end{align*}
At this stage, it is useful to note that 
	\begin{align} \label{E: another blowup argument}
		B(q_{*}') := \liminf_{\epsilon \to 0^{+}} \int_{B_{1}} \left[ \epsilon^{-1} \text{dist}(q_{*}' + \epsilon y', G_{p}) \right]^{2} \rho(y') \, dy' > 0.
	\end{align}
Assuming \eqref{E: another blowup argument} for the moment, one deduces that, for any $\epsilon > 0$ small enough,
	\begin{equation*}
		g_{\nu}^{\epsilon}(q_{*}') \geq \varphi^{*}(e) - \frac{1}{2} \nu^{-1} B(q_{*}') \epsilon^{2}.
	\end{equation*}
Since $q_{*} \in \partial \varphi^{*}(p)$,  the duality identity \eqref{E: subdifferential basic identity dual} implies $\langle q_{*}, e \rangle = \varphi^{*}(e)$, hence 
	\begin{equation*}
		g_{\nu}^{\epsilon}\left( q_{*}' \right) \geq \varphi^{*}(e) \left( 1 - \frac{B(q_{*}') \epsilon^{2}}{2 \nu \varphi^{*}(e)} \right) = \langle q_{*}, e \rangle \left(1 - \frac{B(q_{*}') \epsilon^{2}}{2 \nu \varphi^{*}(e)} \right)
	\end{equation*}
Since $\lim_{\epsilon \to 0^{+}} Dg_{\nu}^{\epsilon}(q_{*}') = Dg_{\nu}(q_{*}') = 0$ by Proposition \ref{P: gradient vanishing ack ack}, it follows that one can fix an $\eta > 0$ such that, for all $\epsilon > 0$ small enough,
	\begin{align*}
		g_{\nu}^{\epsilon} \left( \frac{q_{*}'}{1 + \eta \epsilon^{2}} \right) \geq \frac{\langle q_{*}, e \rangle}{1 + \eta \epsilon^{2}}.
	\end{align*}
In particular, as soon as $\epsilon$ is small enough, one can invoke Proposition \ref{P: crazy monotonicity argument} to deduce that $\psi_{\nu}^{\epsilon}(q_{*}) \leq 1 + \eta \epsilon^{2}$, yielding the desired estimate \eqref{E: parabola estimate thing} by homogeneity.

To check \eqref{E: another blowup argument}, let $H_{p'} = \{x' \in \mathbb{R}^{d-1} \, \mid \, \langle p', x' \rangle \geq 0\}$ and use \eqref{E: supporting hyperplane finally} to obtain
	\begin{equation*}
		\text{dist}(x',G_{p}) \geq \text{dist}(x',q_{*}' - H_{p'}).
	\end{equation*}
By homogeneity and \eqref{E: symmetry mollifier}, this leads to
	\begin{align*}
		B(q_{*}') \geq \lim_{\epsilon \to 0^{+}} \int_{B_{1}} \left[\epsilon^{-1} \text{dist}(\epsilon y', -H_{p'}) \right]^{2} \rho(y') \, dy' = \int_{B_{1}} \text{dist}(y',-H_{p'})^{2} \rho(y') \, dy' > 0.
	\end{align*}
\end{proof}

It only remains to prove Proposition \ref{P: scarier bound}.  Before delving into the proof, it will be convenient to introduce constants $M_{\zeta}(\delta)$ and $R(\delta)$ defined by 
		\begin{align} 
			M_{\zeta}(\delta) &= \sup \left\{ \frac{g_{\zeta}(x') - g_{\zeta}(y')}{\|x' - y'\|} \, \mid \, x', y' \in \{ \text{dist}(\cdot,G) \leq \delta\}, \, \, x' \neq y' \right\}, \label{E: lipschitz constant} \\
			R(\delta) &= \sup \left\{ \|q'\| \, \mid \, \text{dist}(q',G) \leq \delta \right\}. \label{E: radius}
		\end{align}
	
\begin{proof}[Proof of Proposition \ref{P: scarier bound}] Suppose that $x \in \{\psi^{\epsilon}_{\nu} < + \infty\} \cap \{\psi_{\zeta}< + \infty\}$ and assume that \eqref{E: smallness condition necessary assumption} holds for some $\delta > 0$ to be determined below.  As in Section \ref{S: c11 setting}, write $x = \langle x, e \rangle e + x'$, where $x' \in \{e\}^{\perp} \simeq \mathbb{R}^{d-1}$.  

To lighten the notation, define $C, \gamma, \ell > 0$ by
	\begin{equation*}
		C = \psi^{\epsilon}_{\nu}(x), \quad \gamma = \psi^{\epsilon}_{\nu}(x)^{-1} \psi_{\zeta}(x), \quad \ell = \text{dist}\left(\frac{x'}{\psi^{\epsilon}_{\nu}(x)},G_{p}\right).
	\end{equation*}
Note, in particular, that $\psi_{\zeta}(x) = \gamma C$. By Definition \ref{D: admissible}, $\psi^{\epsilon}_{\nu} \geq \psi_{\zeta}$ and, thus, $\gamma \leq 1$.

By Taylor expanding the function $r \mapsto \sqrt{\nu^{2} - r}$ at $r = 0$, we have, for $\delta > 0$ small enough,
	\begin{align} \label{E: definition of delta in the appendix}
		g_{\nu}^{\epsilon}(y') \leq g_{\nu}(y') \leq g_{\zeta}(y') -  \frac{1}{4} (\nu^{-1} - \zeta^{-1}) \text{dist}(y',G_{p})^{2} \quad \text{if} \quad \text{dist}(y',G_{p}) \leq \delta.
	\end{align}
This condition fixes $\delta$ from now on.
	
In view of \eqref{E: smallness condition necessary assumption}, the key property \eqref{E : graph property ack ack ack} of $\psi_{\nu}^{\epsilon}$, and the definition of $C$ and $\ell$, one can write
	\begin{align} \label{E: thankful}
		\frac{x_{d}}{C} = g_{\nu}^{\epsilon} \left( \frac{x'}{C} \right) \leq g_{\zeta} \left(\frac{x'}{C}\right) - \frac{1}{4} (\nu^{-1} - \zeta^{-1}) \ell^{2}.	\end{align}
At the same time, by definition of $\gamma$ and $M_{\zeta}(\delta)$,
	\begin{align*}
			\frac{x_{d}}{\gamma C} = g_{\zeta}\left(\frac{x'}{\gamma C}\right) \geq g_{\zeta} \left(\frac{x'}{C} \right) - M_{\zeta}(\delta) \left\|\frac{x'}{\gamma C} - \frac{x'}{C} \right\| = g_{\zeta}\left(\frac{x'}{C}\right) - M_{\zeta}(\delta) \left(\frac{1}{\gamma} - 1\right) \left\|\frac{x'}{C}\right\|.
	\end{align*}

Combining this with the bound \eqref{E: thankful}, one deduces that
	\begin{align*}
		\frac{x_{d}}{\gamma C} &\geq \frac{x_{d}}{C} + \frac{1}{4} (\nu^{-1} - \zeta^{-1}) \ell^{2} - M_{\zeta}(\delta) \left\| \frac{x'}{C} \right\| \left( \frac{1}{\gamma} - 1 \right).
	\end{align*}
After rearrangement, this becomes
	\begin{align*}
		\frac{1}{4} (\nu^{-1} - \zeta^{-1}) \ell^{2} &\leq \left(\frac{1}{\gamma} - 1 \right) \left(\frac{x_{d}}{C} + M_{\zeta}(\delta) \left\| \frac{x'}{C} \right\| \right) \\
			&= \left(\frac{1}{\gamma} - 1 \right) \left(g_{\nu}^{\epsilon} \left( \frac{x'}{C} \right) + M_{\zeta}(\delta) \left\| \frac{x'}{C} \right\| \right).
	\end{align*}
	
Reinterpreting this last inequality as a lower bound on $\gamma^{-1} - 1$, one arrives at the following lower bound on $\psi^{\epsilon}_{\nu}(x)$:
	\begin{align*}
		\psi^{\epsilon}_{\nu}(x) = C &= \psi_{\zeta}(x) + (1 - \gamma)C \\
			&\geq \psi_{\zeta}(x) + \frac{1}{4} (\nu^{-1} - \zeta^{-1}) \gamma \psi^{\epsilon}_{\nu}(x) \ell^{2} \left( g^{\epsilon}_{\nu}\left( \frac{x'}{C}\right) + M_{\zeta}(\delta) \left\| \frac{x'}{C}\right\|\right)^{-1} \\
			&= \psi_{\zeta}(x) \left[ 1 + \frac{1}{4} (\nu^{-1} - \zeta^{-1}) \ell^{2} \left( g^{\epsilon}_{\nu} \left( \frac{x'}{C} \right) + M_{\zeta}(\delta) \left\| \frac{x'}{C}\right\|\right)^{-1} \right]
	\end{align*}
Finally, observe that $C^{-1} \|x'\| \leq R(\delta)$ and 
	\begin{align*}
		\max \left\{ g_{\nu}^{\epsilon}(q') \, \mid \, \text{dist}(q',G) \leq \delta \right\} \leq \max \left\{ g_{\nu}(q') \, \mid \, \text{dist}(q',G) \leq \delta \right\} &= \varphi^{*}(e),
	\end{align*}
hence \eqref{E: main thing appendix smallness} holds with the constant $A$ given by 
	\begin{align*}
		A= \frac{1}{4} (\nu^{-1} - \zeta^{-1}) \left( \varphi^{*}(e) + M_{\zeta}(\delta) R(\delta) \right)^{-1}.
	\end{align*}
\end{proof}

  \bibliographystyle{plain}
\bibliography{bibliography}

\end{document}